\documentclass[11pt]{report}
\usepackage{amssymb}
\usepackage{amsmath}
\usepackage{amsthm}
\usepackage{calc}
\usepackage[textwidth=16cm, textheight=24cm]{geometry}

\usepackage[english]{babel}
\usepackage[utf8]{inputenc}
\usepackage[T1]{fontenc}
\usepackage{todonotes}
\usepackage{xcolor}
\usepackage[pdfborder={0 0 0}, colorlinks, citecolor=blue]{hyperref}
\hypersetup{
  pdfinfo={
    Author={Joachim Jelisiejew},
    Title={Joachim Jelisiejew -- Hilbert schemes of points and applications},
    CreationDate={D:20170509103200},
    ModDate={D:20170509103200},
    Keywords={},
  }
}

\usepackage{setspace}
\usepackage{booktabs}
\usepackage{tikz}
\usepackage{tikz-cd}
\tikzcdset{>=latex}
\usepackage{etoolbox}
\usepackage{xpatch}

\makeatletter
\xpatchcmd{\part}{\null\vfil}{\vspace*{.3\textheight}}{}{}

\newenvironment{partwithabstract}
  {\begingroup\let\@endpart\relax\part@withabstract}
  {\endquotation\endgroup\@endpart}
\newcommand{\part@withabstract}{\@dblarg\part@@withabstract}
\def\part@@withabstract[#1]#2{%
    \phantom{m}
  \part[#1]{#2}%
  \thispagestyle{empty}
  \begin{center}\bfseries\vspace{-.5em}\vspace{\z@}\end{center}
  \quotation
}
\makeatother

\numberwithin{equation}{chapter}
\newtheorem{theorem}[equation]{Theorem}
\newtheorem{corollary}[equation]{Corollary}
\newtheorem{lemma}[equation]{Lemma}

\newtheorem{proposition}[equation]{Proposition}
\newtheorem{claim}[equation]{Claim}
\newtheorem{assumption}[equation]{Assumption}
\theoremstyle{definition}
\newtheorem{definition}[equation]{Definition}
\newtheorem{remark}[equation]{Remark}
\newtheorem{example}[equation]{Example}
\newtheorem{problem}[equation]{Problem}

\DeclareMathOperator{\Spec}{Spec}
\DeclareMathOperator{\Proj}{Proj}

\DeclareMathOperator{\Supp}{Supp}

\DeclareMathOperator{\Ext}{Ext}
\DeclareMathOperator{\Hom}{Hom}
\DeclareMathOperator{\gr}{gr}
\DeclareMathOperator{\soc}{soc}
\DeclareMathOperator{\rank}{rank}
\DeclareMathOperator{\Ann}{Ann}
\DeclareMathOperator{\tmpAdd}{Add}
\DeclareMathOperator{\linearOp}{Lin}
\DeclareMathOperator{\Grass}{Gr}
\DeclareMathOperator{\Sym}{Sym}
\DeclareMathOperator{\Sing}{Sing}
\DeclareMathOperator{\inv}{inv}
\DeclareMathOperator{\Cat}{Cat}
\DeclareMathOperator{\Diff}{Diff}

\newcommand{\tensor}{\otimes}
\newcommand{\onto}{\twoheadrightarrow}
\newcommand{\into}{\hookrightarrow}
\newcommand{\hook}{\,\lrcorner\,}
\newcommand{\kdp}{\kk_{dp}}

\renewcommand{\geq}{\geqslant}%
\renewcommand{\leq}{\leqslant}%

\def\DPut#1#2{\global\csdef{D#1}{#2}#2}%
\def\DDef#1#2{\global\csdef{D#1}{#2}}%

\definecolor{labelblue}{RGB}{0, 146, 223}

\renewcommand{\AA}{\mathbb{A}}%
\newcommand{\PP}{\mathbb{P}}%
\newcommand{\UU}{\mathcal{U}}%
\newcommand{\RR}{\mathbb{R}}%
\newcommand{\OO}{\mathcal{O}}%
\newcommand{\CC}{\mathbb{C}}%
\newcommand{\KK}{\mathbb{K}}%
\newcommand{\DA}{A}%
\newcommand{\socleA}{\soc A}%
\newcommand{\grDA}{\gr A}%
\newcommand{\mm}{\mathfrak{m}}%
\newcommand{\kk}{\Bbbk}%
\newcommand{\chark}{\operatorname{char} \kk}%
\newcommand{\kkbar}{\overline{\kk}}%
\newcommand{\dimk}{\dim_{\kk}}%
\newcommand{\DS}{S}%
\newcommand{\DShat}{\hat{\DS}}%
\newcommand{\DP}{P}%
\newcommand{\canA}{\omega_A}%
\newcommand{\Dx}{\alpha}%
\newcommand{\Dy}{\beta}%
\newcommand{\Dz}{\gamma}%

\newcommand{\pp}[1]{\left( #1 \right)}%
\newcommand{\reduced}[1]{\pp{#1}_{red}}%
\newcommand{\Apolar}[1]{\operatorname{Apolar}\left(#1\right)}%
\newcommand{\sspan}[1]{\left\langle #1\right\rangle}%
\newcommand{\Speccal}{\mathcal{S}pec}%
\newcommand{\Dgrp}{\mathbb{G}}%

\newcommand{\floor}[1]{\left\lfloor #1 \right\rfloor}%
\newcommand{\ccX}{\mathcal{Z}}%

\AtBeginDocument{\addtocontents{toc}{\protect\thispagestyle{empty}}}
\begin{document}
\pagenumbering{roman}
\pagestyle{empty}
%--------------------------------------------------------------------------------
%--------------------------------------------------------------------------------

\vfill
\vspace*{6cm}

\begin{center}
    \Large\bfseries Long, long ago\footnote{That is, five years}, in a
    galaxy\footnote{Metaphorically speaking} far, far away\footnote{In terms
    of scientific development $\ddot\smile$}...
\end{center}
    The following is the author's PhD Thesis, defended in September 2017. It
    has not been edited since (except for the footnotes in the open problems section
    and for this note). In particular, there is nothing here about
    applications of Bia{\l}ynicki-Birula decomposition to Hilbert
    schemes~\cite{Jelisiejew__Elementary, Jelisiejew__Pathologies, Satriano_Staal}, the
    multigraded Hilbert schemes and border
    apolarity~\cite{Buczyska_Buczynski__border, Mandziuk}, fiber-full Hilbert
    schemes~\cite{CidRuiz_Ramkumar__Fiber_full,
    CidRuiz_Ramkumar__Fiber_full_local} and other recent developments.
    Philosophically speaking, is it reassuring that there are too many
    new interesting papers on Hilbert schemes of points to summarize here.
    It is also very interesting that many of the open problems listed below
    in~\S\ref{ssec:openquestions} remain actual.
    Today, the author would add also the more applied open problems
    (see~\cite{Jelisiejew_Landsberg_Pal__Concise_tensors_of_minimal_brank}): a large part of the classical problem of determining the
    complexity of matrix multiplication can be reduced fruitfully to the study
    of the Hilbert scheme of points.

    While the author is not aware of any errors, it is likely they exist here.
    This work is put on arXiv due/thanks to several people who decided to
    refer to it and only to grant this piece a quasi-permanent place to rest.
\vfill
\newpage
%--------------------------------------------------------------------------------
\vspace*{2cm}
\begin{center}

{\Large University of Warsaw}

\vspace{0.2cm}

{\large Faculty of Mathematics, Informatics and Mechanics}

\vspace{3cm}

{\large\textbf{Joachim Jelisiejew}}

\vspace{2cm}

{\Large\bfseries Hilbert schemes of points and their applications}

\vspace{0.2cm}
\emph{PhD dissertation}

\vspace{1cm}
\end{center}

{\raggedleft\vfill{%
Supervisor

\textbf{dr hab. Jaros{\l{}}aw Buczy{\'n}ski}

Institute of Mathematics

University of Warsaw

and Institute of Mathematics

Polish Academy of Sciences

\vspace{0.5cm}

Auxiliary Supervisor\\
\textbf{dr Weronika Buczy{\'n}ska}\\
Institute of Mathematics\\
University of Warsaw
}\par}

\thispagestyle{empty}
%--------------------------------------------------------------------------------
%--------------------------------------------------------------------------------
\newpage

\pagestyle{empty}

%--------------------------------------------------------------------------------

\phantom{m}
\vfill

\begin{flushleft}
Author’s declaration:

 I hereby declare that I have written this dissertation myself
 and all the contents of the dissertation have been obtained by legal means.

\vspace{0.5cm}

\hfill  May 9, 2017 \hspace*{2em}...................................................\\
 {\small \it \hfill  Joachim Jelisiejew}

\vspace{2cm}

Supervisors' declaration:

The dissertation is ready to be reviewed.

\vspace{0.5cm}

\hfill  May 9, 2017 \hspace*{2em}...................................................\\
{\small \it \hfill dr hab. Jaros{\l{}}aw Buczy{\'n}ski}

\vspace{1.5cm}

\hfill May 9, 2017\hspace*{2em}...................................................\\
\hfill {\small \it dr Weronika Buczy{\'n}ska}
\end{flushleft}

\vfill

\thispagestyle{empty}
%--------------------------------------------------------------------------------
%--------------------------------------------------------------------------------
\newpage
\thispagestyle{empty}
\phantom{m}

\vfill

{ {\LARGE \bf Abstract}
 \vspace{1cm}

 This thesis is concerned with deformation theory of finite subschemes of
 smooth varieties. Of central interest are the smoothable subschemes (i.e., limits of
 smooth subschemes). We prove that all Gorenstein subschemes of degree up to
 $13$ are smoothable. This result has immediate applications to finding equations of
 secant varieties.
 We also give a description of nonsmoothable Gorenstein subschemes
 of degree $14$, together with an explicit condition for smoothability.

 We prove that being smoothable is a local property, that it does not depend
 on the embedding and it is invariant under a base field extension. The above
 results are equivalently stated in terms of the Hilbert scheme of points,
 which is the moduli space for this deformation problem.

 We extensively use the combinatorial framework of Macaulay's inverse systems.
 We enrich it with a pro-algebraic group action and use this to
 reprove and extend recent classification results by Elias and Rossi.
 We provide a relative version of this framework and use it to give a local
 description of the universal family over the Hilbert scheme of points.

 We shortly discuss history of Hilbert schemes of points and provide a list of open
 questions.

 \vspace{0.5cm}
 {\textbf{Keywords:} deformation theory, Hilbert scheme, Gorenstein algebra,
 inverse system, apolarity, smoothability, classification of finite
 commutative algebras.

 }
 %% 13H10 Gorenstein, 13N10 MIS

\vfill

%--------------------------------------------------------------------------------
%--------------------------------------------------------------------------------
\newpage

\thispagestyle{empty}
\tableofcontents
\thispagestyle{empty}

\chapter{Introduction}
\pagenumbering{arabic}
\pagestyle{plain}

        The Hilbert scheme of points on a smooth variety $X$ is of
    central interest for several branches of mathematics:
    \begin{itemize}
        \item in commutative algebra, it is a moduli space of finite
            algebras (presented as quotients of a fixed ring),
        \item in geometry and topology, it is a compact variety containing the space
            of tuples of points on $X$; in many cases this
            space is dense,
        \item in algebraic geometry, its construction (1960-61) is one of the
            advances of the Grothendieck school~\cite{Gro}, it found
            applications in constructing other moduli spaces and hyperk\"ahler
            manifolds, and also in McKay correspondence
            and theory of higher secant varieties,
        \item in combinatorics, the Hilbert
            scheme appears in Haiman's proofs of $n!$ and Macdonald positivity
            conjectures.
    \end{itemize}
    Consider the set of finite algebras, presented as quotients of a fixed polynomial
    ring. There is a \emph{unique} and \emph{natural} topological space structure on
    this set, together with a sheaf of regular functions. These structures
    jointly give a scheme structure, called the \emph{Hilbert scheme of points on affine space}, see
    Chapter~\ref{ssec:hilbertschemes} for precise definition.
%    These structures are \emph{unique} and \emph{natural}, for example invariant under
%    coordinate changes.
%    By definition, the Hilbert scheme of points is the \emph{unique natural
%    scheme structure} on the set of finite quotients of a fixed ring
%    (equivalently, finite subschemes of a given scheme,
%    see~).
    These structures are unique, but they are non-explicit and difficult to
    investigate; many open questions persist, despite continuous research, see
    Section~\ref{ssec:openquestions}.

%    Amusingly, the Hilbert scheme, a
%    development of the highly abstract French school which usually avoided
%    concentrating on a single object, is now studied for its own sake,
%    frequently using classical Italian school tools introduced well before
%    its construction.

    In this thesis we analyse the geometry of Hilbert schemes of points on
    smooth varieties, concentrating on the following question:
    \begin{center}
        \emph{What are the irreducible components of the Hilbert scheme of points?
        What are their intersections and singularities?}

        \smallskip
        An informal, intuitive view of geometry of the components is given on Figure~\ref{fig:bellis}
        below.
    \end{center}
%    describing its irreducible components, their intersections and singularities.
%    local properties, such as
%    irreducibility (the scheme is always connected) or singularities.
%    Therefore we usually consider Hilbert schemes of affine spaces, instead of
%    projective spaces; mathematically it makes no difference.
%    Among our results are \textbf{list: 14 points, local description}

    Our analysis of the Hilbert scheme as a moduli space of finite algebras
    requires tools for working with algebras themselves, which are developed
    in Part~\ref{part:algebras}.
    We then switch our attention to families of algebras (subschemes) in
    Part~\ref{part:families} and analyse Hilbert
    schemes for small numbers of points in Part~\ref{part:applications}.
    Almost all of our original results presented in this thesis
    are also found in~\cite{jelisiejew_1551, cjn13, jelisiejew_VSP,
    jabu_jelisiejew_smoothability, Jel_classifying, Michalek}.

%%%%%%%%%%%%%%%%%%%%%%%%%%%%%

\newpage
\begin{figure}[t]\label{fig:bellis}
    \centering
    \includegraphics[width=0.75\textwidth]{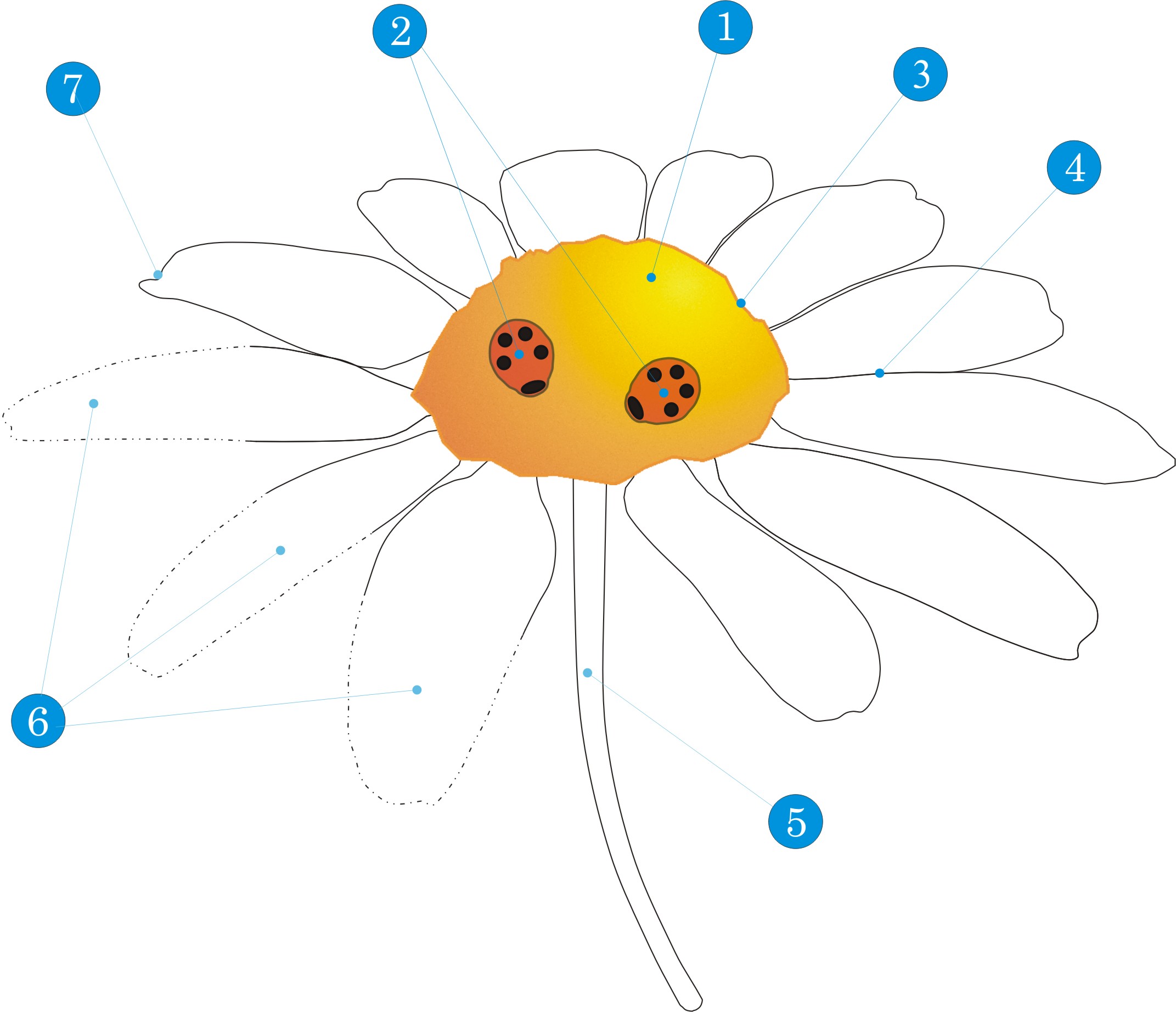}
    \caption{\emph{Bellis Hilbertis.} Components of the Hilbert scheme of $r$
    points on a smooth variety $X$.  Flower, petals and stem correspond to
    irreducible components of the Hilbert scheme of $X$. The analysis of their
    geometry is the main aim of this work.}
\end{figure}%

\newlength{\raising}%
\setlength{\raising}{0.4em}%
\newcommand{\myitem}[2]{
    \item[{\begin{tikzpicture}
            \filldraw[fill=labelblue, draw=labelblue!50!black] (0, 0) circle (0.6em);
            \node[color=white] (0, 0) {\large\textbf{#1}};
    \end{tikzpicture}}]\leavevmode\vadjust{\vspace{-1.7em}}\newline
{#2}}%

%\vspace*{-2.3em}%
\begin{enumerate}
        \myitem{1}{The \emph{smoothable component}, see
        Definition~\ref{ref:smoothablecomponent:def}. It compactifies
        the space of $r$-tuples of points on $X$, thus has dimension $r(\dim
        X)$.
        For
        small $r$, it is the only component -- there are no ``petals'' --
        see Section~\ref{ssec:examplesnonsmoothable},
        Theorem~\ref{ref:cjnmain:thm} and also
        Problems~\ref{prob:irreducibilityA3},~\ref{prob:irreducibilityA4Gor}.}
    \myitem{2}{Points of the Hilbert scheme correspond to finite subschemes of $X$.
        General points on the smoothable component (``ladybirds'') correspond to tuples of
        points on $X$. Thus all subschemes corresponding to points on this component are \emph{limits} of
        tuples of points; they are \emph{smoothable}.}
    \myitem{3}{
            Every component intersects the smoothable one,
            see~\cite{Reeves_radius}. Describing the
            intersection is subtle, see Problem~\ref{prob:iarrobino} and Theorem~\ref{ref:14pointsmain:thm} for
            an example involving cubic fourfolds.  It is even hard to decide
            whether a given subscheme of $X$ is smoothable, see
            Problem~\ref{prob:grobnerfan}.
        }
    \myitem{4}{
            There is a single example~\cite{erman_velasco_syzygetic_smoothability}, where
            components intersect away from the smoothable component. No
            singular points lying on a unique nonsmoothable component are
            known.
        }
    \myitem{5}{
            There are only few known components of dimension smaller than
            the smoothable one, see Section~\ref{ssec:examplesnonsmoothable}
            and
            Problems~\ref{prob:lowerbounddimension},~\ref{prob:smallcomponents}.
    }
    \myitem{6}{
        There are many examples of loci too large to fit inside the
        smoothable component, see Section~\ref{ssec:examplesnonsmoothable}, however the components
        containing these loci are not known.
    }
    \myitem{7}{
        Not much is known about the geometry and singularities of components
        other than the smoothable one, see Problem~\ref{prob:reducedness},
        Problem~\ref{prob:nonrational}.}
\end{enumerate}
\goodbreak

%%%%%%%%%%%%%%%%%%%%%%%%%%%%%

\section{Overview and main results}\label{ssec:intromainresults}

%    Throughout this thesis we strive to avoid technical language and to make
%    this text accessible, in particular for algebraists outside algebraic
%    geometry.

\newcommand{\Hilbfunc}{\mathcal{H}\hspace{-0.25ex}\mathit{ilb\/}}%
\newcommand{\Hilbarged}[2]{\Hilbfunc{}_{#1}\,(#2)}%
\newcommand{\ambient}{X}%
\newcommand{\Hilb}[1]{\Hilbfunc{}_{pts}\,(#1)}%
\newcommand{\Hilbr}[1]{\Hilbfunc{}_{r}\,(#1)}%
\newcommand{\Hilbsm}[1]{\Hilbfunc^{sm}\,(#1)}%
\newcommand{\Hilbsmarged}[2]{\Hilbfunc{}_{#1}^{sm}\,(#2)}%
\newcommand{\Hilbzeror}[1]{\Hilbfunc_r^{\circ}\,(#1)}%
\newcommand{\Hilbsmr}[1]{\Hilbfunc_{r}^{sm}\,(#1)}%
\newcommand{\HilbGorarged}[2]{\Hilbfunc{}^{Gor}_{#1}\,(#2)}%
\newcommand{\HilbGorsmarged}[2]{\Hilbfunc{}^{Gor, sm}_{#1}\,(#2)}%
\newcommand{\HilbGorr}[1]{\Hilbfunc_r^{Gor}\,(#1)}%
\newcommand{\HilbGorsmr}[1]{\Hilbfunc_{r}^{Gor, sm}\,(#1)}%
\newcommand{\HilbGorpuncarged}[2]{\Hilbfunc{}P^{Gor}_{#1}\,(#2)}%
\newcommand{\HilbGorpp}[1]{\Hilbfunc{}P^{Gor}_{r}\,(#1, p)}%
\newcommand{\HilbGorppleq}[1]{\Hilbfunc{}P^{Gor}_{\leq r}\,(#1, p)}%

    Part~\ref{part:algebras} gathers tools for studying finite local algebras
    over a field $\kk$,
    especially Gorenstein algebras. In this part, we speak the language of
    algebra; only basic background on commutative algebras and Lie theory is
    assumed. Part~\ref{part:algebras} is mostly prerequisite, although
    Sections~\ref{ssec:standardforms}-\ref{ssec:examplesthree} contain many
    original
    results published in~\cite{BJMR, Jel_classifying}.

    Theory of Macaulay's inverse systems (also known as apolarity) is a central tool in
    our investigation. To an element $f$ of a (divided
    power) polynomial ring $\DP$ it assigns a finite local Gorenstein algebra $\Apolar{f}$,
    see Section~\ref{ssec:Gorensteininapolarity}.
    A pro-algebraic group $\Dgrp$ acts on $\DP$ so that the
    orbits are isomorphism classes of $\Apolar{-}$, see
    Section~\ref{ssec:Gorensteininapolarity}. In
    Section~\ref{ssec:specialforms} we explicitly describe the action of
    $\Dgrp$ and investigate its Lie group, building an effective tool to
    investigate isomorphism classes of algebras. We then give applications; as a
    sample result, we reprove (with weaker assumptions on the base field) the
    following main theorems
    of~\cite{EliasRossiShortGorenstein, EliasRossi_Analytic_isomorphisms}.
    \begin{theorem}[Example~\ref{ex:1nn1}, Corollary~\ref{ref:eliasrossi:cor}]\label{ref:introeliasrossi:thm}%
        Let $\kk$ be a field of characteristic $\neq 2, 3$.
        Let $(\DA, \mm, \kk)$ be a finite Gorenstein local $\kk$-algebra with
        Hilbert function $(1, n, n, 1)$ or $(1, n, \binom{n+1}{2}, n, 1)$.
        Then $\DA$ isomorphic to its associated graded algebra
        $\grDA$.
    \end{theorem}
    Interestingly, Theorem~\ref{ref:introeliasrossi:thm} fails for small characteristics, see
    Example~\ref{ex:1nn1char2}.
    We also obtain genuine, down to earth classification results, such as the
    following.
    \begin{proposition}[Example~\ref{ex:13331}]
        Let $\kk$ be an algebraically closed field of characteristic $\neq 2, 3$.
        There are exactly eleven isomorphism types of finite local Gorenstein
        algebras with Hilbert function $(1, 3, 3, 3, 1)$, see
        Example~\ref{ex:13331} for their list.
    \end{proposition}

    In Part~\ref{part:families}, our attention shifts towards families of
    algebras.  Accordingly, we change the language from algebra to
    algebraic geometry, from finite algebras to finite schemes.  The main
    object is the \emph{Hilbert scheme of points} on
    a scheme $X$, denoted by $\Hilbr{X}$, together with its
    degree $r$ finite flat \emph{universal family}
    \[
        \pi\colon \UU\to \Hilbr{X}.
    \]
    Intuitively, points of
    $\Hilbr{\AA^n}$ parameterize all finite algebras and the fiber of $\pi$
    over a point is the corresponding algebra;
    see~Section~\ref{ssec:hilbertschemes} for precise definition and
    discussion. Inside the Hilbert scheme of points, we have its
    \emph{Gorenstein locus} $\HilbGorr{\AA^n}$, which is the family of all finite
    \emph{Gorenstein} algebras. We have the restriction
    $\pi_{|\HilbGorr{\AA^n}}:\UU^{Gor} \to
    \HilbGorr{\AA^n}$, which we usually denote simply by $\pi$.

    We make the $\Apolar{-}$ construction relative in
    Section~\ref{ssec:relativemacaulaysystems},
    following~\cite{jelisiejew_VSP}, and prove that every family
    locally comes from this construction; this gives a very satisfactory local
    theory of the Hilbert scheme. In particular, for the Gorenstein locus, we
    obtain the following result.
    \begin{proposition}[Corollary~\ref{ref:localdescriptionGorenstein:cor}]
        Locally on $\HilbGorr{\AA^n}$, the universal family has the form
        $\Apolar{f} \to \Spec A$ for an $f\in A\tensor_{\kk} \DP$.
    \end{proposition}

    The Hilbert scheme of points on $X$ has a distinguished \emph{open} subset
    $\Hilbzeror{X}$, consisting of \emph{smooth} subschemes. Its closure is
    called the \emph{smoothable component} and denoted $\Hilbsmr{X}$, see
    Definition~\ref{ref:smoothablecomponent:def}. Tuples of points on
    smooth $X$ are smooth subschemes and in fact $\Hilbzeror{X}$ is naturally
    the space of tuples of points on $X$, see
    Lemma~\ref{ref:smoothpartdescription:lem}. Thus for proper $X$ the component
    $\Hilbsmr{X}$ is a compactification of the space of (unordered) tuples of
    points, also called the \emph{configuration space}.
    We provide examples of points in and outside $\Hilbsmr{X}$ in
    Sections~\ref{ssec:examplesnonsmoothable}-\ref{ssec:klimits}.
    Schemes $R = \Spec A$ corresponding to points of
    $\Hilbsmr{X}$ are called \emph{smoothable}. For $\kk =
    \kkbar$, i.e., over~an algebraically closed field, they correspond precisely to algebras $A$ which are limits of $\kk^{\times r}$.
    In Chapter~\ref{sec:smoothings} we investigate those limits, following~\cite{jabu_jelisiejew_smoothability}.
    They can be taken abstractly (see
    Definition~\ref{ref:abstractsmoothable:def}) or embedded into $X$, i.e.,
    in $\Hilbr{X}$. The dependence on $X$ is a bit artificial, fortunately it
    is superficial for smooth $X$, as the following shows.
    \begin{theorem}[Theorem~\ref{thm_equivalence_of_abstract_and_embedded_smoothings}]\label{intro_thm_equivalence_of_abstract_and_embedded_smoothings}
        Suppose $X$ is a smooth variety over a field $\kk$ and $R\subset X$ is
        a finite $\kk$-subscheme.  The following conditions are equivalent:
        \begin{enumerate}
            \item $R$ is abstractly smoothable,
            \item $R$ is embedded smoothable in $X$,
            \item every connected component of $R$ is abstractly smoothable,
            \item every connected component of $R$ is embedded smoothable in $X$.
        \end{enumerate}
    \end{theorem}

    In general, for non-smooth $X$, embedded smoothability of $R \subset X$
    depends purely on the local geometry of $X$ around the support of $R$,
    provided that $X$ is separated, see
    Proposition~\ref{prop_smoothability_depends_only_on_sing_type}.
%    Let $\hat{\OO}_{X, x}$ denote
%    the completion of the ring of local functions around $x\in X$.
%    \begin{proposition}[Proposition~\ref{prop_smoothability_depends_only_on_sing_type}]
%        \label{intro_prop_smoothability_depends_only_on_sing_type}%
%        Let $X$ be a separated scheme and $R \subset X$ be a finite scheme,
%        supported at points $x_1, \ldots ,x_k$ of $X$. Then $R$ is smoothable
%        in $X$ if and only if $R$ is smoothable in $\bigsqcup \Spec
%        \hat{\OO}_{X, x_i}$.
%    \end{proposition}

    Part~\ref{part:applications} applies previously developed machinery to
    investigate, for fixed $n$ and $r$, the question of irreducibility
    of the Gorenstein locus. It can be reformulated in the following
    equivalent ways:
    \begin{enumerate}
        \item Consider Gorenstein $\kkbar$-algebras of degree $r$ and embedding dimension at most $n$.
            Are they all limits of $\kkbar^{\times r}$?
        \item Is the Gorenstein locus $\HilbGorr{\AA^n}$ contained in
            $\Hilbsmr{\AA^n}$?
    \end{enumerate}
    In the following theorem, we answer these questions positively for small $r$. This has immediate
    applications for secant varieties, see Section~\ref{ssec:introsecants}.
    \begin{theorem}[Theorem~\ref{ref:cjnmain:thm}]
        Let $\kk$ be a field and $\chark \neq 2, 3$.
        Let $R$ be a finite Gorenstein scheme of degree at most $14$.  Then
        either $R$ is smoothable or it corresponds to a local algebra $(\DA,
        \mm, \kk)$ with $H_A = (1, 6, 6, 1)$. In particular, if $R$ has degree
        at most $13$, then $R$ is smoothable.
        \label{ref:cjnmainintro:thm}
    \end{theorem}
    \newcommand{\introHilbshortother}{\mathcal{H}_{1661}}%
    Although $r\leq 14$ might seem severely restrictive, the result above is
    thought of as a partial classification of algebras up to degree $14$, which
    is quite complex. In the proof (see Chapter~\ref{sec:Gorensteinloci}), we
    avoid most of the classifying work by carefully dividing algebras into several groups
    according to their Hilbert functions and ruling out several distinguished
    classes (e.g., Corollary~\ref{ref:squareshavedegenerations:cor}).

    The nonsmoothable Gorenstein schemes of degree $14$ form a component
    $\introHilbshortother \subset \HilbGorarged{14}{\AA^6}$.
    Such components are of interest, because few are described and because they
    arise ``naturally''; for example they are $SL_6$-invariant.
    The next theorem gives a full description of the component
    $\introHilbshortother$ and, even more
    importantly, of the intersection of $\introHilbshortother$ with the
    smoothable component; see the introduction of Chapter~\ref{sec:Gorensteinloci} for
    details. The most striking result is that the intersection is given by an
    object tightly connected with the theory of cubic fourfolds: the
    Iliev-Ranestad divisor. This is the unique $SL_6$-invariant divisor of
    degree $10$ on $\mathbb{P}(\Sym^3 \kk^6)$, see
    Section~\ref{ssec:14pointsproof} for an explanation.
    \begin{theorem}[Theorem~\ref{ref:14pointsmain:thm}]
        \label{ref:intro14degree:thm}%
        Let $\kk$ be a field of characteristic zero.
        The component $\introHilbshortother$ is a rank $21$ vector bundle over
        an open subset of the space of cubic fourfolds ($\simeq
        \mathbb{P}(\Sym^3 \kk^6)$). In particular, $\dim \introHilbshortother
        = 76$. The intersection of $\introHilbshortother$ with the smoothable
        component is the preimage of the \emph{Iliev-Ranestad} divisor  in
        this space.
    \end{theorem}
    The Gorenstein assumption is very important for the proofs, since
    $\Apolar{-}$ construction is heavily applied. However, the methods can be applied also
    for non-Gorenstein schemes. For example, using the methods similar to the proof of
    Theorem~\ref{ref:cjnmainintro:thm}, one shows that
    $\Hilbarged{11}{\AA^3} = \Hilbsmarged{11}{\AA^3}$, see~\cite{DJNT} and
    the discussion in Section~\ref{ssec:examplesnonsmoothable}.

    \newcommand{\intropunc}{\ccX}%
    We also consider schemes supported at a point. Fix the origin $p\in \AA^n$
    and denote
    \[
        \HilbGorpp{\AA^n} = \left\{ [R] \in \HilbGorr{\AA^n}\ |\ \Supp(R) = \{p\}\right\}
    \]
    In the literature, this is called the \emph{Gorenstein locus of the
    punctual Hilbert scheme}. Inside, the family of \emph{curvilinear}
    schemes, those isomorphic to $\Spec\kk[x]/x^r$, forms a component of dimension
    $(n-1)(r-1)$. The following result is crucial for applications to
    constructing $r$-regular maps, see Section~\ref{ssec:introkregularity}. We
    do not know, whether it holds in all characteristics, but we expect so.
    \begin{theorem}[Theorem~\ref{ref:expecteddim:thm}]
        \label{ref:introsmallpunc:thm}%
        Let $r\leq 9$ and $\chark = 0$. Then
        \[
            \dim \HilbGorpp{\AA^n} = (r-1)(n-1).
        \]
    \end{theorem}
    Note that we do not claim that $\HilbGorpp{\AA^n}$ is irreducible, we
    only compute its dimension.

    In the following sections we present two applications of our results, a
    brief historical survey and a list of open problems.

\section{Application to secant varieties}\label{ssec:introsecants}

    \newcommand{\intropoly}{\CC[x_0, \ldots ,x_n]}%
    \newcommand{\intropolyd}{\CC[x_0, \ldots ,x_n]_d}%
    \newcommand{\intropolyshortd}{\DP_d}%
    \newcommand{\PPintropolyshortd}{\PP(\DP_d)}%
    \newcommand{\introlinear}{\sspan{x_0, \ldots ,x_n}}%
    \newcommand{\introsecant}[2]{\sigma_{#1}(\nu_{#2}(\PP^n))}%
    Consider homogeneous forms of degree $d$ in $n+1$ variables, i.e., elements
    of $\intropolyshortd := \intropolyd$. A classical \emph{Waring problem} for forms
    (e.g.~\cite{SylvesterWaring, vaughan_wooley_Warings_problem_survey,
    Landsberg__tensors}) asks, for a given form $F$, what is the minimal
    number $r$, such that
    \[
        F = \ell_1^d +  \ldots + \ell_r^d
    \]
    for some linear forms $\ell_i$. This is tightly related to finding
    polynomial functions $\intropolyshortd \to \CC$, vanishing of the subset
    \[
        \sigma^{\circ}_r := \left\{ \sum_{i=1}^r \ell_i^d\ |\ \ell_i\in
        \introlinear
    \right\} \subset \intropolyshortd.
    \]
    The Euclidean closure of $\sigma^\circ$ in $\intropolyshortd$ is an algebraic
    variety, the cone over the \emph{$r$-th secant variety
        $\introsecant{r}{d}$ of the $d$-th Veronese
    reembedding}.
    The problem of finding polynomial equations of $\introsecant{r}{d}$ is long
    studied and important for applications, see references
    in~\cite{Landsberg__tensors}.

    However, this problem is difficult, because $\introsecant{r}{d}$ is
    parameterized by tuples of $r$ points of $\CC^{n+1} = \introlinear$, which
    are difficult to describe in terms of equations.  A remedy for this, introduced
    in~\cite{bubu2010}, is to replace tuples of $r$ points by finite
    subschemes of degree $r$.

    First, we have a Veronese reembedding $\nu_d\colon\PP\introlinear\to
    \PPintropolyshortd$, which maps a form $[F]\in \PP\introlinear$
    to $[F^d]\in \PPintropolyshortd$.
    For a subscheme $X \subset \PPintropolyshortd$ by $\sspan{X}$ we denote the
    projective subspace spanned by $X$. In this language, the secant variety
    $\introsecant{r}{d}$ is the closure of
    \[
        \left\{ \sspan{\nu_d(\ell_1), \ldots ,\nu_d(\ell_r)}\ |\ \ell_i\in
            \introlinear \right\} \subset \PPintropolyshortd.
    \]
    A tuple $\{\ell_i\}$ is just a tuple of $r$ points of $\introlinear$; it is
    a smooth degree $r$ subscheme of $\PP\introlinear$. In~\cite{bubu2010}
    Buczy\'nska and Buczy\'nski introduced the
    \emph{cactus variety}, defined as the closure of \emph{all} degree $r$ subschemes
    of $\PP\introlinear$:
    \[
        \left\{ \sspan{\nu_d(R)}\ |\ R \mbox{ degree } r \mbox{ subscheme of
        }{\PP\introlinear} \right\}
        \subset \PPintropolyshortd,
    \]
    \newcommand{\cactus}[2]{\kappa_{#1}(\nu_{#2}(\PP^n))}%
    The cactus variety is denoted by $\cactus{r}{d}$.
    The idea may be summarized as follows: \emph{Since the Hilbert scheme is
        defined in a more natural way than the space of tuples of points,
        the equations of cactus variety,  parameterized by the Hilbert scheme,
        are easier than the equations of secant variety, parameterized by
    tuples of points}.

    Indeed, for $2r\leq d$, the variety $\cactus{r}{d}$ has an easy to describe set of
    equations. For a form $F\in \intropolyshortd$, consider all its partial
    derivatives of degree $a$, i.e. consider the linear space
    \[
        \Diff(F)_a := \sspan{\partial_1  \ldots \partial_{d-a}\circ F\ |\ \partial_i = \sum_j
            a_{ij}
            \frac{\partial}{\partial x_{j}},\ \ a_{ij}\in \CC}.
    \]
    The cactus variety is set-theoretically defined by the condition
    $\dim_{\CC} \Diff(F)_{\floor{\frac{d}{2}}} \leq r$, which corresponds to certain determinantal
    equations called \emph{minors of the catalecticant matrix},
    see~\cite[Theorem~1.5]{bubu2010} and, for special cases,
    Section~\ref{ssec:secants}.

    To get equations of $\introsecant{r}{d}$, we would like to have an
    equality $\introsecant{r}{d} = \cactus{r}{d}$.
    If all \emph{Gorenstein} schemes of degree $r$ are smoothable (see
    Section~\ref{ssec:intromainresults} or Chapter~\ref{sec:smoothings}), then
    indeed equality happens,
    see~\cite[Theorem~1.6]{jabu_jelisiejew_smoothability}.
     If $r\leq 13$, then all Gorenstein subschemes are smoothable by
     Theorem~\ref{ref:cjnmainintro:thm} and we obtain the following theorem.

     \begin{theorem}\label{ref:intro:secants:thm}
         Let $r, d$
         be integers such that $0 < r < 14$ and $d\geq 2r$. Then the $r$-th secant
         variety of the $d$-th Veronese reembedding of $\PP^n$ is cut out by minors of
         the middle catalecticant matrix.
         More explicitly, the Euclidean closure of the subset
         \[
             \left\{ \sum_{i=1}^r \ell_i^d\ |\ \ell_i\in \introlinear \right\}
             \subset \intropolyshortd.
         \]
         consists precisely of the forms $F\in \intropolyshortd$ such that
         $\Diff(F)_{\floor{\frac{d}{2}}} \leq r$.
     \end{theorem}

     Theorem~\ref{ref:intro:secants:thm} was proven, in the case $r\leq 10$,
     in~\cite{bubu2010}.  Our contribution to this theorem is the extension to
     $r\leq 13$ and to fields of arbitrary characteristic $\neq 2, 3$,
     see~\cite{jabu_jelisiejew_smoothability}.  Applications of
     Theorem~\ref{ref:intro:secants:thm}, for example to signal processing,
     are found in~\cite{Landsberg__tensors}.

\section{Application to constructing $r$-regular
maps}\label{ssec:introkregularity}

An Euclidean-continuous map $f\colon\RR^n \to \RR^N$ or $f\colon\CC^n \to \CC^N$ is called \emph{$r$-regular}
if the images of every $r$ points are linearly independent. The existence of $r$-regular maps to $\CC^N$ for
 given $(r, n)$ is highly nontrivial and has attracted the interest of algebraic
 topologists, including
 Borsuk~\cite{haar__kreg,kolmogorov__kreg,Borsuk,kreg1,kreg2,kreg4,kreg5,kreg6},
 and interpolation theorists~\cite{kreg3,wulbert1999interpolation,
 interpolation__Shekhtman__poly, interpolation__Shekhtman__ideal}. Their
 developments improved the lower bounds on $N$, depending on $n$ and $r$,
 see~\cite{blagojevic_cohen_luck_ziegler_On_highly_regular_embeddings_2}, but
 few examples or sharp upper bounds were known. Instead, many new examples are provided
 by~\cite{Michalek}. Below we outline the ideas of~this paper.

 First, we consider a Veronese map $\CC^n \to \CC^{N_0}$, given by all monomials of degree $\leq d$,
 for $d$ fixed.
 Such a map, when the degree $d$ of the monomials is sufficiently high, is known to be $r$-regular.
 Then, we project from a sufficiently high dimensional linear subspace $H$.
 It turns out that the dimension of possible $H$ is closely related to the numerical properties
 of the smoothable and Gorenstein loci of the punctual Hilbert scheme.
 We obtain the following results, see~\cite{Michalek}.

 \begin{theorem} \label{thm_main_bound_intro_any_k}
     There exist $r$-regular maps $\RR^n \to \RR^{(n+1)(r-1)}$ and $\CC^n \to \CC^{(n+1)(r-1)}$.
 \end{theorem}
 For small values of $r$ or $n$, we find $r$-regular maps into smaller
 dimensional spaces:
 \begin{theorem} \label{thm_main_bound_intro_small_k}
     If $r \leq 9$ or $n\leq 2$, then there exist $r$-regular maps $\RR^n \to \RR^{n(r-1)+1}$  and $\CC^n \to \CC^{n(r-1)+1}$.
 \end{theorem}

 Theorem~\ref{thm_main_bound_intro_small_k} is a consequence of the following
Theorem~\ref{thm_bound_by_hilbert_scheme_intro} together with our estimate of
the dimension of the punctual Hilbert scheme, obtained in Theorem~\ref{ref:introsmallpunc:thm}.
\begin{theorem}[{\cite[Theorem~1.13]{Michalek}}]\label{thm_bound_by_hilbert_scheme_intro}
   Suppose $n$ and $r\geq 2$ are positive integers.
   Let $d_i$ be the dimension of the locus of Gorenstein schemes in the
   punctual Hilbert scheme of degree
     $i$ subschemes of $\CC^n$.
   Then there exist $r$-regular maps $\RR^n \to \RR^N$ and $\CC^n \to \CC^N$,
      where $N = \max\left\{d_i + i \mid 2\le i \le r\right\}$.
\end{theorem}

\newcommand{\introareole}{\mathfrak{b}_{r}(p)}
\begin{proof}[Sketch of proof of Theorem~\ref{thm_bound_by_hilbert_scheme_intro}]
    \def\bigspace{\CC^{N_0}}%
    \def\bigspaceproj{\PP^{N_0}}%
    For $N_0$ large enough, an $r$-regular map $\CC^n \to \bigspace$ exists.
    In fact for $N_0 = \binom{r+n}{r}$ the Veronese map
    \[
        \nu:\CC^n \to \bigspace,
    \]
    given by all monomials of degree $\leq r$, is an example.
    Fix a point $p\in
    \CC^n$. Every Euclidean ball centered at $p$ is homeomorphic to $\CC^n$,
    so it is enough to find a map $\CC^n\to \CC^N$ which is \emph{$r$-regular
    near $p$}: there exists a ball $B$ centered at $p$ such that for every $r$
    points \emph{in $B$}, their images are linearly independent.

    We aim at finding a vector subspace $H \subset \bigspace$ such that
    the composition $\CC^n \to \bigspace\to \bigspace/H$ is also $r$-regular
    near $p$. Denote by $\sspan{R}$ the linear span of $R \subset
    \bigspace$ and consider the following subset of $\bigspace$:
    \begin{equation}
        \introareole := \overline{ \bigcup \left\{ \sspan{R} \mid
            {R\in \HilbGorppleq{\CC^n}}\right\}}.
        \label{eq:areole}
    \end{equation}
    Suppose $H$ is a linear subspace with $H\cap \introareole = \{0\}$.
    Consider the composition $f_H\colon \CC^n\to \bigspace/H$ and suppose it is
    not $r$-regular near $p$. Then, for every $\varepsilon > 0$ there exists
    a tuple $R_\varepsilon$ of $r$ points in $B(p, \varepsilon)$ whose images under $f_H$ are
    linearly dependent, so $\sspan{\nu(R_{\varepsilon})}\cap H\neq \{0\}$.
    Pick a subsequence of $R_{\varepsilon}$ which converges in the Hilbert
    scheme (to assure it exists we should replace $\CC^n$ with $\PP^n$, so
    that the Hilbert scheme is compact); its limit is a finite subscheme $R \subset \CC^n$ of degree $r$ supported
    at $p$ and such that $\sspan{\nu(R)}\cap H\neq\{0\}$. Here we
    implicitly use the fact that $\sspan{\nu(R)}$ behaves well in families,
    see~\cite[Section~2]{jabu_ginensky_landsberg_Eisenbuds_conjecture}. By~\cite[Lemma~2.3]{bubu2010},
    the span $\sspan{\nu(R)}$ is covered by spans of Gorenstein subschemes of
    $R$. For a scheme $R'$ among those subschemes, we have $\sspan{\nu(R')}\cap H\neq
    \{0\}$, so $H\cap \introareole \neq \{0\}$, a contradiction.
    It remains to see that $\dim \introareole \leq N$ and it is fixed under
    the usual $\CC^*$-action, so there exists a linear
    space $H$ not intersecting it and such that $\dim \bigspace/H =
    N$.
\end{proof}

\goodbreak
\section{A brief historical survey}\label{ssec:historical}

        Below we give a brief historical survey of the literature on
        Hilbert schemes of points and finite algebras. We hope that such a
        summary might be helpful for the reader primarily as suggestions for
        further reading.  We should note that there are many great
        introductions to Hilbert schemes from different angles,
        e.g.~\cite{fantechi_et_al_fundamental_ag,
            Gottsche_Hilbert_schemes_and_Betti_numbers,
            hartshorne_deformation_theory, Miller_Sturmfels,
            nakajima_lectures_on_Hilbert_schemes, stromme,
        literature_on_hilberts_workshop},~\cite[Appendix~C]{iarrobino_kanev_book_Gorenstein_algebras}.
        Our viewpoint is very specific; we are interested in
        an explicit, down-to-earth approach and on Hilbert schemes of higher
        dimensional varieties; thus we limit ourselves to connected results.
        For example, we omit the beautiful theory of Hilbert schemes of
        surfaces.

        Several distinguished researchers commented on this survey, however
        their suggestions are not yet incorporated.
        All errors and omissions are entirely due to the author's ignorance.

        Below $X$ is smooth projective variety over $\mathbb{C}$.

        \paragraph{Hilbert schemes.}
        \newcommand{\Hilbabstr}[1]{\Hilbfunc{}(#1)}
        \newcommand{\introHilbpX}{\Hilbfunc{}^p(X)}
        \newcommand{\introHilbpp}{\Hilbfunc{}^p(\PP^n)}

        The Hilbert scheme $\Hilbabstr{X}$ was constructed by Grothendieck~\cite{Gro}.
            It decomposes into a disjoint union of $\introHilbpX$, parameterized by
            the Hilbert polynomials $p\in \mathbb{Q}[t]$.
            Hartshorne~\cite{har66connectedness} proved that all $\introHilbpp$ are
            connected for all $n$. Soon after, Fogarty~\cite{fogarty} proved
            that for constant $p$ the Hilbert scheme $\introHilbpX$ of a smooth
            irreducible surface $X$ is smooth and irreducible.

            At the same time
            Mumford~\cite{Mumford_nonreduced_component} showed that the
            Hilbert scheme $\Hilbabstr{\PP^3}$ is non-reduced: it has a component parameterizing
            certain curves in three dimensional projective space, which is
            even generically non-reduced (see~\cite{Kleppe_Ottem}). Much later
            Vakil~\cite{Vakil_MurphyLaw} vastly generalized this by
            showing that every (up to smooth equivalence) singularity appears
            on $\Hilbabstr{\PP^n_{\mathbb{Z}}}$ for some $n$.

            Much work has been done on finding explicit equations
            of the Hilbert scheme. Grothendieck's proof of existence together
            with Gotzmann's bound on
            regularity~\cite{gotzmann_persistence_theorem} give explicit
            equations of $\introHilbpp$ inside a Grassmannian.  But
            the extremely large number of variables involved makes
            computational
            approach ineffective.
            There is an ongoing progress in simplifying equations and
            understanding the geometry, usually using Borel fixed points, see in
            particular Iarrobino, Kleiman~\cite[Appendix~C]{iarrobino_kanev_book_Gorenstein_algebras},
            Roggero, Lella~\cite{lella_roggero_Rational_Components},
            Bertone, Lella, Roggero~\cite{bertone_borel_open_cover} and
            Bertone, Cioffi, Lella~\cite{bertone_cioffi_roggero_division_algorithm}.
            Staal~\cite{Staal_smooth_hilbert_schemes} showed that over a half
            of Hilbert schemes has only one component: the
            Reeves-Stillman~\cite{reeves_stillman__component} component.
            Roggero and Lella~\cite{lella_roggero_Rational_Components} proved
            that every smooth component of the Hilbert scheme is rational and
            asked, whether \emph{each} component is rational~\cite[Problem
            list]{aimpl}.

            Haiman and Sturmfels~\cite{Haiman_Sturmfels__multigraded}
            introduced the multigraded Hilbert scheme. Independently,
            Huibregtse~\cite{Huibregtse_elementary_construction} and
            Peeva~\cite{Peeva_Stillman} gave
            similar constructions. Smoothness and irreducibility
            of this more general version of the Hilbert scheme for the plane are
            proven by Evain~\cite{Evain_equivariant_punctual} and Maclagan,
            Smith~\cite{Maclagan_Smith__multigraded}.

            Below we are exclusively interested in the \emph{Hilbert scheme of
            points}, which is the union of $\Hilbr{X}$ where $r\in \mathbb{Z}$
            denotes constant Hilbert polynomials. In other words, this scheme
            parameterizes zero-dimensional subschemes of degree $r$.
            This scheme has a distinguished component, called the
            \emph{smoothable component} or \emph{principal component}. It is
            the closure of the set of smooth subschemes (tuples of
            points). Schemes corresponding to points of this component are
            called \emph{smoothable}. In particular, $\Hilbr{X}$ is
            irreducible if and only if every
            subscheme is smoothable.

            Gustavsen, Laksov and Skjelnes~\cite{Gustavsen_Laksov_Skjelnes__Elementary_explicit_const}
            provided a construction of the Hilbert scheme of points for
            \emph{every} affine scheme. Rydh, Skjelnes~\cite{Skjelnes_an_intrinsic_princ_component_via_blowup}
            and Ekedahl, Skjelnes~\cite{Ekedahl_Skjelnes_construction}
            gave an intrinsic construction of the \emph{smoothable component} of
            $\Hilbr{X}$, without reference to $\Hilbr{X}$, while Lee~\cite{Lee_singularities_of_the_principal}
            proved that the smoothable component is not Cohen-Macaulay for $X
            = \mathbb{C}^9$.

            By Fogarty's result, if $\dim X \leq 2$, then $\Hilbr{X}$ is
            irreducible. In fact there is a beautiful theory of Hilbert schemes
            of points on surfaces, which we do not discuss here, see
            e.g.~\cite{Gottsche_Hilbert_schemes_and_Betti_numbers,
             Haiman_macdonald,
        haiman_factorial_conjecture, Kumar__Frobenius, Lederer_Components_of_Groebner_strata, nakajima_lectures_on_Hilbert_schemes}).
            In his paper, Fogarty~\cite[p.~520]{fogarty} asked whether all
            Hilbert schemes of points on smooth varieties are irreducible.
            Iarrobino~\cite{iarrobino_reducibility} disproved this entirely
            and showed that $\Hilbr{\PP^n}$ is reducible for every $n\geq 3$ and
            $r\gg 0$. Fogarty also asked whether Hilbert schemes of points on
            smooth varieties are always
            reduced. This question remains completely
            open, even though progress is made, see Erman~\cite{erman_Murphys_law_for_punctual_Hilb}.

            \newcommand{\Briancon}{Brian{\c{c}}on}%
            Schemes concentrated at a point are of special interest. Their
            locus inside $\Hilbr{X}$ is called the \emph{punctual Hilbert
            scheme}. Since $X$ is smooth, the punctual Hilbert scheme is, at
            the level of points, equal to $\Hilbr{\mathbb{C}[[x_1,
            \ldots ,x_{\dim X}]]}$.
            It has strong connections with germs of mappings~\cite{Galligo__mapping_germs} and topological
            flattenings~\cite{Galligo__flattener}.
            \Briancon{}
            proved~\cite{briancon} that for $\dim X \leq 2$ the punctual
            Hilbert scheme is irreducible: every degree
            $r$ quotient is a limit of quotients isomorphic to
            $\mathbb{C}[t]/t^r$, see also~\cite{iarrobino_punctual,
            iarrobino_10years, yameogo}. The state of the art for year 83 is
            nicely summarized in Granger~\cite{Granger_memoir}.
            Gaffney~\cite{Gaffney_chaining_upperbounds_onsmoothablepunctual}
            gave a lower bound for the dimensions of components of punctual Hilbert
            schemes whose points correspond to smoothable algebras and
            conjectured that it bounds the dimensions of all components.

            Much is known about schemes $\Hilbr{\PP^n}$ for small $r$.
            Mazzola~\cite{mazzola_generic_finite_schemes} proved that they are
            irreducible for $r\leq 7$.
            Emsalem and Iarrobino~\cite{emsalem_iarrobino_small_tangent_space}
            proved that $\Hilbarged{8}{\PP^n}$ is reducible for $n\geq 4$.
            Cartwright, Erman, Velasco and
            Viray~\cite{cartwright_erman_velasco_viray_Hilb8}
            proved that $\Hilbarged{8}{\PP^n}$ is irreducible for $n\leq 3$ and has
            exactly two components for $n\geq 4$; they also gave a full
            description of the non-smoothable component and the intersection.
            These result imply that $\Hilbr{\PP^n}$ is reducible for all
            $n\geq 4$ and $r\geq 8$, which leaves only the case $r = 3$ open.
            Borges dos Santos, Henni and Jardim~\cite{dosSantos} proved
            irreducibility of $\Hilbr{\PP^3}$ for $r\leq 10$, using the
            results of
            {\v{S}}ivic~\cite{Sivic__Varieties_of_commuting_matrices} on
            commuting matrices. Douvropoulos, Utst{\o{}}l
            N{\o{}}dland, Teitler and the author~\cite{DJNT} proved irreducibility of
            $\Hilbarged{11}{\PP^3}$. The case $\Hilbr{\PP^3}$ is interesting,
            because the Hilbert scheme can be presented as a singular locus of
			a hypersurface on a smooth manifold; such presentation restricts
            possible singularities, see~Dimca,
			Szendr{\H{o}}i~\cite{Dimca_Szendroi__Hilbert_of_A3} and Behrend,
            Bryan, Szendr{\H{o}}i~\cite{Behrend}.

            The \emph{Gorenstein locus} of $\Hilbr{X}$ is the open subset
            $\HilbGorr{X}$
            consisting of points corresponding to Gorenstein algebras.
            Casnati, Notari and the
            author~\cite{casnati_notari_irreducibility_Gorenstein_degree_9, cn10, cn11, cjn13} proved
            irreducibility of the Gorenstein locus of up to degree $13$ and
            investigated its singular locus. The author~\cite{jelisiejew_VSP}
            also
            described the geometry of the Gorenstein locus for degree
            $14$, the first reducible case, using results of Ranestad,
            Iliev and Voisin~\cite{Ranestad_Iliev__VSPcubic,
            Ranestad_Iliev__VSPcubic_addendum, Ranestad_Voisin__VSP} on Varieties of
            Sums of Powers.

            \vspace{-2mm}
        \paragraph{Finite algebras.}

			Algebraically, Hilbert schemes of points parameterize
			zero-dimensional quotients of polynomial rings.
			Historically those were considered far before Hilbert schemes;
			perhaps the first mention of zero-dimensional Gorenstein algebras
			is Macaulay's
            paper~\cite{macaulay_enumeration}, which describes possible
            Hilbert functions of complete intersections on $\AA^2$.
            Macaulay~\cite{Macaulay_inverse_systems} also gave his famous
            structure theorem, describing all local zero-dimensional
            algebras in terms of inverse systems and duality
            between functions and constant coefficients differential operators
            on affine space. This duality can be also viewed as a case of
            Matlis duality~\cite{Matlis} or in the language of Hopf
            algebras~\cite{Ehrenborg}.

            A new epoch started with the construction of the Hilbert scheme.
            Fogarty's result~\cite{fogarty} implies that,
            for every $r$, finite rank $r$ quotients of $\CC[x, y]$ are
            \emph{smoothable},~i.e., are limits of
            $\mathbb{C}^{\times r}$. Iarrobino's~\cite{iarrobino_reducibility}
            proves that there are non-smoothable quotients of $\mathbb{C}[x,
            y, z]$ and higher dimensional polynomial rings (there examples
            are not Gorenstein). Interestingly, the result is
            non-constructive: Iarrobino produces a family too large to fit
            inside the smoothable component, but no specific point of this
            family is known to be nonsmoothable; more generally no explicit example of a non-smoothable quotient of
            $\mathbb{C}[x, y, z]$ is known.
            Fogarty's smoothness result follows also from the Hilbert-Burch theorem,
            saying that deformations of zero-dimensional
            quotients of $\mathbb{C}[x, y]$ are controlled by deformations of
            a certain matrix (the ideal is generated by its maximal minors). The same result
            holds for codimension two Cohen-Macaulay algebras,
            see~Ellingsrud~\cite{ellingsrud} and
            Laksov~\cite{Laksov_defs_of_determinantal}.  Buchsbaum and
            Eisenbud~\cite{BuchsbaumEisenbudCodimThree} described resolutions
            of zero-dimensional Gorenstein quotients
            of $\mathbb{C}[x, y, z]$ and showed that they are controlled by an
            anti-symmetric matrix (the ideal is defined by its Pfaffians).
            This is used~\cite{kleppe_pfaffians} to show that zero-dimensional Gorenstein
            quotients of $\mathbb{C}[x, y, z]$ are
            smoothable.
            A classical and very accessible survey of these results
            is Artin's~\cite{artin_deform_of_sings}.
            Later Eisenbud-Buchsbaum result was
            generalized to arbitrary codimension three Gorenstein (or
            arithmetically Gorenstein)
            quotients, see Kleppe and Mir\`o-Roig \cite{roig_codimensionthreeGorenstein,kleppe__smoothness,
            kleppe_roig_codimensionthreeGorenstein}. Codimension four remains
            in progress, see Reid's~\cite{reid_codimension_four_Gorenstein},
            however one does not except as striking smoothness results as above.

            In the following years progress was made in several directions
            (state of the art for 1987 are nicely summarized in
            Iarrobino's~\cite{iarrobino_10years}).
            An influential article of Bass~\cite{Bass} discussed various
            appearances of Gorenstein algebras in literature.
            Schlessinger~\cite{Schlessinger} investigated deformations and
            asked for classifications of zero-dimensional rigid algebras (the question
            remains open).
            Emsalem and Iarrobino~\cite{emsalem_iarrobino_small_tangent_space} produced
            a nonsmoothable, degree $8$ quotient of $\mathbb{C}[x,y,z,t]$ and
            a nonsmoothable, degree $14$ Gorenstein quotient of
            $\mathbb{C}[x_1, \ldots ,x_6]$.
            Much later Shafarevich~\cite{Shafarevich_Deformations_of_1de}
            generalized the degree $8$ example and produced several new
            classes of nonsmoothable algebras with Hilbert function $(1, d,
            e)$ where $\binom{d+1}{2} \gg e > 2$.

            Emsalem~\cite{emsalem} announced several milestone results concerning
            deformations and classification of zero-dimensional Gorenstein
            algebras: he translated those problems into language of inverse
            systems, thus enabling a combinatorial approach.

            These developments prompted work on the classification.
            Sally~\cite{SallyStretchedGorenstein} classified Gorenstein algebras with Hilbert function
            $(1, n, 1,  \ldots , 1)$, though she was primarily interested in
            higher-dimensional case.
            Mazzola~\cite{Mazzola_CNdegreefive,
            mazzola_generic_finite_schemes} proved that all algebras of degree
            up to $7$ are smoothable and provided a table of their
            deformations up to degree $5$, assuming that the base field is
            algebraically closed of characteristic different from $2$, $3$.
            Poonen~\cite{Poonen} classified algebras of degree up to $6$
            without assumptions on the characteristic. He also
            investigated the moduli space of algebras with fixed
            basis~\cite{poonen_moduli_space}, calculating its asymptotic
            dimension.

            Iarrobino~\cite{ia_deformations_of_CI} considered deformations of complete
            intersections. His subsequent
            paper~\cite{iarrobino_compressed_artin} introduced
            \emph{compressed algebras} and proved that there are
            nonsmoothable degree $78$ quotients of $\mathbb{C}[x, y, z]$;
            this bound has not been sharpened since. He also proved that there
            are nonsmoothable Gorenstein quotients of $\mathbb{C}[x_1, \ldots
            , x_n]$ for all $n\geq 4$.

            Stanley~\cite{stanley_hilbert_functions_of_graded_algebras,
            StanleyCombinatoricsAndCommutative} used Buchsbaum-Eisenbud
            results and classified possible Hilbert
            functions of \emph{graded} Gorenstein quotients
            of $\mathbb{C}[x, y, z]$ (here
            and below
            an algebra is \emph{graded} if it is isomorphic to a quotient of
            polynomial ring by a homogeneous ideal with respect to the standard grading. This
            is usually named \emph{standard
            graded},~\cite{StanleyCombinatoricsAndCommutative}).
            Which Hilbert functions appear for the graded
            Gorenstein quotients of $\mathbb{C}[x_1, \ldots ,x_n]$ remains a
            hard problem, see Migliore, Zanello~\cite{Migliore_Hvectors}.
            Diesel~\cite{diesel__irreducibility} and
            Kleppe~\cite{kleppe__smoothness} showed that the locus of graded
            quotients of $\mathbb{C}[x, y, z]$ with given function is smooth
            and irreducible.
            Boij~\cite{Boij__components} showed that this is no longer true in
            higher number of variables, see
            also~\cite{kleppe__unobstructedness}. Boij~\cite{Boij_betti_numbers}
            investigated also the Betti tables of graded Gorenstein algebras,
            in particular compressed ones, conjectured that these tables are
            minimal (in the spirit of Minimal Resolution Conjecture) and proved
            this conjecture in several classes.
            Conca, Rossi and Valla~\cite{Conca_Rossi_Valla__Koszulness_of_general_1nn1}
            proved that general graded Gorenstein algebras with Hilbert function $(1, n, n, 1)$ are Koszul.

            Iarrobino~\cite{iarrobino_associated_graded} considered
            Hilbert functions of Gorenstein, not necessarily graded, algebras.
            He developed and investigated the notion of \emph{symmetric
            decomposition of the Hilbert function}; all subsequent
            classification work relies on this memoir.

            Graded Gorenstein algebras are intrinsically connected to secant
            varieties and Waring problem for forms, which enjoyed much
            research activity following the paper of Alexander and
            Hirschowitz~\cite{AlexHirsch}. We discuss this most briefly.
            A nice introduction is found in
            Geramita~\cite{geramita_inverse_systems_of_fat_points}.
            See also the papers of
            Bernardi, Gimigliano, Ida~\cite{BGIComputingSymmetricRank},
            Bernardi, Ranestad~\cite{BRonTheCactusRank},
            Buczyński, Ginensky, Landsberg~\cite{jabu_ginensky_landsberg_Eisenbuds_conjecture},
            Buczyński, Buczyńska~\cite{bubu2010},
            Carlini, Catalisano, Geramita \cite{carlini_monomials},
            Derksen, Teitler~\cite{derksen_teitler},
            Landsberg, Ottaviani~\cite{LO},
            Landsberg, Teitler~\cite{landsberg_teitler}, and
            Buczyński, the author~\cite{jabu_jelisiejew_smoothability} for an
            overview of possible directions and connections.
            Many of the aforementioned results on Gorenstein algebras, their
            loci, and about Waring problems are discussed in Iarrobino, Kanev
            book~\cite{iarrobino_kanev_book_Gorenstein_algebras}.

            Casnati and Notari analysed smoothability
            of finite Gorenstein algebras in~\cite{Casnati_Notari_6points}.
            In a subsequent series of
            papers~\cite{casnati_notari_irreducibility_Gorenstein_degree_9, cn10, cn11} they established smoothability of all
            Gorenstein algebras for degrees up to $11$ and then, jointly with
            the author~\cite{jelisiejew_1551, cjn13}, smoothability for all
            Gorenstein algebras of degree up to $14$,
            except those with Hilbert function $(1, 6, 6, 1)$, see
            also~\cite{jelisiejew_VSP}.
            Bertone, Cioffi and
            Roggero~\cite{bertone_cioffi_roggero_division_algorithm} proved
            smoothability of Gorenstein algebras with Hilbert function $(1, 7, 7, 1)$.
            Iarrobino's~\cite{iarrobino_compressed_artin} results show that a
            general Gorenstein algebra with Hilbert function $(1, n, n, 1)$, $n\geq 8$ is
            nonsmoothable.

            Some of the above results for $(1, n, n, 1)$ depended on Elias' and
            Rossi's~\cite{EliasRossiShortGorenstein}
            proof that all Gorenstein algebras with Hilbert functions $(1, n,
            n, 1)$ are \emph{canonically graded} (isomorphic to their associated
            graded algebra). Elias and Rossi also proved that algebras
            with Hilbert function~$(1, n, \binom{n+1}{2}, n, 1)$ are
            canonically graded, see~\cite{EliasRossi_Analytic_isomorphisms}.
            Fels, Kaup~\cite{Fels_Kaup_nilpotent_and_homogeneous} and
            Eastwood, Isaev~\cite{Eastwood_Isaev} showed that
            an algebra is canonically graded if and only if certain
            hypersurfaces, associated to this algebra, are affinely equivalent.
%            The problem of checking gradedness was also tackled from the angle
%            of differential geometry, in particular in articles by Fels,
%            Kaup~\cite{Fels_Kaup_nilpotent_and_homogeneous} and Eastwood,
%            Isaev~\cite{Eastwood_Isaev}.
            Jelisiejew~\cite{Jel_classifying} considered classification of
            algebras, extending the ideas of Emsalem. He proved that the above
            results of Elias and Rossi are
            consequences of a group action. He also conjectured when
            are ``general'' Gorenstein algebras graded and classified
            Gorenstein algebras
            with Hilbert function $(1, 3, 3, 3, 1)$, giving examples of nongraded algebras. Masuti and
            Rossi~\cite{masuti_rossi_Artinian_level_algebras_of_socle_degree_four}
            provided many examples of non-graded algebras for all Hilbert
            functions $(1, n, m, s, 1)$ with $s\neq n$ or $m\neq
            \binom{n+1}{2}$.

            Meanwhile, some more classification results were obtained,
            primarily using inverse systems.
            Casnati~\cite{Casnati__Isomorphisms_types_of_aG_to_nine} gave
            a complete classification of Gorenstein algebras of degree at most
            $9$. Elias, Valla~\cite{EliasVallaAlmostStretched} and Elias,
            Homs~\cite{EliasHoms} classified all Gorenstein
            quotients of $\mathbb{C}[[x, y]]$,
            which are \emph{almost stretched}: their Hilbert function $H$
            satisfies $H(2) \leq 2$. Casnati and
            Notari~\cite{CN2stretched} investigated Gorenstein algebras with $H(3) = 1$
            and classified those with Hilbert function $(1, 4, 4, 1, 1)$.
            The classification results were also used to prove rationality of the
            Poincar\`e series of zero-dimensional Gorenstein algebras,
            see~\cite{CN_PoincareUpto10, CENR_Poincare, CJNPoincare}.

            Outside the Gorenstein world, Erman and
            Velasco~\cite{erman_velasco_syzygetic_smoothability} gave new
            obstructions to smoothability of algebras with Hilbert function
            $(1, n, e)$ and obtained a complete picture for $(1, 5, 3)$.
            Huibregtse~\cite{Huibregtse_elementary} built a framework for
            finding nonsmoothable algebras
            and presented several of them.
            Many unsolved problems exist. We gather some of
            them in the following section.

\section{Open problems}\label{ssec:openquestions}

Hilbert schemes are rich in natural, but open
questions. We list some of them below. Many come from the~American Institute
of Mathematics workshop on Hilbert schemes~\cite{aimpl}. I am indebted to
organizers of this workshop, who made their problem list publicly available,
however any errors below or the selection of problems remain my sole
responsibility.

\begin{problem}[{\cite[p.~520]{fogarty}},
    {\cite[p.~794]{cartwright_erman_velasco_viray_Hilb8}},
    {\cite[p.~1349]{Lee_singularities_of_the_principal}}]\label{prob:reducedness}
    Is $\Hilbr{\AA^n}$ always reduced? If not, what are the
    examples?\footnote{Examples found in~\cite{Jelisiejew__Pathologies} and
    for $13$ points on $\mathbb{A}^6$, in~\cite{Szachniewicz}.}
\end{problem}

\begin{problem}[{\cite{aimpl}},
\cite{lella_roggero_Rational_Components}]\label{prob:nonrational}
    Does there exist a~nonrational component of $\Hilbr{\AA^n}$?
\end{problem}

\begin{problem}[{\cite{aimpl}}]\label{prob:rigid}
    Is there a~rigid local Artinian $\kk$-algebra besides $\kk^{\times n}$?
    An algebra $A$ is \emph{rigid} if all its nearby deformations are abstractly
    isomorphic to $A$, see~\cite{Schlessinger}.
\end{problem}

\begin{problem}[{\cite[p.~794]{cartwright_erman_velasco_viray_Hilb8}}]\label{prob:irreducibilityA3}
    What is the smallest $r$ such that $\Hilbr{\mathbb{A}^3}$ is
    reducible?
    We know
    that $12\leq r\leq 77$ at least in characteristic zero.
\end{problem}

\begin{problem}[{\cite[Section~6]{bubu2010},
\cite{jabu_jelisiejew_smoothability}}]\label{prob:irreducibilityA4Gor}
What is the smallest $r$ such that $\HilbGorr{\AA^4}$ is reducible?  We know
that $15\leq r \leq 140$ at least in characteristic $\neq 2, 3$.
\end{problem}

\begin{problem}[{\cite{aimpl}}]\label{prob:grobnerfan}
    Is the Gr\"obner fan a~discrete invariant that distinguishes the
    irreducible components of $\Hilbr{\AA^n}$? More generally, what are the
    components of $\Hilbr{\AA^n}$?\footnote{The Gr\"obner fan does not work,
        see~\cite{Jelisiejew__Elementary}. The question remains interesting.}
\end{problem}

\begin{problem}[{{\cite[9H,
    p.~258]{iarrobino_kanev_book_Gorenstein_algebras}}}]\label{prob:iarrobino}
    Can we produce components of the Hilbert scheme from special geometric
    configurations of points or schemes?
    See~\cite[Chapter~9]{iarrobino_kanev_book_Gorenstein_algebras} for
    several open questions.
\end{problem}

\begin{problem}[{\cite{aimpl}}]\label{prob:tangenttosmoothable}
    Can we describe the Zariski tangent space to the smoothable component of
    $\Hilbr{\AA^n}$?
\end{problem}

\begin{problem}[{\cite{aimpl}}]\label{prob:characteristic}
    Is there a~component of $\Hilbr{\AA_{\kk}^n}$ which exists only for
    $\kk$ of characteristic $p$, for some $p > 0$?\footnote{Yes, for every
        $p$, see~\cite{Jelisiejew__Pathologies}.}
\end{problem}

\begin{problem}[{\cite{iarrobino_10years}}]\label{prob:lowerbounddimension}
    Consider the scheme $\mathcal{Z} \subset \Hilbr{\AA^n}$ parameterizing
    subschemes supported at the origin. Is there a~component of $\mathcal{Z}$
    of dimension \emph{less} than $(n-1)(r-1)$?\footnote{There is such a component
        already for $n = 4$, see~\cite{Satriano_Staal}.}
\end{problem}

\begin{problem}[{\cite{jelisiejew_VSP}}]\label{prob:smallcomponents}
    Can we classify those irreducible components of $\Hilbr{\AA^n}$ which have
    dimension \emph{less} than $rn$?
\end{problem}

\begin{problem}[{\cite[9G, p.~257]{iarrobino_kanev_book_Gorenstein_algebras},
\cite{Huibregtse_elementary}}]\label{prob:elementarymanymany}
A component of $\Hilbr{\AA^n}$ is \emph{elementary} if its general point
corresponds to an irreducible subscheme. Are there infinitely many elementary
components of $\Hilbr{\AA^3}$?
\end{problem}

\newpage
\section*{Acknowledgements}

    I thank my advisors, Jarek and Weronika, for their help in preparing this
    thesis. During these years Jarek has taught me a lot of beautiful math,
    was always patient to answer my questions, and wise/busy enough to
    gradually give me freedom in research. Thanks a lot for making my PhD
    a happy and productive experience!

    I thank fellow PhD students from \emph{kanciapa}: Maciek Ga{\l{}}{\k{a}}zka,
    Maks Grab, {\L{}}ukasz Sienkiewicz and Maciej Zdanowicz for our creative
    and friendly small world, with its math discussions, problem solving and board game
    playing. I thank all my coauthors, especially Gianfranco Casnati and
    Roberto Notari, for our enjoyable work. Countless people have given me
    advice, suggestions and helped me in all possible mathematical and
    non-mathematical ways. Among them, I'd like to thank Piotr
    Grzeszczuk, Marcin Emil Kuczma, and Jarek Wi{\'s}niewski.

\begin{partwithabstract}{Finite algebras}
%\part{Finite algebras}
\label{part:algebras}

    We discuss finite algebras, their
    numerical properties and embeddings. We restrict ourselves to linear
    algebra and Lie theory tools, which are sufficient for our purposes. We
    also use the language of commutative algebra rather than algebraic geometry.
    This part is essentially a prerequisite of Part~\ref{part:families},
    however Sections~\ref{ssec:specialforms}--\ref{ssec:examplesthree} contain
    some original research results, published
    in~\cite{Jel_classifying} and~\cite{cjn13}.
\end{partwithabstract}

    \chapter{Basic properties of finite algebras}\label{sec:basicprops}

Throughout this thesis $\kk$ is a field (of arbitrary characteristic, not necessarily
algebraically closed) and $\DA$ is a finite
$\kk$-algebra,~i.e., a finite dimensional $\kk$-vector space equipped with an
associative and commutative multiplication $\DA\tensor_{\kk} \DA\to \DA$ and a
unity $1 \in \DA$. Every such algebra is isomorphic to a product of
\emph{local} algebras and in fact we will be mostly interested in local
algebras. We denote local algebra by $(\DA, \mm, \kk)$, where $\mm$ is the
maximal ideal of $\DA$. We make a global assumption that $\kk\to \DA/\mm$ is
an \emph{isomorphism}. This assumption is automatic if $\kk$ is algebraically
closed. It can also be satisfied by replacing $\kk$ with the residue field
$\kappa = \DA/\mm$; the fact that there is an embedding $\kappa\into\DA$ is an ingredient of
Cohen structure theorem, see~\cite[Section~7.4]{Eisenbud}. We should
be aware that while $\kappa$ and $\DA$ are $\kk$-algebras, it is not clear,
whether the embedding $\kappa\into \DA$ can be chosen $\kk$-linearly~\cite[Section~7.4]{Eisenbud}.
The assumption that $\kk \to \DA/\mm$ is an isomorphism implies that all numerical objects
defined later: Hilbert functions and their symmetric decompositions, socle
dimensions etc., are invariant under field extension.

We usually consider local algebras presented as quotients of power series
rings and we use the following observation.
\begin{lemma}\label{ref:embeddingdimension:lem}
    A finite local algebra $(\DA, \mm, \kk)$ can be presented as quotient of
    power series $\kk$-algebra $\DShat = \kk[[\Dx_1, \ldots ,\Dx_n]]$ if and
    only if $n = \dim \DShat
    \geq \dimk \mm/\mm^2$.
\end{lemma}

The phrase \emph{dimension of $\DA$} is ambiguous: the Krull dimension of $\DA$ is
zero, whereas its dimension as a linear space is finite and positive. We
resolve this as follows: we never refer to the Krull dimension and we refer to the
\emph{degree} of $\DA$ when speaking about $\dimk \DA$.

\section{Gorenstein algebras}

Gorenstein algebras are tightly connected with the notion of \emph{duality},
in fact they can be thought of as simplest algebras from the dual point of
view, as we explain below. In this presentation we
follow~\cite[Section~21.2]{Eisenbud}.

\begin{definition}\label{ref:canonicalmodule:def}
    Let $\DA$ be a finite $\kk$-algebra. Its \emph{canonical module} $\canA$ is
    the vector space $\Hom_{\kk}(\DA, \kk)$ endowed with an $A$-module
    structure via
    \begin{equation}
        (a \circ f)(a') = f(aa')\quad\mbox{for }a, a'\in A,\
        \mbox{ and } f\in \canA.
        \label{eq:dualmodule}
    \end{equation}
\end{definition}
The canonical module does not depend on the choice of $\kk$; only on the ring
structure of $A$, see~\cite[Proposition~21.1]{Eisenbud}. Also $\dimk \canA =
\dimk A$ for all $A$.
Note that $\canA$ is torsion-free. Indeed, $a\cdot \canA = 0$ implies that
$ 0 = (af)(1) = f(a)$ for every functional $f\colon\DA\to \kk$, so $a = 0$.

\begin{definition}\label{ref:Gorenstein:def}
    A finite $\kk$-algebra $\DA$ is \emph{Gorenstein} if and only if
    $\canA$ is isomorphic to $A$ as an $A$-module. In this case every $f\in
    \canA$ such that $Af = \canA$ is called a \emph{dual generator} of $A$.
\end{definition}

We observe that Definition~\ref{ref:Gorenstein:def} is local, as explained in the following lemma.
\begin{lemma}\label{ref:locality:lem}
    A finite $\kk$-algebra $\DA$ is Gorenstein if and only if all its
    localisations at maximal ideals are Gorenstein.
\end{lemma}
\begin{proof}
    For every maximal ideal $\mm \subset \DA$, we have $\omega_{\DA_{\mm}}
    \simeq \left( \canA \right)_{\mm}$. Hence, if $\DA$ is Gorenstein, then
    every its localisation is Gorenstein.
    Conversely, if all localisations of $\DA$ at maximal ideals are
    Gorenstein, then $\canA$ is locally free of rank one. Since $\DA$ is
    finite, this implies that $\canA  \simeq \DA$, so $\DA$ is Gorenstein.
\end{proof}

We stress that dual generators are by no means unique: even in
the trivial case $\DA = \kk$ every non-zero functional on $\DA$ is its dual
generator.
Before we give examples, we present two equivalent but even more
explicit conditions on Gorenstein algebras. First, one can give a definition
not involving $\canA$.
\begin{proposition}\label{ref:perfectpairing:prop}
    Let $\DA$ be a finite $\kk$-algebra and $f\colon\DA\to \kk$ be a functional. Then $\DA$ is Gorenstein with dual
    generator $f$ if and only if the pairing
    \begin{equation}
        \DA \times \DA\ni (a, b) \to f(ab) \in \kk
        \label{eq:perfectpairing}
    \end{equation}
    is nondegenerate.
\end{proposition}
\begin{proof}
    The pairing~\eqref{eq:perfectpairing} descents to $\DA\otimes \DA\to \kk$
    and can be rewritten as $\DA\to \Hom(\DA, \kk)$ sending unity to $f$ and
    sending $a$ to $a\circ f$; it is an $A$-module homomorphism $\DA\to
    \canA$. Now $f$ is a dual generator iff this homomorphism is onto
    iff this homomorphism is into iff there is no non-zero $a\in A$ such that
    $a \circ f = 0$ iff there is no non-zero $a\in A$ such that $f(Aa) =
    \{0\}$ iff the pairing~\eqref{eq:perfectpairing} is nondegenerate.
\end{proof}

\begin{definition}\label{ref:socle:def}
    Let $(\DA, \mm, \kk)$ be a local algebra.
    The \emph{socle} of $A$ is the annihilator of its maximal ideal. It is
    denoted by $\socleA$.
\end{definition}
Note that for every $a\in A$ and an appropriate exponent $r$ we have $a\cdot
\mm^r \neq 0$ and $a\cdot \mm^{r+1} = 0$, thus $a\cdot\mm^r \subset \socleA$.
Therefore $\socleA$ intersects every nonzero ideal in $A$.

\begin{proposition}[{\cite[Proposition~21.5]{Eisenbud}}]\label{ref:Gorchar:eis:prop}
    Let $(\DA, \mm, \kk)$ be a finite local $\kk$-algebra. The following conditions are
    equivalent:
    \begin{enumerate}
        \item $\DA$ is Gorenstein,
        \item $\DA$ is injective as an $\DA$-module,
        \item the socle of $\DA$ is a one-dimensional $\kk$-vector space,
        \item the $A$-module $\DA$ is principal.
    \end{enumerate}
\end{proposition}
\begin{corollary}\label{ref:dlajarka:cor}
    Let $(\DA, \mm, \kk)$ be a finite local $\kk$-algebra. Then $\DA$ is
    Gorenstein if and only if there is a unique quotient $B = \DA/I(B)$ with
    $\dimk B = \dimk A - 1$.
\end{corollary}
\begin{proof}
    We have $\socleA \neq 0$. For each $B$
    as above its ideal $I(B)\subset A$ is given by a single element of
    $\socleA$.
    Thus the space of possible $B$'s is isomorphic to $\mathbb{P}(\socleA)$
    and $B$ is unique if and only if $\dimk\socleA = 1$. By
    Proposition~\ref{ref:Gorchar:eis:prop}, this is equivalent to $\DA$ being
    Gorenstein.
\end{proof}

As seen in the Proposition~\ref{ref:Gorchar:eis:prop}, Gorenstein property
depends only on the socle of $\DA$. We now show that dual generators are
distinguished among functionals on $\DA$ by their nonvanishing on the socle.
\begin{corollary}\label{ref:dualgens:cor}
    Let $(\DA, \mm, \kk)$ be a local Gorenstein $\kk$-algebra. Then $f\colon \DA\to
    \kk$ is a dual generator of $\DA$ if and only if $f(\socleA) \neq \{0\}$.
\end{corollary}
\begin{proof}
    We will use the characterization of dual generators from
    Proposition~\ref{ref:perfectpairing:prop}. Suppose first that $f(\socleA)
    = \{0\}$. Since $A\socleA \subset \socleA$, this implies that $f(A\socleA)
    = \{0\}$, thus the pairing~\eqref{eq:perfectpairing} is degenerate.

    Suppose $f(\socleA) \neq \{0\}$.
    Since $\socleA$ is one-dimensional, the condition $f(\socleA)\neq \{0\}$
    implies that $f(b)\neq 0$ for every non-zero $b\in \socleA$. Choose any
    $a\in A$. Then $Aa \cap \socleA\neq \{0\}$ so $f(Aa) \neq \{0\}$ and the
    pairing~\eqref{eq:perfectpairing} is nondegenerate.
\end{proof}

%As two important corollaries from Proposition~\ref{ref:Gorchar:eis:prop} we
%obtain that $\dimk \DA = \dimk \canA$.

\begin{example}
    The smallest degree non-Gorenstein
    algebra is $A = \kk[[\Dx, \Dy]]/(\Dx^2, \Dx\Dy, \Dy^2)$. Indeed, $\socleA
    = (\Dx, \Dy)$ is
    two dimensional.
\end{example}

\begin{example}\label{ex:firstGorenstein}
    The algebra $\DA = \kk[[\Dx, \Dy]]/(\Dx\Dy, \Dx^2-\Dy^2)$ is Gorenstein, with socle
    generated by the class of $\Dx^2+\Dy^2$. More generally, any complete
    intersection is Gorenstein~\cite[Corollary~21.19]{Eisenbud} and its socle
    is generated by an element which may be interpreted as the Jacobian of the
    minimal generating set, see~\cite[Exercise~21.23c]{Eisenbud}.
\end{example}

The following Proposition~\ref{ref:Gorensteinbasechange:prop} shows that we may investigate, whether an algebra is
Gorenstein, after arbitrary field extension, for example after a base change to
$\kkbar$.
\begin{proposition}\label{ref:Gorensteinbasechange:prop}
    \def\KK{\mathbb{K}}%
    Let $\DA$ be a finite $\kk$-algebra. Then the following conditions are
    equivalent
    \begin{enumerate}
        \item $\DA$ is a Gorenstein $\kk$-algebra,
        \item $\DA \tensor_{\kk} \KK$ is a Gorenstein $\KK$-algebra for every field extension
            $\kk \subset \KK$,
        \item $\DA \tensor_{\kk} \KK$ is a Gorenstein $\KK$-algebra for some field extension
            $\kk \subset \KK$.
    \end{enumerate}
\end{proposition}
\begin{proof}
    \def\KK{\mathbb{K}}%
    \def\DApr{\DA'}%
    \def\canApr{\omega_{\DApr}}%
    By Lemma~\ref{ref:locality:lem}, we may assume $\DA$ is local, with
    maximal ideal $\mm$ and residue field $\kappa \supset\kk$.
    Fix a field extension $\kk \subset \KK$ and $\DApr := \DA\tensor_{\kk}
    \KK$.
    If $\DA$ is Gorenstein, then any
    isomorphism $A\to \canA$ induces an isomorphism $\DApr \to
    \canApr$ and $\DApr$ is Gorenstein.
    Suppose $\DApr$ is Gorenstein, so $\canApr  \simeq \DApr$ as
    $\DApr$-modules. Therefore,
    $\canApr/\mm \canApr  \simeq \DApr/\mm \DApr  \simeq
    \kappa\tensor_{\kk} \KK$ and $\dim_{\KK} \canApr/\mm\canApr = \dim_{\kk} \kappa$.
    Moreover, we have $\canApr  \simeq \canA
    \tensor_{\kk} \KK$ as $\DApr$-modules, so that
    \[
         \dim_{\kk} \canA/\mm\canA = \dim_{\KK} \canApr/ \mm\canApr = \dim_{\kk} \kappa.
    \]
    It follows that $\canA/\mm\canA$ is a one-dimensional $\kappa$-vector
    space, so there exists an epimorphism $\DA\to \canA/\mm\canA$ and hence, by
    Nakayama's lemma, an
    epimorphism $\DA\to \canA$. Since $\dimk \DA = \dimk \canA$, it follows
    that $\canA$ is isomorphic to $\DA$ as an $\DA$-module and $\DA$ is Gorenstein.
\end{proof}

\begin{example}\label{ex:smoothareGorenstein}
    Every finite smooth $\kk$-algebra $\DA$ is Gorenstein. Indeed,
    $\DA\tensor_{\kk} \kkbar$ is a smooth $\kkbar$-algebra, so it is
    isomorphic to $\kkbar^{\times \deg \DA}$. This algebra is Gorenstein by
    Proposition~\ref{ref:Gorchar:eis:prop}, so $\DA$ is Gorenstein as well
    by Proposition~\ref{ref:Gorensteinbasechange:prop}.
\end{example}

For further use, we note below in Lemma~\ref{ref:dualizingfunctors:lem} that
two natural notions of the dual module agree for
Gorenstein algebras.
\newcommand{\Homthree}[3]{\operatorname{Hom}_{#1}\pp{#2, #3}}%
\begin{lemma}\label{ref:dualizingfunctors:lem}
    Let $(\DA, \mm, \kk)$ be a local Gorenstein $\kk$-algebra.
    Then for every $\DA$-module $M$ we have $\Homthree{\DA}{M}{\DA}  \simeq \Homthree{\kk}{M}{\kk}$ naturally.
\end{lemma}
\begin{proof}
    In the language of~\cite[Section~21.1]{Eisenbud}, both $\Homthree{\DA}{-}{\DA}$ and $\Homthree{\kk}{-}{\kk}$ are dualizing
    functors, so they are isomorphic, see Proposition~\cite[Proposition~21.1, Proposition~21.2]{Eisenbud}.
\end{proof}

\section{Hilbert function of a local
algebra}\label{ssec:hilbertfunctionGeneral}

Let $(\DA, \mm, \kk)$ be a finite local $\kk$-algebra. Its \emph{associated
graded algebra} is $\grDA = \bigoplus_{i\geq 0} \mm^i/\mm^{i+1}$. Of course,
$\grDA$ is also a finite local $\kk$-algebra.
    \begin{definition}
        The \emph{Hilbert function} of $\DA$ is defined as
        \[H_{\DA}(i) = \dim_{\kk} \mm^i/\mm^{i+1}.\]
    \end{definition}
    Note that $H_{\DA} = H_{\grDA}$.
    We have $H_{\DA}(i) = 0$ for $i\geq \dimk A$ and it is usual to write $H_{\DA}$ as a
    vector of its nonzero values.
    Lemma~\ref{ref:embeddingdimension:lem} proves that if $\DA$ is a quotient
    of a power series ring $\DShat$, then $\DShat$ has dimension at least $H_{\DA}(1)$.
    Therefore $H_{\DA}(1)$ is named the \emph{embedding dimension} of $\DA$.

    Hilbert functions are usually considered in the setting
    of standard graded algebras. A $\kk$-algebra $\DA$ is \emph{standard
    graded} if
    it is graded, $\DA = \oplus_{i\geq 0} \DA_i$, the map $\kk\to \DA_0$ is an isomorphism
    and $\DA$ generated by $\DA_1$ as a $\kk = A_0$-algebra. These intrinsic
    conditions are summarized by saying that $\DA$ has a presentation $\DA =
    \kk[\Dy_1, \ldots ,\Dy_r]/I$, where $I$ is a homogeneous ideal.
    We now note that $\grDA$ is standard graded.
    \begin{lemma}\label{ref:associatedStandardGraded:lem}
        Let $(\DA, \mm, \kk)$ be a finite algebra. Then $\grDA$ is a standard
        graded algebra.
    \end{lemma}
    \begin{proof}
        Clearly $\grDA$ is graded. The map $\kk \to A/\mm = \grDA_0$ is an
        isomorphism by assumption. Every homogeneous piece $(\grDA)_i =
        \mm^i/\mm^{i+1}$ is in the $\kk$-subalgebra generated by $(\grDA)_1 =
        \mm/\mm^2$, so that $\grDA$ is standard graded.
    \end{proof}
    The possible Hilbert
    functions of standard graded algebras are classified by Macaulay's growth theorem. Before stating it,
    we need to define binomial expansions. We follow the classical
    presentations, details are found in~\cite[Section~4.2]{BrunsHerzog}.

    Fix a positive integer $i$. For a positive integer $h$ there exist
    uniquely determined integers $a_{i} > a_{i-1} >  \ldots > a_{1} \geq 0$ such that
    \begin{equation}
        h = \binom{a_i}{i} + \binom{a_{i-1}}{i-1} +  \ldots  + \binom{a_1}{1}.
        \label{eq:binomialexpansion}
    \end{equation}
    Here we assume that $\binom{a_j}{j} = 0$ for $a_j < j$. We call the
    numbers $a_j$ the \emph{$i$-th binomial expansion of $h$}.
    These numbers can be determined by a greedy algorithm, choosing first
    $a_i$ largest possible, then $a_{i-1}$ etc. For $i$, $h$
    and~$a_j$'s determined as in Equation~\eqref{eq:binomialexpansion}, we
    define
    \newcommand{\binfunc}[2]{#1^{\langle #2\rangle}}%
    \begin{equation}
        \binfunc{h}{i} = \binom{a_i+1}{i+1} + \binom{a_{i-1}+1}{i} + \ldots +
        \binom{a_{1}+1}{2}.
        \label{eq:binomialshifted}
    \end{equation}

    \begin{theorem}[Macaulay's growth theorem, {\cite{macaulay_enumeration},
            \cite[p.~4.2.10]{BrunsHerzog}}]\label{ref:MacaulayGrowth:thm}
        Let $\DA$ be standard graded algebra and $H$ be its Hilbert function.
        Then
        \begin{equation}
            H(i+1) \leq \binfunc{H(i)}{i}\quad\mbox{for all } i.
            \label{eq:MacaulayBound}
        \end{equation}
    \end{theorem}
    Macaulay also proved that if $H:\mathbb{N} \to \mathbb{N}$ satisfies $H(0)
    = 1$ and $H(i+1) \leq \binfunc{H(i)}{i}$ for all $i$, then there exists a
    standard graded algebra with this Hilbert function. Therefore,
    Theorem~\ref{ref:MacaulayGrowth:thm} gives a full classification of Hilbert
    functions of standard graded algebras. By
    Lemma~\ref{ref:associatedStandardGraded:lem} also the Hilbert function of a
    local algebra satisfies Inequality~\eqref{eq:MacaulayBound} and every
    function $H:\mathbb{N}\to\mathbb{N}$ satisfying~\eqref{eq:MacaulayBound}
    and $H(0) = 1$ is a Hilbert function of a local algebra.

    \begin{corollary}\label{ref:nonincreasinghf:cor}
        Let $\DA$ be a standard graded $\kk$-algebra with Hilbert function $H$.
        If $i \geq 0$ is such that $H(i) \leq i$, then we have
        $H(i) \geq H(i+1) \geq H(i+2) \geq  \ldots$.
    \end{corollary}
    \begin{proof}
        In the $i$-th binomial equation of $H(i)$ each $a_j$ is either
        $j$ or $j-1$, thus $\binfunc{H(i)}{i} = H(i)$.
    \end{proof}

    Once the Macaulay bound is attained then it will also be attained for all
    higher degrees provided that no new generators of the ideal appear:
    \begin{theorem}[Gotzmann's Persistence Theorem, {\cite{gotzmann_persistence_theorem} or \cite[Theorem~4.3.3]{BrunsHerzog}}]
        \label{ref:Gotzmann:thm}
        Let $\DS$ by a polynomial ring, $I$ be a homogeneous ideal and $\DA =
        \DS/I$ be a standard graded algebra with Hilbert function $H$.
        If $i \geq 0$ is an integer such that $H(i+1) = H(i)^{\langle i
        \rangle}$ and $I$ is generated in degrees $\leq i$,
        then we have $H(j+1) = H(j)^{\langle j \rangle}$ for all $j \geq i$.
    \end{theorem}

In the following we will mostly use the following consequence of
Theorem~\ref{ref:Gotzmann:thm}, for which we introduce some (non-standard) notation. Let
$I \subseteq \DS = \kk[\Dx_1, \ldots ,\Dx_n]$ be a graded ideal in a polynomial ring and
$m\geq 0$. We say that $I$ is \emph{$m$-saturated} if for all $l\leq m$ and
$\sigma\in \DS_l$ the condition  $\sigma\cdot (\Dx_1, \ldots ,\Dx_n)^{m - l}\subseteq I$
implies $\sigma\in I$.

\begin{lemma}\label{ref:P1gotzmann:lem}
    \def\Dnn{\mathfrak{n}}%
    Let $\DS = \kk[\Dx_1, \ldots ,\Dx_n]$ be a polynomial ring with
    maximal ideal $\Dnn = (\Dx_1,
    \ldots ,\Dx_n)$. Let $I \subseteq \DS$ be a graded ideal and $A =
    \DS/I$. Suppose that $I$ is $m$-saturated for some
    $m\geq 2$.
    Then
    \begin{enumerate}
        \item if $H_A(m) = m+1$ and $H_A(m+1) = m+2$, then $H_A(l) = l+1$ for
            all $l\leq m$, in particular $H_A(1) = 2$.
        \item if $H_A(m) = m+2$ and $H_A(m+1) = m+3$, then $H_A(l) = l+2$ for
            all $l\leq m$, in particular $H_A(1) = 3$.
    \end{enumerate}
\end{lemma}

\begin{proof}
    1. First, if $H_A(l) \leq l$ for some $l < m$, then by Macaulay's Growth
    Theorem $H_A(m)\leq l < m+1$, a contradiction. So it suffices to prove
    that $H_A(l) \leq l+1$ for all $l < m$.

    Let $J$ be the ideal generated by elements of degree at most $m$ in $I$.
    We will prove that the graded ideal $J$ of $\DS$  defines a
    $\mathbb{P}^1$ linearly embedded into $\mathbb{P}^{n-1}$.

    Let $B = \DS/J$. Then $H_B(m) = m+1$ and $H_B(m+1) \geq m+2$. Since $H_B(m)
    = m+1 =
    \binom{m+1}{m}$, we have $H_B(m)^{\langle m\rangle} = \binom{m+2}{m+1} =
    m+2$ and by Theorem \ref{ref:MacaulayGrowth:thm} we get $H_B(m+1)\leq m+2$, thus $H_B(m+1) = m+2$. Then by
    Gotzmann's
    Persistence Theorem $H_B(l) = l+1$ for all $l > m$. This implies that the
    Hilbert polynomial of $\Proj B \subseteq \mathbb{P}^{n-1}$ is $h_B(t) = t+1$, so
    that $\Proj B \subseteq \mathbb{P}^{n-1}$ is a linearly embedded $\mathbb{P}^1$. In
    particular the Hilbert function and Hilbert polynomial of $\Proj B$ are
    equal
    for all arguments.
    By assumption, we have $J_l = J^{sat}_l$ for all $l < m$. Then  $H_{A}(l)
     = H_{\DS/J}(l) = H_{\DS/J^{sat}}(l) = l+1$ for all $l < m$ and the claim of
    the lemma follows.

    2. The proof is similar to the above one; we mention only the points,
    where it changes. Let $J$ be the ideal generated
    by elements of degree at most $m$ in $I$ and $B = \DS/J$. Then $H_B(m) = m
    +2 = \binom{m+1}{m} + \binom{m-1}{m-1}$, thus $H_B(m+1) \leq
    \binom{m+2}{m+1} +
    \binom{m}{m} = m+3$ and $B$ defines a closed subscheme of $\mathbb{P}^{n-1}$ with
    Hilbert polynomial $h_B(t) = t+2$. There are two isomorphism types of such
    subschemes: $\mathbb{P}^1$ union a point and $\mathbb{P}^1$ with an
    embedded double point. One checks that for these schemes the Hilbert
    polynomial is equal to the Hilbert function for all arguments and then
    proceeds as in the proof of Point 1.
\end{proof}

\begin{remark}\label{ref:GorensteinSaturated:rmk}
    \def\Dnn{\mathfrak{m}_S}%
    \def\Dan#1{\Ann{#1}}%
    If $\DA = \DS/I$ is a finite graded Gorenstein algebra with socle
    concentrated in degree $d$, then
    $\DA$ is $m$-saturated for every $m\leq d$.
    Indeed, fix an $m\leq d$ and suppose that $\sigma\in \DS_l$ is such that $\sigma(\Dx_1, \ldots
    ,\Dx_n)^{m-l} \subset I$. Then also $\sigma(\Dx_1, \ldots
    ,\Dx_n)^{d-l} \subset I$. Let $\bar{\sigma}\in A$ be the image of
    $\sigma$, then $\bar{\sigma}(\Dx_1, \ldots
    ,\Dx_n)^{d-l} = 0$. Since the socle of $A$ is concentrated in degree $d$,
    this implies $\bar{\sigma}\DS \cap \socleA = 0$. But then $\bar{\sigma} =
    0$ because $\socleA$ intersect every nonzero ideal of $A$, see the
    discussion below~Definition~\ref{ref:socle:def}.
\end{remark}

\section{Hilbert function of a local Gorenstein
algebra}\label{ssec:hilbertfunctionGorenstein}

In contrast to the general case, the classification of Hilbert functions of
\emph{Gorenstein} algebras is not known. However a well-developed theory
exists. We first discuss standard graded algebras. As usually with Gorenstein
property, the slogan is duality and hence symmetry. To define the center
of symmetry, we first introduce the notion of socle degree.
\begin{definition}
    The \emph{socle degree} of a finite local Gorenstein algebra $\DA$ is the
    largest $d$ such that $H_{\DA}(d) \neq 0$.
\end{definition}
We will see that necessarily $H_{\DA}(d) = 1$.

\begin{proposition}[Symmetry of the Hilbert
    function]\label{ref:symmetryGorensteingraded:prop}
    Let $\DA$ be a standard graded Gorenstein algebra with Hilbert function $H$. Let $d$
    be the socle degree of $\DA$. Then $H(i) = H(d-i)$ for all $0\leq i\leq
    d$.
\end{proposition}
\begin{proof}
    The subspace $\DA_d$ is nonzero and annihilated by the maximal ideal of
    $\DA$, so it is the socle of $\DA$. Then the pairing $A_i \times
    A_{d-i} \to A_d  \simeq \kk$ is non-degenerate by
    Corollary~\ref{ref:dualgens:cor} and
    Proposition~\ref{ref:perfectpairing:prop}, hence the claim.
\end{proof}

\begin{remark}
Stanley~\cite{stanley_hilbert_functions_of_graded_algebras} gave the
following characterisation of Hilbert functions $H$ of graded Gorenstein
algebras of socle degree $d$ under the assumption $H(1)\leq 3$. He proved that $H$ is a Hilbert
function of such algebra if and only if $H(0)= 1$ and $H(d-i) = H(i)$ for all $i$ and the
sequence
\begin{equation}\label{eq:tmpfirsthalf}
    H(0), H(1) - H(0),  \ldots , H(t) - H(t-1)
\end{equation}
with $t = \lfloor \frac{d}{2} \rfloor$ consists of nonnegative integers and
satisfies Macaulay's Bound~\eqref{eq:MacaulayBound}. He also showed the
necessity of assumption $H(1) \leq 3$ by giving an example of a graded Gorenstein algebra with
Hilbert function $(1, 13, 12, 13, 1)$; then~\eqref{eq:tmpfirsthalf} becomes
$(1, 12, -1)$, which contradicts the assumption $H(2) - H(1) \geq 0$.
\end{remark}

Now we investigate the Hilbert function $H$ of a local Gorenstein algebra $(\DA,
\mm, \kk)$. This
Hilbert function need not be symmetric (Example~\ref{ex:trivialcubicHf}), however it admits a
decomposition into symmetric factors. The decomposition is canonically
obtained from $\DA$, but decompositions may be different for different
algebras with equal Hilbert functions.

\newcommand{\Dmmperp}[1]{(0:\mm^{#1})}%
\newcommand{\Loevy}{L\"ovy}%
Let $d$ be the socle degree of $\DA$.  Let us denote by $(0:I)$ the
annihilator of an ideal $I \subset A$ and assume that $\mm^i = A$ for all
$i\leq 0$. It follows from Proposition~\ref{ref:symmetryGorensteingraded:prop}
that for \emph{graded} $\DA$ we have $\Dmmperp{d-i} = \mm^{i+1}$. However this is
not true for all local algebras, see~Example~\ref{ex:trivialcubicHf}.
We thus obtain two canonical filtrations on $\DA$. One is a descending filtration
\[A \supset \mm \supset \mm^2  \ldots \supset \mm^d \supset \{0\}\]
by powers of $\mm$ and the other is an
ascending filtration
\[\{0\} \subset \Dmmperp{} \subset \Dmmperp{2} \subset  \ldots \subset
\Dmmperp{d+1} = A,\] called the \Loevy{} filtration.
We begin by relating the \Loevy{} filtration to duality.
\begin{lemma}\label{ref:dualitycomparison:lem}
    Let $(\DA, \mm, \kk)$ be a local Gorenstein algebra and $I \subset A$ be
    an ideal. Choose any dual generator $f\in \canA$ and consider the
    associated pairing $A \times A \ni (a, b)\to f(ab)\in \kk$. Then
    $I^{\perp} = (0:I)$.
\end{lemma}
\begin{proof}
    Clearly $(0:I) \subset I^{\perp}$. On the other hand, if $a\in A$ does
    not annihilate $I$, then $aI$ is a nonzero ideal, thus $aI \cap \socleA
    \neq 0$ by the discussion below Definition~\ref{ref:socle:def}, so
    $f(aI)\neq \{0\}$ by Corollary~\ref{ref:dualgens:cor} and hence $a\not\in I^{\perp}$.
\end{proof}

\begin{lemma}\label{ref:loewyHilbertfunc:lem}
    Let $(\DA, \mm, \kk)$ be a local Gorenstein $\kk$-algebra. Then
    \[H_{\DA}(i) = \dimk \frac{\Dmmperp{i+1}}{\Dmmperp{i}}.\]
\end{lemma}
\begin{proof}
    Fixing any pairing as in Lemma~\ref{ref:dualitycomparison:lem} we have
        \[
        \left(\frac{\mm^i}{\mm^{i+1}}\right)^{\vee}  \simeq
        \frac{\Dmmperp{i+1}}{\Dmmperp{i}}.\qedhere
        \]
\end{proof}
The result of Lemma~\ref{ref:loewyHilbertfunc:lem} may be interpreted as a
duality between subquotients of the two filtrations.
Let $d$ be the socle degree of $\DA$.
\newcommand{\iaq}[2]{Q(#1)_{#2}}%
\newcommand{\iac}[2]{C(#1)_{#2}}%
\newcommand{\iainter}[2]{\mm^{#1}\cap \Dmmperp{#2}}%
\newcommand{\iasum}[2]{\mm^{#1} + \Dmmperp{#2}}%
\newcommand{\iaqvect}[1]{Q(#1)}%
\newcommand{\iacvect}[1]{C(#1)}%
We introduce the summands $\iacvect{a} = \bigoplus_i \iac{a}{i} \subset
\grDA$ and $\iaqvect{a} = \bigoplus_i \iaq{a}{i}$ by the following formulas
\begin{equation}
    \iac{a}{i} := \frac{\iainter{i}{d+1-a-i}}{\iainter{i+1}{d+1-a-i}},
    \label{eq:Csummands}
\end{equation}
\begin{equation}
    \iaq{a}{i} := \frac{\iac{a}{i}}{\iac{a+1}{i}} =
    \frac{\iainter{i}{d+1-a-i}}{\iainter{i+1}{d+1-a-i} +
    \iainter{i}{d-a-i}}.
    \label{eq:Qsummands}
\end{equation}
Since $\mm\cdot\Dmmperp{d+1-a-i} \subset \Dmmperp{d+1-a-(i+1)}$, each $\iacvect{a} \subset \grDA$ is an ideal, so that
$\iaqvect{a}$ are $\grDA$-modules.
If $i > d-a$ then $\Dmmperp{d-a-i} = \Dmmperp{d+1-a-i} = A$ and so
$\iaq{a}{i} = 0$. Also if $i < 0$ then $\iaq{a}{i} = 0$.
Therefore $\iaqvect{a}$ may be interpreted as a vector of length $d-a$.
The following result proves that this vector is symmetric up to taking duals.

\begin{lemma}\label{ref:Qduality:lem}
    Let $(\DA, \mm, \kk)$ be a local Gorenstein algebra and
    $\iaqvect{a}$ be defined as in~\eqref{eq:Qsummands}. Then
    $\iaq{a}{i}^{\vee} = \iaq{a}{d-a-i}$
    naturally. Hence $\iaqvect{a}^{\vee}  \simeq \iaqvect{a}$ as a $\grDA$-module.
\end{lemma}
\begin{proof}
    We note that if $M, N, P$ are $A$-modules and $N \subset M$, then
    \[\frac{M+P}{N+P}  \simeq \frac{M}{M\cap(N+P)} = \frac{M}{N+M\cap P}.\]
    Fix a dual generator of $\DA$ and hence a perfect pairing $A\times A\to \kk$. Lemma~\ref{ref:dualitycomparison:lem}
    implies that $\Dmmperp{d+1-a-i}^{\perp} = \mm^{d+1-a-i}$ and
    $\pp{\mm^i}^{\perp} = \Dmmperp{i}$ so that
    \[
        \iac{a}{i}^{\vee} =
        \pp{\frac{\iainter{i}{d+1-a-i}}{\iainter{i+1}{d+1-a-i}}}^{\vee} =
        \frac{\Dmmperp{i+1} + \mm^{d+1-a-i}}{\Dmmperp{i} + \mm^{d+1-a-i}}
        \simeq \frac{\Dmmperp{i+1}}{\Dmmperp{i} + \mm^{d+1-a-i}\cap
        \Dmmperp{i+1}}.
    \]
    Now $\iaq{a}{i}$ is the cokernel of $\iac{a+1}{i}\to \iac{a}{i}$, so
    $\iaq{a}{i}^{\vee}$ is the kernel of the natural map
    \[
        \frac{\Dmmperp{i+1}}{\Dmmperp{i} + \mm^{d+1-a-i}\cap \Dmmperp{i+1}}
        \to
        \frac{\Dmmperp{i+1}}{\Dmmperp{i} + \mm^{d-a-i}\cap \Dmmperp{i+1}}.
    \]
    Therefore
    \[
        \iaq{a}{i}^{\vee}  \simeq \frac{\Dmmperp{i} + \mm^{d-a-i}\cap
        \Dmmperp{i+1}}{\Dmmperp{i} + \mm^{d+1-a-i}\cap \Dmmperp{i+1}}  \simeq \frac{\mm^{d-a-i}\cap
        \Dmmperp{i+1}}{\mm^{d-a-i}\cap\Dmmperp{i} + \mm^{d+1-a-i}\cap
        \Dmmperp{i+1}},
    \]
    which is exactly $\iaq{a}{d-i-a}$.
\end{proof}

Let us note that $\iaq{a}{0} = \iaq{a}{d-a} = 0$ for all $a > 0$. Indeed,
since $a >0$ we have $d+1-a \leq d$. Therefore
$\Dmmperp{d+1-a} \subset \mm$ and $\iainter{0}{d+1-a} =
\iainter{1}{d+1-a}$. Thus $\iaq{a}{0} = 0$; then $\iaq{a}{d-a} = 0$ follows by symmetry.
Also $\iaq{0}{0} = A/\mm = \kk$ and $\iaq{0}{d}  \simeq \Dmmperp{}  \simeq \kk$.
\begin{example}[Nonzero $\iaq{i}{j}$ for socle degree $d =
    3$]\label{ex:tmpQsoclethree}
    If $d = 3$, then $\iaqvect{0} = (\kk,
    \iaq{0}{1}, \iaq{0}{1}^{\vee}, \kk)$ and $\iaqvect{1} = (0, \iaq{0}{2}, 0)$.
\end{example}

\newcommand{\Dhd}[2]{\Delta_{#1}\left(#2\right)}
\newcommand{\Dhdvect}[1]{\Delta_{#1}}
\begin{definition}\label{ref:symmetricdecomposition:def}
    Let $(\DA, \mm, \kk)$ be a local Gorenstein algebra of socle degree $d$. The \emph{symmetric
    decomposition of Hilbert function of $A$} is a tuple of $d-1$ vectors
    $\Dhdvect{A, a} = \Dhdvect{a}$, for $a = 0, 1,  \ldots , d-2$, defined by
    \begin{equation}
        \Dhd{a}{i} := \dimk \iaq{a}{i}.
        \label{eq:deltadefn}
    \end{equation}
    We call $\Dhdvect{a}$ the \emph{$a$-th symmetric summand} and identify it
    with the vector $(\Dhd{a}{0}, \ldots , \Dhd{a}{d-a})$.
\end{definition}
By Lemma~\ref{ref:Qduality:lem} the vector $\Dhdvect{a}$ is symmetric around
$d-a$; we have $\Dhd{a}{i} = \Dhd{a}{d-a-i}$.
We now briefly justify why $\Dhdvect{\bullet}$ form a decomposition of the
Hilbert function of $\DA$.

\begin{lemma}\label{ref:hilbertfromdecomposition:lem}
    Let $(\DA, \mm, \kk)$ be a local Gorenstein algebra of socle degree $d$
    and with Hilbert function $H$.
    Then $H(i) = \sum_{a=0}^{d-i} \Dhd{a}{i}$.
\end{lemma}
\begin{proof}
    The spaces $\iacvect{a}_i$ form a filtration of
    $\mm^i/\mm^{i+1}$, so that $H(i) = \sum_{a=0}^{\infty} \Dhd{a}{i}$, but
    $\iaq{a}{i} \neq 0$ only when $a+i\leq d$, so it is enough to sum over $a\leq d-i$.
\end{proof}

The following example shows how the mere existence of the symmetric
decomposition forces some constraints on the Hilbert function of a Gorenstein
algebra. One obvious constraint is that $H(d) = \dimk \iaq{0}{d} =
\dimk\iaq{0}{0}^{\vee} = 1$.
\begin{example}\label{ex:hilbertfunctioncubics}
    Let $(\DA, \mm, \kk)$ be a local algebra of socle degree three. Then by
    Example~\ref{ex:tmpQsoclethree} we have $H_{\DA} = (1,
    \Dhd{0}{1}+\Dhd{1}{1}, \Dhd{0}{1}, 1)$, so $H_{\DA}(1) \geq H_{\DA}(2)$.
\end{example}
We may restrict this function even further, by restricting possible
$\Dhdvect{a}$. The key observation is the following lemma.
\begin{lemma}\label{ref:sumsofdhdaraOseqence:lem}
    Let $(\DA, \mm, \kk)$ be a finite local Gorenstein algebra and let
    $\Dhdvect{\bullet}$ be the symmetric decomposition of its Hilbert
    function. Fix $i\geq 0$. The partial sum $\sum_{a=0}^i\Dhdvect{i}$
    is the Hilbert function of a standard graded algebra
    $\grDA/\iacvect{i+1}$ defined in~\eqref{eq:Qsummands}.
    In particular $\sum_{a=0}^i\Dhdvect{i}$ satisfies Macaulay's
    Bound~\eqref{eq:MacaulayBound}.
\end{lemma}

\begin{proof}
    Immediate, arguing as in Lemma~\ref{ref:hilbertfromdecomposition:lem}.
\end{proof}

    \begin{example}
        From Lemma~\ref{ref:sumsofdhdaraOseqence:lem}
        it follows that there does not exist a finite local Gorenstein algebra
        $\DA$ with Hilbert function decomposition
            \[\begin{matrix}
                \Dhdvect{0}=&(1,& 1,& 1,& 1, & 1, & 1)\\
                \Dhdvect{1}=&(0,& 0,& 1, & 0, & 0)\\
                \Dhdvect{2}=&(0,& 0,& 0, &0)\\
                \Dhdvect{3}=&(0,& 1,& 0).
            \end{matrix}\]
        Indeed, $\Dhdvect{0} + \Dhdvect{1} = (1, 1, 2, 1, 1, 1)$ violates
        Corollary~\ref{ref:nonincreasinghf:cor}.
        Note that $H_{\DA} = (1, 2, 2, 1, 1, 1)$ seems possible to obtain and
        indeed there exist Gorenstein algebras with such function and
        decomposition
            \[\begin{matrix}
                \Dhdvect{0}=&(1,& 1,& 1,& 1, & 1, & 1)\\
                \Dhdvect{1}=&(0,& 0,& 0, & 0, & 0)\\
                \Dhdvect{2}=&(0,& 1,& 1, &0)\\
                \Dhdvect{3}=&(0,& 0,& 0).
            \end{matrix}\]
    \end{example}

    \begin{example}
        The Hilbert function $H = (1, 3, 3, 2, 1, 1)$ has exactly two
        possible decompositions:
            \[\begin{matrix}
                \Dhdvect{0}=&(1,& 1,& 1,& 1, & 1, & 1)\\
                \Dhdvect{1}=&(0,& 1,& 1, & 1, & 0)\\
                \Dhdvect{2}=&(0,& 1,& 1, &0)\\
                \Dhdvect{3}=&(0,& 0,& 0).
            \end{matrix}\]
            \[\begin{matrix}
                \Dhdvect{0}=&(1,& 1,& 1,& 1, & 1, & 1)\\
                \Dhdvect{1}=&(0,& 1,& 2, & 1, & 0)\\
                \Dhdvect{2}=&(0,& 0,& 0, &0)\\
                \Dhdvect{3}=&(0,& 1,& 0).
            \end{matrix}\]
            We will later see in
            Examples~\ref{ex:133211firstdecomposition},~\ref{ex:133211snddecomposition}
            see that both decompositions are possible. In
            Proposition~\ref{ref:squares:prop}
            we will also see that $\Dhdvect{3} \neq (0, 0, 0)$ in the second
            decomposition constraints the corresponding algebra.
            This example is treated in depth
            in~\cite[Section~4B]{iarrobino_associated_graded}.
    \end{example}

    \section{Betti tables, Boij and S\"oderberg theory}\label{ssec:betti}

    The study of Betti tables is indispensable for analysis of graded
    algebras. We will use it only sparsely, mainly in
    Section~\ref{ssec:secants}, so we content ourselves with
    an informal discussion. An excellent reference is~\cite{eisenbud:syzygies}.

    \newcommand{\Dfreeres}{\mathcal{F}}%
    Let $(\DA, \mm, \kk)$ be a finite $\kk$-algebra of socle degree $d$. Suppose
    that $\DA = \DS/I$ is presented as a quotient of polynomial ring $\DS$ of
    dimension $n$ by a homogeneous ideal.
    If $\DA$ is Gorenstein and its Hilbert function is symmetric
    (Proposition~\ref{ref:symmetryGorensteingraded:prop}) and, as we now
    recall, this is a consequence of
    the symmetry of its resolution.
    First, by Hilbert's theorem, the $\DS$-module $\DS/I$ has a minimal graded free
    resolution of length $n$:
    \newcommand{\betti}[1]{\beta_{#1}}%
    \begin{equation}\label{eq:freeresolution}
        0\to \bigoplus_i \DS(-i)^{\oplus\betti{n, i}} \to \bigoplus_i \DS(-i)^{\oplus\betti{n-1,
        i}} \to  \ldots \to \bigoplus_i \DS(-i)^{\oplus\betti{1, i}} \to \DS.
    \end{equation}
    Here $M(i)$ denotes the module $M$ with grading shifted by $i$.
    Since $\DA$ is finite, it is Cohen-Macaulay, so $\Ext^i(\DA, \DS(-n)) = 0$ for
    all $i < n$. Moreover $\Ext^n(\DA, \DS(-n))  \simeq  \canA(d)$,
    see~\cite[Corollary~21.16]{Eisenbud} for details. Therefore, if we denote
    complex~\eqref{eq:freeresolution} by $\Dfreeres$, then $\Hom(\Dfreeres,
    S(-n-d))$ is the minimal free resolution of $\canA$. In particular, $\sum_i \betti{n,
    i}$ is the minimal number of generators of $\canA$; by
    Proposition~\ref{ref:Gorchar:eis:prop} it is equal to one if and only if
    $\DA$ is Gorenstein.

    Suppose now and for the remaining part of this section that $\DA$ is
    Gorenstein. Then $\canA  \simeq \DA$ so that $\Hom(\Dfreeres, S(-n-d))
    \simeq \Dfreeres$ by uniqueness.
    This implies that the Betti table
    \[
        \begin{bmatrix}
            1            & \betti{1, 1} &  \betti{2, 2} &  \ldots & \betti{n-1,n-1} & 0\\
            0            & \betti{1, 2} &  \betti{2, 3} & \ldots  & \betti{n-1,n}  & 0\\
                         &              &     \ldots    &         &                 &\\
            0            & \betti{1, d+1} & \betti{2,d+2}&  \ldots & \betti{n-1,n+d-1} & 1\\
        \end{bmatrix}
    \]
    is symmetric around its center.
    \begin{example}
        Let $\DA = \kk[\Dx_1, \ldots ,\Dx_4]/I$
        with $I$ generated by $\Dx_i\Dx_j, \Dx_i^3 - \Dx_j^3$ for $i\neq j$.
        We compute its Betti table
        \[
            \begin{bmatrix}
                1 & 0 & 0 & 0 & 0\\
                0 & 6 & 8 & 3 & 0\\
                0 & 3 & 8 & 6 & 0\\
                0 & 0 & 0 & 0 & 1\\
            \end{bmatrix}
        \]
        and see that in the last column there is a single one, so $\DA$ is
        Gorenstein and that indeed the table is symmetric.
    \end{example}

    Boij-S\"oderberg theory gives a beautiful description of the cone of all
    Betti tables of graded quotients of fixed $\DS$,
    see~\cite{boij_Soederberg}.
    In the following we never use this theory explicitly, but below we give an
    example showing how it restricts the possible shapes of Betti tables of
    finite Gorenstein algebras (for another example,
    see~\cite[Section~4.1]{erman_velasco_syzygetic_smoothability}).

    \begin{example}
        Let $\DA = \kk[\Dx_1, \ldots ,\Dx_7]/I$ be a finite graded Gorenstein algebra with
        $H_{\DA} = (1, 7, 7, 1)$. Then there exist $a, b, c\in
        \mathbb{Q}_{\geq 0}$
        with $a+b+c\leq 1$, such that the Betti table of $\DA$ is
        \[
            \begin{bmatrix}
                1 & 0  & 0        & 0                       & 0 & 0 & 0 & 0\\
                0 & 21 & 64 + 16a & 70+70a + 35b            & \frac{224}{5}(3a+2b+c)  & 70a+35b  &  16a & 0\\
                0 & 16a& 70a + 35b& \frac{224}{5}(3a+2b+c)  & 70+70a + 35b  &
                64 + 16a  & 21  & 0\\
                0 & 0  & 0        & 0                       & 0 & 0 & 0 & 1\\
            \end{bmatrix}.
        \]
    \end{example}

\chapter{Macaulay's inverse systems
(apolarity)}\label{sec:macaulaysinversesystems}

    \newcommand{\myN}{\mathbb{N}}%
    \newcommand{\DmmS}{\mathfrak{m}_S}%
    \newcommand{\Ddual}[1]{{#1}^{\vee}}%
    \newcommand{\ip}[2]{\left\langle #1, #2\right\rangle}%
    \newcommand{\ord}[1]{\operatorname{ord}\left(#1\right)}%

    So far we have analysed finite $\kk$-algebras abstractly. Now we switch to
    embedded setting; we consider finite quotients of a polynomial ring
    $\DS$
    over $\kk$.  In fact we restrict to finite local algebras, so we consider
    finite quotients of a power series ring $\DShat$, the
    completion of $\DS$.
    Macaulay's inverse systems view this situation through a dual setting.
    Namely, for each finite quotient $\DShat/I$ we have $\omega_{\DShat/I} =
    \Hom_{\kk}(\DShat/I, \kk) \subset \Hom_{\kk}(\DShat, \kk)$.
    We analyse the generators of $\omega_{\DShat/I}$ and, more generally, the
    action of $\DShat$ on $\Hom_{\kk}(\DShat, \kk)$.
    In our presentation we closely
    follow~\cite{Jel_classifying}.

    In the first two sections we develop the theory of $\DShat$ action on $\Hom_{\kk}(\DShat,
    \kk)$ and the theory of inverse systems. In Section~\ref{ssec:Gorensteininapolarity} we explain, how
    this theory gives examples and even classifies Gorenstein quotients
    of $\DShat$.

\section{Definition of contraction action}\label{ssec:contraction}
By $\myN$ we denote the set of non-negative integers.  Let $\DShat$ be a power
series ring over $\kk$ of dimension $\dim \DShat = n$ and let $\DmmS$ be its
maximal ideal.
By $\ord{\sigma}$ we denote the order of a non-zero $\sigma\in \DShat$ i.e.~the largest $i$
such that $\sigma\in \DmmS^i$. Then $\ord{\sigma} = 0$ if and only if
$\sigma$ is invertible.
Let $\Ddual{\DShat} = \Homthree{\kk}{\DShat}{\kk}$ be the space of functionals on $\DShat$.
We denote the pairing between $\DShat$ and $\Ddual{\DShat}$ by
\[
    \ip{-}{-}: \DShat \times \Ddual{\DShat} \to \kk.
\]

\begin{definition}\label{ref:contraction:propdef}
    The dual space $\DP \subset \Ddual{\DShat}$ is the linear subspace of functionals
    eventually equal to zero:
    \[
        \DP = \left\{ f\in \Ddual{\DShat}\ |\ \forall_{D\gg 0}\ \ip{\DmmS^D}{f} = 0\right\}.
    \]
    On $\DP$ we have a structure of $\DShat$-module via precomposition:
    for every $\sigma\in \DShat$ and $f\in \DP$ the element $\sigma\hook f\in
    \DP$ is defined via the equation
    \begin{equation}\label{eq:hookvsip}
        \ip{\tau}{\sigma\hook f} = \ip{\tau\sigma}{f}\quad \mbox{for
        every}\ \tau\in \DShat.
    \end{equation}
    This action is called \emph{contraction}.
\end{definition}
    We will soon equip $\DP$ with topology and a structure of a ring
    (Definition~\ref{ref:dividedpows:def}), but its vector space structure is
    sufficient for most purposes.

    The existence of contraction action is a special case of the following
    construction, which is basic and foundational for our approach.
    Let $L\colon \DShat \to \DShat$ be a $\kk$-linear map. Assume that $L$ is
    $\DmmS$-adically continuous: there is sequence of integers $o_i$
    such that $L(\DmmS^{i}) \subset \DmmS^{o_i}$ for all $i$ and
    $\lim_{i\to \infty} o_i = \infty$.
    Then the dual map $\Ddual{L}:\Ddual{\DShat}\to \Ddual{\DShat}$ restricts to
    $\Ddual{L}:\DP \to \DP$. Explicitly, $\Ddual{L}$ is given by the equation
    \begin{equation}
        \ip{\tau}{\Ddual{L}(f)} = \ip{L(\tau)}{f}\quad \mbox{for
        every}\ \tau\in \DShat,\ f\in \DP.
        \label{eq:dual}
    \end{equation}
    To obtain contraction with respect to $\sigma$ we use $L(\tau) =
    \sigma\tau$, the multiplication by $\sigma$. Later in this thesis we will also
    consider maps $L$ which are automorphisms or derivations of $\DShat$.

\renewcommand{\aa}{\mathbf{a}}%
\newcommand{\bb}{\mathbf{b}}%
\newcommand{\xx}{\mathbf{x}}%
\newcommand{\DPel}[2]{#1^{[#2]}}%
To get a down to earth description of $\DP$, choose $\Dx_1, \ldots
,\Dx_n\in \DShat$ such that $\DShat = \kk[[\Dx_1, \ldots ,\Dx_n]]$.
Write $\Dx^{\aa}$ to denote $\Dx_1^{a_1} \ldots
\Dx_n^{a_n}$. For every $\aa\in \myN^{n}$ there is a unique element
$\DPel{\xx}{\aa}\in \DP$ dual to $\Dx^{\aa}$, given by
\[
    \ip{\Dx^{\bb}}{\DPel{\xx}{\aa}} = \begin{cases}
        1 & \mbox{if }\aa = \bb\\
        0 & \mbox{otherwise.}
    \end{cases}
\]
Additionally, we define $x_i$ as the functional dual to $\Dx_i$, so that $x_i
= \DPel{\xx}{(0, \ldots 0, 1, 0,  \ldots , 0)}$ with one on $i$-th position.
Let us make a few remarks:
\begin{enumerate}
    \item The functionals $\DPel{\xx}{\aa}$ form a basis of $\DP$. We have a
        natural isomorphism
        \begin{equation}
            \Ddual{\DP} = \DShat.
            \label{eq:dualdual}
        \end{equation}
    \item The contraction action is given by the formula
        \[
            \Dx^{\aa} \hook \DPel{\xx}{\bb} =
            \begin{cases}
                \DPel{\xx}{\bb - \aa} & \mbox{if } \bb\geq \aa,\mbox{ that is, }\forall_{i}\ b_i \geq a_i\\
                0 & \mbox{otherwise}.
            \end{cases}
        \]
        Therefore our definition agrees with the one
        from~\cite[Definition~1.1,
        p.~4]{iarrobino_kanev_book_Gorenstein_algebras}.
\end{enumerate}
We say that $\DPel{\xx}{\aa}$ has degree $\sum a_i$. We will speak
about constant forms, linear forms, (divided) polynomials of bounded degree
etc. Note that the forms of degree $d$ are just those elements of $\Ddual{\DS}
\subset \Ddual{\DShat}$ which are perpendicular to all forms of degree $\neq
d$. Thus this notion is independent of choice of basis. However it depends on
$\DS$, so it is not intrinsic to $\DShat$. What is intrinsic is the space
$\DP_{\leq d}$; indeed it is the perpendicular of $\DmmS^{d+1}$.

We endow $\DP$ with a topology, which is the Zariski topology of an affine
space, in particular $\DP_{\leq d}$ inherits the usual Zariski topology of
finite dimensional affine space. This topology will be used when speaking about
general polynomials and closed orbits.

Now we will give a ring structure on $\DP$.
%It will be used
%crucially in Proposition~\ref{ref:dualautomorphism:prop}.
For multi-indices $\aa, \bb\in \myN^{n}$ we define $\aa! = \prod (a_i!)$, $\sum
\aa = \sum a_i$ and $\binom{\aa + \bb}{\aa} = \prod_i
\binom{a_i + b_i}{a_i} = \binom{\aa + \bb}{\bb}$.
\begin{definition}\label{ref:dividedpows:def}
    We define multiplication on $\DP$ by
    \begin{equation}\label{eq:divpowmult}
        \DPel{\xx}{\aa} \cdot \DPel{\xx}{\bb} := \binom{\aa + \bb}{\aa}
        \DPel{\xx}{\aa + \bb}.
    \end{equation}
    In this way $\DP$ is a \emph{divided power ring}. We denote it by $\DP =
    \kdp[x_1, \ldots ,x_n]$.
\end{definition}

The multiplicative structure on $\DP$ can be defined in a coordinate-free
manned using a natural comultiplication on $\DS$.
Since $\Spec \DS$ is an affine space, it has an group scheme structure and in
particular an addition map $\tmpAdd:\Spec\DS \times
\Spec\DS \to \Spec\DS$, which induces a comultiplication homomorphism
$\tmpAdd^{\#}:\DS \to \DS \otimes \DS$ and in turn a dual map
$\tmpAdd^{\vee}:(\DS \otimes \DS)^{\vee} \to \DS^{\vee}$ which can be
restricted to $\Ddual{\tmpAdd}:\DS^{\vee} \otimes \DS^{\vee} \to \DS^{\vee}$.
Explicitly, $\Ddual{\tmpAdd}(x) = 1\tensor x + x\tensor 1$ for all $x\in
\DS_1$, so $\Ddual{\tmpAdd}(\mm^d) \subset \sum \mm^i \tensor \mm^{d-i}$.
Therefore, if $f\in \DP$ is annihilated by $\mm^{d}$ and $g\in\DP$ by $\mm^e$,
then $\Ddual{\tmpAdd}(f, g)$ is annihilated by $\mm^{d+e}$. Hence
$\Ddual{\tmpAdd}$ restrict to a map $\DP \otimes \DP \to \DP$, which in
coordinates in given by~\eqref{eq:divpowmult}.
We refer to \cite[\S A2.4]{Eisenbud} for
details in much greater generality. See Ehrenborg, Rota~\cite{Ehrenborg} for
an interpretation in terms of Hopf algebras. Once more, we stress that the
multiplicative structure on $\DShat$ depends on $\DS$.
\begin{example}\label{ex:dividedpows}
    Suppose that $\kk$ is of characteristic $p > 0$. Then $\DP$ is not isomorphic to
    a polynomial ring. Indeed, $(x_1)^p =
    p!\DPel{x_1}{p} = 0$. Moreover, $\DPel{x_1}{p}$
    is not in the subring generated by $x_1, \ldots ,x_n$.
\end{example}

\newcommand{\ithpartial}[2]{#1^{(#2)}}%
 For an element
 $\sigma\in \DShat = \kk[\Dx_1, \ldots ,\Dx_n]$ denote its $i$-th
partial derivative by $\ithpartial{\sigma}{i}\in \DShat$, for example $\ithpartial{(\Dx_1^2)}{1} = 2\Dx_1$ and
$\ithpartial{(\Dx_1^2)}{2} = 0$.
Note that the linear forms of $\DShat$ act on $\DP$ as derivatives.
Therefore we can interpret $\DShat$ as lying inside the ring of differential
operators on $\DP$.  The following related fact is very useful in
computations.
\begin{lemma}\label{ref:commutator:lem}
    Let $\sigma\in \DShat$. For every $f\in \DP$ we have
    \begin{equation}\label{eq:weylalgebra}
        \sigma\hook (x_i\cdot f) - x_i\cdot (\sigma\hook f) =
        \ithpartial{\sigma}{i}\hook f.
    \end{equation}
\end{lemma}

\begin{proof}
    Since the formula is linear in $\sigma$ and $f$ we may assume these are
    monomials. Let $\sigma = \Dx_i^{r}\tau$, where $\Dx_i$ does not appear in
    $\tau$. Then $\ithpartial{\sigma}{i} = r\Dx_i^{r-1}\tau$.
    Moreover $\tau\hook (x_i \cdot f) = x_i\cdot (\tau\hook f)$. By replacing
    $f$ with $\tau\hook f$, we reduce to the case $\tau = 1$, $\sigma =
    \Dx_i^{r}$.

    Write $f = \DPel{x_i}{s}g$ where $g$ is a monomial in variables other
    than $x_i$. Then $x_i\cdot f = (s+1)\DPel{x_i}{s+1} g$ according
    to~\eqref{eq:divpowmult}.
    If $s+1 < r$ then both sides of~\eqref{eq:weylalgebra} are
    zero. Otherwise
    \[
        \sigma\hook (x_i \cdot f) = (s+1)\DPel{x_i}{s+1-r}g,\ \
        x_i\cdot (\sigma\hook f) = x_i\cdot\DPel{x_i}{s-r}g =
        (s-r+1)\DPel{x_i}{s-r+1}g,\ \ \ithpartial{\sigma}{i}\hook f =
        r\DPel{x_i}{s-(r-1)}g,
    \]
    so Equation~\eqref{eq:weylalgebra} is valid in this case also.
\end{proof}

\begin{remark}
    Lemma~\ref{ref:commutator:lem} applied to $\sigma = \Dx_i$ shows that
    $\Dx_i\hook (x_i\cdot f) - x_i\cdot (\Dx_i\hook f) = f$. This can be
    rephrased more abstractly by saying that $\Dx_i$ and $x_i$
interpreted as linear operators on $\DP$ generate a Weyl algebra. Since these
operators commute with other $\Dx_j$ and $x_j$, we see that $2n$ operators
$\Dx_i$ and $x_i$ for $i=1, \ldots ,n$ generate the $n$-th Weyl algebra.
\end{remark}

Example~\ref{ex:dividedpows} shows that $\DP$ with its ring structure has
certain properties distinguishing it from the polynomial ring, for example it contains nilpotent elements.
Similar phenomena do not occur in degrees lower than the characteristic or in
characteristic zero, as we show in Proposition~\ref{ref:divpowssmallarepoly:prop} and
Proposition~\ref{ref:divpowerispoly:prop} below.
\begin{proposition}\label{ref:divpowssmallarepoly:prop}
    Let $\DP_{\geq d}$ be the linear span of $\{\DPel{\xx}{\aa}\ |\  \sum \aa \geq d\}$.
    Then $\DP_{\geq d}$ is an ideal of $\DP$, for all $d$.
    Let $\kk$ be a field of characteristic $p$.
    The ring $\DP/\DP_{\geq p}$ is isomorphic to the truncated polynomial ring. In fact
    \[
        \Omega:\DP/\DP_{\geq p}\to \kk[x_1, \ldots ,x_n]/(x_1, \ldots ,x_n)^p
    \]
    defined by
    \[
        \Omega\left( \DPel{\xx}{\aa} \right) = \frac{x_1^{a_1}  \ldots
        x_n^{a_n}}{a_1! \ldots a_n!}.
    \]
    is an isomorphism.
\end{proposition}
\begin{proof}
    Since $\Omega$ maps a basis of $\DP/I_p$ to a basis of $\kk[x_1, \ldots
    ,x_n]/(x_1, \ldots ,x_n)^p$, it is
    clearly well defined and bijective. The fact that $\Omega$ is a $\kk$-algebra homomorphism
    reduces to the equality $\binom{\aa + \bb}{\aa} = \prod \frac{(a_i +
    b_i)!}{a_i!b_i!}$.
\end{proof}

\paragraph{Characteristic zero case.}
In this paragraph we assume that $\kk$ is of characteristic zero.
This case is technically easier, but there are two competing conventions:
contraction and partial differentiation. These agree up to an
isomorphism. The main aim of this section is clarify this isomorphism and
provide a dictionary between divided power rings, used in this thesis, and
polynomial rings in characteristic zero.
Contraction was already defined above, now we define the action of $\DShat$ via
partial differentiation.
\begin{definition}\label{ref:partialdiff:def}
    Let $\kk[x_1, \ldots ,x_n]$ be a polynomial ring. There is a (unique) action
    of $\DShat$ on $\kk[x_1, \ldots ,x_n]$ such that the element $\Dx_i$ acts a
    $\frac{\partial}{\partial x_i}$. For $f\in \kk[x_1, \ldots ,x_n]$ and
    $\sigma\in \DShat$ we denote this action as $\sigma\circ f$.
\end{definition}

The following Proposition~\ref{ref:divpowerispoly:prop} shows that in
characteristic zero the ring $\DP$ is isomorphic to a
polynomial ring and the isomorphism identifies the $\DShat$-module structure on $\DP$
with that from Definition~\ref{ref:partialdiff:def} above.
\begin{proposition}\label{ref:divpowerispoly:prop}
    Suppose that $\kk$ is of characteristic zero. Let $\kk[x_1, \ldots ,x_n]$ be a
    polynomial ring with $\DShat$-module structure as defined
    in~\ref{ref:partialdiff:def}. Let $\DPut{stale}{\Omega}:\DP \to \kk[x_1,
    \ldots ,x_n]$ be defined via
    \[
        \Dstale\left(\DPel{\xx}{\aa}\right) = \frac{x_1^{a_1}  \ldots x_{n}^{a_n}}{a_1! \ldots
        a_n!}.
    \]
    Then $\Dstale$ is an isomorphism of rings and an isomorphism of $\DShat$-modules.
\end{proposition}
\begin{proof}
    The map $\Dstale$ is an isomorphism of $\kk$-algebras by the same argument as in
    Proposition~\ref{ref:divpowssmallarepoly:prop}.
     We leave the check that $\Dstale$ is a $\DShat$-module
    homomorphism to the reader.
\end{proof}

Summarizing, we get the following corresponding notions.
\begin{center}

\begin{tabular}{@{}l l l @{}}
    Arbitrary characteristic && Characteristic zero \\ \midrule
    divided power series ring $\DP$ && polynomial ring $\kk[x_1, \ldots ,x_n]$\\
    $\DShat$-action by contraction (precomposition) denoted $\sigma\hook f$  && $\DShat$ action by
    derivations denoted $\sigma\circ f$\\
    $\DPel{\xx}{\aa}$ && $\xx^{\aa}/\aa!$\\
    $x_i = \DPel{\xx}{(0, \ldots 0, 1, 0,  \ldots , 0)}$ && $x_i$\\
\end{tabular}
\end{center}

\section{Automorphisms and derivations of the power series
ring}\label{ssec:automorphisms}

Let as before $\DShat = \kk[[\Dx_1, \ldots ,\Dx_n]]$ be a power series ring with
maximal ideal $\DmmS$. This ring has a huge automorphism group: for every
choice of elements $\sigma_1, \ldots ,\sigma_n\in \DmmS$ whose images span
$\DmmS/\DmmS^2$ there is a unique
automorphism $\varphi\colon\DShat\to \DShat$ such that $\varphi(\Dx_i) = \sigma_i$.
Note that $\varphi$ preserves $\DmmS$ and its powers. Therefore the dual
map $\Ddual{\varphi} : \Ddual{\DShat}\to \Ddual{\DShat}$ restricts to
$\DPut{phid}{\Ddual{\varphi}}:\DP\to \DP$.
The map $\Dphid$ is defined (using the pairing of
Definition~\ref{ref:contraction:propdef}) via the condition
\begin{equation}\label{eq:definitionofdual}
    \ip{\varphi(\sigma)}{f} = \ip{\sigma}{\Dphid(f)}\quad\mbox{for all
    }\sigma\in \DShat,\ f\in \DP.
\end{equation}

\def\DDD{\mathbf{D}}%
Now we will describe this action explicitly.
\begin{proposition}\label{ref:dualautomorphism:prop}
    Let $\varphi\colon\DShat\to\DShat$ be an automorphism. Let $D_i = \varphi(\Dx_i) -
    \Dx_i$. For $\aa\in \myN^{n}$ denote $\DDD^{\aa} = D_1^{a_1} \ldots
    D_n^{a_n}$.
    Let $f\in \DP$.
    Then
    \[
        \Ddual{\varphi}(f) = \sum_{\aa\in \myN^n} \DPel{\xx}{\aa}\cdot \left(
        \DDD^{\aa} \hook f \right) = f + \sum_{i=1}^n x_i\cdot (D_i
        \hook f) +  \ldots .
    \]
\end{proposition}

\begin{proof}

    We need to show that
    \[
        \ip{\sigma}{\sum_{\aa\in \myN^n} \DPel{\xx}{\aa}\cdot \left(
        \DDD^{\aa} \hook f \right)} = \ip{\varphi(\sigma)}{f}
    \]
    for all $\sigma\in \DShat$. Since $f\in \DP$, it is
    enough to check this for all $\sigma\in \kk[\Dx_1, \ldots ,\Dx_n]$. By
    lineality, we may assume that $\sigma = \Dx^{\aa}$.
    For every $g\in \DP$ let $\varepsilon(g) = \ip{1}{g}\in \kk$.
    We have
    \def\parenthese#1{\left( #1 \right)}%
    \begin{align*}
        \ip{\varphi(\sigma)}{f} = \ip{1}{\varphi(\sigma)\hook f} &=
        \varepsilon\parenthese{\varphi(\sigma)\hook f} \\
        &=
        \varepsilon\parenthese{\sum_{\bb\leq \aa} \binom{\aa}{\bb}\left(\Dx^{\aa -
        \bb}\DDD^{\bb}\right)
    \hook f} = \sum_{\bb\leq \aa}
    \varepsilon\parenthese{\binom{\aa}{\bb}\Dx^{\aa -
    \bb}\hook (\DDD^{\bb}\hook f)}.
    \end{align*}
    Consider a term of this sum. Observe that for every $g\in \DP$
    \begin{equation}\label{eq:aux}
        \varepsilon\parenthese{\binom{\aa}{\bb}\Dx^{\aa - \bb}\hook g} =
        \varepsilon\parenthese{\Dx^{\aa} \hook \left( \DPel{\xx}{\bb}\cdot
        g \right)}.
    \end{equation}
    \def\cc{\mathbf{c}}%
    Indeed, it is enough to check the above equality for $g = \DPel{\xx}{\mathbf{c}}$
    and both sides are zero unless $\cc = \aa - \bb$, thus it is enough
    to check the case $g = \DPel{\xx}{\aa - \bb}$, which is straightforward.
    Moreover note that if $\bb \not\leq \aa$, then the right hand side is zero
    for all $g$, because $\varepsilon$ is zero for all $\DPel{\xx}{\cc}$ with
    non-zero $\cc$.
    We can use~\eqref{eq:aux} and remove the restriction $\bb\leq \aa$,
    obtaining
    \begin{align*}
        \sum_{\bb} \varepsilon\parenthese{\Dx^{\aa}\hook
        \left(\DPel{\xx}{\bb}\cdot (\DDD^{\bb}\hook f)\right)} =
        \varepsilon\parenthese{ \sum_{\bb} \Dx^{\aa}\hook
        \left(\DPel{\xx}{\bb}\cdot (\DDD^{\bb}\hook f)\right)} &=
        \ip{\Dx^{\aa}}{ \sum_{\bb}\DPel{\xx}{\bb}\cdot (\DDD^{\bb}\hook f)}
        =\\
        \ip{\Dx^{\aa}}{\Ddual{\varphi}(f)} &=
        \ip{\sigma}{\Ddual{\varphi}(f)}.\qedhere
    \end{align*}
\end{proof}

Consider now a derivation $D\colon \DShat\to \DShat$, i.e., a $\kk$-linear map
satisfying $D(\sigma\tau) = \sigma D(\tau) + D(\sigma)\tau$ for all $\sigma, \tau\in \DShat$. It gives rise to a dual map
$\Ddual{D}:\DP \to \DP$, which we now  describe explicitly.
\begin{proposition}\label{ref:dualderivation:prop}
    Let $D\colon \DShat\to \DShat$ be a derivation and $D_i := D(\Dx_i)$.
    Let $f\in \DP$. Then
    \[
        \Ddual{D}(f) = \sum_{i=1}^n x_i\cdot (D_i\hook f).
    \]
\end{proposition}

\begin{proof}
    The proof is similar to the proof of
    Proposition~\ref{ref:dualautomorphism:prop}, although it is easier.
\end{proof}

\begin{remark}\label{ref:lowersdegree:rmk}
    Suppose $D\colon \DShat\to \DShat$ is a derivation such that $D(\DmmS) \subseteq
    \DmmS^2$. Then $\deg(\Ddual{D}(f)) < \deg(f)$. We say that $D$ lowers
    the degree.
\end{remark}

\newcommand{\Dglcot}{\operatorname{GL}(\DmmS/\DmmS^2)}%
A special class of automorphisms of $\DShat$ are linear automorphisms.
\begin{definition}\label{ref:linearautomorphism:defn}
    Under the identification of $\sspan{\Dx_1, \ldots ,\Dx_n}$ with
    $\DmmS/\DmmS^2$, every linear map $\varphi\in \Dglcot$ induces a linear
    transformation of $\sspan{\Dx_1, \ldots ,\Dx_n}$ and consequently an
    automorphism $\varphi\colon\DShat\to \DShat$. We call such automorphisms
    \emph{linear}.
\end{definition}
The group $\Dglcot$ acts also on $\Ddual{\sspan{\Dx_1, \ldots
,\Dx_n}} = \sspan{x_1, \ldots ,x_n}$, hence on
$\DP =\kdp[x_1, \ldots ,x_n]$. The action of $\varphi\in \Dglcot$ is precisely
$\Ddual{\varphi}$. In particular, in this special case, $\Ddual{\varphi}$ is
an automorphism of $\kdp[x_1, \ldots ,x_n]$.

\paragraph{Characteristic zero case.}
Let $\kk$ be a field of characteristic zero. By $\xx^{\aa}$ we denote the
monomial $x_1^{a_1} \ldots x_n^{a_n}$ in the polynomial ring $\kk[x_1, \ldots
,x_n]$. Then, in the notation of Proposition~\ref{ref:divpowerispoly:prop}, we
have
\[\Dstale\left(\DPel{\xx}{\aa}\right) = \frac{1}{\aa!}\xx^{\aa}.\]
Clearly, an automorphism of $\DShat$ gives rise to an linear map $\kk[x_1, \ldots
,x_n]\to \kk[x_1, \ldots ,x_n]$.
We may restate Proposition~\ref{ref:dualautomorphism:prop} and
Proposition~\ref{ref:dualderivation:prop} as
\begin{corollary}\label{ref:dualautcharzero:cor}
    Let $\varphi\colon\DShat\to\DShat$ be an automorphism. Let $D_i = \varphi(\Dx_i) -
    \Dx_i$. For $\aa\in \myN^{n}$ denote $\DDD^{\aa} = D_1^{a_1} \ldots
    D_n^{a_n}$.
    Let $f\in \kk[x_1, \ldots ,x_n]$.
    Then
    \[
        \Ddual{\varphi}(f) = \sum_{\aa\in \myN^n} \frac{\xx^\aa}{\aa!} \left(
        \DDD^{\aa} \circ f \right) = f + \sum_{i=1}^n x_i (D_i
        \circ f) +  \ldots .
    \]
    Let $D\colon \DShat\to \DShat$ be a derivation and $D_i := D(\Dx_i)$. Then
    \[
        \Ddual{D}(f) = \sum_{i=1}^n x_i(D_i\circ f).
    \]
\end{corollary}

\begin{example}
    Let $n = 2$, so that $\DShat = \kk[[\Dx_1, \Dx_2]]$  and consider an
    automorphism $\varphi\colon\DShat \to \DShat$ given by
    $\varphi(\Dx_1) = \Dx_1$ and $\varphi(\Dx_2) = \Dx_1 + \Dx_2$. Dually,
    $\Ddual{\varphi}(x_1) = x_1 + x_2$ and $\Ddual{\varphi}(x_2) = x_2$.
    Since $\varphi$ is linear, $\Ddual{\varphi}$ is an automorphism of
    $\kk[x_1, x_2]$.
    Therefore $\Ddual{\varphi}(x_1^3) = (x_1 + x_2)^3$.
    Let us check this equality using Proposition~\ref{ref:dualautcharzero:cor}.
    We have $D_1 = \varphi(\Dx_1) - \Dx_1 = 0$ and $D_2 = \varphi(\Dx_2) - \Dx_2 = \Dx_1$.
    Therefore $\DDD^{(a, b)} = 0$ whenever $a > 0$ and $\DDD^{(0, b)}
    = \Dx_1^b$.

    We have
    \[
        \Ddual{\varphi}\left( x_1^3 \right) = \sum_{(a, b)\in \mathbb{N}^2}
        \frac{x_1^{a}x_2^{b}}{a!b!}(\DDD^{(a, b)} \circ x_1^3) = \sum_{b\in
            \mathbb{N}} \frac{x_2^{b}}{b!}(\Dx_1^{b} \circ x_1^3) = x_1^3 +
            \frac{x_2}{1} \cdot(3x_1^2) + \frac{x_2^2}{2} \cdot(6x_1) +
            \frac{x_2^3}{6}\cdot(6) = (x_1 + x_2)^3,
    \]
    which indeed agrees with our previous computation.

    When $\varphi$ is not linear, $\Ddual{\varphi}$ is not an endomorphism of
    $\kk[x_1, x_2]$ and computing it directly from definition becomes harder. For
    example, if $\varphi(\Dx_1) = \Dx_1$ and $\varphi(\Dx_2) = \Dx_2 +
    \Dx_1^2$, then
    \[
        \Ddual{\varphi}(x_1) = x_1,\quad \Ddual{\varphi}(x_1^4) = x_1^4 + 12x_1^2x_2 + 12x_2^2.
    \]
\end{example}

\section{Classification of local embedded Gorenstein algebras\\ via
apolarity}\label{ssec:Gorensteininapolarity}

\newcommand{\DAut}{\operatorname{Aut}(\DShat)}%
In this section fix a power series ring $\DShat$ and consider finite local
algebras $\DA$ presented as quotients $\DA = \DShat/I$. Note that every $\DA$
can be embedded into every $\DShat$ of dimensions at least $\dimk \DA$.
This section is classical; Macaulay's theorem first appeared
in~\cite{Macaulay_inverse_systems}, while the classification using the group
$\Dgrp$ defined below in~\eqref{eq:groupdefinition}, was
first noticed, without proof, by Emsalem~\cite{emsalem}.

Let $\DP$ be defined as in Definition~\ref{ref:contraction:propdef}.
For every subset $X \subset \DP$ by $\Ann(X) \subset \DShat$ we denote the set of
all $\sigma\in \DShat$ such that $\sigma\hook x = 0$ for all $x\in X$. Note
that if $X$ is an ideal of $\DP$ then $\Ann(X) = X^{\perp} = \left\{ \sigma\in
\DS\ |\ \ip{\sigma}{X} = \{0\} \right\}$.

Consider a finite algebra $\DA = \DShat/I$, then $\DA$ is local with maximal
ideal $\mm$ which is the image of $\DmmS$.
The surjection $\DShat \onto \DA$ gives an inclusion
\begin{equation}\label{eq:dualinclusion}
    \canA = \Homthree{\kk}{\DA}{\kk} \subset \Homthree{\kk}{\DShat}{\kk} =
    \Ddual{\DShat}.
\end{equation}
We note the following fundamental lemma.

Recall that $\DP_{\leq d-1} \subset
\DP \subset \Ddual{\DShat}$ is the space of elements annihilated by
contraction with $\DmmS^d$.
\begin{lemma}\label{ref:agreement:lem}
    Let $\DA = \DShat/I$ be a finite algebra of degree $d$. Let $\canA \subset
    \Ddual{\DShat}$ be defined as in~\eqref{eq:dualinclusion}.
    The subspace $\canA$ lies in $\DP_{\leq d-1}$ and it is an $\DShat$-submodule of
    $\DP_{\leq d-1}$.
\end{lemma}
\begin{proof}
    According to Definition~\ref{ref:contraction:propdef}, we have
    $(\sigma\hook f)(\tau) = f(\sigma\tau)$ for all $\sigma, \tau\in \DShat$.
    In particular if $\sigma\in I$, then $\sigma\hook f = 0$.
    Since $\DA$ is
    finite of degree $d$, we have $\mm^d = 0$ and so $\DmmS^d \subset
    I$. Then we have $\DmmS^d \hook \canA$, so $\canA \subset \DP_{\leq d-1}$.
    Similarly, if $\sigma\in \DShat$ and $f\in \canA$, then $(\sigma\hook
    f)(I) = f(\sigma I) \subset f(I) = \{0\}$, so $\sigma\hook f$ is an
    element of $\canA$. Hence, $\canA$ is an $\DShat$-submodule of
    $\DP_{\leq d-1}$.
\end{proof}
Note that there are two actions applicable to element of $\canA$: one is the contraction action of
$\DShat$ on $\Ddual{\DShat}$, as defined in
Definition~\ref{ref:contraction:propdef} and the other is the action of $\DA$
on $\canA$ as in Definition~\ref{ref:canonicalmodule:def}.
These actions agree, as shown in the proof of Lemma~\ref{ref:agreement:lem}.
Below we consistently use contraction.

\begin{definition}\label{ref:apolar:def}
    Let $\mathcal{F} \subset \DP$ be a subset. The \emph{apolar algebra} of
    $\mathcal{F}$ is the
    quotient
    \[
        \Apolar{\mathcal{F}} := \DShat/\Ann(\mathcal{F}).
    \]
\end{definition}

\begin{theorem}[Macaulay's theorem~\cite{Macaulay_inverse_systems}]
    \label{ref:Macaulaytheorem:thm}
    Let $(\DA, \mm, \kk)$ be a finite local $\kk$-algebra. Fix $\DA =
    \DShat/I$. Then there exist $f_1, \ldots ,f_r\in \DP$ such that $\DA = \Apolar{f_1, \ldots ,f_r}$.
\end{theorem}
\begin{proof}
    Consider the subspace $\canA \subset \Ddual{\DShat}$. Since $\DA$ is local
    and finite, we have $\mm^d \subset I$ for $d$ large enough, so $\canA
    \subset \DP$. By discussion after Definition~\ref{ref:canonicalmodule:def}
    no element of $\DA$ annihilates $\canA$, so
    $\Ann(\canA) \subset \DShat$ is equal to $I$. Choose any set $f_1, \ldots
    ,f_r$ of generators of $\DA$-module $\canA$, then $\DA =\Apolar{f_1, \ldots ,f_r}$.
\end{proof}

\begin{theorem}[Macaulay's theorem for Gorenstein algebras]
    \label{ref:MacaulaytheoremGorenstein:thm}
    Let $(\DA, \mm, \kk)$ be a finite local Gorenstein $\kk$-algebra. Fix $\DA =
    \DShat/I$. Then there exists $f\in \DP$ such that $\DA = \Apolar{f}$.
    Conversely, if $f\in \DP$, then $\Apolar{f}$ is a finite local Gorenstein
    algebra.
\end{theorem}
\begin{proof}
    Let $\DA = \DShat/I$ be Gorenstein and $f\in \canA$ be its dual generator.
    Since $Af = \canA$ is torsion-free, no non-zero element of $A$ annihilates
    $f$. If we interpret $f\in \canA \subset \DP$ as an element of $\DP$, then
    $\Ann(f) = I$, thus $\Apolar{f} = \DA$. Conversely, take $f\in \DP$. Then
    $f\in \DP_{\leq d-1}$ for some $d$, so that $\mm^{d}\hook f = 0$ and $A = \Apolar{f} = \DShat/\Ann(f)$
    is finite and local. By definition, no element of $A$ annihilates $f$, so
    $\dimk Af = \dimk A = \dimk \canA$, hence $Af = \canA$ and $f$ is a dual
    generator of $\DA$.
\end{proof}

Before we delve into deeper considerations, let us point out that
Theorem~\ref{ref:MacaulaytheoremGorenstein:thm} enables us to explicitly
describe Gorenstein algebras and in particular give examples.
\begin{example}
    The algebra $\Apolar{\DPel{x_1}{2}+\DPel{x_2}{2}} = \kk[[\Dx_1,
    \Dx_2]]/(\Dx_1^2 - \Dx_2^2, \Dx_1\Dx_2)$ already appeared in
    Example~\ref{ex:firstGorenstein}.
\end{example}
\begin{example}\label{ex:trivialcubicHf}
    Let $f = \DPel{x_1}{2} +  \ldots + \DPel{x_k}{2} + \DPel{x_{k+1}}{3} +
    \ldots + \DPel{x_{n}}{3}$. Then
    \[\DShat  f = \sspan{1, x_1, \ldots ,x_n, \DPel{x_{k+1}}{2}, \ldots
    ,\DPel{x_n}{2},
f}.\] Thus $\Ann(f) = (\Dx_i\Dx_j)_{i\neq j} + (\Dx_i^2 - \Dx_j^2)_{i, j \leq k} +
(\Dx_i^3 - \Dx_j^3)_{k < i, j\leq n} + (\Dx_i^2 - \Dx_{j}^3)_{i \leq k <
j}$. We compute that $\DA = \Apolar{f}$ has socle
    degree three and that $H_{\DA} = (1, n, n-k, 1)$. The maximal ideal of
    $\DA$ is generated by images of $\Dx_i$. The nonzero images of
    $\Dx_1, \ldots ,\Dx_k$ lie in $(0:\mm^2)$ but not in $\mm^2$, in
    contrast with the graded case.
\end{example}
\begin{remark}\label{ref:Macaulaygraded:rmk}
    If $\DA = \DShat/I$ is given by homogeneous ideal $I$, then $f_1, \ldots ,
    f_r\in \DP$ in Theorem~\ref{ref:Macaulaytheorem:thm} may be
    chosen homogeneous. Also, if $\DA$ is Gorenstein, then $f\in \DP$ in
    Theorem~\ref{ref:MacaulaytheoremGorenstein:thm} may be chosen homogeneous;
    indeed choose any $f'$ with leading form $f'_d$, then $I f_d' = 0$, since
    $I$ is homogeneous, so $f_d'\in \DShat f'\setminus \DmmS f'$ is a dual
    generator.
\end{remark}
\begin{example}\label{ex:monomialGorenstein}
    There are few finite monomial Gorenstein $\kk$-algebras. Indeed,
    such an algebra $\DA = \DS/I$ is graded, hence local and $\DA =
    \Apolar{f}$. Since $I$ is
    monomial, also $\canA \subset \DP$ is spanned by monomials, so that $f$ can be chosen
    to be monomial: $f = \DPel{x_1}{s_1}\cdot\DPel{x_2}{s_2}\cdot \ldots \cdot\DPel{x_n}{s_n}$. Then $I =
    (\Dx_1^{s_1+1}, \Dx_2^{s_2+1}, \ldots ,\Dx_{n}^{s_n+1})$ and $\DA$ is a complete intersection.
\end{example}

Every finite algebra of degree $d$ can be presented as a
quotient of a fixed power series algebra $\DShat$ by
Lemma~\ref{ref:embeddingdimension:lem}. We
now consider the question ``When are two Gorenstein quotients of
$\DShat$ isomorphic?''.
Let $\DShat^*$ denote the group of invertible elements of $\DShat$ and let
\begin{equation}\label{eq:groupdefinition}
\Dgrp := \DAut \ltimes \DShat^*
\end{equation}
be the group generated by $\DAut$ and $\DShat^*$ in the space
$\Homthree{\kk}{\DShat}{\DShat}$. As the notation suggests, the group $\Dgrp$ is a semidirect product
of those groups: indeed $\varphi\circ \mu_s \circ \varphi^{-1} =
\mu_{\varphi(s)}$, where $\varphi$ is an automorphism, $s\in \DShat$ is
invertible and $\mu_{s}$ denotes the multiplication by $s$. We have an action of
$\Dgrp$ on $\DP$ described by Equation~\eqref{eq:dual}. Here $\DShat^*$ acts
by contraction and $\DAut$ acts as described in
Proposition~\ref{ref:dualautomorphism:prop}.

\begin{proposition}\label{ref:Gorenstein:prop}
    Let $A = \DShat/I$ and $B = \DShat/J$ be two finite local Gorenstein
    $\kk$-algebras.
    Choose $f, g\in \DP$ so that $I = \Ann(f)$ and $J = \Ann(g)$.
    The following conditions are equivalent:
    \begin{enumerate}
        \item\label{it:isom} $A$ and $B$ are isomorphic,
        \item\label{it:conj} there exists an automorphism $\varphi\colon\DShat\to
            \DShat$ such that
            $\varphi(I) = J$,
        \item\label{it:dualconj} there exists an automorphism $\varphi\colon\DShat\to
            \DShat$ such that
            $\Ddual{\varphi}(f) = \sigma\hook g$, for an
            invertible element $\sigma\in \DShat$.
        \item\label{it:grp} $f$ and $g$ lie in the same $\Dgrp$-orbit of $\DP$.
    \end{enumerate}
\end{proposition}

\begin{proof}
    Taking an isomorphism $A  \simeq B$, one obtains $\varphi':\DShat\to
    B = \DShat/J$, which can be lifted to an automorphism of $\DShat$ by choosing
    lifts of linear forms. This proves $\ref{it:isom}\iff \ref{it:conj}$.

    $\ref{it:conj}\iff\ref{it:dualconj}$. Let $\varphi$ be as
    in Point~\ref{it:conj}. Then $\Ann(\Ddual{\varphi}(f)) = \varphi(\Ann(f)) =
    \varphi(I) = J$. Therefore the principal $\DShat$-submodules of $\DP$ generated by
    $\Ddual{\varphi}(f)$ and $g$ are equal, so that there is an invertible
    element $\sigma\in \DShat$ such that $\Ddual{\varphi}(f) = \sigma\hook g$.
    The argument can be reversed.

    Finally, Point~\ref{it:grp}~is just a reformulation of
    \ref{it:dualconj}.
\end{proof}

\begin{remark}[Graded algebras]\label{ref:graded:rem}
    In the setup of Proposition~\ref{ref:Gorenstein:prop} one could specialize
    to homogeneous ideals $I$, $J$ and homogeneous polynomials $f, g\in \DP$.
    Then Condition 1.~is equivalent to the fact that $f$ and $g$ lie in
    the same $\operatorname{GL}(\DmmS/\DmmS^2)$-orbit. The proof of
    Proposition~\ref{ref:Gorenstein:prop} easily restricts to this case,
    see~\cite{geramita_inverse_systems_of_fat_points}.
\end{remark}

\begin{theorem}
    The set of finite local Gorenstein algebras of degree $r$ is naturally
    in bijection with the set of orbits of $\Dgrp$-action on $\DP =
    \kk_{dp}[x_1, \ldots ,x_r]$.
    \label{ref:classification:thm}
\end{theorem}
\begin{proof}
    Every local Gorenstein algebra of degree $r$ can be presented as a
    quotient of $\DShat = \kk[[\Dx_1, \ldots ,\Dx_r]]$, so the
    claim follows from Proposition~\ref{ref:Gorenstein:prop}.
\end{proof}

Proposition~\ref{ref:Gorenstein:prop} shows the central role of
$\Dgrp$ in the classification. Having this starting point, we may consider at
least two directions. First, we may construct elements of $\Dgrp$ explicitly
and, using the explicit description of their action on $\DP$ given
in~Section~\ref{ssec:automorphisms}, actually classify some algebras or
prove some general statements. Second, we may consider $\Dgrp$ as a whole and
investigate the Lie theory of this group to gain more knowledge about its orbits.
We will illustrate the first path in Example~\ref{ex:1nn1} and the second path
in Section~\ref{ssec:specialforms}. Both paths are combined to obtain the
examples from Sections~\ref{ssec:examplesone}-\ref{ssec:examplesthree}.

Now we give one complete, non-trivial, explicit example to illustrate the core
ideas before they will be enclosed into a more formal apparatus.

\begin{example}[compressed cubics, {\cite[Theorem~3.3]{EliasRossiShortGorenstein}}]\label{ex:1nn1}
    Assume that $\kk$ has characteristic not equal to two. Let $(\DA, \mm,
    \kk)$ be a
    local Gorenstein $\kk$-algebra such that $H_{\DA} = (1, n, n, 1)$ for some $n$.
    By Macaulay's theorem~\ref{ref:MacaulaytheoremGorenstein:thm} there exists
    an $f\in \DP = \kdp[x_1, \ldots ,x_n]$ such that $\DA =
    \DShat/\Ann(f)$.
    Since $\mm^3\neq 0$ and $\mm^4 = 0$ we have $\deg f = 3$. Let $f_3$ be its
    leading form.

    We claim that there is an element $\varphi\in\Dgrp$
    such that $\Ddual{\varphi}(f_3) = f$.

    Since $\dimk \mm^3 = 1$, we have $\dimk \mm^2 = n+1$, so $\dim \DmmS^2 f =
    n+1$. But $\DmmS^2 f \subset \DP_{\leq 1}$ and $\dimk \DP_{\leq 1} = n+1$, so
    $\DmmS^2 f = \DP_{\leq 1}$; every linear form in $\DP$ is obtained as
    $\delta\hook f$ for some operator $\delta\in \DmmS^2$.
    We pick operators $D_1, \ldots ,D_n\in \DmmS^2$ so that $\sum x_i \cdot (D_i
    \hook f) = -(f_2 + f_1)$.
    Explicitly, $D_i$ is such that $D_i\hook f = -(\Dx_i \hook f_2)/2 -
    \Dx_i\hook f_1$. Here we
    use the assumption on the characteristic.

    Let $\varphi\colon\DShat\to \DShat$ be an automorphism defined via $\varphi(\Dx_i) =
    \Dx_i + D_i$. Since $(D_i D_j)\hook f = 0$ by degree reasons, the explicit formula in
    Proposition~\ref{ref:dualautomorphism:prop} takes the form
    \[
        \Ddual{\varphi}(f) = f + \sum x_i \cdot (D_i \hook f) = f
        - f_2 - f_1 = f_3 + f_0.
    \]
    The missing term $f_0$ is a constant, so that we may pick an order three
    operator $\sigma\in \DShat$ with $\sigma\hook \Ddual{\varphi}(f) = -f_0$.
    Then $(1 + \sigma)\hook (\Ddual{\varphi}(f)) = f_3$, so by
    Proposition~\ref{ref:Gorenstein:prop} we have $\Apolar{f}  \simeq  \Apolar{f_3}$
    as claimed.
    By taking associated graded algebras, we obtain
    \[
        \gr \Apolar{f}  \simeq  \gr \Apolar{f_3}  \simeq  \Apolar{f_3}  \simeq
        \Apolar{f},
    \]
    so in fact $\grDA \simeq \DA$, a
    rather rare property for a local algebra.
    The assumption $\chark\neq 2$ is necessary, as we later see in
    Example~\ref{ex:1nn1char2}.
\end{example}

\newcommand{\linear}[2]{\linearOp(#1)^{#2}}%
\newcommand{\DSf}{\Diff(f)}%
\newcommand{\DSfin}[2]{\DSf_{#1}^{#2}}
\section{Hilbert functions of Gorenstein algebras in terms of inverse
systems}\label{ssec:HFininversesystems}

In this section we translate the numerical data of
$\DA = \Apolar{f}$: its Hilbert function and symmetric decomposition, into
properties of $\DA f = \DShat f$. This is straightforward; we just use
the isomorphism $\DA\to \DA f$ of $\DA$-modules (or $\DShat$-modules). We
include this section mainly to explicitly state the results we later use
implicitly in examples.
We
follow~\cite{JelMSc} and~\cite{BJMR}.

Let $d$ be the socle degree of $\DA$.
It is equal to the degree of $f$.
Recall that $\DShat = \kk[[\Dx_1, \ldots,\Dx_n]]$.
Since the linear forms $\Dx_i$ act as derivatives on the divided power
polynomial ring $\DP$, we
call elements of $\DShat f$ the \emph{partials} of $f$ and we denote
\[\DSf := \DShat f.\]
The filtrations on $\DA$, the \Loevy{} filtration and the usual $\mm$-adic
filtration, translate respectively to filtration by degree and filtration by \emph{order}.
For a nonzero element $g\in \DSf$ the \emph{order} of $g$ is the maximal $i$ such
that $g\in\mm^i f$.
We define the following subspaces of $\DP$:
\begin{equation}\label{eq:diffdefs}
    \DSfin{i}{} = \DSf \cap \DP_{\leq i},\quad \DSfin{i}{a} =
    (\DmmS^{d-a-i} f)\cap \DP_{\leq i}.
\end{equation}
From Lemma~\ref{ref:loewyHilbertfunc:lem} in follows that $H_{\DA}(i) = \dimk
\DSfin{i}{} - \dimk \DSfin{i-1}{}$ is the dimension of the space of partials
of degree exactly $i$. Moreover $\DSfin{i}{a}$ is the image of $\mm^{d-a-i}
\cap \Dmmperp{i+1} \subset \DA$ in $\DSf$. Therefore, the space
$\DSfin{i}{a}/(\DSfin{i-1}{a} + \DSfin{i}{a+1})$ is the image of
\begin{equation}\label{eq:Qintermsoff}
    \frac{\mm^{d-a-i} \cap \Dmmperp{i+1}}{\mm^{d-a-i} \cap \Dmmperp{i} +
    \mm^{d-a-i+1} \cap \Dmmperp{i+1}} = \iaq{a}{d-a-i} = \iaq{a}{i}^{\vee};
\end{equation}
the equalities follow from~\eqref{eq:Qsummands} and Lemma~\ref{ref:Qduality:lem}. Hence, we have
\begin{equation}
    \Dhd{a}{i} = \dimk \frac{\DSfin{i}{a}}{\DSfin{i-1}{a} + \DSfin{i}{a+1}},
    \label{eq:diffdeltas}
\end{equation}
this may be thought of as the space of partials of $f$ which have degree
$i$ and order $a$. The spaces corresponding to
$\Dhd{a}{1}$ are of special importance. We define
$\linearOp(f):=\DSf \cap \DP_1$,
and its linear subspaces
${\linear{f}{a}= \{l\in \DP_1 \mid l \in \mm^{d-a-1}f \}} = \DSfin{1}{a}
    \cap \DP_1$.
We easily see that for each ${a\ge0}$, we have an isomorphism
${\linear{f}{a}\simeq \DSfin{1}{a}/\DSfin{0}{}}$ and $\DSfin{0}{} = \kk$, so
\[\Dhd{a}{1}=\dimk \linear{f}{a} - \dimk \linear{f}{a-1}.\] We
obtain a canonical flag of subspaces of $\DP_1$:
\begin{equation}\label{eq:flagsubspaces}
    \linear{f}{0} \subseteq \linear{f}{1} \subseteq  \cdots \subseteq
    \linear{f}{d-2}= \linearOp(f) \subseteq 
    \DP_1.
\end{equation}

\begin{example}\label{ex:133211firstdecomposition}
    Let ${f=\DPel{x_1}{5} + \DPel{x_2}{4} + \DPel{x_3}{3}}$. Its space of partials is generated by
    the elements in the following table, where the generators of each
    $\iaqvect{a}^{\vee}$
    are arranged by degree; next to it, we have the symmetric decomposition of its
    Hilbert function:

    \begin{center}
        \newcommand{\oldarraystretch}{\arraystretch} 
        \renewcommand{\arraystretch}{1.2}
        \begin{tabular}{c @{\hspace{4em}} c}
            Generators of the space of partials & Hilbert function decomposition \\[2mm]
            \begin{tabular}{ l cccccc }
                degree & $0$ &$1$ & $2$ & $3$ & $4$ & $5$\\
                \hline
                $\iaqvect{0}^\vee$ & $1$ & $x_1$  & $\DPel{x_1}{{2}}$ & $\DPel{x_1}{{3}}$ &
                $\DPel{x_1}{4}$ & $f$\\
                $\iaqvect{1}^\vee$ & & $x_2$ & $\DPel{x_2}{{2}}$ &
                $\DPel{x_2}{3}$ &  &\\
                $\iaqvect{2}^\vee$ & & $x_3$  & $\DPel{x_3}{2}$ & & & \\
                $\iaqvect{3}^\vee$ & &  & & & & \\
                & &  &  &  & & \\
            \end{tabular}
            &
            $\begin{array}{cccccccc}
                \mbox{degree} && 0 &1 & 2 & 3 & 4 & 5\\
                \hline
                \Dhdvect{0}& = & 1 & 1 & 1 & 1 & 1 & 1\\
                \Dhdvect{1}& = & 0 & 1 & 1 & 1 & 0 &\\
                \Dhdvect{2}& = & 0 & 1 & 1 & 0 & &\\
                \Dhdvect{3}& = & 0 & 0 & 0 &  & & \\
                \hline
                H_{\DA}            & = & 1 & 3 & 3 & 2 & 1 & 1\\
            \end{array}$
        \end{tabular}
        \renewcommand{\arraystretch}{\oldarraystretch}
    \end{center}
    We have $\linear{f}{0} = \sspan{x_1} \subset \linear{f}{1} = \sspan{x_1,
    x_2} \subset \linear{f}{2} = \linear{f}{3} = \sspan{x_1, x_2, x_3}$.
\end{example}
\begin{example}\label{ex:133211snddecomposition}
    Let ${f=\DPel{x_1}{5} + x_1\DPel{x_2}{3} + \DPel{x_3}{2}}$. Its space of partials is generated by
    the elements in the following table, where the generators of each
    $\iaqvect{a}^{\vee}$
    are arranged by degree; next to it, we have the symmetric decomposition of its
    Hilbert function:

    \begin{center}
        \newcommand{\oldarraystretch}{\arraystretch} 
        \renewcommand{\arraystretch}{1.2}
        \begin{tabular}{c @{\hspace{4em}} c}
            Generators of the space of partials & Hilbert function decomposition \\[2mm]
            \begin{tabular}{ l cccccc }
                degree & $0$ &$1$ & $2$ & $3$ & $4$ & $5$\\
                \hline
                $\iaqvect{0}^\vee$ & $1$ & $x_1$  & $\DPel{x_1}{{2}}$ & $\DPel{x_1}{{3}}$ &
                $\DPel{x_1}{4}+\DPel{x_2}{3}$ & $f$\\
                $\iaqvect{1}^\vee$ & & $x_2$ & $x_1x_2$,\ \,$\DPel{x_2}{{2}}$ &
                $x_1\DPel{x_2}{2}$ &  &\\
                $\iaqvect{2}^\vee$ & &  & & & & \\
                $\iaqvect{3}^\vee$ & & $x_3$  & & & & \\
                & &  &  &  & & \\
            \end{tabular}
            &
            $\begin{array}{cccccccc}
                \mbox{degree} && 0 &1 & 2 & 3 & 4 & 5\\
                \hline
                \Dhdvect{0}& = & 1 & 1 & 1 & 1 & 1 & 1\\
                \Dhdvect{1}& = & 0 & 1 & 2 & 1 & 0 &\\
                \Dhdvect{2}& = & 0 & 0 & 0 & 0 & &\\
                \Dhdvect{3}& = & 0 & 1 & 0 &  & & \\
                \hline
                H_{\DA}            & = & 1 & 3 & 3 & 2 & 1 & 1\\
            \end{array}$
        \end{tabular}
        \renewcommand{\arraystretch}{\oldarraystretch}
    \end{center}
    We have $\linear{f}{0} = \sspan{x_1} \subset \linear{f}{1} = \linear{f}{2}
    = \sspan{x_1,
    x_2} \subset \linear{f}{3} = \sspan{x_1, x_2, x_3}$.
\end{example}

\begin{example}\label{partialsofapolynomial}
Let ${f=\DPel{x_1}{{3}}x_2+\DPel{x_3}{{3}}+\DPel{x_4}{{2}}}$. Its space of partials is generated by
the elements in the following table, where the generators of each
$\iaqvect{a}^{\vee}$
are arranged by degree; next to it, we have the symmetric decomposition of its
Hilbert function:

\begin{center}
\newcommand{\oldarraystretch}{\arraystretch} 
\renewcommand{\arraystretch}{1.2}
\begin{tabular}{c @{\hspace{4em}} c}
    Generators of the space of partials & Hilbert function decomposition \\[2mm]
\begin{tabular}{ l ccccc }
degree & $0$ &$1$ & $2$ & $3$ & $4$ \\
\hline
$\iaqvect{0}^\vee$ & $1$ & $x_1$,\ \,$x_2$  & $\DPel{x_1}{{2}}$,\ \,$x_1x_2$ & $\DPel{x_1}{{3}}$,\ \,$\DPel{x_1}{{2}}x_2$ & $f$\\
$\iaqvect{1}^\vee$ & & $x_3$ & $\DPel{x_3}{{2}}$ & & \\
$\iaqvect{2}^\vee$ & & $x_4$ & & & \\
& &  &  &  & \\
\end{tabular}
&
$\begin{array}{ccccccc}
    \mbox{degree} && 0 &1 & 2 & 3 & 4 \\
\hline
\Dhdvect{0}& = & 1 & 2 & 2 & 2 & 1 \\
\Dhdvect{1}& = & 0 & 1 & 1 & 0 & \\
\Dhdvect{2}& = & 0 & 1 & 0 &  & \\
\hline
H_{\DA}            & = & 1 & 4 & 3 & 2 & 1\\
\end{array}$
\end{tabular}
\renewcommand{\arraystretch}{\oldarraystretch}
\end{center}
For instance $\DPel{x_3}{{2}}$ is a partial of order $1$, since it is obtained as
${\Dx_3\hook f=\DPel{x_3}{{2}}}$ and cannot be attained by a higher order element of $T$, so
it is a generator of $\iaq{1}{2}^\vee$. Here we have ${\linear{f}{0}=\langle x_1,x_2\rangle}$, ${\linear{f}{1}=\langle x_1,x_2,x_3\rangle}$, and ${\linear{f}{2}=\langle x_1,x_2,x_3,x_4\rangle}$.
\end{example}

From the above examples, we might notice the natural fact that the lower
degree terms of $f$ do not appear in $\iaqvect{a}$ for small $a$. The
following Proposition~\ref{ref:degreecomparison:prop} makes this observation precise.
\begin{proposition}\label{ref:degreecomparison:prop}
        Suppose that polynomials $f, g\in \DP$ of degree $d$ are such that
        $\deg (f - g) \leq d - \delta$.
        Then $\Dhdvect{\Apolar{f}, a} = \Dhdvect{\Apolar{g}, a}$ for all $a < \delta$.
\end{proposition}
\begin{proof}
    By~\eqref{eq:Qintermsoff} the spaces $\iaq{a}{\Apolar{f}}$ and
    $\iaq{a}{\Apolar{g}}$ are spanned by partials of degree $i$ and order $d-a-i$
    of $f$ and $g$, respectively. But $\deg
    \sigma(f-g) \leq a+i-\delta < i$ for all $\sigma\in \DmmS^{d-a-i}$, so
    that leading forms of elements of degree $i$ are equal for $f$ and $g$,
    see~\cite[Lemma~1.10]{iarrobino_associated_graded} or \cite[Lemma~4.34]{JelMSc} for details.
\end{proof}

\begin{corollary}\label{ref:topdegreefirstrow:cor}
    Let $f\in \DP$ and $\DA = \Apolar{f}$.
    The vector $\Dhdvect{0}$ is equal to the Hilbert function of
    $\Apolar{f_d}$. If the Hilbert function $H_{\DA}$ satisfies $H_{\DA}(d-i)
    = H_{\DA}(i)$ for all $i$, then $H_{\DA} = \Dhdvect{0}$.
\end{corollary}
\begin{proof}
    Since $\deg (f - f_d) < d$, we have $\Dhdvect{0} = \Dhdvect{\Apolar{f_d}, 0}$ by
    Proposition~\ref{ref:degreecomparison:prop}.
    We also have $H_{\DA} = \sum_{a=0}^{d} \Dhdvect{a}$ by
    Lemma~\ref{ref:hilbertfromdecomposition:lem}. Each $\Dhdvect{a}$
    is symmetric around $d - a$, so if any $\Dhdvect{a}$ with $a\neq 0$ is
    non-zero, then the center of gravity of $H_{\DA}$ is smaller than $d/2$,
    so $H_{\DA}$ cannot be symmetric around $d/2$ as assumed.
\end{proof}

\begin{example}\label{ex:lowersummands}
Proposition~\ref{ref:degreecomparison:prop} might give an impression that if
$f$ with $d = \deg f$ has no homogeneous parts of degrees less that $d - i$,
then $\Dhdvect{a} = 0$ for all $a > i$. This is false for $f = \DPel{x_1}{4} +
\DPel{x_1}{2}x_2$. Indeed, $\DSf = \sspan{f,\ \DPel{x_1}{3} + x_1x_2,\ \DPel{x_1}{2},\ \DPel{x_1}{2} + x_2,\
x_1,\ 1}$, so that $H_{\Apolar{f}} = (1, 2, 1, 1, 1)$ with the unique
symmetric decomposition $\Delta_0 = (1, 1, 1, 1, 1)$ and $\Delta_2 = (0, 1,
0)$.
\end{example}

Conclusion of Corollary~\ref{ref:topdegreefirstrow:cor} may be lifted from
the level of Hilbert functions to algebras, as we present below in
Corollary~\ref{ref:symmetricHf:cor}. Recall the ideal $\iacvect{1}\subset
\grDA$ defined in~\eqref{eq:Csummands}.
\begin{corollary}\label{ref:symmetricHf:cor}
    Let $\DA = \Apolar{f}$ for $f\in \DP$ of degree $d$. Then $\grDA/\iacvect{1} \simeq \Apolar{f_d}$.
    If $H_{\DA}(d-i) = H_{\DA}(i)$ for all $i$, then $\grDA \simeq
    \Apolar{f_d}$ is also
    a Gorenstein algebra.
\end{corollary}
\begin{proof}
%    Algebra $\DA$ is a quotient of $\DShat/\DmmS^{d+1}$, which is a graded
%    ring. Therefore $\grDA$ is also a quotient of $\DShat/\DmmS^{d+1}$.
    Let $I = \Ann(f)$. The algebra $\grDA$ is a quotient of $\DShat$ by the
    ideal generated by all lowest degree forms of elements of $I$. If $i\in I$
    and $i'$ is its lower degree form, then the top degree form of $i\hook
    f$ is $i'\hook f_d$. Since $i\hook f = 0$ also $i' \hook f_d = 0$. This proves that
    $\Apolar{f_d}$ is a graded quotient of $\grDA$ and that it makes sense to speak
    about the action of an element of $\grDA$ on $f_d$.
    Consider any nonzero element $a\in \iacvect{1}_i$.
    Then $a\hook f_d$ is of degree $d-i$. By definition $a\in
    \Dmmperp{d-i}$, so $\deg(a\hook f) < d-i$ and so $\deg(a\hook f_{d}) <
    d-i$. This implies that $a\hook f_d = 0$. Thus $\Apolar{f_d}$ is a
    quotient of $\grDA/\iacvect{1}$. The Hilbert functions of these algebras
    are equal to $\Dhdvect{0}$ by Corollary~\ref{ref:topdegreefirstrow:cor}
    and Equation~\eqref{eq:Qsummands},
    so $\grDA/\iacvect{1}  \simeq \Apolar{f_d}$.  If additionally
    $H_{\DA}(d-i) = H_{\DA}(i)$ for all $i$, then
    $H_{\DA} = \Dhdvect{0}$ by Corollary~\ref{ref:topdegreefirstrow:cor} again,
    so $\iacvect{1} = 0$ and $\grDA  \simeq \Apolar{f_d}$.
\end{proof}

\section{Standard forms of dual generators}\label{ssec:standardforms}

Let $(\DA, \mm, \kk)$ be a finite local Gorenstein algebra. We have seen that
$\DA$ may be presented as a quotient of $\DShat$ if and only if $H_{\DA}(1)
\leq \dim \DShat$, see Lemma~\ref{ref:embeddingdimension:lem}. This can be
rephrased as saying that if $\DA  \simeq \Apolar{f}$ then necessarily $f\in
\DP$ depends on at least $H_{\DA}(1)$ variables. In this section we refine
this statement by considering each homogeneous piece of such $f$ separately
and finding minimal number of variables it must depend on.
The existence of standard forms was proven by Iarrobino in~\cite[Theorem
5.3AB]{iarrobino_associated_graded}.

Recall from~\eqref{eq:flagsubspaces} the filtration of $\DP_1$ by
$\linear{f}{a} = \DP_1 \cap \DmmS^{d-a-1}f$. Let $n_a = \sum_{i=0}^a\Dhd{i}{1} = \dim
\linear{f}{a}$ and fix a basis of linear forms ${x_1, \ldots , x_n}$ in
$\DP_1$ that agrees with the filtration by $\linear{f}{i}$:
\begin{multline}\label{Linfiltration}
\linear{f}{0}=\langle x_1, \ldots , x_{n_0}\rangle \subseteq
    \linear{f}{1}=\langle x_1, \ldots , x_{n_1} \rangle \subseteq  \cdots\\
  \cdots \subseteq \linear{f}{d-2}=\langle x_1, \ldots , x_{n_{d-2}} \rangle
  \subseteq \DP_1 = \langle x_1, \ldots , x_n \rangle.
\end{multline}
None of the considerations below depend on this choice; it is done only to
improve presentation. The foundational property of standard forms is the following proposition.
\begin{proposition}\label{ref:linearpartsofdecomposition:prop}
    Fix $f\in \DP$ and $i\geq 0$ and let $f_{\geq d-i} = f_{d-i} +  \ldots  + f_{d-1} + f_d$.
    Then the linear forms from $\linear{f}{i}$ are partials of $f_{\geq d-i}$.
\end{proposition}
\begin{proof}
    Pick $\ell\in \linear{f}{i}$. By construction, $\ell = \sigma\hook f$ for
    an operator $\sigma\in \DmmS^{d-i-1}$. For such $\sigma$
    we have $\deg(\sigma\hook (f - f_{\geq d-i})) \leq 0$, so that
    $\ell = \sigma\hook f = \sigma\hook f_{\geq d-i} \mod \DP_0$; the form
    $\ell$ is a partial of $f_{\geq d-i}$ modulo a constant form. But constant
    forms are also partials of $f_{\geq d-i}$, so $\ell\in\DShat \hook
    f_{\geq d-i}$.
\end{proof}

\begin{definition}\label{ref:standardform:def}
Let ${f\in \DP}$ be a polynomial with homogeneous decomposition ${f = f_{d}+
\cdots + f_0}$. Let $\Delta$ be the symmetric decomposition of the Hilbert
function of $\DShat/\Ann(f)$. We say that $f\in S$ is in \emph{standard form}
if
\[
    f_{d - i} \in \kdp\left[\linear{f}{i}\right] \quad \mbox{for all } i,
\]
\end{definition}
This is equivalent to $f_{d-i}\in \kdp\left[x_1, \ldots , x_{n_i} \right]$,
where ${x_1, \ldots , x_n}$ is any choice of basis for $\DP_1$ as in \eqref{Linfiltration}.
The standard form is a way to write $f$ using as few variables as
possible, as we explain now. For $i$ and $f_{\geq d-i} = f_{d-i} +
f_{d-i+1} +  \ldots + f_d$ we have $\linear{f}{i} \subset \DShat f_{\geq
d-i}$, by Proposition~\ref{ref:linearpartsofdecomposition:prop}. The
polynomial $f$ is in standard form, if and only if for each $i$ we have
conversely $f_{\geq d-i} \in \kdp \left[ \linear{f}{i} \right]$, no additional
variables appear.

Now we prove that in the orbit $\DAut f \subset \Dgrp f$ of every $f\in \DP$ there is an element in the standard
form.

\begin{theorem}[Existence of standard forms]\label{ref:stform:thm}
    Let ${f\in \DP}$.  Then there is an automorphism ${\varphi\colon\DShat\to
    \DShat}$ such that $\Ddual{\varphi}(f)$ is in a standard form.
%    Moreover one can choose $\varphi$ of order two.
\end{theorem}
\begin{proof}
    Choose a basis of $\DP_1$ as in~\eqref{Linfiltration} and the dual basis
    $\Dx_1, \ldots ,\Dx_n$.
        Consider $\DA= \Apolar{f}$ and the sequence of ideals as defined
        in~\eqref{eq:Qsummands}:
        \def\iaintersnd#1{\mm\cap \Dmmperp{#1}}%
        \def\fracced#1{\frac{#1 + \mm^2}{\mm^2}}
        \begin{equation}
            \frac{\mm^{\phantom{1}}}{\mm^2} =
            \iacvect{0}_1 =
            \fracced{\iaintersnd{d}} \supseteq
            \iacvect{1}_1 = \fracced{\iaintersnd{d-1}} \supseteq \dots
            \supseteq \iacvect{d}_1 \supseteq 0.
            \label{eq:standardform}
        \end{equation}
%        Note that elements of $\iacvect{1+a}$ annihilate $\linear{f}{a}$.
%        Note that $\dimk \iacvect{a} = H(1) - n_a = H(1) -\sum_{i=0}^a
%        \Dhdvect{i}(1)$.
        We choose lifts of $\kk$-vector spaces $\iacvect{a}$
        to $\DS$. This gives a flag of subspaces
        \begin{equation}\label{eq:tmpflag}
            \sspan{z_1, \ldots , z_n} \supset \sspan{z_{n-n_0+1}, \ldots
            ,z_{n}} \supset \sspan{z_{n-n_1+1}, \ldots
                ,z_{n}} \supset  \ldots \supset \{0\}
        \end{equation}
        spanned by elements $z_i$ of order one. Take an automorphism
        $\varphi\colon\DShat\to \DShat$ sending $\Dx_i$ to $z_i$.
        We claim that $\Ddual{\varphi}(f)$ is in the standard
        form. Indeed, for every $i$ and $a$ such that $i > n - n_a$ we have $z_i$
        in the lift of $\iacvect{a}$ so
        \[
            \ip{\mm^{d-a}\Dx_i}{\Ddual{\varphi}(f)} = \ip{\mm^{d-a}z_i}{f} =
            \ip{1}{\mm^{d-a}z_i\hook f} = 0.
        \]
        This implies that $\deg(\Dx_i\hook \Ddual{\varphi}(f)) < d - a$.
        Let $\Ddual{\varphi}(f) = g$ and $g = g_d +  \ldots  + g_0$ be
        decomposition into homogeneous summands.
        By induction we conclude that $g_{d-i}\in \kdp[x_1, \ldots ,x_{n_i}]$
        for all $i=0,1 \ldots , d$.
\end{proof}

Standard form of a given polynomial is by no means unique: if $f\in \DP$ is a
polynomial in a standard form and $\varphi$ is linear (see
Definition~\ref{ref:linearautomorphism:defn}), then $\Ddual{\varphi}(f)$
is also in the standard form. Note also that $z_i$
in the proof of Theorem~\ref{ref:stform:thm} may be chosen such that $z_i
\equiv \Dx_i \mod \DmmS^2$, so that $\varphi(\Dx) - \Dx\in \DmmS^2$ for all $\Dx\in \DmmS$.

The following example illustrates how to obtain the standard form
of a given polynomial.
\begin{example}
    Take $f = \DPel{x_1}{4} + \DPel{x_1}{2}x_2$ from Example~\ref{ex:lowersummands}. The
    nonzero summands of the symmetric decomposition of $H_{\Apolar{f}}$ are
    $\Dhdvect{0} = (1, 1, 1, 1, 1)$ and $\Dhdvect{2} = (0, 1,
    0)$, so we have $\linear{f}{0} = \linear{f}{1} = \sspan{x_1}$ and
    $\linear{f}{2} = \sspan{x_1, x_2}$. Since $f_3 = \DPel{x_1}{2}x_2\not\in
    \kdp[x_1] = \kdp[\linear{f}{1}]$, the polynomial $f$ is not in a standard form. The flag
    from~\eqref{eq:tmpflag} for $f$ is equal to
    \[
        \iacvect{0} = \sspan{\Dx_1, \Dx_2} \supset \iacvect{1} = \iacvect{2} = \sspan{\Dx_2
        - \Dx_1^2}  \supset \iacvect{3} = \iacvect{4} = \{0\}.
    \]
    Take an automorphism $\varphi\colon\DShat\to \DShat$ defined by $\varphi(\Dx_1)
    = \Dx_1$ and $\varphi(\Dx_2) = \Dx_2 - \Dx_1^2$. By Proposition~\ref{ref:dualautomorphism:prop}, we have
    \[
        \Ddual{\varphi}(f) = f + x_2\cdot (-\Dx_1^2 \hook f) + \DPel{x_2}{2}\cdot
        (\Dx_1^4 \hook f) = f - \pp{\DPel{x_1}{2}x_2 + 2\DPel{x_2}{2}} + \DPel{x_2}{2} = \DPel{x_1}{4} - \DPel{x_2}{2},
    \]
    which is in a standard form.
\end{example}

\section{Simplifying dual generators}\label{ssec:specialforms}

\newcommand{\orbit}[1]{\Dgrp\!\cdot\!#1}%
\newcommand{\tang}[1]{\Dgrptang #1}%
\newcommand{\perpspace}[1]{\left( #1 \right)^{\perp}}%
    While the standard form of $f\in \DP$ is highly useful, it is not unique.
    In this section we investigate refinements of the standard form,
    using the Lie theory of $\Dgrp$. Our aim is to remove or simplify the
    lower degree homogeneous components of $f$ by replacing it with another
    element of $\orbit f$; ideally, we would like to show that $f\in \Dgrp
    f_d$, as in Example~\ref{ex:1nn1}; of course this is not always true. All
    results of this subsection first appeared in~\cite{Jel_classifying}.

    Let $\DPut{Der}{\mathfrak{aut}}$ denote the space of derivations of
    $\DShat$ preserving $\DmmS$, i.e.~derivations such that $D(\DmmS)
    \subseteq \DmmS$. Let $\DShat \subset \Homthree{\kk}{\DShat}{\DShat}$ be
    given by sending $\sigma\in \DShat$ to multiplication by $\sigma$.
    Let
    \[
        \DPut{grptang}{\mathfrak{g}} := \DDer  + \DShat,
    \]
where the sum is taken in the space of linear maps from $\DShat$ to $\DShat$.
Then $\Dgrptang$ acts on $\DP$ as defined in Equation~\eqref{eq:dual}. The
space $\Dgrptang$ is actually the tangent space to the group scheme $\Dgrp$,
see Serre~\cite[Theorem~5, p.~4 and discussion below]{Serre__Lie_algebras_and_Lie_groups}.
Similarly, $\tang{f}$ is
naturally contained in the tangent space of the orbit $\orbit{f}$ for every
$f\in \DP$. There is a subtlety here, though. The space $\tang{f}$ is the image of
the tangent space $\Dgrptang$ under $\Dgrp \to \orbit{f}$. If $\kk$ is of
characteristic zero then this map, being a map of homogeneous spaces, is
smooth, so $\tang{f}$ is the tangent space to $\orbit{f}$.
However the map need not be smooth in positive characteristic, so in principle it may happen
that $\tang{f}$ is strictly contained in the tangent space of
$\orbit{f}$. Presently we do not have an example of such behavior.

Sometimes it is more convenient to work with
equations in $\DShat$ than with subspaces of $\DP$. Recall from~\eqref{eq:dualdual} that $\Ddual{\DP} =
\DShat$. For each subspace $W\subset \DP$ we may consider the orthogonal space
\[
    W^{\perp} = \{\sigma\in \DShat \ |\ \forall_{f\in W} f(\sigma) = 0 \}\subset
\DShat.
\]
Below we describe the linear space
$\perpspace{\tang{f}}$.
For $\sigma\in \DShat$ by $\ithpartial{\sigma}{i}$ we denote the $i$-th partial
derivative of $\sigma$. We use the convention that $\deg(0) <
0$.

\begin{proposition}[tangent space description]\label{ref:tangentspacepoly:prop}
        Let $f\in \DP$. Then
    \[
        \DDer \cdot f = \sspan{ x_i \cdot (\delta\hook f)\ |\ \delta\in \DmmS,\
        i=1, \ldots ,n},\qquad \tang{f} = \DShat f + \sum_{i=1}^n \DmmS (x_i\cdot
        f).
    \]
    Moreover
    \[
        \perpspace{\tang{f}} = \left\{ \sigma\in \DShat\ |\ \sigma\hook f = 0,\
            \ \forall_i\ \ \deg(\ithpartial{\sigma}{i}\hook f) \leq 0\right\}.
    \]
    Suppose further that $f\in \DP$ is homogeneous of degree $d$. Then  $\perpspace{\tang{f}}$ is
    spanned by homogeneous operators and
    \[
        \perpspace{\tang{f}}_{\leq d} = \left\{ \sigma\in \DShat\ |\ \sigma\hook f = 0,\
            \ \forall_i\ \ \ithpartial{\sigma}{i}\hook f = 0 \right\}.
    \]

\end{proposition}

\begin{proof}
    Let $D\in \DDer$ and $D_i := D(\Dx_i)$.
    By Proposition~\ref{ref:dualderivation:prop} we have $\Ddual{D}(f) =
    \sum_{i=1}^n x_i\cdot (D_i\hook f)$. For any $\delta\in \DmmS$ we may choose
    $D$ so that $D_i = \delta$ and all other $D_j$ are zero. This proves the
    description of $\DDer \cdot f$.
    Now $\tang{f} = \DShat f + \sspan{ x_i \cdot (\delta\hook f)\ |\ \delta\in \DmmS,\
        i=1, \ldots ,n}$. By Lemma~\ref{ref:commutator:lem} we have
        $x_i(\delta\hook f) \equiv \delta\hook(x_i f) \mod \DShat f$. Thus
        \[\tang{f} = \DShat f + \sspan{ \delta\hook (x_i\cdot f)\ |\ \delta\in \DmmS,\
        i=1, \ldots ,n} = \DShat f + \sum \DmmS (x_i f).
    \]
    Now let $\sigma\in \DShat$ be an operator such that $\ip{\sigma}{\tang{f}}
    = 0$. This is equivalent to $\sigma\hook (\tang{f}) = 0$, which
    simplifies to $\sigma\hook f = 0$ and $(\sigma\DmmS)\hook (x_i f) = 0$ for all $i$.
    By Lemma~\ref{ref:commutator:lem}, we have $\sigma\hook (x_if) = x_i (\sigma\hook f) +
    \DPut{tmp}{\ithpartial{\sigma}{i}}\hook f = \Dtmp\hook f$, thus
    we get equivalent conditions:
    \[
        \sigma\hook f = 0\quad \mbox{and} \quad \DmmS\hook (\Dtmp \hook f) = 0,
    \]
    and the claim follows. Finally, if $f$ is homogeneous of degree $d$ and $\sigma\in \DShat$
    is homogeneous of degree at most $d$ then $\ithpartial{\sigma}{i}\hook
    f$ has no constant term and so $\deg(\ithpartial{\sigma}{i}\hook f) \leq 0$
    implies that $\ithpartial{\sigma}{i}\hook f = 0$.
\end{proof}

\begin{remark}\label{ref:cotangent:rmk}
    Let $f\in \DP$ be homogeneous of degree $d$.
    Let $j \leq d$ and $K_j := \perpspace{\tang{f}}_{j}$.
    Proposition~\ref{ref:tangentspacepoly:prop} gives a connection of $K_j$
    with the conormal sequence.
    Namely, let $I = \Ann(f)$ and $\DA = \Apolar{f} = \DShat/I$. We have $(I^2)_j \subseteq
    K_{j}$ and the quotient space fits into the conormal sequence of $\DShat\to
    \DA$, making that sequence exact:
    \begin{equation}\label{eq:cotangent}
        0\to \left(K/I^2\right)_j\to \left( I/I^2 \right)_j \to
        \left(\Omega_{\DShat/\kk}\tensor
        \DA\right)_j\to \left(\Omega_{\DA/\kk}\right)_{j}\to
        0.
    \end{equation}
    This is expected from the point of view of deformation theory.
    Recall that by \cite[Theorem~5.1]{hartshorne_deformation_theory} the
    deformations of $\DA$ over $\kk[\varepsilon]/\varepsilon^2$ are in one-to-one
    correspondence with elements of a $\kk$-linear space $T^1(\DA/\kk, \DA)$. On the other hand,
    this space fits \cite[Proposition~3.10]{hartshorne_deformation_theory} into the sequence
    \def\Hom#1#2{\operatorname{Hom}(#1, #2)}%
    \[
        0\to \Homthree{\DA}{\Omega_{\DA/\kk}}{\DA} \to
        \Homthree{\DA}{\Omega_{S/\kk}\tensor \DA}{\DA}\to \Homthree{\DA}{I/I^2}{\DA}
        \to T^1(\DA/\kk, \DA)\to 0.
    \]
    Since $\DA$ is Gorenstein, $\Homthree{\DA}{-}{\DA}$ is exact and we have
    $T^1(\DA/\kk, \DA)_j  \simeq \Hom{K/I^2}{\DA}_j$ for all $j \geq 0$.
    The restriction $j\geq 0$ appears because $\Hom{K/I^2}{\DA}$ is the tangent
    space to deformations of $\DA$ inside $\DShat$, whereas $T^1(\DA/\kk, \DA)$
    parameterizes all deformations.
\end{remark}

    Below we use terminology concerning Lie groups. From the onset we note
    that $\Dgrp$ is not a Lie group, because it is not even a finitely
    dimensional algebraic group when we take it as a subgroup of linear
    transformations of $\DShat$. However, $\Dgrp$ is a projective limit of
    algebraic groups.
    Namely, for every $r$ we have a map $\DAut \to
    \operatorname{Aut}(\DShat/\DmmS^r)$ and the image of $\Dgrp$ under this
    map is an algebraic group $\Dgrp_r$. Moreover $\Dgrp = \projlim_r
    \Dgrp_r$.
    In all considerations involving orbits of $f\in \DP$ such that $\deg \Apolar{f} \leq
    r$ we see that $\deg f < r$, the ideal $\DmmS^r$ acts trivially on $f$. Therefore
    we can replace $\Dgrp$ by $\Dgrp_r$.

\newcommand{\DAutunip}{\operatorname{Aut}^+(\DShat)}%
\newcommand{\Dgrpunip}{\Dgrp^+}
\newcommand{\DDerunip}{\mathfrak{aut}^+}%

Now we introduce a subgroup $\Dgrpunip$ of $\Dgrp$, which is an
analogue of the unipotent radical of an algebraic group.
In particular:
\begin{enumerate}
    \item $\Dgrp/\Dgrpunip$ is a finitely dimensional reductive algebraic
        group,
    \item the image $\Dgrpunip_r \subset \Dgrp_r$ of $\Dgrpunip$ is a
        unipotent subgroup and $\Dgrpunip = \projlim_r
        \Dgrpunip_r$.
\end{enumerate}
By Kostant-Rosenlicht
theorem~\cite[Theorem~2,~p.~221]{rosenlicht_unipotent_closures}, over an algebraically
closed field $\kk$ every $\Dgrpunip$-orbit is Zariski closed in $\DP$.
The subgroup $\Dgrpunip$ is also useful in applications, because it preserves the top
degree form, which allows induction on the degree.

Each automorphism of $\DShat$ induces a linear
map on the cotangent space: we have a restriction $\DAut\to
\Dglcot$. Let us denote by $\DAutunip$
the group of automorphisms which act as identity on the tangent space:
$\DAutunip = \left\{ \varphi\in \DShat\ |\ \forall_{i}\ \varphi(\Dx_i) -
\Dx_i\in\DmmS^2 \right\}$. We have the following sequence of groups:
\begin{equation}\label{eq:groups}
    1 \to \DAutunip \to \DAut \to \Dglcot\to 1.
\end{equation}
We define
\[\Dgrpunip = \DAutunip \ltimes (1 + \DmmS) \subseteq \Dgrp.\]
Note that we have the following exact sequence:
\begin{equation}\label{eq:unipradical}
    1\to \Dgrpunip \to \Dgrp \to \Dglcot \times \kk^*\to 1.
\end{equation}

Correspondingly, let $\DPut{Der}{\mathfrak{aut}}$ denote the
space of derivations preserving $\DmmS$, i.e.~derivations such that $D(\DmmS) \subseteq
\DmmS$.
Let $\DDerunip$ denote the space of derivations
such that $D(\DmmS) \subseteq \DmmS^2$. Denoting by
$\DPut{glliecot}{\mathfrak{gl}\left( \DmmS/\DmmS^2 \right)}$ the space of
linear endomorphisms of $\DmmS/\DmmS^2$, we have we following sequence of linear
spaces:
\begin{equation}\label{eq:derivations}
    0\to \DDerunip \to \DDer \to \Dglliecot\to 0.
\end{equation}
We define
\[
    \DPut{grpuniptang}{\Dgrptang^+} = \DDerunip + \DmmS.
\]
Following the proof of Proposition~\ref{ref:tangentspacepoly:prop} we get the
following proposition.
\def\tangunip#1{\Dgrpuniptang #1}%
\def\orbitunip#1{\Dgrpunip\!\cdot\!#1}%
\begin{proposition}\label{ref:uniptangentspacepoly:prop}
    Let $f\in \DP$. Then $\tangunip{f} = \DmmS f + \sum \DmmS^2(x_i f)$
    so that
    \[
        \perpspace{\tangunip{f}} = \left\{ \sigma\in \DShat\ |\
            \deg(\sigma\hook f) \leq 0,\ \ \forall_i\ \deg(\ithpartial{\sigma}{i}\hook f)\leq 1\right\}.
    \]
    If $f$ is homogeneous of degree $d$ then $\tangunip{f}$ is spanned
    by homogeneous polynomials and
    \[
        \perpspace{\tangunip{f}}_{<d} = \left\{ \sigma\in \DShat\ |\
            \sigma\hook f = 0,\ \ \forall_i\ \ithpartial{\sigma}{i}\hook f = 0
        \right\} = \perpspace{\tang{f}}_{<d}.\hfill\qed
    \]
\end{proposition}

We end this section by illustrating above theory with an example.
In Example~\ref{ex:1nn1} we presented a proof of
\cite[Theorem~3.3]{EliasRossiShortGorenstein}.  Below we give a different,
more conceptual proof under additional restrictions on $\kk$.
\begin{example}[compressed cubics, using Lie theoretic ideas]\label{ex:1nn1Lie}
    Assume that $\kk$ be algebraically closed of characteristic different than
    $2$.
    Let $(\DA, \mm, \kk)$ be a local Gorenstein $\kk$-algebra such that
    $H_{\DA} = (1, n, n, 1)$ for some $n$.  By Macaulay's
    theorem~\ref{ref:MacaulaytheoremGorenstein:thm} there exists a degree
    three polynomial $f\in \DP
    = \kdp[x_1, \ldots ,x_n]$ such that $\DA = \DShat/\Ann(f)$. Let $f_3$ be its
    leading form, then $\Apolar{f_3} \simeq \grDA$ by
    Corollary~\ref{ref:symmetricHf:cor} and so $H_{\Apolar{f_3}} = (1, n, n,
    1)$.

    We claim that there is an element $\varphi$ of $\Dgrp$
    such that $\Ddual{\varphi}(f_3) = f$. This proves that $\Apolar{f}
    \simeq \Apolar{f_{3}} = \gr\Apolar{f}$. We say that the
    apolar algebra of $f$ is \emph{canonically graded}.

    In fact, we claim that already $\orbitunip{f_3}$ is the whole
    space:
    \begin{equation}\label{eq:tmpcubics}
        \orbitunip{f_3} = f_3 + \DP_{\leq 2}.
    \end{equation}

    From the explicit formula in
    Proposition~\ref{ref:dualautomorphism:prop} we see that $\orbitunip{f_3}
    \subseteq f_3 + \DP_{\leq 2}$. It is a Zariski closed
    subset by the Kostant-Rosenlicht theorem. To prove
    equality~\eqref{eq:tmpcubics} it enough to
    check that
    $\tangunip{f_3} = \DP_{\leq 2}$.
    Let $\sigma\in \perpspace{\tangunip{f_3}}_{\leq 2}$ be non-zero. Since
    $\perpspace{\tangunip{f_3}}$ is spanned by homogeneous elements, we take
    $\sigma$ homogeneous. By
    Proposition~\ref{ref:uniptangentspacepoly:prop} we get that $\sigma\hook
    f_3 = 0$ and $\ithpartial{\sigma}{i} \hook f_3 = 0$ for all $i$.
    Since $\chark \neq
    2$, there exists $i$ such that $\ithpartial{\sigma}{i}\neq 0$. Either
    $\sigma$ has degree one or $\ithpartial{\sigma}{i}$ has degree one, so
    there is a nonzero degree one operator annihilating $f_3$. But this
    contradicts the fact that $H_{\Apolar{f_3}}(1) = n = H_{\DShat}(1)$.
    Therefore $\perpspace{\tangunip{f_3}}_{\leq 2} = 0$ and the claim
    follows.

    The assumption $\chark \neq 2$ is necessary, as Example~\ref{ex:1nn1char2}
    shows.
\end{example}

\section{Examples I --- compressed algebras}\label{ssec:examplesone}

In this section we gather some corollaries of the machinery from
Section~\ref{ssec:specialforms} and present the theory of compressed algebras
as in~\cite{Jel_classifying}. In particular, we prove that certain local algebras
are isomorphic to their associated graded algebras.
\begin{assumption}
    Throughout Section~\ref{ssec:examplesone} we assume that $\kk =
    \kkbar$ is algebraically closed and, if $\chark$ is positive, then it is
    greater that the degrees of all considered polynomials.
\end{assumption}

We begin we a closer comparison between the orbits of $\Dgrpunip$ and
$\Dgrpuniptang$.
\newcommand{\tdf}[1]{\operatorname{tdf}\pp{#1}}%
For every $f\in \DP$ let $\tdf{f}$ denote the top degree form of
$f$,~so that $\tdf{\DPel{x_1}{3} + \DPel{x_2}{2}x_3 + \DPel{x_4}{2}} = \DPel{x_1}{3} + \DPel{x_2}{2}x_3$.
\begin{proposition}\label{ref:topdegreecomp:prop}
    Let $f\in \DP$.
    Suppose that $\chark > d = \deg(f)$.
    Then the top degree form of every element of $\orbitunip{f}$ is equal to the top degree form of $f$.
    Moreover,
    \begin{equation}\label{eq:tdf}
        \left\{ \tdf{g - f}\ |\ g\in \orbitunip{f} \right\} = \left\{ \tdf{h}\ |\ h\in
    \tangunip{f} \right\}.
    \end{equation}
    If $f$ is homogeneous, then both sides of~\eqref{eq:tdf} are equal to
    the set of homogeneous elements of $\tangunip{f}$.
\end{proposition}

\begin{proof}
    Consider the $\DShat$-action on
    $\DP_{\leq d}$. This action descents to an $\DPut{Strunc}{\DShat/\DmmS^{d+1}}$ action.
    Further in the proof we implicitly replace $\DShat$ by $\DStrunc$, thus also
    replacing $\DAut$ and $\Dgrp$ by appropriate truncations.
    Let $\varphi\in \Dgrpunip$. Since $(\operatorname{id} - \varphi)\left( \DmmS^i \right) \subseteq \DmmS^{i+1}$ for
    all $i$, we have $(\operatorname{id} - \varphi)^{d+1} = 0$. By our
    assumption on the characteristic of $\kk$, the element $D :=
    \log(\varphi)$ is well-defined and $\varphi = \exp(D)$.
    We get an injective map $\exp:\Dgrpuniptang \to
    \Dgrpunip$ with left inverse $\log$. Since $\exp$ is algebraic we see by
    dimension count that its image is open in $\Dgrpunip$. Since $\log$ is
    Zariski-continuous, we get that $\log(\Dgrpunip) \subseteq \Dgrpuniptang$, then
    $\exp:\Dgrpuniptang \to \Dgrpunip$ is an isomorphism.

    Therefore
    \[
        \Ddual{\varphi}(f) = f + \sum_{i=1}^{d}
        \frac{\left(\Ddual{D}\right)^i(f)}{i!} = f + \Ddual{D}(f) +
        \left(\sum_{i=1}^{d-1} \frac{\left(\Ddual{D}\right)^i}{(i+1)!}\right) \Ddual{D}(f).
    \]
    By Remark~\ref{ref:lowersdegree:rmk} the derivation $D\in \Dgrpuniptang$
    lowers the degree, we see that $\tdf{\Ddual{\varphi}f} =
    \tdf{f}$ and
        $\tdf{\Ddual{\varphi}f - f} = \tdf{\Ddual{D}(f)}$. This proves~\eqref{eq:tdf}.
    Finally, if $f$ is homogeneous then $\tangunip{f}$ is equal to the $\sspan{ \tdf{h}\ |\ h\in
    \tangunip{f}}$ by
    Proposition~\ref{ref:uniptangentspacepoly:prop}, and the last claim
    follows.\qedhere

    For an elementary proof, at least for the subgroup $\DAutunip$,
    see~\cite[Proposition~1.2]{Matczuk_unipotent_derivations}.
\end{proof}

The following almost tautological
Corollary~\ref{ref:leadingformremoval:cor} enables one to prove that a
given apolar algebra is canonically graded inductively, by lowering the degree
of the remainder.
\begin{corollary}\label{ref:leadingformremoval:cor}
    Let $F$ and $f$ be polynomials. Suppose that the leading
    form of $F - f$ lies in $\tangunip{F}$. Then there is an element $\varphi\in
    \Dgrpunip$ such that $\deg(\Ddual{\varphi}f - F) < \deg(f - F)$.
\end{corollary}

\begin{proof}
    Let $G$ be the leading form of $f - F$ and $e$ be its degree.
    By Proposition~\ref{ref:topdegreecomp:prop} we may find $\varphi\in
    \DAutunip$ such that $\tdf{\Ddual{\varphi}(F) - F} = -G$, so that
    $\Ddual{\varphi}(F) \equiv F - G \mod \DP_{\leq e-1}$.
    By the same proposition we have $\deg(\Ddual{\varphi}(f - F) - (f -
    F)) < \deg(f - F) = e$, so that $\Ddual{\varphi}(f - F) \equiv f - F \mod
    \DP_{\leq e-1}$. Therefore
    $\Ddual{\varphi}(f) -F = \Ddual{\varphi}(F) + \Ddual{\varphi}(f - F) -F \equiv f
    - G - F \equiv 0 \mod \DP_{\leq e-1}$, as claimed.
\end{proof}

Example~\ref{ex:1nn1} is concerned with a degree three polynomial
$f$ such that the Hilbert function of $\Apolar{f}$ is maximal i.e. equal to
$(1, n, n, 1)$ for $n = H_{\DShat}(1)$.  Below we generalize the results obtained
in this example to polynomials of arbitrary degree.

Recall that a finite local Gorenstein algebra $A$ of socle degree $d$ is called \emph{compressed} if
\[H_A(i) = \min\left(H_{\DShat}(i), H_{\DShat}(d - i)\right) =
    \min\left(\binom{i+n-1}{i}, \binom{d-i+n-1}{d-i}
    \right)\quad\mbox{for all}\ \
i=0, 1,  \ldots , d.\]
Here we introduce a slightly more general notation.

\begin{definition}[$t$-compressed]\label{ref:compressed:def}
    Let $A = S/I$ be a finite local Gorenstein algebra of socle degree $d$. Let
    $t\geq 1$.
    Then $A$ is called \emph{$t$-compressed} if the following conditions are
    satisfied:
    \begin{enumerate}
        \item $H_A(i) = H_{\DShat}(i) = \binom{i+n-1}{i}$ for all $0\leq i \leq t$,
        \item $H_{A}(d-1) = H_{\DShat}(1)$.
    \end{enumerate}
\end{definition}

\begin{example}
    Let $n = 2$. Then $H_A = (1, 2, 2, 1, 1)$ is not $t$-compressed, for any
    $t$. The function $H_A = (1, 2, 3, 2, 2, 2, 1)$ is $2$-compressed. For any
    sequence $*$ the function $(1, 2, *, 2, 1)$ is $1$-compressed.
\end{example}

Note that it is always true that $H_{A}(d-1) \leq H_{A}(1) \leq H_{\DShat}(1)$,
thus both conditions above assert that the Hilbert function is maximal
possible. Therefore they are open in $\DP_{\leq d}$.

\begin{remark}\label{ref:compressedtrivia:rmk}
    The maximal value of $t$, for which $t$-compressed algebras exists, is $t
    = \left \lfloor d/2\right\rfloor$.  Every compressed algebra is
$t$-compressed for $t = \floor{d/2}$ but not
    vice versa. If $A$ is graded, then $H_A(1) = H_A(d-1)$, so the condition $H_{A}(d-1) = H_{\DShat}(1)$
    is satisfied automatically.
\end{remark}

The following technical Remark~\ref{ref:socleminusone:rmk} will be useful later. Up to some extent, it
explains the importance of the second condition in the definition of
$t$-compressed algebras.
\begin{remark}\label{ref:socleminusone:rmk}
    \def\DmmA{\mathfrak{m}_A}%
    Let $A = \Apolar{f}$ be a $t$-compressed algebra with maximal ideal $\DmmA$. We have
    $\dim \DP_{\leq 1} = H_{A}(d-1) + H_{A}(d) = \dim \DmmA^{d-1}/\DmmA^d + \dim \DmmA^d =
    \dim \DmmA^{d-1}$. Moreover $\DmmA^{d-1}  \simeq
    \DmmS^{d-1} f$ as linear spaces and $\DmmS^{d-1} f \subseteq
    \DP_{\leq 1}$. Thus
    \[
        \DmmS^{d-1} f = \DP_{\leq 1}.
    \]
\end{remark}

The definition of $t$-compressed algebras explains itself in the following
Proposition~\ref{ref:sp:nosmallorder:prop}.
\begin{proposition}\label{ref:sp:nosmallorder:prop}
    Let $f\in \DP$ be a polynomial of degree $d\geq 3$ and $A$ be its
    apolar algebra.
    Suppose that $A$ is $t$-compressed.
    Then the $\Dgrpunip$-orbit of $f$ contains $f + \DP_{\leq t+1}$. In particular
    $f_{\geq t+2} \in \orbitunip{f}$, so that $\Apolar{f}  \simeq \Apolar{f_{\geq t+2}}$.
\end{proposition}

\begin{proof}
\DDef{tmplowdeg}{\DP_{\leq t+1}}%
    First we show that $\Dtmplowdeg\subseteq\tangunip{f}
    $,~i.e. that no non-zero operator of order at most $t+1$ lies in
    $\perpspace{\tangunip{f}}$.
    Pick such an operator. By Proposition~\ref{ref:uniptangentspacepoly:prop} it
    is not constant. Let $\sigma'$ be any of its non-zero partial
    derivatives. Proposition~\ref{ref:uniptangentspacepoly:prop} asserts that
    $\deg(\sigma'\hook f)\leq 1$.
    Let $\ell := \sigma'\hook f$.
    By Remark~\ref{ref:socleminusone:rmk} every linear polynomial is contained
    in $\DmmS^{d-1} f$. Thus we may choose a $\delta\in \DmmS^{d-1}$ such that
    $\delta\hook f = \ell$. Then $(\sigma' - \delta)\hook f = 0$. Since
    $d\geq 3$, we have $d - 1 > \floor{d/2} \geq t$, so that $\sigma - \delta$ is an operator of order
    at most $t$ annihilating $f$. This contradicts the fact that $H_{A}(i) =
    H_{\DShat}(i)$ for all $i\leq t$. Therefore $\Dtmplowdeg\subseteq\tangunip{f}$.

    Second, pick a polynomial $g\in f + \DP_{\leq t+1}$. We prove that $g\in
    \orbitunip{f}$ by induction on $\deg(g - f)$.
    The top degree form of $g - f$ lies
    in $\tangunip{f}$. Using Corollary~\ref{ref:leadingformremoval:cor} we
    find $\varphi\in \Dgrpunip$ such that $\deg(\Ddual{\varphi}(g) - f) <
    \deg(g - f)$.
\end{proof}

For completeness, we state the following consequence of the previous result.
\begin{corollary}\label{ref:compressedpoly:cor}
    Let $f\in \DP$ be a polynomial of degree $d\geq 3$ and $A$ be its
    apolar algebra. Suppose that $A$ is compressed. Then $A  \simeq
    \Apolar{f_{\geq \floor{d/2} + 2}}$.
\end{corollary}

\begin{proof}
    The algebra $A$ is
    $\floor{d/2}$-compressed and the claim follows from
    Proposition~\ref{ref:sp:nosmallorder:prop}.
\end{proof}

As a corollary we reobtain the result of Elias and Rossi, see \cite[Theorem~3.1]{EliasRossi_Analytic_isomorphisms}.
\begin{corollary}\label{ref:eliasrossi:cor}
    Suppose that $A$ is a finite compressed Gorenstein local $\kk$-algebra of
    socle degree $d\leq 4$. Then $A$ is canonically graded~i.e. isomorphic to
    its associated graded algebra $\operatorname{gr} A$.
\end{corollary}

\begin{proof}
    The case $d\leq 2$ is easy and left to the reader. We assume $d\geq 3$, so
    that $3\leq d\leq 4$.

    Fix $n = H_A(1)$ and choose $f\in \DP = \kk_{dp}[x_1, \ldots ,x_n]$ such
    that $A  \simeq \Apolar{f}$.
    Let $f_d$ be the top degree part of $f$. Since $\floor{d/2} + 2 = d$,
    Corollary~\ref{ref:compressedpoly:cor} implies that $f_d\in \orbitunip{f}$. Therefore the
    apolar algebras of $f$ and $f_d$ are isomorphic.
    The algebra $\Apolar{f_d}$ is a quotient of $\operatorname{gr}
    \Apolar{f}$. Since $\dimk \operatorname{gr}\Apolar{f} = \dimk \Apolar{f} =
    \dimk \Apolar{f_d}$ it follows that
    \[
        \Apolar{f}  \simeq \Apolar{f_{d}}  \simeq
        \operatorname{gr}\Apolar{f}.\qedhere
    \]
\end{proof}

The above Corollary~\ref{ref:eliasrossi:cor} holds under the assumptions that
    $\kk$ is algebraically closed and of characteristic not equal to $2$ or $3$. The
    assumption that $\kk$ is algebraically closed is unnecessary as proven for
    cubics in Example~\ref{ex:1nn1}, the cases of quartics is similar.

    The assumption on the characteristic is necessary.

    \begin{example}[compressed cubics in characteristic two]\label{ex:1nn1char2}
        Let $\kk$ be a field of characteristic two.
        Let $f_3\in \DP_{3}$ be a cubic form such that $H_{\Apolar{f_3}} = (1,
        n, n, 1)$ and $\Dx_1^2\hook f_3 = 0$.
        Then there is a degree three polynomial $f$ with leading form $f_3$,
        whose apolar algebra is compressed but not canonically graded.

        Indeed, take $\sigma = \Dx_1^2$. Then all derivatives of $\sigma$ are
        zero because the characteristic is two. By Proposition~\ref{ref:tangentspacepoly:prop} the element
        $\sigma$ lies in $\perpspace{\tang{f_3}}$. Thus $\tang{f_3}$ does not
        contain $\DP_{\leq 2}$ and so $\orbit{f_3}$ does not contain $f_3 +
        \DP_{\leq 2}$. Taking any $f\in f_3 + \DP_{\leq 2}$ outside the orbit
        yields the desired polynomial. For example, the polynomial $f = f_3 +
        \DPel{x_i}{2}$ lies outside the orbit.

        A similar example shows that over a field of characteristic three
        there are compressed quartics which are not canonically graded.
    \end{example}

\section{Examples II --- $(1, 3, 3, 3,
1)$}\label{ssec:examplestwo}
In this section we present an example (\cite{Jel_classifying}), where we
actually explicitly classify up to isomorphism finite Gorenstein
algebras with Hilbert function $(1, 3, 3, 3, 1)$.
\begin{example}[Hilbert function $(1, 3, 3, 3, 1)$]\label{ex:13331}
    Assume $\chark \neq 2, 3$ and $\kk = \kkbar$.
    Consider a polynomial $f\in \DP = \kk_{dp}[x, y, z]$ whose Hilbert function is
    $(1, 3, 3, 3, 1)$.
    Let $F$ denote the leading form of $f$.
    By \cite{LO} or \cite[Proposition~4.9]{cjn13} the form $F$ is linearly equivalent to
    one of the following:
    \[
        F_1 = \DPel{x}{4} + \DPel{y}{4} + \DPel{z}{4},\quad F_2 = \DPel{x}{3}y + \DPel{z}{4},\quad F_3 = \DPel{x}{3}y + \DPel{x}{2}\DPel{z}{2}.
    \]

    Since $\Apolar{f}$ is $1$-compressed, we have
    $\Apolar{f}  \simeq \Apolar{f_{\geq 3}}$; we may assume that the quadratic part
    is zero. In fact by the explicit description of top degree form in
    Proposition~\ref{ref:topdegreecomp:prop} we see that
    \[\orbitunip{f} = f + \tangunip{F} + P_{\leq 2}.\]
    Recall that $\Dgrp/\Dgrpunip$ is the product of the group of linear
    transformations and $\kk^*$ acting by multiplication.

    \paragraph{The case $F_1$.} Since $\Ann(F)_{\leq 3} = (\alpha\beta,
    \alpha\gamma, \beta\gamma)$, we see that $\perpspace{\tangunip{F}}_{\leq 3}$ is spanned by $\alpha\beta\gamma$.
    Therefore we may assume $f = F_1 + c\cdot xyz$ for some $c\in \kk$. By
    multiplying variables
    by suitable constants and then multiplying whole $f$ by a constant, we may
    assume $c = 0$ or $c = 1$. As before, we get
    two non-isomorphic algebras. Summarizing, we got two isomorphism types:
    \[
        f_{1,0} = \DPel{x}{4} + \DPel{y}{4} + \DPel{z}{4},\quad f_{1,1} = \DPel{x}{4} + \DPel{y}{4} + \DPel{z}{4} + xyz.
    \]
    Note that $f_{1, 0}$ is canonically graded, whereas $f_{1,1}$ is a
    complete intersection.

    \paragraph{The case $F_2$.} We have $\Ann(F_2)_2 = (\alpha\gamma,
    \beta^2, \beta\gamma)$, so that $\perpspace{\tangunip{F_2}}_{\leq
    3} = \sspan{\beta^3, \beta^2\gamma}$. Thus we may assume $f =
    F_2 + c_1 \DPel{y}{3} + c_2 \DPel{y}{2}z$. As before, multiplying $x$, $y$ and $z$ by
    suitable constants we may assume $c_1, c_2\in \{0, 1\}$. We get four
    isomorphism types:
    \[
        f_{2, 00} = \DPel{x}{3}y + \DPel{z}{4},\ \ f_{2, 10} = \DPel{x}{3}y +
        \DPel{z}{4} + \DPel{y}{3},\ \ f_{2,
        01} =
        \DPel{x}{3}y + \DPel{z}{4} + \DPel{y}{2}z,\ \ f_{2, 11} = \DPel{x}{3}y + \DPel{z}{4} + \DPel{y}{3} + \DPel{y}{2}z.
    \]
    To prove that the apolar algebras are pairwise non-isomorphic one shows
    that the only linear maps preserving $F_2$ are diagonal and argues
    as described in the case of $F_3$ below.

    \def\complement{\sspan{\DPel{y}{3},\, \DPel{y}{2}z,\, y\DPel{z}{2}}}%
    \paragraph{The case $F_3$.} We have $\Ann(F_3)_2 = (\beta^2,\,
    \beta\gamma,\,\alpha\beta - \gamma^2)$ and
    \[\perpspace{\tangunip{F_3}}_{\leq 3} = \sspan{\beta^2\gamma,\ \beta^3,\
        \alpha\beta^2 -
    2\beta\gamma^2}.\]
    We choose $\complement$ as the complement of $\tangunip{F_3}$
    in $P_3$. Therefore the apolar algebra of each $f$ with top degree form $F_3$ is
    isomorphic to the apolar algebra of
    \[
        f_{3, *} = \DPel{x}{3}y + \DPel{x}{2}\DPel{z}{2} + c_1\DPel{y}{3} + c_2\DPel{y}{2}z + c_3y\DPel{z}{2}
    \]
    and two distinct such polynomials $f_{3, *1}$ and $f_{3, *2}$ lie in
    different $\Dgrpunip$-orbits.
    We identify the set of $\Dgrpunip$-orbits with
    $P_3/\tangunip{F_3}  \simeq \complement$.
    We wish to determine isomorphism classes, that is,
    check which such $f_{3, *}$ lie in the same $\Dgrp$-orbit. A little care
    should be taken here, since $\Dgrp$-orbits will be bigger than in the
    previous cases.

    Recall that $\Dgrp/\Dgrpunip  \simeq \Dglcot \times \kk^*$ preserves the
    degree. Therefore, it is enough to look at the operators stabilizing
    $F_3$. These are $c\cdot g$, where $c\in \kk^*$ is a constant and $g\in
    \Dglcot$ stabilizes $\sspan{F_3}$, i.e.~$\Ddual{g}(\sspan{F_3}) =
    \sspan{F_3}$. Consider such a $g$.
    It is a linear automorphism of $\DP$ and maps $\Ann(F)$ into itself.
    Since $\beta(\lambda_1\beta +
    \lambda_2\gamma)$ for $\lambda_i\in \kk$ are the only reducible quadrics in $\Ann(F_3)$ we see that
    $g$ stabilizes $\sspan{\beta, \gamma}$, so that $\Ddual{g}(x) = \lambda x$ for a
    non-zero $\lambda$. Now it is straightforward to check directly that the
    group of linear maps stabilizing $\sspan{F_3}$ is generated by the following
    elements
\def\DtmpA{a}%
\def\DtmpB{b}%
    \begin{enumerate}
        \item homotheties: for a fixed $\lambda\in \kk$ and for all linear
            forms
            $\ell\in \DP$ we have $\Ddual{g}(\ell) = \lambda \ell$.
        \item for every $\DtmpA , \DtmpB \in \kk$ with $\DtmpB \neq 0$,
            the map $\DPut{tab}{t_{\DtmpA , \DtmpB }}$ given by
            \[
                \Dtab(x) = x,\ \ \Dtab(y) = -\frac{3}{2}\DtmpA ^2x + \DtmpB ^2
                y - 3\DtmpA \DtmpB  z,\ \
                \Dtab(z) = \DtmpA x + \DtmpB z.
            \]
            which maps $F_3$ to $\DtmpB ^2 F_3$.
    \end{enumerate}
    The action of $\Dtab$ on $P_3/\tangunip{F_3}$ in the basis $(\DPel{y}{3}, \DPel{y}{2}z, y\DPel{z}{2})$
    is given by the matrix (its entries slightly differ from~\cite[p.~23]{Jel_classifying} due to a different choice of
    basis):
    \[
        \begin{pmatrix}
            \DtmpB ^6 & 0 & 0\\
            -3\DtmpA \DtmpB ^5 & \DtmpB ^5 & 0\\
            \frac{39}{4}\DtmpA ^2\DtmpB ^4 & -13\DtmpA \DtmpB ^4 & \DtmpB ^4\\
        \end{pmatrix}
    \]
    Suppose that $f_{3, *} = \DPel{x}{3}y + \DPel{x}{2}\DPel{z}{2} + c_1\DPel{y}{3} + c_2\DPel{y}{2}z + c_3y\DPel{z}{2}$ has
    $c_1 \neq 0$. The above matrix shows that we may choose $\DtmpA $ and
    $\DtmpB $ and a homothety $h$ so that
    \[(h\circ \Dtab)(f_{3, *}) = c(\DPel{x}{3}y + \DPel{x}{2}\DPel{z}{2} + \DPel{y}{3} + c_3 y\DPel{z}{2}),\quad
    \mbox{where}\ c\neq 0,\ c_3\in \{0, 1\}.\]
    Suppose $c_1 = 0$. If $c_2 \neq 0$ then we may choose $\DtmpA $,
    $\DtmpB $ and $\lambda$ so
    that $(h\circ \Dtab)(f_{3, *}) = \DPel{x}{3}y + \DPel{x}{2}\DPel{z}{2} + \DPel{y}{2}z$. Finally, if $c_1 = c_2 =
    0$, then we may choose $\DtmpA = 0$ and $\DtmpB$, $\lambda$ so that $c_3 = 0$ or $c_3 = 1$. We get
    at most five isomorphism types:
    \begin{align*}
        &f_{3, 100} = \DPel{x}{3}y + \DPel{x}{2}\DPel{z}{2} +
        \DPel{y}{3},\quad\ f_{3, 101} = \DPel{x}{3}y + \DPel{x}{2}\DPel{z}{2} +
        \DPel{y}{3} + y\DPel{z}{2},\\
        &f_{3, 010} = \DPel{x}{3}y + \DPel{x}{2}\DPel{z}{2} +
        \DPel{y}{2}z,\quad f_{3, 001} =
        \DPel{x}{3}y + \DPel{x}{2}\DPel{z}{2} + y\DPel{z}{2},\\
        &f_{3, 000} = \DPel{x}{3}y + \DPel{x}{2}\DPel{z}{2}.
    \end{align*}
    By using the explicit description of the $\Dgrp$ action on $P_3/\tangunip{F_3}$
    one checks that the apolar algebras of the above polynomials are pairwise
    non-isomorphic.

    \paragraph{Conclusion:} There are $11$ isomorphism types of algebras with
    Hilbert function $(1, 3, 3, 3, 1)$.
    We computed the tangent spaces to the corresponding orbits in
    characteristic zero, using a computer implementation of the description in
    Proposition~\ref{ref:tangentspacepoly:prop}.
    The dimensions of the orbits are as follows:
\begin{center}

\begin{tabular}{@{}l l c @{}}
    orbit && dimension\\ \midrule
    $\orbit{(\DPel{x}{4} + \DPel{y}{4} + \DPel{z}{4} + xyz)}$   && $29$\\
    $\orbit{(\DPel{x}{4} + \DPel{y}{4} + \DPel{z}{4})}$ && $28$\\
    $\orbit{\left( \DPel{x}{3}y + \DPel{z}{4} + \DPel{y}{3} + \DPel{y}{2}z \right)}$ && $28$\\
    $\orbit{\left( \DPel{x}{3}y + \DPel{z}{4} + \DPel{y}{3} \right)}$ && $27$\\
    $\orbit{\left(\DPel{x}{3}y + \DPel{z}{4} + \DPel{y}{2}z\right)}$ && $27$\\
    $\orbit{\left(\DPel{x}{3}y + \DPel{z}{4}\right)}$ && $26$\\
\end{tabular}
\hspace{.5cm}\begin{tabular}{@{}l l c @{}}
    orbit && dimension\\ \midrule
    $\orbit{\left(\DPel{x}{3}y + \DPel{x}{2}\DPel{z}{2} + \DPel{y}{3} + y\DPel{z}{2} \right)}$ && $27$\\
    $\orbit{\left(\DPel{x}{3}y + \DPel{x}{2}\DPel{z}{2} + \DPel{y}{3} \right)}$ && $26$\\
    $\orbit{\left( \DPel{x}{3}y + \DPel{x}{2}\DPel{z}{2} + \DPel{y}{2}z \right)}$ && $26$\\
    $\orbit{\left( \DPel{x}{3}y + \DPel{x}{2}\DPel{z}{2} + y\DPel{z}{2} \right)}$ && $25$\\
    $\orbit{\left( \DPel{x}{3}y + \DPel{x}{2}\DPel{z}{2}\right)}$ && $24$\\
    &&
\end{tabular}
\end{center}
\def\DGLthree{\operatorname{GL}_3}
    The closure of the orbit of $f_{1, 1} = \DPel{x}{4} + \DPel{y}{4} + \DPel{z}{4} + xyz$ is
    contained in $\DGLthree(\DPel{x}{4} + \DPel{y}{4} + \DPel{z}{4}) + \DP_{\leq 3}$, which is
    irreducible of dimension $29$. Since the orbit itself has dimension $29$
    it follows that it is dense inside. Hence the orbit closure contains
    $\DGLthree(\DPel{x}{4} + \DPel{y}{4} + \DPel{z}{4}) + \DP_{\leq 3}$.
    Moreover, the set $\operatorname{GL}_{3}(\DPel{x}{4} + \DPel{y}{4} + \DPel{z}{4})$ is dense inside
    the set $\sigma_3$ of forms $F$ whose apolar algebra has Hilbert function
    $(1, 3, 3, 3, 1)$.
    Thus the orbit of $f_{1, 1}$ is dense inside the set of polynomials with
    Hilbert function $(1, 3, 3, 3, 1)$. Therefore, the latter set is
    irreducible and of dimension $29$.

    It
    would be interesting to see which specializations between different
    isomorphism types are possible. There are some obstructions. For example,
    the $\DGLthree$-orbit of $\DPel{x}{3}y + \DPel{x}{2}\DPel{z}{2}$ has smaller dimension than the
    $\DGLthree$-orbit of $\DPel{x}{3}y + \DPel{z}{4}$. Thus $\DPel{x}{3}y + \DPel{x}{2}\DPel{z}{2} + \DPel{y}{3} + y\DPel{z}{2}$ does
    not specialize to $\DPel{x}{3}y + \DPel{z}{4}$ even though its $\Dgrp$-orbit has higher
    dimension.
\end{example}

\vspace{-5mm}
\section{Examples III --- preliminaries for
    Chapter~\ref{sec:Gorensteinloci}}\label{ssec:examplesthree}

In this subsection we gather several technical results which we later
use when discussing ray families in Section~\ref{ssec:rayfamilies}. They imply
that a given dual generator $f$ can be transformed, using nonlinear change of
coordinates, to an ``easier'' form, e.g., which some monomials absent.
Geometrically, we change the embedding of $\Spec\Apolar{f}$ inside $\Spec
\DShat$. These results first appeared, in a somewhat partial form,
in~\cite{cjn13} and here are recast using the presentation
of~\cite{Jel_classifying}, given in Section~\ref{ssec:automorphisms}.

\begin{lemma}\label{ref:topdegreetwist:lem}
    \def\hatf{g}%
    Fix $d$ and assume $\chark = 0$ or $\chark > d$.
    Let $f\in \DP$ be a polynomial of degree $d$ and $\Dx_1\in \DmmS$ be such
    that $\Dx_1^d \hook f \neq 0$. Then in $\orbit{f}$ there is a polynomial
    $\hatf$, such that
    \begin{enumerate}
        \item $\Dx_1^d \hook \hatf = 1$,
        \item polynomial $\hatf$ contains no monomials of the form $\DPel{x_1}{i}$ with $i< d$.
        \item polynomial $\hatf$ contains no monomials of the form $\DPel{x_1}{i}x_j$ for $j\neq
            1$ and $i$ arbitrary,
    \end{enumerate}
\end{lemma}
\begin{proof}
    Acting with $\Dglcot$ on $f$ we may assume $\Dx_1^d \hook f = 1$ and
    $\Dx_1^{d-1}\Dx_j \hook f = 0$ for all $j \neq 1$.
    We will modify $f$, so that it satisfies Condition~3.
    Suppose $i$ is the largest exponent such that a monomial $\DPel{x_1}{i}x_j$ with
    $j\neq 1$ appears in $f$. We argue by downward induction on $i$.
    Considers all terms of $f$ having form
    $\lambda_j \DPel{x_1}{i} x_j$
    with $\lambda_j\in \kk$. Let $\ell = \sum \lambda_j x_j$.
    Then $\DPel{x_1}{i}\ell = \Dx_1^{d-i}\hook(\ell \DPel{x_1}{d})$ is the sum of all terms of
    $\Dx_1^{d-i}\hook (\ell f)$ which have the form $\DPel{x_1}{i}x_j$. Also $\Dx_1^{d-i}(x_jf) \in \tangunip{f}$ by
    Proposition~\ref{ref:uniptangentspacepoly:prop}. Let $D\in \Dgrpuniptang$ be
    any element such that $\Dx_1^{d-i}(x_jf) = Df$, then $\exp(-D)f = f - Df +
     \ldots \in
     \orbitunip{f}$ contains no terms of the form $\DPel{x_1}{i}x_j$ with $j\neq 1$. We
    replace $f$ by $\exp(-D)f$ and continue by induction. Hence, we obtain $f$
    satisfying Conditions 1.~and~3. An appropriate partial of $f$ satisfies
    all three conditions.
\end{proof}

\begin{example}\label{ref:standardformofstretched:ex}
    Suppose that a finite local Gorenstein algebra $A$ of socle degree $d$ has
    Hilbert function equal to $(1, H_1, H_2, \dots, H_c, 1, \dots, 1)$. The
    standard form of the dual generator of $A$ is
    \[f = \DPel{\DPut{yd}{x_1}}{d} + \kappa_{d-1}\DPel{\Dyd}{d-1} + \dots +
    \kappa_{c+2}\DPel{\Dyd}{c+2} + g,\] where $\deg g\leq c+1$ and
    $\kappa_{\bullet}\in \kk$. By adding a suitable
    derivative we may furthermore make all $\kappa_{i} = 0$ and assume
    that $\Dx_1^{c+1}\hook g = 0$. Using Lemma \ref{ref:topdegreetwist:lem} we
    may also assume that $g$ contains no monomials of the form $\DPel{x_1}{c}x_j$ with
    $j\neq 1$. By adding a suitable derivative of $f$ again, we may assume
    that $g$ does not contain the monomial $\DPel{x_1}{c}$, so in fact
    $\Dx_{1}^c \hook g = 0$. This gives a dual generator
    \[f = \DPel{x_1}{d} + g,\]
    where $\deg g \leq c+1$ and $g$ does not contain monomials divisible by
    $\DPel{\Dyd}{c}$.
\end{example}

The following is a seemingly easy yet subtle enhancement of the symmetric
decomposition. It was proven in a slightly weaker form in \cite{CN2stretched}
and in~\cite{cjn13}.
\begin{proposition}\label{ref:squares:prop}
    Let $\kk = \kkbar$ be a field of $\chark \neq 2$.
    Let $A$ be finite local Gorenstein algebra of socle degree $d\geq 2$ whose Hilbert function decomposition
    has $\Dhdvect{d-2} =
    (0, q, 0)$. Then $A$ is isomorphic to the apolar
    algebra of a polynomial $f$ such that $f$ is in the standard form and the
    quadric part $f_2$ of $f$ is a sum of
    $q$ squares of variables not appearing in $f_{\geq 3}$.
\end{proposition}

\begin{proof}
    Take a dual generator $f\in \DP := \kdp[x_1,\dots,x_n]$
    of algebra $A$ in the standard form. We will twist $f$ to obtain the required form of
    $f_2$. We may assume that $H_{\Apolar{f}}(1) = n$.

    If $d = 2$, then the theorem follows from the fact that the quadric $f$ may be
    diagonalized. Assume $d\geq 3$.
    Let
    $\DPut{parte}{e} := n_{d-3} = \sum_{a=0}^{d-3} \Dhd{a}{1}$. We have
    $\DPut{tote}{n} = n_{d-2} = e + q$,
    so that $f_{\geq 3}\in \kdp[x_1,\dots,x_{\Dparte}]$ and $f_2\in
    \kdp[x_1,\dots,x_{\Dtote}]$. Note that $f_{\geq 3}$ is also in the standard
    form, so that every linear form in $x_1, \ldots ,x_{\Dparte}$ is a
    derivative of $f_{\geq 3}$.

    If
    $\Dx_{\Dtote}\hook f\in
    \kdp[x_1,\dots,x_{\Dparte}]$ then there exists an operator $\partial\in \DmmS^2$ such
    that $\pp{\Dx_{\Dtote} - \partial}\hook f = 0$. This contradicts the fact
    that $f$ was in the standard form.
    So we get that $\Dx_{\Dtote} \hook f$ contains some $x_r$ for $r >
    \Dparte$, i.e.~$f$
    contains a monomial $x_rx_{\Dtote}$. A linear change
    of variables involving only $x_r$ and $x_{\Dtote}$ preserves the standard form and gives
    $\Dx_{\Dtote}^2 \hook f \neq 0$. Another change asserts that
    $\Dx_{\Dtote}^2\hook f =1$ and $\Dx_{\Dtote}\Dx_j \hook f = 0$ for $j\neq
    \Dtote$. Repeating, we obtain $f_2 = f_{2, 0} + \DPel{x_{\Dparte+1}}{2} +  \ldots
    + \DPel{x_{\Dtote}}{2}$ with $f_{2, 0}\in \kdp[x_1, \ldots ,x_{\Dparte}]$.

    It remains to prove that $f - f_{2, 0}\in \orbit{f}$. By
    Proposition~\ref{ref:topdegreecomp:prop} it is enough to prove that
    $x_ix_j\in \tangunip{f}$ for all $i, j\leq \Dparte$. Suppose that this is
    not the case and pick $\sigma\in \perpspace{\tangunip{f}}$ containing a
    monomial $\Dx_i\Dx_j$. By Proposition~\ref{ref:uniptangentspacepoly:prop}
    we have $\deg(\ithpartial{\sigma}{i}\hook f) \leq 1$ and clearly
    $\ithpartial{\sigma}{i}$ is of order one. Let $\tau\in \DA$ be the image
    of $\ithpartial{\sigma}{i}$, then $\tau\in \mm\cap \Dmmperp{}$.

    Since $f$ is in the standard form, the images of operators
    $\Dx_{\Dparte+1}, \ldots ,\Dx_{\Dtote}$ in $\DA$ span
    $\frac{\mm\cap\Dmmperp{}}{\mm^2\cap \Dmmperp{}} = \iaq{d-2}{1}$. Therefore
    the image of $\tau\in \iaq{d-2}{1}$ is zero, so $\tau\in \mm^2\cap
    \Dmmperp{}$. This means that there is an operator $\tau_2\in \DmmS^2$ such
    that $\ithpartial{\sigma}{i} - \tau_2$ annihilates $f$. But
    $\ithpartial{\sigma}{i} - \tau_2$ is of order one;
    this is a contradiction with $H_{\Apolar{f}} = n$. Hence we conclude that
    no $\tau$ exists, so that $x_ix_j\in \tangunip{f}$ for all $i, j \leq
    \Dparte$ and hence $f - f_{2, 0} = f_{\geq 3} + \DPel{x_{\Dparte+1}}{2} +  \ldots
    + \DPel{x_{\Dtote}}{2}\in \orbitunip{f}$.\qedhere
    \vspace*{-1cm}
\end{proof}

\begin{partwithabstract}{Hilbert schemes}\label{part:families}%
In this part we shift our attention from single algebras to families, specifically
to families of quotients of a polynomial ring (and others rings).
We change the language from algebras to schemes, so we speak about
families of finite subschemes of affine space (and other varieties) over $\kk$.
We investigate the geometry of the ``largest'' such family, which is called
\emph{the Hilbert scheme of points} of an affine space.

Its geometry is given naturally and uniquely, but remains to a large extent
unknown, see the Open Problems in Section~\ref{ssec:openquestions}.  After
the introduction, we review an
abstract framework of smoothings.  We compare abstract smoothings, embedded
smoothings and the geometry of the smoothable component of the Hilbert scheme.
We give examples of smoothings (Section~\ref{ssec:kollar},
Section~\ref{ssec:klimits}) and of nonsmoothable schemes
(Section~\ref{ssec:examplesnonsmoothable}).
We follow~\cite{jabu_jelisiejew_smoothability} and include some
folklore or unpublished results and examples.

The language of schemes, although necessary, is notably technical.
A good introduction is, for
instance, \cite{eisenbud_harris}. Much more details are provided in
\cite{hartshorne}, \cite{gortz_wedhorn_algebraic_geometry_I},
\cite{Vakil_foag} or \cite{stacks_project}. Most of the notions needed for our
purposes are briefly summarized in~\cite{jabu_jelisiejew_smoothability}. Our
main interest lies in the local theory; an algebraically minded reader is free
to, for example, replace the central notion of finite flat family
(Definition~\ref{ref:finiteflatfamily:def}) with a flat and finite
homomorphism of $\kk$-algebras.
\end{partwithabstract}

\chapter{Preliminaries}

\section{Moduli spaces of finite algebras, Hilbert
schemes}\label{ssec:hilbertschemes}

\newcommand{\OT}{\OO_T}%
\newcommand{\OX}{\OO_{\ccX}}%
A \emph{finite $\kk$-scheme} is an affine scheme $\Spec\DA$ for a finite $\kk$-algebra $\DA$.
We do not impose any conditions on the residue fields of $\DA$.
We transfer properties of $\DA$ to $\Spec \DA$, for example we say that $\Spec
\DA$ is \emph{Gorenstein} if and only if $\DA$ is Gorenstein
(Definition~\ref{ref:Gorenstein:def}).

A morphism
$\pi\colon \ccX\to T$ is \emph{affine} if for every affine open subset $U =
\Spec B$ of
$T$ the preimage $\pi^{-1}(U)$ is affine: $\pi^{-1}(U) = \Spec B'$.
An affine morphism $\pi\colon \ccX\to T$ is \emph{finite} (resp.~\emph{flat})
if $B'$ above is a
finite $B$-module (resp.~a flat $B$-module). Note that if $\pi\colon\ccX\to T$
is both finite and flat, then $B'$ is a locally free $B$-module
(see~\cite[Exercise~6.2]{Eisenbud}).
\begin{definition}\label{ref:finiteflatfamily:def}
    A \emph{family} of finite $\kk$-schemes over $T$ is a finite flat morphism $\ccX\to
    T$.
\end{definition}
Flatness and finiteness of a family $\ccX\to T$ together imply that
the sheaf $\pi_*\OX$ is a vector bundle on $T$. If $T$ is connected, this bundle has constant
rank $r$, so that each fiber of $\pi$ is an algebra of degree $r$, we then say
that $\pi$ has degree $r$. See Example~\ref{ex:properness} for some
pathologies without flatness or finiteness assumptions.

Since a finite algebra is a vector space with multiplication, a family,
intuitively, should be a vector space with continuously varying
multiplication. We explain why it is (locally) so under our definition.
Locally on $T$, the bundle $\pi_*\OX$ is free, so it is
isomorphic to $\OT\tensor_{\kk} V$ for an $r$-dimensional $\kk$-vector space
$V$. The multiplication on $\OX$ gives rise to
a $\OT$-linear map $\mu:(\OT\tensor_{\kk} V) \tensor_{\OT}
(\OT\tensor_{\kk} V)\to \OT\tensor_{\kk}
V$, which is equivalent to a $\kk$-linear map
\[
    \mu:V\otimes_{\kk} V\to \OT\tensor_{\kk} V.
\]
Fixing a basis $e_1, \ldots ,e_r$ of $V$ gives $\mu$ a form
$\mu(e_i\tensor e_j) = \sum_k a_{ijk} e_k$ for $a_{ijk}\in \OT$, which
precisely reflects the intuition of a continuously varying multiplication.
Conversely, given a map $\mu$, we obtain a family $\Speccal_T \mathcal{A} \to T$,
where $\mathcal{A}$ is the algebra $\OT\tensor_{\kk} V$ with multiplication $\mu$.

\begin{example}\label{ex:jumphenonmenon}
    An example of a finite flat family above is $\pi\colon \ccX\to T$, where
    \[\ccX = V(x^2-tx) \subset \mathbb{A}^2 =
        \Spec \kk[t, x]\quad\mbox{and}\quad T = \mathbb{A}^1 = \Spec \kk[t].\]
        The fiber $\pi^{-1}(\lambda)$ over every $\lambda\in \kk^*$ is
    \[
        \Spec \frac{\kk[x, t]}{(x^2-tx, t-\lambda)}   \simeq  \Spec
        \frac{\kk[x]}{(x(x-\lambda))}\simeq \Spec (\kk^{\times 2})
    \]
    and the fiber over zero is
    \[
        \Spec \frac{\kk[x, t]}{(x^2-tx, t)} \simeq \Spec\frac{\kk[x]}{(x^2)},
    \]
    see Figure~\ref{fig:doublecover}.
    The bundle $\pi_*\OX$ is free, $\pi_*\OX = \OT\tensor_{\kk} V$ for
    $V = \sspan{1, x}$. The corresponding
    $\mu:V\tensor_{\kk} V\to \OT\tensor_{\kk} V$ is given by $\mu(x\tensor x)
    = tx$ and $\mu(1\tensor x) = \mu(x\tensor 1) = x$, $\mu(1\tensor 1) = 1$.
    \begin{figure}[h]
        \label{fig:doublecover}%
        \centering
        \includegraphics{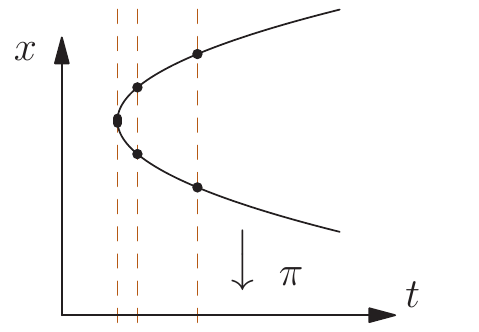}%
        \caption{Ramified double cover as a family of degree two schemes, see
            Example~\ref{ex:jumphenonmenon}.}%
    \end{figure}
\end{example}

\begin{example}\label{ex:properness}
    The morphism $\Spec \kk[x, t]/(xt -1)\to \Spec \kk[t]$ is flat but not
    finite. All fibers over $\kk$-points are isomorphic to $\Spec \kk$
    expect for the fiber over $t=0$, which is empty.

    The morphism $\Spec \kk[x, t]/(x, t)\to \Spec \kk[t]$ is finite but not
    flat. All fibers are zero except for
    the fiber over $t=0$, which is $\Spec \kk$.

    Taking the union of these morphisms we obtain $\Spec \kk[x,
    t]/(x^2t -x, xt^2 - t)\to \Spec \kk[t]$ which is neither
    finite nor flat and such that each fiber is isomorphic to $\Spec \kk$.
\end{example}

\newcommand{\SchOp}{\operatorname{Sch_{\kk}}^{op}}%
\newcommand{\Set}{\operatorname{Set}}%
\newcommand{\moduli}{\mathcal{M}}%
\newcommand{\univ}{\mathcal{U}}%
    Even with the notion of (finite flat) family we
    still lack some geometry. For example, we would like to compare families
    and think about the largest family, containing all possible finite schemes.
    We do this using the notion of representable functors. We do not
    really use the strength of this theory and the language of functors is
    notably technical, so below we slightly change the presentation:
    while everything which we say is precise, it differs from the usual
    presentation of representable functors. For a nice presentation of those,
    we refer to~\cite[Chapter~VI]{eisenbud_harris} and, specifically for the
    Hilbert functor, to~\cite{stromme}.
    \newcommand{\Artbare}{\mathcal{F}in\mathcal{S}ch}%

    Let us aggregate families as follows. For each
    $\kk$-scheme $T$ we define the set
    \[
        \Artbare(T) = \left\{ \ccX\to T \mbox{ finite flat} \right\}.
    \]
    For every morphism $\varphi\colon T'\to T$ we have a pullback map $\Artbare(\varphi):\Artbare(T) \to \Artbare(T')$
    pulling back each family $\ccX\to T$ to an element $\ccX \times_{T} T'\to
    T'$ of $\Artbare(T')$. Formally, $\Artbare:\SchOp\to \Set$ defined in this
    way is a \emph{functor}.

    The intuition about a largest family is encoded into the notion of
    functor represented by a scheme (a \emph{representable functor}).
    \newcommand{\Func}{\mathcal{F}un}
    \begin{definition}
        Let $\Func\colon \SchOp \to \Set$ be a functor. We say that $\Func$ is
        \emph{represented} by a \mbox{$\kk$-scheme} $\moduli$ if there exists a \emph{universal
        family} $\univ\in \Func(\moduli)$, such that for every $T$ and
        $\ccX\in \Func(T)$ there is a \emph{unique} morphism $\varphi\colon T\to \moduli$
        such that $\ccX = \Func(\varphi)(\univ)$, i.e., the family $\ccX$ is a
        pullback via $\varphi$ of $\univ$.
    \end{definition}

    The intuition behind $\Func$ being represented by $\moduli$ is that $\univ$ is a
    largest family, containing every family. A crucial additional part
    is that every family comes from a \emph{unique} pullback of
    universal family. For example, an element of $\Func(\kk)$ corresponds to
    a unique $\kk$-point of $\moduli$ and more generally, $\Func(T)  \simeq
    \Homthree{}{T}{\moduli}$ naturally. This natural isomorphism is usually
    taken as a definition of representability.
    Under this isomorphism, the element $\univ$ corresponds to
    $\operatorname{id}_{\moduli}\in
    \Homthree{}{\moduli}{\moduli}$.

    If there exists a scheme
    $\moduli$ representing $\Func$, it is uniquely determined up to a unique isomorphism: if $\moduli_1$
    and $\moduli_2$ are two representing schemes, then they
    induce unique maps $\moduli_1\to \moduli_2$ and $\moduli_2\to \moduli_1$.
    By uniqueness, the compositions $\moduli_1\to \moduli_2\to \moduli_1$ and
    $\moduli_2\to \moduli_1\to \moduli_2$ are identities, so $\moduli_1 \simeq
    \moduli_2$.

    If $\Artbare$ were represented by $\moduli$, then the $\kk$-points of
    $\moduli$ would correspond bijectively to elements of $\Artbare(\kk)$, i.e., to
    finite $\kk$-schemes.
%    The challenge is to construct a representing space $\moduli$ for
%    $\Artbare$.
    We
    will now show in Example~\ref{ex:moduliofartinnonexistent} that in fact
    such $\moduli$ does
    not exist; $\Artbare$
    is \emph{not} represented by any scheme. This example does not render the
    above discussion irrelevant, it just motivates the
    advantage of changing $\Artbare$ by adding more information.

\begin{example}\label{ex:moduliofartinnonexistent}
    \def\moduli{\mathcal{M}}%
    The scheme representing $\Artbare$ does not exist.
    More precisely, there is no scheme $\moduli$ over $\kk$ such that:
    \begin{enumerate}
        \item the $\kk$-points of $\moduli$ correspond to finite
            $\kk$-schemes and only finitely many points correspond to a given
            scheme.
        \item every finite flat family $\ccX \to C$ over $\kk$ induces a
            morphism $\varphi\colon C\to \moduli$ sending each $\kk$-point $c\in C$
            to a $\kk$-point $\varphi(c)$ corresponding to scheme $\ccX_c$; no uniqueness
            assumed.
    \end{enumerate}
    Indeed, in Example~\ref{ex:jumphenonmenon} we have seen a family $\ccX\to \mathbb{A}^1$ such that the fibers over
    $\mathbb{A}^1\setminus \{0\}$ are all isomorphic and not isomorphic to the fiber $\ccX_0$.
    If $\moduli$ existed, the induced morphism $\mathbb{A}^1 \to \moduli$
    would map $\mathbb{A}^1 \setminus \{0\}$ to a closed $\kk$-point and $0\in
    \mathbb{A}^1$
    to another $\kk$-point; this is impossible.

    The presented problem persists across deformation theory and is known as
    \emph{jump phenomenon}, see~\cite[Section~23, in particular
    Remark~23.0.4]{hartshorne_deformation_theory}.
\end{example}

\newcommand{\famil}{\mathcal{Z}}%
In view of Example~\ref{ex:moduliofartinnonexistent} we refine $\Artbare$ to
a functor $\SchOp\to \Set$ parameterizing families embedded as closed subschemes of the
product of base and a fixed ambient variety $\ambient$:
\begin{equation}\label{eq:hilbertfunc}
    \Hilb{\ambient}(T) = \left\{ \ccX \subset T \times
        \ambient,\ \ccX\to T \mbox{ finite flat} \right\}.
\end{equation}
This functor parameterizes families of finite schemes which are subschemes of a given
ambient scheme $\ambient$. For every flat family $\pi\colon \ccX\to T$ over
connected $T$ the degree
of fibers is constant (as $\pi_*\OX$ is a vector bundle), so we subdivide
$\Hilb{\ambient}$ into a family of functors parameterized by
$r\geq 1$:
\[
    \Hilbr{\ambient}(T) = \left\{ \ccX \subset T \times
        \ambient,\ \ccX\to T \mbox{ finite flat of degree }r \right\}.
\]
\begin{theorem}\label{ref:representability:thm}
    If $\ambient$ is either quasi-projective or affine, then $\Hilb{\ambient}$
    is represented by a scheme $\Hilb{\ambient}$. Also, for all $r\geq
    1$, the functor $\Hilbr{\ambient}$ is represented by a scheme
    $\Hilbr{\ambient}$ and
    \begin{equation}\label{eq:subdivision}
        \Hilb{\ambient} = \coprod_{r\geq 1} \Hilbr{\ambient}.
    \end{equation}
    If $\ambient$ is projective then each $\Hilbr{\ambient}$ is also
    projective.
\end{theorem}
The scheme $\Hilb{\ambient}$ is called \emph{the Hilbert scheme of
points of $\ambient$}, while $\Hilbr{\ambient}$ is called \emph{the Hilbert
scheme of $r$ points of $\ambient$}. By representability, there is
a unique up to isomorphism $\UU \subset \ambient \times \Hilbr{\ambient}$, such that $\pi\colon \UU
\to \Hilbr{\ambient}$ is (finite flat) a family.
Every finite subscheme $R \subset \ambient$ gives a family $R\to \Spec \kk$,
hence a \emph{unique} $\kk$-point of $\Hilb{a\ambient}$, which we denote by
$[R]$. Conversely, any $\kk$-point of $\Hilb{\ambient}$ gives a finite
$\kk$-scheme $\pi^{-1}([R]) \subset \ambient$. Thus $\kk$-points of
$\Hilb{\ambient}$ bijectively correspond to finite subschemes of $\ambient$.
\begin{proof}[Proof of Theorem~\ref{ref:representability:thm}]
    Once we prove the existence of $\Hilbr{\ambient}$, the existence
    of $\Hilb{\ambient}$ and~\eqref{eq:subdivision} follow formally.
    The existence of $\Hilbr{\ambient}$ for quasi-projective $\ambient$ was proven by
    Grothendieck, see~\cite[Theorem~5.14]{fantechi_et_al_fundamental_ag}. We
    also showed that if
    $\ambient$ is projective, then $\Hilbr{\ambient}$ is also
    projective.
The
existence for affine $\ambient$ was proven by
Gustavsen-Laksov-Skjelnes~\cite{Gustavsen_Laksov_Skjelnes__Elementary_explicit_const}.
Grothendieck's and Gustavsen-Laksov-Skjelnes' proofs are quite technical, so we do not reproduce them
here, but we discuss the intuition in a very special case of $\ambient =
\mathbb{A}^n$.

\def\Oone{\OO(1)}
%\begin{proof}[Proof of existence of $\Hilb{\mathbb{A}^N}$ -- sketch]
    \def\IX{\mathcal{I}_{\ccX}}%
    For a monomial ideal $\lambda$ or degree $r$ let $B_{\lambda}$ denote
    the set of monomials not in $\lambda$. Consider the
    subfunctor
    \[
        \Hilbr{\mathbb{A}^N, \lambda}(T) = \left\{ \ccX \subset T \times
            \mathbb{A}^N,\ \pi\colon \ccX\to T \mbox{ finite flat, } \OT\cdot
            B_{\lambda}\mbox{ spans } \pi_*\OX\right\}.
    \]
    For every family $\famil\to T$ we find, at least locally on $T$,
    a $\lambda$ such that $[\famil\to T]\in \Hilbr{\mathbb{A}^n, \lambda}$.
    Therefore it is enough to prove that $\Hilbr{\mathbb{A}^N, \lambda}$
    is representable.
    Since $\OT\cdot B_{\lambda}$ spans $\pi_*\OX$, every monomial $m\in
    \lambda$ can be presented as $m - \sum_{b\in B_{\lambda}} a_{b,m} b\in
    \IX$, where $a_{b, m}\in
    \OT$. These structure constants give a map $T\to \mathbb{A}^{\aleph_0}$ to
    a countably dimensional
    affine space.

    Conversely given a scheme $T$ and a morphism $T\to \mathbb{A}^{\aleph_0}$
    we obtain a set of coefficients $a_{b,m}\in \OT$ such that
    $\sspan{m-\sum a_{b,
    m}b}$ is an ideal and we recover a family $\ccX\to T$ which is
    automatically flat, even free with basis $B_{\lambda}$. Thus $\Hilbr{\mathbb{A}^N,
    \lambda}$ is represented by a subscheme of $\mathbb{A}^{\aleph_0}$.
    This concludes the proof, however the embedding into $\mathbb{A}^{\aleph_0}$
    is far from minimal, as we explain now. Note that for every variable $x_i$ its
    powers $1, x_i,  \ldots , x_i^r$ are dependent over $\OT$ modulo $\IX$, thus there
    are elements $x_i^{r_0} - \sum_{j<r_0} c_j x_i^j\in \IX$ with $c_j\in \OT$
    and $r_0 \leq r$. Using these
    equations, we may reduce modulo $\IX$ all monomials of degree $\geq Nr$ to
    monomials of smaller degree. Therefore we actually need only finitely many
    coefficients $[a_{b,m}]$ to recover $\IX$ and the embedding into
    $\mathbb{A}^{\aleph_0}$ refines to an embedding into a finitely
    dimensional affine space.
\end{proof}

The above construction, while explicit and constructive, embeds the Hilbert
scheme into a very large affine space, thus preventing any direct computation.
There is much work done in lowering the embedding dimension and working with
the Hilbert schemes explicitly, see~for example
\cite{bertone_borel_open_cover, lella_roggero_Rational_Components}.

There are only a few cases when the Hilbert scheme of points can be described explicitly. We
present them below.

\begin{example}\label{ex:rone}
    Let $\ambient$ be quasi-projective or affine, so that $\Hilbr{\ambient}$
    exists.
    A family $\ccX \subset T \times \ambient\to T$ of degree $r
    = 1$ is isomorphic to $T$, so it induces a section $T  \simeq  \ccX\to
    \ambient$. It follows that $\Hilbarged{1}{\ambient}  \simeq \ambient$.
    The universal family is the diagonal $\Delta \subset \ambient \times
    \ambient$ together with a projection onto a factor.
\end{example}

\begin{example}
    If $\ambient$ is a curve, then
    $\Hilbr{\ambient}$ parameterizes hypersurfaces of degree $r$ in
    $\ambient$. For example, $\Hilbr{\PP^1}  \simeq \mathbb{P}^{r}$
    by~\cite[(4), p.~111]{fantechi_et_al_fundamental_ag}. More invariantly, if $\ambient = \mathbb{P}V$ for a
    two-dimensional vector space $V$, then $\Hilbr{\ambient}  \simeq
    \mathbb{P}\pp{\Sym^r \Ddual{V}}$.
\end{example}

%\paragraph{Tangent spaces.}
\newcommand{\Dtangspace}[1]{\mathbb{T}_{#1}}%
\newcommand{\Homthreesh}[3]{\mathcal{H}om_{#1}(#2, #3)}%
\newcommand{\Oamb}{\OO_{\ambient}}%
One information about the Hilbert scheme, which is easily obtained, is the
tangent space. For a $\kk$-point $[R]\in \Hilbr{\ambient}$, corresponding to a
subscheme $R \subset X$, we denote the tangent space by $\Dtangspace{\Hilbr{\ambient}, [R]}$.
\begin{proposition}[tangent space
    description]\label{ref:tangentspaceHilb:prop}
        \def\IR{\mathcal{I}_R}
        Let $\ambient$ be a scheme such that $\Hilbr{\ambient}$ exists and $R
        \subset \ambient$ be
        a finite $\kk$-subscheme of degree $r$ given by ideal sheaf $\IR
        \subset \Oamb$. Then $
        \Dtangspace{\Hilbr{\ambient}, [R]} = H^0(\ambient,
        \Homthreesh{\Oamb}{\IR}{\Oamb/\IR})$.
\end{proposition}
\begin{proof}
    The proof goes by classifying families over $\Spec
    \kk[\varepsilon]/\varepsilon^2$,
    see~\cite[Section~1.2]{hartshorne_deformation_theory} for an accessible and
    expanded presentation.
\end{proof}

\begin{example}[tangent space locally]\label{ex:tangentforGorenstein}
    In the setting of Proposition~\ref{ref:tangentspaceHilb:prop}, suppose
    additionally that $R =
    \Spec \DA \subset
    U=\Spec B$. Then $R \subset U$ is given by an ideal $I \subset B$ and
    $\Dtangspace{\Hilbr{\ambient}, [R]} = \Homthree{B}{I}{B} = \Homthree{A}{I/I^2}{A}$.
    If furthermore the scheme $R$ is supported on a single
    $\kk$-point and Gorenstein, then
    we have
    $\Homthree{A}{I/I^2}{A}  \simeq  \Homthree{\kk}{I/I^2}{\kk}$ by Lemma~\ref{ref:dualizingfunctors:lem}, so
    \[
        \dimk \Dtangspace{\Hilbr{\ambient}, [R]} = \dimk\Homthree{A}{I/I^2}{A} = \dimk \Homthree{\kk}{I/I^2}{\kk} = \dimk
        I/I^2.
    \]
\end{example}

\begin{example}[tangent space for $\kk$-points]\label{ex:smoothpoints}
    Suppose $r = 1$ and $x \in \ambient$ is a $\kk$-point. Then
    $\Hilbr{\ambient} = \ambient$, see Example~\ref{ex:rone}, and
    $\Dtangspace{\Hilbr{\ambient}, x} = \Dtangspace{\ambient, x}$. Similarly for general $r$, if $x_1,
    \ldots ,x_r\in \ambient$ are pairwise different $\kk$-points and $R = \{x_1,
    \ldots ,x_r\}$, then $\Dtangspace{\Hilbr{\ambient}, [R]}  \simeq
    \bigoplus \Dtangspace{\ambient, x_i}$.
\end{example}

When considering the Hilbert scheme from the perspective of classifying finite
subschemes (or, equivalently, finite algebras), it is of prime importance to
understand, which properties are independent of the choice of embedding.
Fortunately, the rough answer is: all properties are independent provided that the
ambient space $\ambient$ is smooth. We will justify this  later (see
Theorem~\ref{thm_equivalence_of_abstract_and_embedded_smoothings},
Proposition~\ref{prop_smoothability_depends_only_on_sing_type}). Now
we show that the tangent space dimension is independent of the embedding.
Consider the baby case of a tuple $R$ of smooth $\kk$-points on a variety
$\ambient$. By
Example~\ref{ex:smoothpoints}, the dimension of $\Dtangspace{\Hilbr{\ambient}, [R]}$
is $(\deg R)(\dim \ambient)$. Thus $0 = \dimk\Dtangspace{[R]} - (\deg R)(\dim
\ambient)$ is
independent of $\ambient$. We show that this independence generalizes to
arbitrary $R$. This result is known and appeared, for example, in the
arXiv version of~\cite{cartwright_erman_velasco_viray_Hilb8} and, for
Gorenstein subschemes,
in~\cite[Lemma~2.3]{casnati_notari_irreducibility_Gorenstein_degree_9}.

\begin{proposition}[invariance of tangent
    space]\label{ref:invarianceoftangentspace:prop}
    Let $\ambient$ be an affine or quasi-projective scheme.
    Let $R \subset \ambient$ a finite $\kk$-scheme of degree $r$. Then the number
    \begin{equation}\label{eq:tangentinvariant}
        \dimk \Dtangspace{\Hilbr{\ambient}, [R]} - r(\dim \ambient)
    \end{equation}
    does not depend on the embedding of $R$ into a quasi-projective variety
    $\ambient$, provided that
    $R$ does not intersect the singular locus of $\ambient$.
\end{proposition}
\begin{proof}
    Irreducibility of $\ambient$ is used only to assert that the dimension
    near each point is independent of the point chosen.
    \def\ccY{\mathcal{Y}}%
    \def\speceps{\Spec \kk[\varepsilon]/\varepsilon^2}%
    We give a slightly involved proof, which, however, can be easily augmented
    to prove other independence properties. The proof goes by series of
    reductions.
    We consider each reducible component of $R$ separately, hence we reduce to the
    case $R$ irreducible, supported at smooth $x\in \ambient $.

    \emph{First, we check that~\eqref{eq:tangentinvariant} is invariant
        under a change from $X$ to $X \times \AA^m$.}
%    Note that since $\ambient $ is
%    quasi-projective over $\kk$, the Hilbert scheme of $\ambient $ exists.
    Fix an embedding
    $i\colon R\into \ambient $ and $m\geq 0$. Consider $i' = (i, 0)\colon R \into \ambient  \times
    \AA^m$. We claim that
    \begin{equation}\label{eq:step}
        \Dtangspace{\Hilbr{\ambient  \times \AA^m}, [i'(R)]} = \Dtangspace{\Hilbr{\ambient },
    [i(R)]} + m\cdot(\deg R).
    \end{equation}
    Indeed, pick an neighbourhood $U = \Spec B$ of $x\in \ambient$ and
    coordinates $\Dx_i$ on $\AA^m$. Let $I
    = I(i(R)) \subset B$ and
    \[
        J = I(i'(R)) = (I) + (\Dx_1, \ldots ,\Dx_m) \subset B[\Dx_1, \ldots ,\Dx_m].
    \]
    Let $\DA = B[\Dx_1, \ldots ,\Dx_m]/J$. By
    Proposition~\ref{ref:tangentspaceHilb:prop}, an element of
    $\Dtangspace{\Hilbr{\ambient \times \AA^m}, [i'(R)]}$ corresponds to a
    homomorphism $\Homthree{B[\Dx_1, \ldots ,\Dx_m]}{J}{\DA}$. Let
    $\sspan{\Dx_1, \ldots ,\Dx_m}$ denote the $\kk$-linear span. We have a
    restriction map
    \begin{equation}\label{eq:restriction}
        \Homthree{B[\Dx_1, \ldots ,\Dx_m]}{J}{\DA} \into
        \Homthree{B}{I}{\DA} \oplus \Homthree{\kk}{\sspan{\Dx_1, \ldots
        ,\Dx_m}}{\DA}.
    \end{equation}
    The only relations between the generators of $J$ involving elements
    $\Dx_i$ have the form $\sum_i \Dx_i J \in (I)$, so the
    map~\eqref{eq:restriction} is in fact onto. Counting dimensions of both
    sides of~\eqref{eq:restriction}, we obtain
    \[
        \dimk \Homthree{B[\Dx_1, \ldots ,\Dx_m]}{J}{\DA} = \dimk
        \Homthree{B}{I}{\DA} + (\dimk \sspan{\Dx_1, \ldots ,\Dx_m})(\dimk
        \DA),
    \]
    which is precisely~\eqref{eq:step}.

    \def\Ohatone{\hat{\OO}_{\ambient, x}}%
    \def\Ohattwo{\hat{\OO}_{\ambient', x'}}%
        \emph{Second, we compare embeddings into two varieties of the same
    dimension.} Consider two embeddings $i\colon R\into \ambient$ and $i'\into
    \ambient'$, with $R$, $R'$ supported at $x$, $x'$ respectively.
    By assumption, $x\in \ambient$ and $x'\in \ambient'$ are smooth points.
    Therefore, the completions $\Ohatone$ and $\Ohattwo$ of local rings are
    isomorphic, in fact isomorphic to power series rings in $\dim n$
    variables. As in
    Proposition~\ref{ref:Gorenstein:prop}, we fix an isomorphism
    $\varphi\colon \Spec\Ohatone \to \Spec\Ohattwo$, such that $i' = \varphi\circ i$.
    Let $I \subset \Ohatone$, $I' \subset \Ohattwo$ be the ideals of $i(R)$,
    $i'(R)$ respectively, then $I' = (\varphi^{\#})^{-1}(I)$.
    Then
    \begin{align}\label{eq:tmpinvariance}
        \Dtangspace{\Hilbr{\ambient}, x}  &\simeq
        \Homthree{\OO_{i(R)}}{I}{\OO_{i(R)}}  \simeq
        \Homthree{(\varphi^{\#})^{-1}(\OO_{i(R)})}{(\varphi^{\#})^{-1}(I)}{(\varphi^{\#})^{-1}(\OO_{i(R)})}
        \simeq \\
         &\simeq
        \Homthree{\OO_{i'(R)}}{I'}{\OO_{i'(R)}} =
        \Dtangspace{\Hilbr{\ambient'}, x'}.\nonumber
    \end{align}
    \emph{Third and final, we conclude.} Choose two embeddings of $R$.
    By~\eqref{eq:step} we may assume that the ambient varieties have the same
    dimension, then~\eqref{eq:tmpinvariance} proves
    that~\eqref{eq:tangentinvariant} is equal for both spaces.
\end{proof}

\newcommand{\ArtWithBasis}{\mathcal{F}in\mathcal{S}ch\mathcal{B}ased}%
Now we briefly discuss one more natural candidate for a parameter space of
finite algebras, to conclude that it gives a topology
equivalent to the Hilbert scheme. If we recall the definition of a finite
algebra $\DA$ as a ``vector space with multiplication'', then a natural
solution to the problems with $\Artbare$ is to fix a basis.
Let
\[
    \ArtWithBasis(T) = \left\{ (\pi, \phi)\ |\ \pi\colon \ccX \to T \mbox{ finite
    flat of degree }r,\
    \phi\colon \OT^{\oplus r} \to \pi_*\OX \mbox{ isomorphism}\right\}.
\]
For every element of $\ArtWithBasis(T)$ the multiplication on $\OX$ is
translates by $\phi$ to a map $\OT^{\oplus r} \tensor \OT^{\oplus r} \to
\OT^{\oplus r}$ which gives $r^3$ structure constants. The unity of
$\OX$ translates into a vector of $r$ constants, so we obtain a map $T\to
\AA^{r^3 + r}$. Commutativity, associativity and properties of the
unity translate into algebraic equations inside $\AA^{r^3+ r}$, so that
$\ArtWithBasis$ is represented by an affine subscheme of $\AA^{r^3+r}$,
see~\cite[Proposition~1.1]{poonen_moduli_space} for details.
The functors $\ArtWithBasis$ and $\Hilbfunc$ can be compared by means of a
common refinement parameterizing finite subschemes with a fixed basis and thus
their topology is essentially the same,
see~\cite[Chapter~4]{poonen_moduli_space} for details.

\section{Base change of Hilbert schemes}

\newcommand{\ambientbasechanged}{\ambient_{\KK}}%
\newcommand{\UUbasechanged}{\UU_{\KK}}%
    The Hilbert scheme behaves well with respect to field extensions. This
    property is very important, as it enables us, for example, to reduce to an
    algebraically closed base field $\kk$.

    Let $\ambient $ be a $\kk$-scheme such that the Hilbert scheme of points
    on $\ambient$ exists. Recall, that the functor~\eqref{eq:hilbertfunc} is defined
    for $\kk$-schemes, hence the obtained scheme $\Hilb{\ambient}$
    depends on $\kk$. In this section we stress this by writing
    $\Hilb{\ambient/\kk}$ instead of $\Hilb{\ambient}$. For a field extension
    $\kk \subset \KK$ we also
    write $\kk$ and $\KK$ instead of $\Spec \kk$ and $\Spec \KK$, respectively.

\begin{proposition}\label{prop_base_change_for_Hilb_r}
    Suppose $\ambient$ is a scheme such that $\Hilb{\ambient/\kk}$ exists.
    Let $\kk \subset \KK$ be a field extension and $\ambientbasechanged :=
    \ambient  \times_{\kk} \KK$. Then
    $\Hilb{\ambientbasechanged/\KK}$ exists and
    \begin{equation}\label{eq:basechangehilb}
        \Hilb{\ambientbasechanged/\KK} \simeq \Hilb{\ambient/\kk} \times_{\kk} \KK.
    \end{equation}
    Let $\UU\to \Hilb{\ambient/\kk} \subset \Hilb{\ambient/\kk} \times_{\kk}
    \ambient$ be the universal family for $\Hilb{\ambient/\kk}$, then the universal
    family for $\Hilb{\ambientbasechanged/\KK}$ is
    \[
        \UUbasechanged = \UU \times_{\kk} \KK \subset
        \Hilb{\ambientbasechanged/\KK} \times_{\KK}
    \ambientbasechanged.
\]
   Moreover, all the relevant maps are commutative:
   \begin{equation}\label{eq:basechangediagram}
     \begin{tikzcd}
         \UUbasechanged \arrow[r]\arrow[d,hook,"cl"]& \UU\arrow[d,hook, "cl"] \\
         \Hilb{\ambientbasechanged/\KK}\times_{\KK}\ambientbasechanged
         \arrow[r]\arrow[d] & \Hilb{\ambient/\kk} \times_{\kk} \ambient
         \arrow[d] \\
         \Hilb{\ambientbasechanged/\KK} \arrow{r} & \Hilb{\ambient/\kk}
     \end{tikzcd}
 \end{equation}
   The same applies with $\Hilbfunc_{pts}$ replaced by $\Hilbfunc_{r}$.
\end{proposition}
\begin{proof}
    See~\cite[(5) pg. 112]{fantechi_et_al_fundamental_ag}.
\end{proof}

\newcommand{\RRK}{R_{\KK}}%
Let us briefly discuss, how the maps of Diagram~\ref{eq:basechangediagram}
behave at the level of points.
Let $\ambient$ be as in Proposition~\ref{prop_base_change_for_Hilb_r} and $R
\subset \ambient$ be a finite $\kk$-scheme. It corresponds to a $\kk$-point
$[R]\in \Hilb{\ambient}$. The scheme $\RRK := R \times_{\kk} \KK$ is a closed
subscheme of $\ambientbasechanged$ corresponding to a $\KK$-point $[\RRK] \in
\Hilb{\ambientbasechanged}  \simeq \Hilb{\ambient} \times_{\kk} \KK$. The
projection of $[\RRK]$ to $\Hilb{\ambient}$ is equal to $[R]$.

\section{Loci of Hilbert schemes of points}

Fix $r$ and let $\ambient $ be a $\kk$-scheme such that $\Hilbr{\ambient }$ exists. In this section we
define various loci of the Hilbert scheme of points, in particular the
Gorenstein locus.
Recall that the Hilbert
scheme of $r$ points comes with a universal finite flat family
\[
    \pi\colon \UU \to \Hilbr{\ambient}
\]
such that the fiber over a $\kk$-point $[R]\in \Hilbr{\ambient }$ is exactly
$R \subset \ambient $.

Perhaps the easiest example of a finite subscheme $R \subset \ambient $ of degree $r$ is a
tuple of $\kk$-points of $\ambient $; such subscheme is smooth. If $\ambient $ is defined over an algebraically closed
field $\kk = \kkbar$ then all finite smooth subschemes are such
tuples (if not, we also have $\kappa$-points, for $\kk \subset \kappa$
separable).
\begin{lemma}
    Let $\ambient $ be a scheme over $\kk = \kkbar$. Then a finite subscheme $R \subset
    \ambient $ is smooth over $\kk$ if and only if $R$ is a disjoint union of
    $\kk$-points.
\end{lemma}
\begin{proof}
    Smoothness is a local property, so we may assume $R$ is irreducible,
    corresponding to finite local $\kk$-algebra $(\DA, \mm, \kk)$. Then $R$ is smooth over
    $\kk$ if and only if the cotangent module $\Omega_{\DA/\kk}$ vanishes.
    This happens if and only if $\mm/\mm^2 = 0$
    if and only if $\mm = 0$ if and only if $\DA = \kk$.
\end{proof}
We now gather all smooth subschemes into a locus.
Let $\Hilbzeror{\ambient } \subset \Hilbr{\ambient }$ denote the subset of points $[R]\in
\Hilbr{\ambient }$ corresponding to smooth subschemes $R \subset \ambient $. This is
precisely the image of all smooth fibers of $\pi$, so that $\Hilbzeror{\ambient }$ is
an \emph{open} subset of $\Hilbr{\ambient }$ by~\cite[Theorem~12.1.1]{ega4-3} and we can endow
it with an open scheme structure.
\begin{definition}\label{ref:smoothablecomponent:def}
    The \emph{smoothable component} of $\Hilbr{\ambient }$ is the closure of
    $\Hilbzeror{\ambient }$ in $\Hilbr{\ambient }$. It is denoted by $\Hilbsmr{\ambient }$.
\end{definition}

\begin{example}
    If $r = 1$, then $\Hilbr{\ambient } = \ambient $ and $\pi$ is an isomorphism (see
    Example~\ref{ex:rone}), so that $\Hilbzeror{\ambient } = \Hilbsmr{\ambient } =
    \Hilbr{\ambient } = \ambient $. This shows that $\Hilbsmr{\ambient }$ may be arbitrary pathological
    (provided that $\ambient $ is): it may be nonreduced, not irreducible etc.
\end{example}

\begin{proposition}\label{ref:basechangehilb:prop}
    Let $\kk \subset \KK$ be a field extension. Then
    \begin{align}\label{eq:basechangehilbzeror}%
        \Hilbzeror{\ambientbasechanged} & \simeq  \Hilbzeror{\ambient} \times_{\kk}
        \KK,\\
        \label{eq:basechangehilbsmr}%
        \Hilbsmr{\ambientbasechanged} &  \simeq \Hilbsmr{\ambient} \times_{\kk}
        \KK.
    \end{align}
\end{proposition}
\begin{proof}
    \def\II{\mathcal{I}}%
    \def\JJ{\mathcal{J}}%
    Let $\kk \subset \KK$ be a field extension. A scheme $R$
    is smooth if and only if $\RRK = R \times_{\kk} \KK$ is smooth, so
    Isomorphism~\ref{eq:basechangehilb} restricts to
    Isomorphism~\ref{eq:basechangehilbzeror}. The ideal sheaf $\II$ of $\Hilbsmr{\ambient}
    \subseteq \Hilbr{\ambient}$ consists of functions vanishing on
    $\Hilbzeror{\ambient}$. The ideal sheaf $\JJ$ likewise consists of
    functions vanishing on $\Hilbzeror{\ambientbasechanged}  \simeq
    \Hilbzeror{\ambient} \times_{\kk} \KK$. Therefore, $\JJ  \simeq  \II
    \otimes_{\kk} \KK \subset \OO_{\ambient} \otimes_{\kk} \KK  \simeq
    \OO_{\ambientbasechanged}$, which proves
    Isomorphism~\eqref{eq:basechangehilbsmr}.
\end{proof}

The scheme $\Hilbzeror{\ambient }$ has an explicit description, at least for
quasi-projective $\ambient $. Let $\ambient ^{r} = \ambient  \times
\ambient \times \ldots  \times \ambient $ be the $r$-fold product of $\ambient $ over $\kk$ and let
$\Delta_{ij} \subset \ambient ^r$ be the subscheme where the $i$-th and $j$-th
coordinates are equal. Let $\Delta = \bigcup_{i\neq j} \Delta_{ij}$ and
$\ambient ^{r, \circ} = \ambient ^{r} \setminus \Delta$. On $\ambient ^r$ we
have an action of the symmetric
group $\Sigma_r$ by permuting coordinates. This action
preserves $\Delta$ and restricts to a \emph{free} action on $\ambient ^{r, \circ}$.
\begin{lemma}\label{ref:smoothpartdescription:lem}
    Let $\ambient $ be quasi-projective.
    We have a natural isomorphism $\ambient ^{r, \circ}/\Sigma_r \to \Hilbzeror{\ambient }$.
\end{lemma}
\begin{proof}
    Over $\ambient ^{r}$ we have a natural degree $r$ family given as follows: let
    $\Delta_i \subset \ambient ^r \times \ambient $ be the locus where $i$-th and the last
    coordinates agree. Each projection $\Delta_i\to \ambient ^r$ is an isomorphism.
    Let $U = \bigcup \Delta_i$ and let $U^{\circ} \subset \ambient ^{r, \circ} \times
    \ambient $ be its restriction to $\ambient ^{r, \circ}$. Over $\ambient ^{r, \circ}$ all
    $\Delta_i$ are disjoint, so $U^{\circ}$ is flat of
    degree $r$ over $\ambient ^{r, \circ}$ and induces a unique map $\ambient ^{r, \circ} \to
    \Hilbzeror{\ambient }$. This map is, by uniqueness, $\Sigma_r$-invariant and so
    factors as $\ambient ^{r, \circ} \to \ambient ^{r, \circ}/\Sigma_r \to \Hilbzeror{\ambient }$.
    Now, there is a Hilbert-Chow morphism $\Hilbr{\ambient } \to \ambient ^{r}/\Sigma_r$, which sends
    a scheme $R$ to its support counted with multiplicities,
    see~\cite[Section~7.1]{fantechi_et_al_fundamental_ag}
    and~\cite[Corollary~5.10]{Mumford__GIT}.
    It restricts to a map $\Hilbzeror{\ambient } \to \ambient ^{r, \circ}/\Sigma_r$, which is the
    inverse of $\ambient ^{r, \circ}/\Sigma_r \to \Hilbzeror{\ambient }$.
\end{proof}

We now describe $\Hilbsmr{\ambient }$ for a nicely behaved scheme. We say that $\ambient $
is \emph{geometrically irreducible} if $\ambient ' = \ambient  \times_{\Spec\kk} \Spec
\kkbar$ is irreducible. Similarly we define \emph{geometrically reduced}
schemes.
A geometrically irreducible scheme $\ambient $ is automatically irreducible as an image of
irreducible $\ambient '$.
For $\kk = \kkbar$ a scheme is geometrically irreducible if and only
if it is irreducible. The $\RR$-scheme $\Spec \CC$ is irreducible, but it is
not geometrically irreducible: $\Spec
\CC \times_{\Spec \RR} \Spec \CC  \simeq \Spec \CC \sqcup \Spec \CC$.
\begin{proposition}\label{ref:smoothablecomponentdescription:prop}
    Let $\ambient $ be a geometrically irreducible, smooth scheme over $\kk$. Then
    $\Hilbsmr{\ambient }$ is geometrically reduced and geometrically irreducible for all $r$. If $\ambient $
    is additionally a quasi-projective variety, then $\Hilbsmr{\ambient }$ is a
    quasi-projective variety of dimension $r \dim \ambient $.
\end{proposition}
\begin{proof}
    By assumption $\ambient ' = \ambient  \times_{\Spec \kk} \Spec \kkbar$ is irreducible and
    smooth over
    $\kkbar$. By Proposition~\ref{ref:basechangehilb:prop} we have
    $\Hilbsmr{\ambient '} = \Hilbsmr{\ambient } \times \Spec \kkbar$ and the
    claim is that $\Hilbsmr{\ambient '}$ is reduced and irreducible. First,
    since $\ambient '$
    is smooth over $\kkbar$, each its $\kkbar$-point is smooth, so that
    $\Hilbzeror{\ambient '}$ is smooth  by Example~\ref{ex:smoothpoints}. Moreover,
    the variety $(\ambient ')^{\times r}$ is irreducible as a product of irreducible
    varieties over $\kkbar$, so also $\Hilbzeror{\ambient '}$ is irreducible
    by the proof of Lemma~\ref{ref:smoothpartdescription:lem}.
    If $\ambient $ is quasi-projective, then $\dim\Hilbzeror{\ambient '} = r \dim \ambient $ again by
    Lemma~\ref{ref:smoothpartdescription:lem}. Then $\Hilbsmr{\ambient '}$ is reduced
    and irreducible as the closure of $\Hilbzeror{\ambient '}$.
    If $\ambient $ is quasi-projective, then it is dense inside a projective
    $\bar{\ambient }$. In this case $\Hilbsmr{\ambient }$ is an open subset of
    the projective scheme $\Hilbsmr{\bar{\ambient }}$, so it is quasi-projective.
\end{proof}

\begin{example}
    In the setting of Example~\ref{ex:smoothpoints}, let $X$ be a smooth,
    geometrically irreducible scheme over $\kk$. Then for each tuple $x_1,
    \ldots ,x_r$ of $\kk$-points of $X$ and $R = \{x_1, \ldots ,x_r\}
    \subset X$, we have
    \[
        \dimk \Dtangspace{\Hilbr{\ambient}, [R]} = \sum \dimk
        \Dtangspace{\ambient, x_i} = r\cdot (\dim X) = \dim \Hilbsmr{X},
    \]
    so $[R]\in \Hilbsmr{X}$ is a smooth point.
\end{example}

\newcommand{\dualT}{\pi^{!}\OT}%
As before, we are especially interested in Gorenstein subschemes and their
families. To define them naturally, we introduce a relative version of
canonical modules (Definition~\ref{ref:canonicalmodule:def}).
\begin{definition}\label{ref:relativecanonical:def}
    The \emph{relative canonical sheaf} of a (finite flat) family $\pi\colon \ccX\to T$ is
    the $\OX$-module
    \[
       \dualT := \Homthreesh{\OT}{\pi_*\OX}{\OT},
    \]
    where the
    multiplication is by precomposition.
\end{definition}
The relative canonical sheaf is also known as the \emph{relative dualizing
sheaf}, see~\cite[Exercise~6.10]{hartshorne}
or~\cite[Tag~0BZI]{stacks_project}.
Since $\pi$ is flat and finite, the $\OT$-module $\pi_*\OX$ is
locally free, so for every section of $\OX$ there is a nonzero functional
in $\Homthreesh{\OT}{\pi_*\OX}{\OT}$ nonvanishing on this section.
Consequently, $\dualT$ is
torsion-free. For every point $t\in T$ with residue field $\kappa(t)$ we have
\[
    (\dualT)_{|t} = (\dualT)\tensor_{\OT} \kappa(t) =
    \Homthreesh{\kappa(t)}{\pi_*\OX\tensor_{\OT} \kappa(t)}{\kappa(t)} =
    \Homthreesh{\kappa(t)}{\pi_*{\OO_{\ccX|t}}}{\kappa(t)},
\]
which is precisely the canonical sheaf of
the fiber $\ccX_{|t}$, see Definition~\ref{ref:canonicalmodule:def}.
\begin{definition}\label{ref:Gorensteinfamily:def}
    A (finite flat) family $\pi\colon \ccX\to T$ is \emph{Gorenstein} if $\dualT$ is a line
    bundle.
\end{definition}
Fibers
of a Gorenstein family are Gorenstein.
Conversely, a family with Gorenstein fibers is
Gorenstein because a generator of $(\dualT)_{|t}$ may be lifted to a neighborhood of
$t$. Similarly, if $\pi\colon \ccX\to T$ is any family, then
the set $\{t\in T \ |\ (\dualT)_{|t}\mbox{ invertible}\}$ is open in $T$.
Applying the above considerations to $\UU\to \Hilbr{\ambient }$ we make the following
definition.
\begin{definition}
    The \emph{Gorenstein locus} of $\Hilbr{\ambient }$ is the open subset consisting
    of points $[R]$ corresponding to Gorenstein subschemes $R \subset \ambient $. We
    endow it with open subscheme structure and denote by $\HilbGorr{\ambient }$.
\end{definition}
Similarly, we define $\HilbGorsmr{\ambient } = \HilbGorr{\ambient } \cap \Hilbsmr{\ambient }$.
A smooth scheme is Gorenstein by
Example~\ref{ex:smoothareGorenstein}, so $\HilbGorsmr{\ambient } \supset
\Hilbzeror{\ambient }$. The Gorenstein locus and $\HilbGorsmr{\ambient}$ behave well with respect to base
change, thanks to Proposition~\ref{ref:Gorensteinbasechange:prop}.
We investigate the local structure of $\HilbGorr{\ambient }$ for $\ambient  = \mathbb{A}^n$
much more closely in Section~\ref{ssec:relativemacaulaysystems} and in
Chapter~\ref{sec:Gorensteinloci}.

We conclude this section with two fundamental results on $\Hilbr{\ambient }$ for
low-dimensional $\ambient $.
These results may be thought of as global versions of Hilbert-Burch
and Eisenbud-Buchsbaum structure theorems for ideals,
see~\cite[Chapters~8 and~9]{hartshorne_deformation_theory}, together with
connectedness arguments. We do not prove
those deep results here, but we encourage the reader to consult Fogarty's
paper, which is, in our opinion, highly enlightening.
\begin{theorem}[\cite{fogarty}]
    The Hilbert scheme of $r$ points on a smooth, geometrically irreducible
    quasi-projective surface is
    smooth and irreducible for all $r$.
    \label{ref:fogarty:thm}
\end{theorem}

\begin{theorem}[\cite{kleppe_pfaffians,
            roig_codimensionthreeGorenstein, kleppe_roig_codimensionthreeGorenstein}]
    The Gorenstein locus of the Hilbert scheme of $r$ points on a smooth,
    geometrically irreducible quasi-projective threefold is smooth and
    irreducible for all~$r$.
    \label{ref:kleppe:thm}
\end{theorem}

The above results generalize to arbitrary codimension for local complete intersections.
\begin{theorem}[{\cite[Theorem~3.10]{huneke_ulrich_cis}}]\label{ref:hunekeulrich:thm}
    Let $R \subset \mathbb{A}^n$ be a local complete intersection. Then $R$ is
    smoothable and $[R]\in \Hilbsmr{\mathbb{A}^n}$ is a smooth point.
\end{theorem}

Beyond those theorems there are almost no results of comparable generality;
see Section~\ref{ssec:examplesnonsmoothable} for some counterexamples on irreducibility.
What is true though, is that $\Hilbr{\ambient }$ is connected for every connected
projective scheme $\ambient $ over $\kk = \kkbar$.
We reproduce Fogarty's proof~\cite{fogarty} of this result. It uses unipotent group
actions.
See~\cite[Proposition~18.12]{Miller_Sturmfels} for a combinatorial approach
or~\cite[Section~4.6.5]{Sernesi__Deformations} for an induction using
Quot-schemes.

\newcommand{\Gadd}{\mathbb{G}_a}%
The affine line $\AA^1$ is an algebraic group with addition. We
denote it by $\Gadd$ to stress the group structure.
\begin{lemma}\label{ref:action:lem}
    Let $C$ be a projective curve over $\kk = \kkbar$ with a non-trivial $\Gadd$ action. Then
    there is exactly one fixed $\Gadd$ point on $C$.
\end{lemma}
\begin{proof}
        Every nontrivial subgroup of $\Gadd$ is infinite, thus dense in the
        Zariski topology. Therefore the stabilizer of $x\in C$ is either
        $\Gadd$ or $0$.
        Suppose that there are no fixed points. Then the orbit of each point
        is isomorphic to $\Gadd$, thus Zariski-dense; hence it contains an
        open subset of $C$. Therefore there is only
        one orbit and $C  \simeq \Gadd$. This is a contradiction since $C$ is
        projective and $\Gadd$ has non-constant global functions.
        The uniqueness of the fixed point is a bit more subtle.
        Consider the normalization $\tilde{C} \to C$. By its universal
        property, the action of $\Gadd$ lifts to an action on $\tilde{C}$.
        Now, $\tilde{C}$ is smooth and rational over $\kkbar$, so it is isomorphic to a $\mathbb{P}^1$.
        Take a fixed point $c\in \tilde{C}$. Then $\Gadd$ acts on
        $\tilde{C}\setminus\left\{ c \right\} = \mathbb{A}^1$ with a dense
        orbit which is also isomorphic to $\mathbb{A}^1$. Hence
        $\tilde{C}\setminus\left\{ c \right\}$ coincides with this orbit and
        there are no other fixed points.
\end{proof}

\begin{proposition}\label{ref:Gaddaction:prop}
    Let $\ambient $ be a connected projective scheme over $\kk = \kkbar$ and assume that the
    group $\Gadd$ acts on $\ambient $. Then the fixed point set $\ambient ^{\Gadd}$ is
    also projective and connected.
\end{proposition}

\begin{proof}
    Since $\Gadd$ acts Zariski-continuously on $\ambient $, $\ambient ^{\Gadd}$ is
    Zariski-closed
    in $\ambient $, thus projective.  Choose two $\Gadd$-invariant points $x_0,
    x_1\in \ambient $. We will find a chain of curves in $\ambient ^{\Gadd}$ linking those
    points, which will prove connectedness.

    By taking successive hyperplane sections through $x_0$, $x_1$ we find
    a dimension one subscheme $C_0 \subset \ambient $ containing these points.
    Consider the point $[C_0]$ on the Hilbert scheme of
    curves (which exists and is projective,
    see~\cite{fantechi_et_al_fundamental_ag}) and the projective curve
    $C' = \overline{\Gadd[C_0]}$.  By Lemma~\ref{ref:action:lem}, this curve
    has a fixed point $[C_1]$, corresponding to a one-dimensional $C_1
    \subset \ambient $. Each point $x_i$ is $\Gadd$-fixed, thus it lies on each
    element of $\Gadd[C_0]$, hence also on $C_1$. Summarizing, we have a
    one-dimensional $C_1$ which is preserved under the action of $\Gadd$
    and links $x_0, x_1$.

    We replace $C_1$ by its reduction. The group
    $\Gadd$ is connected, hence it acts on each irreducible component of
    $C_1$ and preserves the intersections of components.
    Let $C_1^i$ be the components and assume that $C_1^1 \ldots,
    C_1^m$ give the shortest (in terms of number of irreducible curves)
    path from $x_0$ to $x_1$. On each $C_1^i$ we see at
    least two points from the set $\{C_1^i \cap C_1^j\}_{j\neq i} \cup
    \{x_0, x_1\}$. Then $\Gadd$ has two invariant points on $C_1^i$ and
    thence by Lemma~\ref{ref:action:lem} it acts trivially. The chain of
    curves $C_1^1\cup  \ldots C_1^m$ is contained in $\ambient ^{\Gadd}$ and gives
    the required link.
\end{proof}

\begin{proposition}\label{ref:connectedforfinite:prop}
    Assume $\kk = \kkbar$ and consider a finite local $\kk$-algebra
    $(\DA, \mm, \kk)$ of degree $e$.
    The scheme $\Hilbr{\Spec A}$ is a connected closed subscheme of the
    Grassmannian $\Grass(e-r, \DA)$.
\end{proposition}

Here, $\Grass(e-r, \DA)$ parameterizes $e-r$-dimensional $\kk$-subspaces of $\DA$.

\begin{proof}
    A subscheme of $A$ is just an ideal $I \subset A$ and an ideal is just
    a vector space preserved by multiplication by $\mm$. The map
    $i\colon \Hilbr{\Spec A}\to \Grass(e-r, A)$ maps this ideal to the associated
    vector space and it is an embedding. Since $\Spec A$ is finite over $\kk$, it
    is projective over $\kk$, so the scheme $\Hilbr{\Spec A}$ is
    also projective by Theorem~\ref{ref:representability:thm}. Hence
    the image of $i$ is closed.
    It remains to prove connectedness. We will use
    Proposition~\ref{ref:Gaddaction:prop} for $\Grass(e-r,A)$, which is
    projective and connected. A vector space $V \subset A$ is an ideal iff
    it is preserved by the action of the group $1+\mm$, so that
    \[\Hilbr{\Spec A}=\Grass(e-r,A)^{1+\mm}.\]
    The group $1+\mm$ has a composition series
    \[
        1 + \mm \supset 1 + \mm^2 \supset 1 + \mm^3 \supset  \ldots \supset 1
        + \mm^e = \{1\}
    \]
    with quotients
    isomorphic to direct sums of $\Gadd$ and connectedness of the fixed locus follows
    from Proposition~\ref{ref:Gaddaction:prop}.
\end{proof}

A mentioned before, for a projective scheme $\ambient $ there
exists~(\cite[Section~7.1]{fantechi_et_al_fundamental_ag} and \cite[Corollary~5.10]{Mumford__GIT}) the Hilbert-Chow
morphism
\begin{equation}\label{eq:hilbertchow}
    \rho:\Hilbr{\ambient }\to \ambient ^{\times r}/\Sigma_r
\end{equation}
which maps a scheme $R$
to its support counted with multiplicities. Note that $\rho$ is proper as
a morphism between projective schemes.
Recall that $\ambient $ is \emph{geometrically connected} if and only if
$\ambient \times_{\Spec \kk} \Spec \kkbar$ is connected.
We need a small topological lemma.
\begin{lemma}\label{ref:topologyconnected:lem}
    Let $f\colon X\to Y$ be a closed map of topological spaces with $Y$
    connected and fibers of $f$ connected. Then $X$ is connected.
\end{lemma}
\begin{proof}
    Suppose $X$ is disconnected, $X = Z_1 \sqcup Z_2$ for nonempty $Z_i$ open and
    closed. The images $\varphi(Z_1)$, $\varphi(Z_2)$ are closed and
    $\varphi(Z_1) \cup \varphi(Z_2) = Y$. Since $Y$ is connected, there is a
    point $y\in Y$ belonging to both $\varphi(Z_i)$. Then the fiber
    $F = f^{-1}(y)$ is covered by nonempty disjoint closed subsets $Z_i\cap
    F$, hence is not connected; a contradiction.
\end{proof}
\begin{theorem}\label{ref:connectednessFogarty:thm}
    Let $\ambient $ be a geometrically connected projective scheme. Then
    $\Hilbr{\ambient }$ is also connected.
\end{theorem}
\begin{proof}
    By assumption, $\ambient ' = \ambient  \times_{\Spec \kk} \Spec \kkbar$ is connected.
    By Proposition~\ref{ref:basechangehilb:prop} it is enough to show that $\Hilbr{\ambient '}$ is
    connected. Hence we may assume $\kk = \kkbar$, $\ambient = \ambient'$.

    Clearly, $\ambient ^r$ and $\ambient ^r/\Sigma_r$ are connected and $\rho$
    is proper, so $\rho$ is closed. By Lemma~\ref{ref:topologyconnected:lem}, it is
    enough to show that the fibers of $\rho$ are connected.
    Pick a point $P\in \ambient ^r/\Sigma_r$ and its preimage in $\ambient ^r$; suppose that
    points $x_1, \ldots ,x_k$ appear in this preimage with multiplicities
    $n_1, \ldots, n_k$.
    The fiber of $\rho^{-1}(P)$ is the set of schemes whose
    components are supported of $x_i$ with multiplicity $n_i$. Thus is
    equal, at the level of topological spaces, to $\prod_i
    \Hilbarged{n_i}{\Spec \OO_{\ambient }/\mm_{x_i}^{r}}$, which is connected by
    Proposition~\ref{ref:connectedforfinite:prop}.
\end{proof}

\section{Relative Macaulay's inverse
systems}\label{ssec:relativemacaulaysystems}
\newcommand{\DST}{\DS_{T}}%
\newcommand{\DSTt}{\DS_{T, t}}%
\newcommand{\DPT}{\hat{\DP}_{T}}%
\newcommand{\DPTpoly}{\DP_{T}}%
\newcommand{\OTt}{\OO_{T, t}}%
\newcommand{\DSTF}{\DST\,\mathcal{F}}%
\newcommand{\Fsheaf}{\mathcal{F}}%
In this section we generalize Macaulay's inverse systems to cover reducible
subschemes of $\mathbb{A}^n = \Spec \DS$ and their families.
We prove in Proposition~\ref{ref:globaldescription:prop} that every finite flat
family is locally equal to $\Apolar{F_1, \ldots
,F_r}$, where $F_i$ are power series with varying coefficients. In case the family has Gorenstein fibers
it is locally equal to $\Apolar{F}$, see
Corollary~\ref{ref:localdescriptionGorenstein:cor}. This is the relative version of Macaulay's
Theorem~\ref{ref:MacaulaytheoremGorenstein:thm} and this gives an alternative
description of the local structure of $\Hilbr{\AA^n}$.

When all the fibers are supported at the origin, the considered power series
are in fact polynomials. This special case first appeared
in~\cite[Proposition~18]{emsalem}, unfortunately without a proof. We
follow~\cite{jelisiejew_VSP}.

Let $T$ be a $\kk$-scheme. We define quasi-coherent sheaves
\[\DST := \OT \tensor_{\kk} \DS\quad \mbox{and} \quad\DPT :=
\Homthreesh{\OT}{\DST}{\OT}.\] They
play the roles of $\DShat$ and $\DP$ from
Section~\ref{ssec:Gorensteininapolarity} respectively, but there is a slight
change: in Section~\ref{ssec:Gorensteininapolarity} the ring $\DShat$ is a
power series ring and here it is a sheaf of polynomial rings, see
Example~\ref{ex:relativerings}. Also, in
Section~\ref{ssec:Gorensteininapolarity} the space $\DP$ is some subspace of
functional on $\DShat$ and here $\DPT$ is the space of all functionals on
$\DST$.

We have an action of $\DST$ on $\DPT$
by precomposition; for a section $f\in H^0(U, \DPT)$ and $s, t\in H^0(U,
\DST)$ we have $(s\hook f)(t) = f(st)$. Let $\DmmS \subset \DS$ denote
the ideal of the origin.
We define $\DPTpoly \subset \DPT$ as the subsheaf of functionals which are
locally polynomials. More precisely, for open $U \subset T$, a point $t\in U$, and $f\in
H^0(U, \DPT)$ let $f_t$ denote the image of $f$ in
$\Homthreesh{\OTt}{\DSTt}{\OTt}$. We define the subsheaf $\DPTpoly$ on
sections by
\begin{equation}\label{eq:DPTpolydef}
    H^0(U, \DPTpoly) = \left\{ f\in H^0(U, \DPT)\ |\
    \forall_{t\in T} \forall_{D\gg 0} \ip{\DmmS^{D}}{f_t} = 0 \right\}.
\end{equation}
\begin{example}\label{ex:relativerings}
    If $T = \Spec A$ is affine, then we have, after choosing coordinates,
    \[
        H^0(T, \DST) = A[\Dx_1, \ldots , \Dx_n],\quad
        H^0(T, \DPT) = A[[x_1, \ldots ,x_n]]\quad\mbox{and}\quad H^0(T,
        \DPTpoly) = A[x_1, \ldots ,x_n].
    \]
%    Note a difference with Section~\ref{ssec:Gorensteininapolarity}: here
%    $\DST$ is a polynomial ring and $\DPT$ is the set of all functionals on
%    $\DST$, so, as a linear space, it is identified with the power series ring.
\end{example}

For a subsheaf $\Fsheaf \subset \DPT$ by $\Ann(\Fsheaf) \subset \DST$ we denote
its annihilator. We now introduce the relative version of apolar algebras
(Definition~\ref{ref:apolar:def}).
\begin{definition}
    Let $T$ be scheme and let $\Fsheaf \subset \DPT$ be a
    quasi-coherent subsheaf. The
    \emph{apolar family} of
    $\Fsheaf$ is a subscheme
    \[
        \Apolar{\Fsheaf} := \Speccal_T\,\left(\DST/\Ann(\Fsheaf)\right) \subset T
        \times \Spec \DS
    \]
    with a canonical projection
    $\pi_{\Fsheaf}\colon\Apolar{\Fsheaf}\to T$.
\end{definition}

The morphism $\pi_{\Fsheaf}$ is affine by construction, but it need not be finite or flat. The following
notion guarantees these properties.
\begin{definition}\label{ref:flatness:def}
    Suppose that $T$ is locally Noetherian.
    The sheaf $\Fsheaf \subset \DPT$ is \emph{finitely flatly embedded} if
    $\DSTF$ is a finitely generated $\OT$-module and the sequence
    \[
        0 \to \DSTF \to \DPT \to \DPT/(\DSTF) \to 0
    \]
    is a locally split sequence of $\OT$-modules.
\end{definition}
    Definition~\ref{ref:flatness:def} is similar to the definition of a
    subbundle of a vector bundle, where we require that the cokernel is locally free.
    Example~\ref{ex:notionofflatness} provides a family which is finite flat but not
    finitely flatly embedded.

    \begin{lemma}\label{ref:flatness:lem}
        If $T$ is locally Noetherian and $\Fsheaf \subset \DPT$ is
        finitely flatly embedded, then $\DSTF$ and $\DST/\Ann(\Fsheaf)$
        are flat over $T$. The morphism $\pi_{\Fsheaf}$ is
    finite and flat.
    \end{lemma}
    \begin{proof}
    Since $T$ is locally
    Noetherian, the $\OT$-module $\DPT$ is flat by~\cite[4.47,
    p.~139]{Lam__rings_and_modules}; then also its (locally) direct summand $\Fsheaf$ is
    flat over $T$. Since $\DSTF$ is finite over $\OT$, it is even locally
    free, see~\cite[Exercise~6.2]{Eisenbud}.
    Then the composition map $s\colon\DST\to
    \Homthreesh{\OT}{\DPT}{\OT} \to \Homthreesh{\OT}{\DSTF}{\OT}$
    is surjective.
    By local freeness, for each section of $\DSTF$ there is a nonzero functional in
    $\Homthreesh{\OT}{\DSTF}{\OT}$ nonvanishing on this section. Then
    $\Ann(\Homthreesh{\OT}{\DSTF}{\OT}) = \Ann(\DSTF) \subset \DSTF$, so
    that $\ker s = \Ann(\Fsheaf)$ and
    \begin{equation}\label{eq:tmpfiniteness}
        \DST/\Ann(\Fsheaf)  \simeq  \Homthreesh{\OT}{\DSTF}{\OT},
    \end{equation}
    as $\OT$-modules. The $\OT$-module $\Homthreesh{\OT}{\DSTF}{\OT}$
    is locally free (of finite rank) as well, hence flat. Finally,
    $\pi = \pi_{\Fsheaf}$ is affine and
    $\pi_*\OO_{\Apolar{\Fsheaf}} = \DST/\Ann(\Fsheaf)$ is
    flat over $T$, hence $\pi$ is flat. Also, $\pi$ is finite
    by~\eqref{eq:tmpfiniteness}. Summarizing, we obtain the following
    diagram of $\DST$-sheaves
    \begin{equation}\label{eq:diagram}
        \begin{tikzcd}
            0 \arrow{r}{}& \mathcal{I}_{\Apolar{\Fsheaf}}
            \arrow{r}{}\arrow{d}{}& \DST \arrow{r}{}\arrow{d}{}&
            \OO_{\Apolar{\Fsheaf}} \arrow{r}{}\arrow{d}{ \simeq }& 0\\
            0 \arrow{r}{}& \Homthreesh{\OT}{\DPT/(\DSTF)}{\OT} \arrow{r}{}&
            \Homthreesh{\OT}{\DPT}{\OT} \arrow{r}{}& \Homthreesh{\OT}{\DSTF}{\OT} \arrow{r}{}& 0
        \end{tikzcd}
    \end{equation}

\end{proof}

We now prove that each (finite flat) family $\pi\colon \ccX\to T$ is actually an apolar family of a
finitely flatly embedded subsheaf $\Fsheaf \subset \DPT$. We need an
equivalent of a canonical module (Definition~\ref{ref:canonicalmodule:def})
for families. We have already defined it in
Definition~\ref{ref:relativecanonical:def}; it is $\pi^{!}\OT =
\Homthreesh{\OT}{\pi_*\OX}{\OT}$ with multiplication by precomposition.

\begin{proposition}[Description of families]\label{ref:globaldescription:prop}
    Let $T$ be locally Noetherian. Let $i\colon\ccX \into T \times \Spec \DS$ be
    such that $\pi\colon \ccX\to T$ is finite flat.
    Then $\ccX = \Apolar{\Fsheaf}$ for a finitely flatly embedded $\Fsheaf
    \subset \DPT$, in fact for $\Fsheaf = i_*\pi^!\OT$.  If for every $t\in T$ the subscheme $\ccX_t$ is supported
    at the origin, then $\Fsheaf \subset \DPTpoly$.
\end{proposition}
\begin{proof}
    \def\OX{\OO_{\ccX}}%
    \def\IX{\mathcal{I}_{\ccX}}%
    \def\prone{\operatorname{pr}}%
    Denote by $\prone:T \times \Spec\DS\to T$ the projection and by $\omega$
    the $\OX$-module $\pi^!\OT$, it is torsion free by discussion below
    Definition~\ref{ref:relativecanonical:def}.
    Now if $i\colon \ccX \into T \times \Spec \DS$ is the embedding, then $i_*\omega \subset
    \Homthreesh{\OT}{\DST}{\OT} = \DPT$. Let $\Fsheaf = i_*\omega$. Since
    $\omega$ is torsion-free, we have $\Ann(\Fsheaf) = \IX$ and $\ccX =
    \Apolar{\Fsheaf}$.
    In particular, $\pi_*\OX = \DST/\Ann(\Fsheaf)$ is flat and finite, so the
    sequence $0\to \prone_*\IX \to \DST \to \pi_*\OX \to 0$ is locally
    split.
    Applying $\Homthreesh{\OT}{-}{\OT}$ to this sequence we obtain
    \[
        0\to \Homthreesh{\OT}{\pi_*\OX}{\OT} \to \DPT \to
        \Homthreesh{\OT}{\prone_*\IX}{\OT}
        \to 0,
    \]
    which is also locally split. It remains to note that
    $\Homthreesh{\OT}{\pi_*\OX}{\OT} \to \DPT$ is equal to the embedding $\Fsheaf \to \DPT$.
    This proves that $\Fsheaf$ is flatly embedded.
    Suppose now that all fibers are supported at the origin. This means that
    $\DmmS$ is nilpotent in each fiber. Choose a covering of
    $T$ by affine Noetherian schemes $U_i$. Each $U_i$ is quasi-compact, so there exist $d_i$ such that
    $\DmmS^{d_i}$ vanishes on each fiber over $U_i$. Then $\DmmS^{d_i}$
    consists of nilpotent functions on $U_i$. Each $U_i$ is Noetherian,
    so there
    exist $e_i$ such that ${\DmmS^{d_ie_i}}_{|U_i} = 0$. The integers $D_i =
    d_ie_i$ satisfy~\eqref{eq:DPTpolydef} for all $f\in H^0(U_i, \Fsheaf)$
    and assess that $\Fsheaf \subset \DPTpoly$.
\end{proof}

\begin{corollary}[Local description of
        families]\label{ref:tmplocaldescription:cor}
    Let $T$ be locally Noetherian. Let $i\colon \ccX \into T \times \Spec \DS$ be
    such that $\pi\colon \ccX\to T$ is finite flat. For every affine cover $U_i =
    \Spec B_i$ of $T$ we have
    \[
        \pi^{-1}(U_i) = \Apolar{F_{i1}, \ldots ,F_{is_i}}
    \]
    for some $F_{ij}\in \hat{\DP}_{\Spec B_i} \simeq B_i[[x_1, \ldots ,x_n]]$.
    If all fibers of $\pi$ are supported at the origin, we necessarily have $F_{ij}\in
    B_i[x_1, \ldots ,x_n]$.
\end{corollary}
\begin{proof}
    By Proposition~\ref{ref:globaldescription:prop} we have $\ccX =
    \Apolar{\Fsheaf}$. The sheaf $\Fsheaf_{|U_i}$ is globally generated and
    finitely generated; we take $F_{i1}, \ldots ,F_{is_i}$ as its generators.
\end{proof}
\begin{corollary}[Local description, Gorenstein case]\label{ref:localdescriptionGorenstein:cor}
    Let $T$ be locally Noetherian and $i\colon \ccX \into T \times \Spec \DS$ be
    such that $\pi\colon \ccX\to T$ is finite flat with Gorenstein fibers.
    Every affine cover of $T$ can be refined to a cover $U_i = \Spec B_i$ such
    that
    \[
        \pi^{-1}(U_i) = \Apolar{F_{i}}
    \]
    for some $F_i\in \hat{\DP}_{\Spec B_i}  \simeq B_i[[x_1, \ldots ,x_n]]$.
    If all fibers of $\pi$ are supported at the origin, we necessarily have $F_{i}\in
    B_i[x_1, \ldots ,x_n]$.
\end{corollary}
\begin{proof}
    \def\OTt{\OO_{T, t}}%
    By Proposition~\ref{ref:globaldescription:prop} we have $\ccX =
    \Apolar{\Fsheaf}$ where $\Fsheaf = i_*\pi^{!}\OT =
    i_*\Homthreesh{\OT}{\OO_{\ccX}}{\OT}$.
    In particular for $t\in T$ we have $\Fsheaf(t)  \simeq
    \Homthreesh{\kappa(t)}{\OO_{\ccX|t}}{\kappa(t)}$. The morphism $\pi$ has Gorenstein fibers, so each $\Fsheaf(t)$ is
    a principal $\DST$-module. We pick a lift $F_i$ of its generator to a
    neighborhood $U_i$ of $t$ and obtain $\pi^{-1}(U_i) = \Apolar{F_i}$.
\end{proof}

\begin{example}\label{ex:doublecoversMacaulay}
    Let $\Spec \DS = \mathbb{A}^1 = \Spec \kk[x]$.
    As in Example~\ref{ex:jumphenonmenon} consider a branched double cover
    \[\ccX = \Spec \frac{\kk[\Dx, t]}{(\Dx^2 - \Dx t)} \to T = \Spec \kk[t].\]
    Then $\ccX = \Apolar{F}$ for $F = \sum_{n\geq 0}
    \DPel{x}{n+1}t^n\in \kk[t][[x]]$.
    Similarly, consider another branched double cover
    \[\ccX = \Spec \frac{\kk[\Dx, t]}{(\Dx^2 - t)} \to T = \Spec \kk[t].\]
    Then $\ccX = \Apolar{F}$ for $F = \sum_{n\geq 0}
    \DPel{x}{2n+1}t^n\in\kk[t][[x]]$.
\end{example}

\begin{remark}\label{ref:localdescriptiongraded:rmk}
        In the setting of Corollary~\ref{ref:tmplocaldescription:cor}
        or Corollary~\ref{ref:localdescriptionGorenstein:cor}, if the base $T$ is
        reduced and all the fibers are defined by homogeneous polynomials,
        then $F_{i1}, \ldots , F_{is_i}$ or $F_i$ may be chosen homogeneous.
        Indeed, by these assumptions the sheaf $\Fsheaf$ is invariant under
        the dilation action.
\end{remark}

\newcommand{\Tpr}{T'}
\newcommand{\DSTpr}{\DS_{\Tpr}}%
\newcommand{\DPTpr}{\hat{\DP}_{\Tpr}}%
\newcommand{\DPTpolypr}{\DP_{\Tpr}}%
\newcommand{\OTpr}{\OO_{\Tpr}}%
\newcommand{\DSTFpr}{\DSTpr\,\varphi(\Fsheaf)}%
\newcommand{\Pvarphi}{\hat{\DP}_{\varphi}}%
Description~\ref{ref:globaldescription:prop} above reduces the study of
finite flat families to the study of $\Apolar{\Fsheaf}$. Now we reduce this
study to the study of $\Fsheaf$ itself. For this, we need to check that
$\Apolar{-}$ is compatible with base change.
Every morphism $\varphi\colon T'\to T$ induces an isomorphism
$\varphi^*\DST = \DSTpr$ and consequently a natural map $\varphi^*\Homthreesh{\OT}{\DST}{\OT} \to
\Homthreesh{\varphi^*\OT}{\varphi^*\DST}{\varphi^*\OT}$, denoted
\begin{equation}
    \Pvarphi\colon\varphi^*\DPT \to \DPTpr.
\end{equation}
We abbreviate $\Pvarphi(\Fsheaf)$ to $\varphi(\Fsheaf)$. For $\Tpr = \{t\}$ we
denote $\Fsheaf(t) := \Pvarphi(\Fsheaf)$.

\begin{proposition}[base change for apolar]\label{ref:basechangeforapolar:prop}
    Let $\varphi\colon T'\to T$ be a morphism of locally Noetherian schemes and
    $\Fsheaf \subset \DPT$ be finitely flatly embedded. Then
    $\Apolar{\Fsheaf} \times_{T} T' = \Apolar{\varphi(\Fsheaf)}$.
\end{proposition}
If $\Fsheaf$ is not finitely flatly embedded, then the claim is false, see
Example~\ref{ex:notionofflatness}.
\begin{proof}
    \def\OX{\OO_{\ccX}}%
    \def\OXpr{\OO_{\ccX'}}%
    \def\IX{\mathcal{I}_{\ccX}}%
    \def\IXpr{\mathcal{I}_{\ccX'}}%
    Let $\ccX = \Apolar{\Fsheaf}$ with $\pi\colon \ccX\to T$ and $\ccX' = \ccX
    \times_T T'$.
    By Lemma~\ref{ref:flatness:lem} the $\OT$-modules $\pi_*\OX$ and $\DSTF$ are locally
    free. Moreover by~\eqref{eq:diagram} we have $\OX  \simeq
    \Homthreesh{\OT}{\DSTF}{\OT}$, so $\Homthreesh{\OT}{\OX}{\OT} =
    \DSTF$ as $\DST$-submodules of $\DPT$.
    The diagram
    \[\begin{tikzcd}
            \varphi^*\DSTF = \varphi^*\Homthreesh{\OT}{\OX}{\OT}\arrow{r}{ \simeq
            }\arrow{d}{} & \Homthreesh{\OTpr}{\varphi^*\OX}{\OTpr} =
            \Homthreesh{\OTpr}{\OXpr}{\OTpr}\arrow[d, hook]\\
            \varphi^*\DPT =
            \varphi^*\Homthreesh{\OT}{\DST}{\OT}\arrow{r}{\Pvarphi} &
            \Homthreesh{\OTpr}{\DSTpr}{\OTpr} = \DPTpr
    \end{tikzcd}\]
    implies that $\varphi(\DSTF) = \Homthree{\OTpr}{\OXpr}{\OTpr}$.
    The $\OTpr$-module $\pi'_*\OXpr$ is free as a base change of the free
    module $\pi_*\OX$, hence the $\OXpr$-module
    $\Homthree{\OTpr}{\OXpr}{\OTpr}$ is torsion-free, so that
    $\Ann(\varphi(\DSTF)) = \IXpr$ and $\ccX' = \Apolar{\varphi(\DSTF)} =
    \Apolar{\varphi(\Fsheaf)}$.
\end{proof}
\begin{corollary}\label{ref:apolarfibers:cor}
    Let $\Fsheaf \subset \DPT$ be finitely flatly embedded. Then the fiber of
    $\Apolar{\Fsheaf}$ over a point $t\in T$ is equal to
    $\Apolar{\Fsheaf(t)}$.
\end{corollary}
\begin{proof}
    Follows from Proposition~\ref{ref:basechangeforapolar:prop} for $T' = t$.
\end{proof}

In general, it is difficult to construct finitely flatly embedded subsheaves.
In view of Corollary~\ref{ref:apolarfibers:cor} one necessary condition, at
least over connected $T$, is that
\begin{equation}\label{eq:rankcondition}
    \dim_{\kappa(t)} \DSTF(t)\mbox{ is independent of the choice of }t\in T.
\end{equation}
Over reduced $T$ this condition is sufficient.
\begin{proposition}\label{ref:apolarflatness:prop}
    Let $T$ be reduced and $\Fsheaf \subset \DPT$ be such that $\DSTF$ is a
    finitely generated $\OT$-module. Suppose that~\eqref{eq:rankcondition} holds. Then
    $\Fsheaf$ is finitely flatly embedded.
\end{proposition}
\begin{proof}
    By assumption, $\DSTF$ is finitely generated, it remains to prove that
    \begin{equation}\label{eq:tmptosplit}
        0\to \DSTF\to \DPT \to \DPT/\DSTF\to 0
    \end{equation}
    is locally split. For this we may
    assume $T$ is a spectrum of a local ring with closed point $t$. Pick a
    basis of $\DSTF(t)$ and $D$ large enough, so that this basis is
    linearly independent in $\hat{\DP}_{t}/(\hat{\DP}_{t})_{\geq D}$. Then the
    map $\DSTF \to
    \DPT/(\DPT)_{\geq D}$ is injective. Consider the following diagram
    \[
        \begin{tikzcd}
            &   & 0 \arrow[d]& 0\arrow[d] & \\
            & 0 \arrow[d]\arrow[r]& (\DPT)_{\geq D}\arrow[d]\arrow{r}{ \simeq } &
            (\DPT)_{\geq D}\arrow{d} & \\
            0 \arrow[r]& \DSTF\arrow[r]\arrow{d}{ \simeq } &
            \DPT\arrow[r]\arrow[d] & \DPT/\DSTF\arrow[r]\arrow[d] & 0\\
            0 \arrow[r] & \DSTF\arrow[r]\arrow[d] & \DPT/(\DPT)_{\geq D}\arrow[r]\arrow[d] &
            C\arrow[r]\arrow[d] & 0\\
             & 0 & 0 & 0 &
        \end{tikzcd}
    \]
    The $\OT$-module $\DPT/(\DPT)_{\geq D}$ is free of finite rank, so that
    for each point $s\in T$ we have $\dim_{\kappa(s)} C(s) = \rank
    \DPT/(\DPT)_{\geq D} - \dim_{\kappa(s)} \DSTF(s)$ locally constant.
    Since $T$ is reduced, $C$ is a locally free $\OT$-module. Then $\DPT/\DSTF \simeq
    (\DPT)_{\geq D} \oplus C$.
    Also $\DPT  \simeq (\DPT)_{\geq D} \oplus \DPT/(\DPT)_{\geq D}$. Any
    choice of splitting of $\DPT/(\DPT)_{\geq D}\onto C$ yields the desired
    splitting of~\eqref{eq:tmptosplit}.
\end{proof}

\begin{example}\label{ex:tangentvectors}
    \def\divexp#1{\operatorname{exp}_{dp}(#1)}%
    \def\ww{\mathbf{w}}%
    \def\xx{\mathbf{x}}%
    Let $\ww\in \Spec \DS$ be a $\kk$-point. It corresponds to a linear map
    $\DS_1\to \kk$, hence gives a point in $\DP_1$. Denote $\divexp{\ww} := \sum_{i\geq 0} \ww^i\in
    \hat{\DP}$.

    Pick a polynomial $f\in \DP$. It defines a subscheme $R = \Spec
    \Apolar{f}$ supported at zero. We now construct a
    family which moves the support of $R$ along the line
    $\sspan{\ww} \subset \Spec \DS$. The line is chosen for clarity only, the same
    procedure works for arbitrary subvariety.

    For every linear form $\Dx\in \DS$ we have
    \[\Dx\hook (f\divexp{\ww}) = (\Dx\hook f)\divexp{\ww} + c\cdot
        f\divexp{\ww}
    = ((\Dx + c)\hook f)\cdot \divexp{\ww},\] where $c = \Dx\hook \ww\in \kk$.
    Hence for every polynomial $\sigma(\xx)\in \DS$ we have
    $\sigma(\xx)\hook (f\divexp{\ww}) = 0$ if and only if $\sigma(\xx +
    \ww)\hook f = 0$.
    It follows that $\Spec \Apolar{f\divexp{\ww}}$ is the scheme
    $R$ translated by the vector $\ww$; in particular it is supported on $\ww$ and abstractly
    isomorphic to $R$.
    By Proposition~\ref{ref:apolarflatness:prop}, the family $\Spec \Apolar{f\divexp{t\ww}}\to \Spec \kk[t]$ is
    finitely flatly embedded and geometrically corresponds to deformation by moving the support of
    $R$ along the line spanned by $\ww$.
    Its restriction to $\kk[\varepsilon] = \kk[t]/t^2$ gives $\Spec \Apolar{f
    + \varepsilon \ww f}$ corresponding to the tangent vector pointing
    towards this deformation.
\end{example}

For the next proposition recall, that a subset $V$ of an affine space is
called \emph{constructible} (in Zariski topology) if it is a finite union of locally
closed subsets. Each constructible subset possesses a (reduced) scheme structure.
\begin{proposition}[Families from constructible subsets]\label{ref:flatfamiliesforconstructible:prop}
    Suppose $V \subset \DP_{\leq d}$ is a constructible subset such that
    $V\ni f\to \dimk\Apolar{f}$ is constant. Then there exists a (finite flat)
    Gorenstein family with irreducible fibers
    \[
        \pi\colon \ccX \to V,
    \]
    such that $\ccX \subset V \times \Spec \DS$ and for every $f\in V$ we have
    $\pi^{-1}(f) = \Spec \Apolar{f} \subset V$.
\end{proposition}
Intuitively, the family $\pi$ is the incidence scheme $\{(f,
    \Spec\Apolar{f})\ |\ f\in V\}\to V$.
\begin{proof}
    \def\eluniv{\mathcal{V}}%
    Let $\kk[V]$ be the coordinate ring of $V$ and consider an universal element
    $\eluniv \in \kk[V] \tensor \DP_{\leq d}$ such that $\eluniv(f) = f\in
    \DP_{\leq d}$ for all $f\in V$. By
    Proposition~\ref{ref:apolarflatness:prop} the sheaf $\Fsheaf$ on $V$ generated by $\eluniv$ is finitely
    flatly embedded, hence gives a finite flat family $\pi\colon \Apolar{\Fsheaf}\to V$.
    By Corollary~\ref{ref:apolarfibers:cor} for every $f\in V$ we have
    $\pi^{-1}(f) = \Spec
    \Apolar{f}$, which proves that the family is Gorenstein
    (Definition~\ref{ref:Gorensteinfamily:def}) and has
    irreducible fibers.
\end{proof}

\begin{remark}\label{ref:fixedhilbertfunctioncontructible:remark}
    \def\DPH#1{\DP_{#1}}%
    For a fixed vector $H$ of length $d+1$, denote by $\DPH{H}$ the subset of forms $f\in \DP_{\leq d}$ such
    that $H$ is the Hilbert function of $\Apolar{f}$. Then $\DP_H$ is a
    constructible subset. Indeed, we have $H_{\Apolar{f}}(i) = \dim
    (\DmmS^i\hook {f})/(\DmmS^{i+1} \hook f)$.
    By picking a basis of $\DmmS^i$ we can rewrite the equality
    $H_{\Apolar{f}}(i) = H(i)$ as a rank condition on a matrix whose entries
    are coefficients of $f$. Hence, we obtain algebraic (both open and closed) conditions on
    $\DPH{H}$.
\end{remark}
%\begin{example}
%    Suppose $V \subset \DP_{\leq d}$ is a constructible subset such that
%    $V\ni f\to \dimk\Apolar{f}$ is constant. Consider an universal element
%    $\Fsheaf \subset \kk[V] \tensor \DP_{\leq d}$ such that $\Fsheaf(f) = f\in
%    \DP_{\leq d}$ for all $f\in V$. By Proposition~\ref{ref:apolarflatness:prop} it is finitely
%    flatly embedded, hence gives a finite flat family $\Spec\Apolar{\Fsheaf}\to V$.
%    By
%    Corollary~\ref{ref:apolarfibers:cor} the fiber over $f\in V$ is $\Spec
%    \Apolar{f}$, so we view this family as an incidence 
%\end{example}

\begin{example}[Proposition~\ref{ref:basechangeforapolar:prop} fails for
        $\Fsheaf$ not finitely flatly embedded]\label{ex:notionofflatness}
    Take $B = \kk[t]$ and $F = tx\in
    B[x]$. Then $\Ann(F) = (\Dx^2)$, so that the fibers of the associated
    apolar family are all equal to $\kk[\Dx]/\Dx^2$.
    For $t = \lambda$ non-zero we have $\kk[\Dx]/\Dx^2 = \Apolar{F_{\lambda}}$,
    but for $t = 0$ we have $F_{0} = 0$, so the fiber is \emph{not} the
    apolar algebra of $F_0$. It follows that $F$ is not finitely flatly
    embedded even though $\Apolar{F}$ is finite flat.
\end{example}

\begin{example}[Proposition~\ref{ref:apolarflatness:prop} fails for nonreduced base]
    Let use restrict the $B$ from Example~\ref{ex:notionofflatness} to $B =
    \kk[t]/t^2$ and take once more $F = tx\in B[x]$. Then $\Ann(F) = (t, \Dx^2)$, so
    $\Apolar{F}$ is not flat over $B$, even though $\Spec B$ has only one
    point, so~\eqref{eq:rankcondition} holds trivially.
\end{example}

\newcommand{\DualGens}{\mathcal{D}ual\mathcal{G}ens}%
Finitely flatly embedded subsheaves can be put into a functor. We discuss only the Gorenstein
case. Recall that finitely flatly embedded $\Fsheaf$ induces a finite flat
family $\pi_{\Fsheaf}\colon \Apolar{\Fsheaf} \to T$.
We define
\[
    \DualGens_r(T) = \left\{ \Fsheaf\in \DPT \mbox{ finitely flatly
    embedded, } \deg(\pi_{\Fsheaf}) = r\right\}.
\]
\newcommand{\HH}{\mathcal{H}}%
Let $\HH = \Hilbr{\mathbb{A}^n}$ and $\UU \subset \HH \times \mathbb{A}^n\to
\HH$ be the universal family, let $\pi\colon \UU\to \HH$.
\newcommand{\counivbundle}{\Speccal_\HH \Sym \pi_*\OO_\UU}%
\begin{proposition}\label{ref:apolarrepresentability:prop}
    The functor $\DualGens_r$ is representable by an open subset of the vector
    bundle $\Speccal_\HH \Sym \pi_*\OO_\UU$.
\end{proposition}
\begin{proof}
    \def\Repr{\mathcal{R}}
    Points of the bundle $\Speccal_\HH \Sym \pi_*\OO_\UU$ correspond to
    sections of the dual bundle
    $\omega := \Homthree{\OO_{\HH}}{\pi_*\OO_{\UU}}{\OO_{\HH}} =
    \pi_*\pi^!\OO_{\HH}$.
    Let $\Repr \subset \Speccal_\HH \Sym \pi_*\OO_\UU$ be the open subset
    parameterizing sections of $\omega$ which generate it as a $(\OO_{\HH}
    \otimes \DS)$-module.

    For each $T$ and $\Fsheaf\in\DualGens_r(T)$ we obtain a finite flat family
    \[
        \pi_{\Fsheaf}\colon\Apolar{\Fsheaf} \to T
    \]
    of degree $r$, hence a unique map $\varphi\colon T\to
    \HH$. Then $\Fsheaf$ is a generator of the $\DST$-module
    \[
        \Homthree{\OT}{(\pi_{\Fsheaf})_{*}\OO_{\Apolar{\Fsheaf}}}{\OT} = \varphi^*\omega
    \]
    and we get a natural map $T\to
    \Repr$. Conversely, a morphism $T\to \Repr$ gives a morphism
    $\varphi\colon T\to
    \HH$ and a generator $\Fsheaf$ of
    $\varphi^*\omega$, hence a finitely flatly embedded $\Fsheaf$.
\end{proof}

As an example, we provide the following irreducibility result,
useful later in Chapter~\ref{sec:Gorensteinloci}.
\begin{proposition}\label{ref:irreducibleintwovariables:prop}
    Let $\kk = \kkbar$ and $H = (1, 2, 2, \dots, 1)$ be a vector of length
    $d+1$. The set of
    polynomials $f\in \kdp[x_1, x_2]$ such that $H_{\Apolar{f}} = H$ constitutes
    an irreducible, locally closed subscheme of the affine space $\kdp[x_1, x_2]_{\leq d}$.
    A general member of
    this set has, up to the action of the group $\Dgrp$ defined in
    Section~\ref{ssec:Gorensteininapolarity},
    the form $\DPel{x_1}{d} + \DPel{x_2}{{d_2}}$ for some
    $d_2 \leq d$ depending only on $H$.
\end{proposition}
\begin{proof}
    \def\Dan#1{\Ann{#1}}%
    Let $V \subseteq \DP = \kdp[x_1, x_2]$ denote the set of $f$ such that
    $H_{\Apolar{f}} = H$. The subset $V$ is constructible by
    Remark~\ref{ref:fixedhilbertfunctioncontructible:remark}.
    Proposition~\ref{ref:flatfamiliesforconstructible:prop} yields a finite flat
    family $\{(f, \Apolar{f})\}\to V$ and a map $\varphi\colon V\to
    \Hilbr{\mathbb{A}^2}$. By
    Proposition~\ref{ref:apolarrepresentability:prop}, the map $\varphi(V)$
    induces an embedding of $V$ as an open subset of a bundle over
    $\varphi(V)$. The set $\varphi(V)$ is irreducible
    by~\cite[Theorem~3.13]{iarrobino_punctual}.
    Therefore, also $V$ is irreducible.

    Let us take a general polynomial $f$ such that $H_{\Apolar{f}} = H$. Then
    $\Dan{f} = (q_1, q_2)$ is a complete intersection
    by~\cite[Corollary~21.20]{Eisenbud}. Since $H(2) = 2$, we assume that
    $q_1\in\DS$ has
    order $2$, i.e.~$q_1\in \DmmS^2\setminus \DmmS^3$. Since $f$ is general,
    we may assume that the quadric part of $q_1$ has maximal rank, i.e. rank
    two, see also \cite[Theorem~3.14]{iarrobino_punctual}. Then after a change of
    variables $q_1 \equiv \Dx_1\Dx_2 \mod \DmmS^3$.
    Since the leading form $\Dx_1\Dx_2$ of $q_1$ is reducible, $q_1 = \delta_1
    \delta_2$ for some $\delta_1, \delta_2\in \DShat$ such that $\delta_i \equiv
    \Dx_i \mod \DmmS^2$ for $i=1,2$,
    see~e.g.~\cite[Theorem~16.6]{Kunz_plane_algebraic_curves}. After an
    automorphism of $\DShat$ we may assume $\delta_i = \Dx_i$, then $\Dx_1\Dx_2 =
    q_1$ annihilates $f$, so that $f$ is a sum of (divided) powers of $x_1$ and $x_2$.
    Let $d_2$ be largest number such that $\DPel{x_2}{d_2}$ appears in $f$,
    then $f = \DPel{x_1}{d} + \DPel{x_{2}}{d_2} + \partial\hook\pp{\DPel{x_1}{d}
    + \DPel{x_{2}}{d_2}}$
    for some $\partial\in \DmmS$, so that $f\in \Dgrp \pp{\DPel{x_1}{d} +
    \DPel{x_2}{d_2}}$.
\end{proof}

\section{Homogeneous forms and secant varieties}\label{ssec:secants}

    In this section we give an easy application of relative Macaulay's inverse
    defined in Section~\ref{ssec:relativemacaulaysystems}.

    In Chapter~\ref{sec:macaulaysinversesystems} the main emphasis is on
    nonhomogeneous polynomials, since the classification of graded ones (of socle
    degree $d$) is just the classification of orbits of general linear group
    acting on $\DP_d$, see Remark~\ref{ref:Macaulaygraded:rmk}. In this section
    we discuss classifications of homogeneous polynomials corresponding to
    algebras with Hilbert functions $(1, 3, 3,  \ldots , 3, 1)$ and $(1, 4, 4,
    \ldots , 4, 1)$. These results are used in
    Chapter~\ref{sec:Gorensteinloci} and are also of independent interest. In
    the end of the section we discuss connections with secant
    varieties of Veronese reembeddings, precisely with
    $\sigma_{3}(\nu_d(\PP^2))$
    and $\sigma_{4}(\nu_d(\PP^3))$ for $d\geq 4$,
    see~\cite{geramita_inverse_systems_of_fat_points,
    iarrobino_kanev_book_Gorenstein_algebras, bubu2010}.

    Most of this material is algebraic in nature and refers to
    Part~\ref{part:algebras}. An exception is the
    correspondence between families of forms and families of
    finite graded Gorenstein subschemes, which is a special case of relative
    Macaulay's inverse systems from Section~\ref{ssec:relativemacaulaysystems}.

    \begin{proposition}[(1, 3, 3, 3,  \ldots , 3, 1)]\label{ref:thirdsecant:prop}
    Suppose that $\DPut{ff}{F}\in \kdp[x_1, x_2, x_3]$ is a homogeneous polynomial
    of degree $d\geq 4$.
    The following conditions are equivalent
    \begin{enumerate}
        \item the algebra $\Apolar{\Dff}$ has Hilbert function $H$ beginning
            with $H(1) = H(2) = H(3) = 3$, i.e.~$H = (1, 3, 3, 3,
             \ldots 3, 1)$,
        \item after a linear change of variables, $\Dff$ is equal to one of the forms
            \[\DPel{x_1}{d} + \DPel{x_2}{d} + \DPel{x_3}{d},\qquad \DPel{x_1}{d -
            1}x_2 + \DPel{x_3}{d},\qquad \DPel{x_1}{d - 1} x_3 + \DPel{x_1}{d-2}\DPel{x_2}{2}.\]
    \end{enumerate}
    Furthermore, the set of forms in $\kdp[x_1, x_2, x_3]_{d}$ satisfying
    the above conditions is irreducible.
\end{proposition}
For the characteristic zero case see \cite{LO} or \cite[Theorem 4]{BGIComputingSymmetricRank} and references therein.
    See also \cite{bubu2010} for a generalization of this method.

\begin{proof}
    \def\Ddegf{d}%
    \def\span#1{\langle #1 \rangle}%
    \def\Dtmpy{\theta}%
    \def\Dan#1{\Ann{#1}}%
    Let $\DS = \kk[\Dx_1, \Dx_2, \Dx_3]$.
    Let $I := \Ann(\Dff)$ and $I_2 := \span{\Dtmpy_1, \Dtmpy_2,
    \Dtmpy_3}
    \subseteq \DS_2$ be the linear space of operators of degree $2$
    annihilating $\Dff$. Let  $A := \DS/I$, $J := (I_2)
    \subseteq \DS$ and $B := \DS/J$.
    Since $A$ has degree greater than $3\cdot 3 > 2^3$, the ideal $J$ is not a complete
    intersection. Let us analyse the Hilbert function of $A$. By
    Proposition~\ref{ref:symmetryGorensteingraded:prop} we have $H_A(d-1) = H_A(1) = 3$. By
    Corollary~\ref{ref:nonincreasinghf:cor} we have $3 = H_A(3) \geq H_A(4)
    \geq  \ldots \geq H_A(d-1) = 3$, thus
    \[
        H_A(i) = 3\quad\mbox{for all}\quad i = 1, 2,  \ldots , d-1.
    \]
    We will prove that the graded ideal $J$ is saturated and defines a finite scheme of degree
    $3$ in $\Proj \DS = \mathbb{P}^2$. First, $3 = H_A(3)
    \leq H_{B}(3) \leq 4$ by Macaulay's Growth Theorem.
     By Remark~\ref{ref:GorensteinSaturated:rmk}, the algebra $A$
    is $2$-saturated. Since $A_i = B_i$ for $i\leq 2$, also $B$ is
    $2$-saturated. If $H_B(3) = 4$, then applying Lemma~\ref{ref:P1gotzmann:lem}.1 to $B$, we
    obtain $H_B(1) = 2 = H_A(1)$, a contradiction. We have proved that $H_{B}(3) = 3$.

    Now we want to prove that $H_B(4) = 3$. By Macaulay's Growth Theorem applied to
    $H_B(3) = 3$ we have $H_B(4) \leq 3$. If $\Ddegf > 4$ then $H_A(4) = 3$,
    so $H_B(4) \geq 3$. Suppose $\Ddegf = 4$.  By Buchsbaum-Eisenbud result
    \cite[p.~448]{BuchsbaumEisenbudCodimThree} we know that the minimal number of
    generators of $I$ is odd. Moreover, we know that $H_A(i) = H_B(i)$ for $i < 4$,
    thus the generators of $I$ have degree two or four. Since $I_2$ is not a
    complete intersection, there are at least two generators of degree
    $4$, so $H_B(4) \geq H_A(4) + 2 = 3$.

    \def\Jsat{J^{sat}}%
    From $H_B(3) = H_B(4) = 3$ by Gotzmann's Persistence Theorem we see that $H_B(i) =
    3$ for all $i\geq 1$. Thus the scheme $\Gamma := V(J) \subseteq \Proj
    \kk[\Dx_1, \Dx_2, \Dx_3]$ is finite of degree $3$. Let $\Jsat \supset J$
    denote its saturated ideal. Then $H_{S/\Jsat}(i) = H_{S/J}(i) = 3$ for $i\gg 0$. By
    Macaulay's Theorem~\ref{ref:MacaulayGrowth:thm} we have $H_{\DS/\Jsat}(i)
    = 3$ for all $i\geq 2$, hence $\Jsat_i = J_i$ for these $i$. Moreover $\Jsat_1 =
    J_1$, since $B = \DS/J$ is $2$-saturated. Therefore, $J$ is saturated. In
    particular, the ideal $J = I(\Gamma)$ is contained in $I$.

    We will use $\Gamma$ to compute the possible forms of $F$, in the spirit
    of Apolarity Lemma, see
    \cite[Lemma~1.15]{iarrobino_kanev_book_Gorenstein_algebras}. There are four possibilities for $\Gamma$:
    \begin{enumerate}
        \item $\Gamma$ is a union of three distinct, non-collinear points. After a change of basis $\Gamma
            = \left\{ [1:0:0] \right\} \cup \left\{ [0:1:0] \right\} \cup
            \left\{ [0:0:1] \right\}$, then $I_2 = (\Dx_1\Dx_2, \Dx_2\Dx_3,
            \Dx_3\Dx_1)$ and $\Dff = \DPel{x_1}{\Ddegf} + \DPel{x_2}{\Ddegf} +
            \DPel{x_3}{\Ddegf}$.
        \item $\Gamma$ is a union of a point and scheme of degree two, such
            that $\span{\Gamma} = \mathbb{P}^2$. After
            a change of basis $I_{\Gamma} = (\Dx_1^2, \Dx_1\Dx_2, \Dx_2\Dx_3)$,
            so that $\Dff = \DPel{x_3}{\Ddegf-1}x_1 + \DPel{x_2}{\Ddegf}$.
        \item $\Gamma$ is irreducible with support $[1:0:0]$ and it is not a
            $2$-fat point. Then $\Gamma$ is Gorenstein and so $\Gamma$ may
            be taken as the curvilinear scheme defined by $(\Dx_3^2, \Dx_2\Dx_3,
            \Dx_1\Dx_3 - \Dx_2^2)$. Then, after a linear change of variables,
            $\Dff = \DPel{x_1}{\Ddegf-1}x_3 +
            \DPel{x_2}{2}\DPel{x_1}{\Ddegf-2}$.
        \item $\Gamma$ is a $2$-fat point supported at $[1:0:0]$. Then
            $I_{\Gamma} = (\Dx_2^2, \Dx_2\Dx_3, \Dx_3^2)$, so $F =
            \DPel{x_1}{\Ddegf-1}(\lambda_2 x_2 + \lambda_3 x_3)$ for some $\lambda_2,
            \lambda_3\in \kk$. But then there is a degree one operator in $\DS$
            annihilating $F$, a contradiction.
    \end{enumerate}
    The set of forms $\Dff$ which are sums of three powers of
    linear forms is irreducible. To see that the forms satisfying the
    assumptions of the Proposition constitute an irreducible
    subset of $\DP_{\Ddegf}$ we observe that every $\Gamma$ as above is
    smoothable, by
    \cite{cartwright_erman_velasco_viray_Hilb8}. By Local
    Description~\ref{ref:localdescriptionGorenstein:cor} and
    Remark~\ref{ref:localdescriptiongraded:rmk}, the flat family
    proving the smoothability of $\Gamma$ induces a family $\Dff_t \to \Dff$,
    such that $\Dff_{\lambda}$ is a sum of three powers of linear forms for
    $\lambda\neq 0$.
\end{proof}
The claim of Proposition~\ref{ref:thirdsecant:prop} is false for $d = 3$,
i.e., for cubics $F\in \kdp[x_1,x_2,x_3]$ with Hilbert functions $(1, 3, 3, 1)$. Indeed, a general
cubic has such Hilbert function, while a general cubic is not a sum of three
cubes (or a limit of such); the third secant variety of third Veronese
reembedding is a hypersurface given by the Aronhold invariant, see~\cite{LO}.

\begin{proposition}\label{ref:fourthsecant:prop}
    \def\Ddegf{d}%
    Let $\Ddegf \geq 4$.
    Consider the set $\DPut{set}{\mathcal{S}}$ of all forms $\DPut{ff}{F}\in \kdp[x_1,
    x_2, x_3, x_4]$ of degree $\Ddegf$
    such that the apolar algebra of $\Dff$ has Hilbert function $(1, 4, 4,
    4, \dots, 4, 1)$. This set is irreducible and its general member has the
    form $\DPel{\ell_1}{\Ddegf}
    + \DPel{\ell_2}{\Ddegf} + \DPel{\ell_3}{\Ddegf} + \DPel{\ell_4}{\Ddegf}$, where $\ell_1$,
    $\ell_2$, $\ell_3$, $\ell_4$
    are linearly independent linear forms.
\end{proposition}

\begin{proof}
    \def\Dan#1{\Ann{#1}}%
    \def\Ddegf{d}%
    First, the set $\DPut{setzero}{\mathcal{S}_0}$ of forms equal to $\DPel{\ell_1}{\Ddegf}
    + \DPel{\ell_2}{\Ddegf} + \DPel{\ell_3}{\Ddegf} + \DPel{\ell_4}{\Ddegf}$, where $\ell_1$, $\ell_2$, $\ell_3$, $\ell_4$
    are linearly independent linear forms, is irreducible and contained in
    $\Dset$. Then, it is enough
    to prove that $\Dset$ lies in the closure of $\Dsetzero$.

    We follow the proof of Proposition~\ref{ref:thirdsecant:prop}, omitting some details which can be found there.
    Let $\DS = \kk[\Dx_1, \Dx_2, \Dx_3, \Dx_4]$, $I := \Dan{\Dff}$ and $J :=
    (I_2)$. Set $A = S/I$ and $B = S/J$. Then
    $H_B(2) = 4$ and $H_B(3)$ is either $4$ or $5$. If $H_B(3) = 5$, then by
    Lemma~\ref{ref:P1gotzmann:lem} we have $H_B(1) = 3$, a contradiction. Thus
    $H_B(3) = 4$.

    Now we would like to prove $H_B(4) = 4$. By
    Macaulay's Growth Theorem we have $H_B(4) \leq 5$. By
    Lemma~\ref{ref:P1gotzmann:lem} we have $H_B(4) \neq 5$, thus $H_B(4) \leq 4$. If
    $\Ddegf > 4$ then $H_B(4) \geq H_A(4) \geq 4$, so we concentrate on
    the case $\Ddegf = 4$.
    Let us write the minimal free resolution of $A$, which is symmetric, as
    mentioned in Section~\ref{ssec:betti}.
    \[
        0\to S(-8) \to S(-4)^{\oplus a}\oplus S(-6)^{\oplus 6} \to
        S(-3)^{\oplus b}\oplus S(-4)^{\oplus c} \oplus S(-5)^{\oplus b}\to
        S(-2)^{\oplus 6} \oplus S(-4)^{\oplus a}\to S.
    \]
    Calculating $H_A(3) = 4$ from the resolution, we get $b = 8$. Calculating
    $H_A(4) = 1$ we obtain $6 - 2a + c = 0$. Since $1 + a = H_B(4) \leq 4$ we have $a\leq 3$,
    so $a = 3$, $c = 0$ and $H_B(4) = 4$.

    Now we calculate $H_B(5)$. If $\Ddegf > 5$ then $H_B(5) = 4$ as before.
    If $\Ddegf = 4$ then extracting syzygies of $I_2$ from the above resolution
    we see that $H_B(5) = 4
    + \gamma$, where $0\leq \gamma\leq 8$, thus
    $H_B(5) = 4$ and $\gamma = 0$.
    If $\Ddegf = 5$, then the resolution of $A$ is
    \[
        0\to S(-9)\to S(-4)^{\oplus 3}\oplus S(-7)^{\oplus 6}\to S(-3)^{\oplus
        8}\oplus S(-6)^{\oplus 8}\to S(-5)^{\oplus 3}\oplus S(-2)^{\oplus 6}
        \to S.
    \]
    %56 - (20*6) + 8*
    So $H_B(5) = 56 - 20\cdot 6 + 8 = 4$. Thus, as in the previous case we see that $J$ is the
    saturated ideal of a scheme $\Gamma$ of degree $4$. Then $\Gamma$
    is smoothable by \cite{cartwright_erman_velasco_viray_Hilb8} and its smoothing induces a family $\Dff_t\to
    \Dff$, where $\Dff_{\lambda}\in \Dsetzero$ for $\lambda\neq 0$.
\end{proof}

The following Corollary~\ref{ref:equationforsecant:cor} is a consequence of
Proposition~\ref{ref:fourthsecant:prop}. This corollary strengthens
the connection with secant varieties. For simplicity and to refer to some
results from \cite{LO}, we assume that $\kk = \mathbb{C}$, but
the claim holds for all fields of characteristic either zero or large
enough.

\newcommand{\catal}{\Cat_{a, d-a}}%
To formulate the claim we introduce catalecticant matrices.
Let $\catal: \DS_{a}\times \DP_{d}\to\DP_{d-a}$ be the contraction mapping
applied to homogeneous polynomials of degree $d$. For $F\in \DP_d$
we obtain $\catal(F): \DS_a \to \DP_{d-a}$, whose matrix is called the \emph{$a$-catalecticant
matrix}. It is straightforward to see that $\rank \catal(F) =
H_{\Apolar{F}}(a)$.

\begin{corollary}\label{ref:equationforsecant:cor}
    \def\Ddegf{d}%
    Let $\Ddegf \geq 4$ and $\kk = \mathbb{C}$.
    The fourth secant variety to the $\Ddegf$-th Veronese reembedding of
    $\mathbb{P}^n$ is a subset $\sigma_4(\nu_\Ddegf(\mathbb{P}^n)) \subseteq \mathbb{P}(P_d)$ set-theoretically defined
    by the condition $\rank \catal \leq 4$, where $a = \floor{\Ddegf/2}$.
\end{corollary}
We assume $\kk = \mathbb{C}$ only to refer to~\cite[Theorem~3.2.1~(2)]{LO}, we
believe that this assumption can be removed.
\begin{proof}
    \def\Dh#1{H(#1)}%
    \def\Ddegf{d}%
    Since $H_{\Apolar{F}}(a) \leq 4$ for $F$ which is a sum of four powers of
    linear forms, by semicontinuity every $F\in \sigma_4(\nu_\Ddegf(\mathbb{P}^n))$ satisfies the
    above condition.

    Let $F\in \DP_\Ddegf$ be a form satisfying $\rank \catal(F) \leq 4$. Let $A = \Apolar{F}$ and $H = H_A$ be the Hilbert
    function of $A$. We will reduce to the case where $\Dh{i} = 4$ for all $0 < i < \Ddegf$.

    First we prove that $\Dh{i} \geq 4$ for all $0 < i < \Ddegf$.
    If $\Dh{1}\leq 3$, then the claim follows from~\cite[Theorem~3.2.1~(2)]{LO},
    so we assume $\Dh{1} \geq 4$.
    Suppose that for some $i$ satisfying  $4 \leq i < \Ddegf$ we
    have $\Dh{i} < 4$. Then by Corollary~\ref{ref:nonincreasinghf:cor} we
    have $\Dh{j} \leq \Dh{i}$ for
    all $j \geq i$, so that $\Dh{1} = \Dh{\Ddegf-1} < 4$, a contradiction. Thus $\Dh{i}\geq
    4$ for all $i\geq 4$. Moreover, $\Dh{3}\geq 4$ by Macaulay's Growth Theorem. Suppose now
    that $\Dh{2}
    < 4$. By Theorem~\ref{ref:MacaulayGrowth:thm} the
    only possible case is $\Dh{2} = 3$ and $\Dh{3} = 4$. But then $\Dh{1} =2 <
    4$ by Lemma~\ref{ref:P1gotzmann:lem}, a contradiction. Thus we have proved
    that
    \begin{equation}\label{eq:lowerbound}
        \Dh{i} \geq 4\quad \mbox{for all}\quad 0 < i < \Ddegf.
    \end{equation}
    We now prove that $\Dh{i} = 4$ for all $0< i < \Ddegf$.
    By assumption, $\Dh{a} = 4$. If $\Ddegf\geq 8$, then $a\geq 4$, so by
    Corollary~\ref{ref:nonincreasinghf:cor} we have $\Dh{i}
    \leq 4$ for all $i > a$. Then by the symmetry $\Dh{i} = \Dh{\Ddegf - i}$ we have $\Dh{i}
    \leq 4$ for all $i$. Together with $\Dh{i}\geq 4$ for $0 < i < \Ddegf$, we have
    $\Dh{i} = 4$ for $0 < i < \Ddegf$. Then $F\in \sigma_4(\nu_\Ddegf(\mathbb{P}^n))$ by
    Proposition~\ref{ref:fourthsecant:prop}.
    If $a = 3$ (i.e.~$\Ddegf = 6$ or $\Ddegf = 7$), then $\Dh{4}\leq 5$ by
    Macaulay~s Theorem~\ref{ref:MacaulayGrowth:thm} and $\Dh{4} = 5$ would
    contradict
    Lemma~\ref{ref:P1gotzmann:lem}, hence $\Dh{4}\leq 4$ and we finish the proof as in the case
    $\Ddegf
    \geq 8$.
    If $\Ddegf = 5$, then $a = 2$ and the Hilbert function of $A$ is $(1, e, 4, 4,
    e, 1)$. Again arguing using Lemma~\ref{ref:P1gotzmann:lem}, we have $e\leq 4$, thus
    $e = 4$ by \eqref{eq:lowerbound} and
    Proposition~\ref{ref:fourthsecant:prop} applies.
    If $\Ddegf = 4$, then $H = (1, e, 4, e, 1)$. Suppose $e\geq 5$, then
    Lemma~\ref{ref:P1gotzmann:lem} gives $e\leq 3$, a contradiction. Thus
    $e = 4$ and Proposition~\ref{ref:fourthsecant:prop} applies also to this case.
\end{proof}

Note that for $d\geq 8$ and $\kk = \CC$, Corollary~\ref{ref:equationforsecant:cor} was
proved in \cite[Theorem~1.1]{bubu2010}.

\chapter{Smoothings}\label{sec:smoothings}

\newcommand{\amb}{\ambient}%
    One geometric way to obtain a finite scheme of degree $r$ embedded into an
    ambient scheme $\amb$
    is to take $r$ points over $\kk$ of $\amb$ and collide them (we
    make this precise in Section~\ref{ssec:abstractpropsofsmoothability}). The result
    is a \emph{smoothable} finite scheme and the family describing
    the trajectories of points is called an \emph{embedded smoothing}. We also
    consider \emph{abstract smoothings}. In the following
    subsections we formally develop the theory of smoothings. Our main aim is
    to prove the following theorem.
\begin{theorem}\label{thm_equivalence_of_abstract_and_embedded_smoothings}
  Suppose $X$ is a smooth variety over a field $\kk$ and
  $R\subset X$ is a finite $\kk$-subscheme.
  The following conditions are equivalent:
  \begin{enumerate}
      \item\label{it:Rabssm} $R$ is abstractly smoothable,
      \item\label{it:Rembsm} $R$ is embedded smoothable in $X$,
      \item\label{it:Rconabssm} every connected component of $R$ is abstractly smoothable,
      \item\label{it:Rconembsm} every connected component of $R$ is embedded smoothable in $X$.
  \end{enumerate}
\end{theorem}
We closely follow~\cite{jabu_jelisiejew_smoothability}.

\section{Abstract smoothings}\label{ssec:abstractpropsofsmoothability}

In this subsection we introduce smoothings as (finite
flat) generically smooth families with prescribed special fiber. Most
importantly, we prove that once a smoothing exists, also a smoothing over a
small base (one-dimensional complete local ring) exists as well. We deduce
that smoothability can be checked independently on each component of a finite
scheme.

    \begin{definition}[abstract smoothing]\label{ref:abstractsmoothable:def}
        Let $R$ be a~finite scheme over $\kk$. We say that \emph{$R$ is abstractly
        smoothable} if there exist an irreducible scheme
        $T$ and a finite flat family $\famil \to T$
        such that
        \begin{enumerate}
            \item $T$ has a $\kk$-rational point $t$, such that $\famil _t \simeq R$. We call
                $t$ the \emph{special point} of $T$.
            \item $\famil _{\eta}$ is a~smooth scheme over $\eta$, where $\eta$
                is the generic point of $T$.
        \end{enumerate}
        The $T$-scheme $\famil $ is called an \emph{abstract smoothing} of $R$. We
        sometimes
        denote it by $(\famil , R)\to (T, t)$, which means that $t$ is the $\kk$-rational point of $T$, such that
        $\famil _t  \simeq R$.
    \end{definition}

    An abstract smoothing $\famil\to T$ is finite, so the relative cotangent
    sheaf $\Omega_{\famil/T}$ is coherent. Therefore the set of fibers which are smooth is
    \emph{open} and, by assumption, non-empty. For example, if $T$ is a curve, then all but
    finitely many fibers are smooth. A fiber over a $\kk$-point is smooth if
    and only if it is a union of $\deg R$ points.

    \begin{example}
        Any finite smooth scheme $R$ has a trivial smoothing
        $\famil =R$, $T = \Spec \kk$.
    \end{example}

    \begin{example}\label{ex:fieldExtensions}
        \def\KK{\mathbb{K}}%
        For every finite field extension $\kk\subset \KK$ the $\kk$-scheme $R =
        \Spec \KK$ is smoothable.  Suppose first that $\kk \subset
        \KK$ is  a separable extension.  Then $R = \Spec \KK$ is smooth over
        $\kk$, so it is trivially smoothable.
%    To prove this, we work by induction on the degree of extension.
        Suppose now that $\kk \subset \KK$ is not separable.
        Every extension of a finite field is separable, so we may assume $\kk$ is infinite.
        We may decompose $\kk \subset \KK$ as a chain of one-element extensions
        $\kk \subset \kk(t_1) \subset \kk(t_1,t_2) \subset  \ldots
        \subset \kk(t_1, \ldots ,t_n) = \KK$.
        Then $\KK = \kk[\Dx_1, \ldots ,\Dx_n]/(f_1, \ldots ,f_n)$ where $f_i$ is the lifting of minimal polynomial of $t_i$; in
        particular $f_i = \Dx_i^{d_i} + a_{i,d_i-1}\Dx_i^{d_i-1} +  \ldots  + a_{i, 0}$, where
        $a_{i,j}\in \kk[\Dx_1, \ldots ,\Dx_{i-1}]$.
        We now inductively construct, for $i=1, \ldots ,n$, polynomials $F_i\in \kk[\Dx_1, \ldots
        ,\Dx_n][t]$ such that:
        \begin{enumerate}
            \item\label{it:cond1} The family $\famil  = \Spec \kk[\Dx_1, \ldots ,\Dx_n,t]/(F_1, \ldots ,F_n)\to
                \Spec \kk[t]$ is flat and finite,
            \item\label{it:cond2} $F_i(0) = f_i$, so that $\famil _{|t=0} = \Spec \KK$,
            \item\label{it:cond3} The fiber $\famil _{|t=1}$ is a disjoint union of copies of $\kk$.
        \end{enumerate}
        First, we construct polynomials $g_i$ such that $g_i$ has degree $d_i =
        \deg(f_i)$ and
        \begin{equation}\label{eq:formofgi}
            g_i = \Dx_i^{d_i} + b_{i,d_i-1}\Dx_i^{d_i-1} +  \ldots  + b_{i, 0}\quad
            \mbox{where}\quad
            b_{i,j}\in \kk[\Dx_1, \ldots ,\Dx_{i-1}]
        \end{equation}
        and $\kk[\Dx_1, \ldots ,\Dx_n]/(g_1,
        \ldots ,g_n)$ is a product of $\kk$. This is done inductively. We choose $g_1$ as any
        polynomial of degree $d_1$ having $d_1$ distinct roots in $\kk$, then
        $\kk[\Dx_1]/g_1  \simeq \kk^{\times d_1}$.
        We choose $g_2$ as a polynomial of degree $d_2$ in $\kk[\Dx_1,\Dx_2]/g_1 \simeq
        (\kk[\Dx_2])^{\times d_1}$ such that after projecting to each factor $g_i$
        has $d_2$ distinct roots in $\kk$, then $\kk[\Dx_1,\Dx_2]/(g_1, g_2)  \simeq
        \kk^{\times d_1d_2}$ and so we continue.

        Define $F_i = (1-t)f_i + t g_i$, then $F_i =
        \Dx_i^{d_i}+((1-t)a_{i,d_{i}-1}+tb_{i,d_i-1})\Dx_i^{d_i-1} + \ldots
        +((1-t)a_{i,0}+tb_{i,0})$ where $(1-t)a_{j,d_{i}-1}+tb_{j,d_i-1}\in
        \kk[\Dx_1, \ldots ,\Dx_{i-1}][t]$.
        From this ``upper-triangular'' form of $F_i$ we see that the quotient
        \[
            \kk[\Dx_1, \ldots ,\Dx_n,t]/(F_1, \ldots ,F_n)
        \]
        is a free $\kk[t]$-module with basis consisting of monomials
        $\Dx_1^{s_1} \ldots \Dx_n^{s_n}$ such that $s_i < d_i$ for all $i$.
        Hence, Condition~\ref{it:cond1} is satisfied;
        Conditions~\ref{it:cond2} and~\ref{it:cond3} are satisfied by construction. In
        particular, a fiber of $\famil $ is smooth, so the generic fiber is also smooth, thus
        the family $\famil $ is a smoothing of $\kk$-scheme $\Spec
        \KK$.
    \end{example}

We now introduce \emph{embedded smoothings}.
The difference between the previous setting is in the presence of the ambient
scheme $\amb$, where the whole smoothing must be embedded.
According to Theorem~\ref{thm_equivalence_of_abstract_and_embedded_smoothings}
we eventually prove that for smooth $\amb$ the notions are equivalent.

    \begin{definition}[embedded smoothing]\label{ref:smoothable:def}
        Let $\amb$ be a~scheme and $R$ be a~finite closed subscheme of $\amb$. We say that
        $R$ is \emph{smoothable in $\amb$} if there exist an irreducible scheme
        $T$ and a~closed subscheme $\famil  \subseteq \amb \times T$ such that $\famil \to T$ is
        an abstract smoothing of $R$.
        The scheme $\famil $ is called an \emph{embedded
        smoothing} of $R \subseteq \amb$.
    \end{definition}

\begin{example}
   Let $\amb= \Spec \kk[\Dx,\Dy,\Dz]/(\Dx\Dy,\Dx\Dz,\Dy\Dz)$, i.e.~$\amb$ is the union of three coordinate lines 
     in the three dimensional affine space $\AA_{\kk}^3$.
   Let $R= \Spec \kk[\Dx,\Dy,\Dz]/(\Dx-\Dy, \Dx-\Dz, \Dx^2) \simeq \kk[\epsilon]/\epsilon^2$
   be the degree two subscheme of $\amb$,
     which is the intersection of $\amb$ with the affine line $x=y=z$. 
   Then $R$ is abstractly smoothable, but $R$ is not smoothable in $\amb$.
\end{example}

    Smoothings behave well under base-change, modulo the existence of
    a $\kk$-point.

    \begin{lemma}[Base change for smoothings]\label{ref:basechange:lem}
        Let $T$ be an irreducible scheme with the generic point $\eta$ and a
        $\kk$-rational point $t$.  Let $(\famil , R)\to (T, t)$ be an abstract~smoothing of a~finite scheme $R$.

        Suppose $T'$ is another irreducible scheme with a~morphism $f\colon T' \to T$ such that
        $\eta$ is in the image of $f$ and there
        exists a~$\kk$-rational point of $T'$ mapping to $t$. Then
        the base change $\famil ' = \famil  \times_T T' \to T'$ is an abstract~smoothing
        of $R$.
        Moreover, if $R$ is embedded into some $\amb$ and
        $\famil \subseteq \amb\times T$ is an embedded smoothing, then $\famil ' \subseteq
        \amb\times T'$ is also an embedded smoothing of $R\subseteq \amb$.
    \end{lemma}

    \begin{proof}
        $\famil '\to T'$ is finite and flat. The generic point $\eta'$ of $T'$ maps
        to $\eta$ under $f$ so that $\famil '_{\eta'}\to \eta'$ is a~base change of a
        smooth morphism $\famil _{\eta}\to \eta$. In particular it is smooth, so
        that $\famil '\to T'$ is a~smoothing of $R$. If $\famil  \subseteq \amb\times T$ was
        a~closed subscheme, then $\famil '\subseteq \amb\times T'$ is also a~closed
        subscheme.
    \end{proof}

    The next lemma can be informally summarised as follows: if $U$ is an open
    subset of a scheme $T$, and $t$ is a point in the closure of $U$, then
    there exists a curve in $T$ through $t$ intersecting $U$.
    \begin{lemma}\label{lem_finding_a_curve_through_point_and_open_subset}
        Suppose $T$ is a scheme, $U \subset T$ is an open subset and $t\in T$
        is a point in $\overline{U}$.
        Suppose the residue field of $t$ is $\kappa$.
        Then there exists a one-dimensional Noetherian complete local domain
        $A'$ with residue field $\kappa$,
           and a morphism $T'= \Spec A' \to T$, such that the closed point $t'
           \in T'$ is $\kappa$-rational and it is mapped to $t$ and the generic point $\eta' \in T'$ is mapped into $U$.

        If in addition $\kappa$ is algebraically closed, we may furthermore
        assume that $A' = \kappa[[x]]$.
   \end{lemma}
    \begin{proof}
        \def\gotp{\mathfrak{p}}%
        First, we may replace $T$ with $\Spec A_1 := \Spec \OO_{T, t}$, and $U$ with the preimage under $\Spec A_1 \to T$.
        The ring $A_1$ is Noetherian by our global assumption.
        Let $A$  be the completion of $A_1$ at the maximal ideal $\mathfrak{m}$ of $A_1$.
        Then $A$ is Noetherian and flat over $A_1$, so that
        $\Spec A\to \Spec A_1$ is surjective
        \cite[Tag~0316, Tag~0250, items~(6)\&(7)]{stacks_project}.
        Let any prime ideal $ \gotp \subset A$ mapping to the generic point of $T$.
        Then $\Spec A/\mathfrak{p}$ is integral and satisfies the assertions on $T'$
        except, perhaps, one-dimensionality.
        Replace $T$ with $\Spec A/\mathfrak{p}$ and $U$ with the preimage under $\Spec A/\mathfrak{p} \to T$.

        If $\dim T =0$, then $T = \{t\} = \Spec \kappa$ and $U = T$, and $T'
        = \Spec \kappa[[x]]$ with a morphism $T' \to T$ corresponding to
        $\kappa \to \kappa[[x]]$ will satisfy the claim of the lemma.
        So suppose $\dim T>0$.

        Let $\eta$ be a generic point of $T$.
        If $U = \{\eta\}$, then $T$ is at most one-dimensional by the Theorem of
        Artin-Tate, see~\cite[Corollary~B.62]{gortz_wedhorn_algebraic_geometry_I}.
        If not, then we may take an irreducible closed subset $V \subsetneq T$ such that the generic point of $V$ is in $U$ and again replace $T$ with $V$. 
        Since $\dim V < \dim T <\infty$, after a finite number of such replacements we obtain that $\dim T = 1$.
        Thus $T$ is a spectrum of a Noetherian complete local domain with
        quotient field $\kappa$ and we may take the identity $T'=T$ to finishes the proof of the first part.

        Suppose now that $\kappa$ is algebraically closed. 
        We may assume that $T = \Spec A$, is as above. 
        Let $\mathfrak{m}$ be the maximal ideal of $A$.
        The normalisation $\tilde{A}$ of $A$ is a finite $A$-module, see
        e.g.~\cite[Appendix 1, Corollary~2]{nagata_algebraic_geo_over_Dedekind_two}.
        Then $\tilde{T} = \Spec \tilde{A} \to T$ is finite and dominating, thus it is onto.
        Since $\kappa$ is algebraically closed, any point in the preimage of the
        special point is a $\kappa$-rational point, thus $\tilde{T}\to T$ satisfies claim of the lemma.
        Now $\tilde{A}$ is a one-dimensional normal Noetherian domain which is a finite $A$-module.
        By \cite[Corollary~7.6]{Eisenbud} the algebra $\tilde{A}$ 
          is a finite product $\tilde{A} = \prod B_i$,
          where each $B_i$ is local and complete, and the residue field of
          each $B_i$ is $\kappa$.
        From the first part of the proof it follows, that
        we may replace $\tilde{A}$ by one of the factors $B_i$, 
        which is a one-dimensional
        Noetherian normal local complete domain with quotient field $\kappa$.
        Thus $B_i$ is regular by Serre's criterion~\cite[Theorem~11.5]{Eisenbud}, so from
        the Cohen Structure Theorem~\cite[Theorem~7.7]{Eisenbud}, it follows
        that $B_i$ is isomorphic to $\kappa[[x]]$.	
    \end{proof}

 \begin{example}
    Suppose $\kk=\RR$ and consider the $\RR$-algebra
    $A:=\RR \oplus x \CC[[x]] \subset \tilde{A} := \CC[[x]]$.
    Then the normalisation of $A$ is $\tilde{A}$, which has no $\RR$-points. 
    This illustrates that in the proof of the final part of
    Lemma~\ref{lem_finding_a_curve_through_point_and_open_subset}
    the assumption that $\kappa$ is algebraically closed is necessary.
\end{example}

    The following Theorem~\ref{ref:goodbaseofsmoothing:thm} is the key
    result of this section. It allows us to shrink the base of smoothing to an
    algebraic analogue of a small one-dimensional disk.
    \begin{theorem}\label{ref:goodbaseofsmoothing:thm}
        Let $\famil \to T$ be an abstract or embedded smoothing of some scheme.
        Then, after a~base change, we may assume that $T  \simeq \Spec A$, where $A$
        is a one-dimensional Noetherian complete local domain with quotient field $\kk$.

        If $\kk$ is algebraically closed, we may furthermore
        assume that $A = \kk[[x]]$.
    \end{theorem}
    \begin{proof}
        Since $\famil \to T$ is finite, the relative differentials sheaf is coherent
        over $T$, so that there exists an open neighbourhood $U$ of the generic point $\eta$
        such that $\famil _{u}$ is smooth for any $u\in U$.
        Thus the claim is a combination of Lemmas~\ref{ref:basechange:lem} and \ref{lem_finding_a_curve_through_point_and_open_subset}.
    \end{proof}

    Now we recall the correspondence between the smoothings of $R$ and of its
    connected components. Intuitively, by Theorem~\ref{ref:goodbaseofsmoothing:thm} we
    may choose such a small basis of the smoothings, that smoothings of connected components are
    connected components of the smoothing.
    \begin{proposition}\label{ref:smoothingcomponents:prop}
        Let $R = R_1\sqcup R_2\sqcup  \ldots  \sqcup R_k$ be a~finite scheme.
        If $(\famil _i, R_i)\to (T, t)$ are abstract smoothings of $R_i$ over some
        base $T$, then $\famil  =
        \bigsqcup \famil _i\to T$ is an abstract~smoothing of $R$.

        Conversely, let $(\famil , R)\to (T, t)$ be an abstract smoothing of $R$
        over $T=\Spec A$, where $A = (A,\mathfrak{m}, \kk)$ is a~local
        complete $\kk$-algebra. Then $\famil  = \famil _1\sqcup  \ldots \sqcup \famil _k$, where
        $(\famil _i, R_i)\to (T, t)$ is an abstract smoothing of $R_i$.
    \end{proposition}
    \begin{proof}
        The first claim is clear, since we may check that $\famil  = \bigsqcup \famil _i$
        is flat and finite locally on connected components of $\famil $.
        Let $\eta$ be the generic point of $T$, then
        $\famil _{\eta} = \bigsqcup (\famil _i)_{\eta}$ is smooth over $\eta$ since
        $(\famil _i)_{\eta}$ are all smooth.

        \def\nn{\mathfrak{n}}
        \def\mm{\mathfrak{m}}
        For the second part, note that $\famil $ is affine by definition.
        Let $\famil  = \Spec B$, then
        $B$ is a finite $A$-module and $R = \Spec B/\mm B$.
        Let $\nn_i$ be the maximal ideals in $B$ containing $\mm$. They
        correspond bijectively to maximal ideals of $B/\mm B$ and thus to
        components of $R$. Namely, $R_i = (B/\mm B)_{\nn_i}$ for appropriate indexing of $\nn_i$.
        Since $A$ is complete Noetherian
        $\kk$-algebra, by~\cite[Theorem~7.2a,~Corollary~7.6]{Eisenbud} we get that $B =
        B_{\nn_1} \times  \ldots \times B_{\nn_n}$.
        Then $B_{\nn_i}$ is a flat $A$-module, as a localisation of $B$,
        and also a finite $A$-module, since it may be regarded as a quotient
        of $B$. The fiber of $B_{\nn_i}$ over the generic point of $\Spec A$ is a
        localisation of the fiber of $B$. Therefore $\Spec B_{\nn_i} \to \Spec
        A$ is a smoothing of $\Spec (B_{\nn_i})/\mm B_{\nn_i} = \Spec (B/\mm
        B)_{\nn_i} = R_i$.
    \end{proof}

    \begin{corollary}\label{ref:smoothingcomponentsresult:cor}
        Let $R = R_1\sqcup R_2\sqcup  \ldots  \sqcup R_k$ be a~finite scheme.
        Then $R$ is abstractly smoothable if and only if each $R_i$ is
        abstractly smoothable.
    \end{corollary}

    \begin{proof}
        If each $R_i$ is abstractly smoothable, then we may choose smoothings
          over the same base $T$, for instance by taking the product of the all bases of the individual smoothings.
        The claim follows from
        Proposition~\ref{ref:smoothingcomponents:prop}.

        Conversely, if $R$ is
        smoothable, then we may choose a smoothing over a one-dimensional
        Noetherian complete local domain by
        Theorem~\ref{ref:goodbaseofsmoothing:thm}. Again the result is implied by Proposition~\ref{ref:smoothingcomponents:prop}.
    \end{proof}

\section{Comparing abstract and embedded
smoothings}\label{sec:smoothings_abstract_embedded}

    Now we will compare the notion of abstract smoothability and embedded
    smoothability of a scheme $R$ and prove
    Theorem~\ref{thm_equivalence_of_abstract_and_embedded_smoothings}.
    We begin with a technical lemma.
    \begin{lemma}\label{ref:closedimmersions:lem}
        Let $(\DA, \mm, \kk)$ be a local $\kk$-algebra and $T = \Spec A$ with a
        $\kk$-rational point $t$ corresponding to $\mm$.
        Let $\famil $ be a scheme with a unique closed point and $\famil \to T$ be a finite flat morphism.
        Let $\amb$ be a separated scheme and $f\colon \famil \to \amb \times T$ be a
        morphism such that the following diagram is commutative:
        \[
          \begin{tikzcd}
             \famil  \arrow{rr}{}\arrow{dr}{} & & \amb \times T\arrow{dl}{}\\
             & T
          \end{tikzcd}
        \]
        If $f_t:\famil _t\to \amb$ is a closed
        immersion, then $f$ is also a closed immersion.
    \end{lemma}

    \begin{proof}
        Since $\amb\times T\to T$ is separated and $\famil \to T$ is finite, from the
        cancellation property (\cite[Theorem~10.1.19]{Vakil_foag} or \cite[Exercise~II.4.8]{hartshorne}) it follows that
        $f\colon \famil \to \amb\times T$ is
        finite, thus the image of $\famil $ in $\amb\times T$ is closed.
        Then it is enough to prove that $\famil \to \amb\times T$ is a~locally closed
        immersion.

        Let $U\subseteq \amb$ be an open affine neighbourhood of $f_t(p)$, where
        $p$ is the unique
        closed point of $\famil $. Since the preimage of $U \times T$ in $\famil $ is open
        and contains $p$, the morphism $\famil \to \amb\times T$ factors through
        $U\times T$.
        We claim that $\famil \to U\times T$ is a~closed immersion.
        Note that it is a morphism of affine schemes. 
        Let $B$, $C$ denote the coordinate rings of $\famil $ and $U\times T$, respectively.
        Then the morphism of schemes $\famil \to \amb\times T$ corresponds to a~morphism of $A$-algebras $C\to B$. 
        Since the base change $A\to A/\mm$
        induces an isomorphism $C/\mm C\to B/\mm B$, we have $B = \mm B + C$,
        thus $C\to B$ is onto by Nakayama Lemma and the fact that $B$ is
        a finite $A$-module. Hence, the morphism $f\colon \famil \to U \times T\to \amb\times T$ is
        a~locally closed immersion.
    \end{proof}

    The following Theorem~\ref{ref:abstractvsembedded:thm} together with its
    immediate Corollary~\ref{ref:smoothingeverytwhere:cor} is a generalization
    of \cite[Lemma~2.2]{casnati_notari_irreducibility_Gorenstein_degree_9} and
    \cite[Prop.~2.1]{bubu2010}.  Similar ideas are mentioned in
    \cite[Lemma~4.1]{cartwright_erman_velasco_viray_Hilb8} and in
    \cite[p.~4]{artin_deform_of_sings}.

    The theorem uses the notion of \emph{formal smoothness}, see \cite[Def~17.1.1]{ega4-4}.
    A scheme $X$ is \emph{formally smooth} if for every affine scheme $Y$
    and every closed subscheme $Y_0 \subset Y$ defined by a nilpotent ideal of
    $\OO_Y$,
    every morphism $Y_0\to X$ extends to a morphism $Y\to X$.

    \begin{theorem}[Abstract smoothing versus embedded smoothing]\label{ref:abstractvsembedded:thm}
        Let $R$ be a~finite scheme over $\kk$ which is embedded into
        a~formally smooth, separated scheme $X$.
        Then $R$ is smoothable in $X$ if and only if it is abstractly
        smoothable.
    \end{theorem}

    \begin{proof}
        Clearly from definition, if $R$ is smoothable in $X$, then it is
        abstractly smoothable. It remains to prove the other implication.

        Let us consider first the case when $R$ is irreducible.
        Let $(\famil , R)\to (T, t)$ be an abstract smoothing of $R$. Using
        Theorem~\ref{ref:goodbaseofsmoothing:thm} we may assume that $T$ is
        a~spectrum of a~complete local ring $(A, \mm, \kk)$. Since $\famil \to T$ is
        finite, $\famil  \simeq \Spec B$, where $B$ is a~finite $A$-algebra. In
        particular, since $B$ is irreducible, it is complete
        by~\cite[Corollary~7.6]{Eisenbud},
          and by~\cite[Theorem~7.2a]{Eisenbud}, the algebra $B$ is the inverse limit of Artinian $\kk$-algebras $B/(\mm B)^n$, where $n\in \mathbb{N}$.

        By definition of $X$ being formally smooth, the morphism $R = \Spec
        B/\mm B \to X$ lifts to $\Spec B/(\mm B)^2\to X$
          and subsequently $\Spec B/(\mm B)^n\to X$ lifts to $\Spec B/(\mm B)^{n+1}\to X$ for every $n \in \mathbb{N}$.
        Together these morphisms give a morphism $\famil  = \Spec B \to X$, which in
        turn gives rise to a~morphism of $T$-schemes $\famil \to X\times
        T$. This morphism is a closed immersion by
        Lemma~\ref{ref:closedimmersions:lem}.
        This finishes the proof in the case of
        irreducible $R$.

        Now consider a not necessarily irreducible $R$. Let $R = R_1 \sqcup
        \ldots  \sqcup R_k$ be the decomposition into irreducible (or
        connected) components. By
        Proposition~\ref{ref:smoothingcomponents:prop}, the smoothing $\famil $
        decomposes as $\famil  = \famil _1 \sqcup  \ldots  \sqcup \famil _k$, where $(\famil _i,
        R_i)\to (T, t)$ are smoothings of $R_i$. The schemes $R_i$ are
        irreducible, so by the previous case, these smoothings give
        rise to embedded smoothings $\famil _i \subseteq X\times T$. 
        In particular, each subscheme $\famil _i$ is closed.
        Moreover, the images of closed points of $\famil _i$ are pairwise different in $X\times T$, 
          thus $\famil _i$ are pairwise disjoint and 
          we get an embedding of $\famil  = \famil _1\sqcup  \ldots  \sqcup \famil _k \subset X\times T$, which is the required embedded smoothing.
    \end{proof}

    \begin{corollary}\label{ref:smoothingeverytwhere:cor}
        Suppose that $R$ is a finite scheme and $X$ and $Y$ are two smooth
        separated schemes. 
        If $R$ can be embedded in $X$ and in $Y$, then $R$ is
        smoothable in $X$ if and only if $R$ is smoothable in $Y$.
    \end{corollary}
    \begin{proof}
        Follows directly from Theorem~\ref{ref:abstractvsembedded:thm}.
    \end{proof}

%    \begin{corollary}\label{ref:smoothingcomponentsembedded:cor}
%        Let $R = R_1\sqcup R_2\sqcup  \ldots  \sqcup R_k\subset X$ be a~finite subscheme of a scheme $X$.
%        Then $R$ is  smoothable in $X$ if and only if each $R_i$ is
%        smoothable in $X$.
%    \end{corollary}
%    \begin{proof}
%        Follows directly from Theorem~\ref{ref:abstractvsembedded:thm} and
%        Corollary~\ref{ref:smoothingcomponentsresult:cor}.
%    \end{proof}

    \begin{proof}[Proof of Theorem~\ref{thm_equivalence_of_abstract_and_embedded_smoothings}]
        Corollary~\ref{ref:smoothingcomponentsresult:cor} gives the
        equivalence of \ref{it:Rabssm} and \ref{it:Rconabssm}.
        Smooth variety is formally smooth and separated by definition, so
        Theorem~\ref{ref:abstractvsembedded:thm} implies equivalence of
        \ref{it:Rabssm} and \ref{it:Rembsm} as well as
        \ref{it:Rconabssm} and \ref{it:Rconembsm}.
    \end{proof}

    \section{Embedded smoothability depends only on singularity type}

    While the comparison between abstract and embedded smoothings
    given in Theorem~\ref{ref:abstractvsembedded:thm} above is
    satisfactory, it is natural to ask what is true without formal smoothness
    assumption. This assumption cannot be removed altogether: for a projective
    curve $C$ all its non-zero tangent
    vectors, regarded as $\Spec (\kk[\varepsilon]/\varepsilon^2) \subset C$, are
    smoothable in $C$ if and only if all tangent spaces are contained in
    tangent stars,
    see~\cite[Section~3.3]{jabu_ginensky_landsberg_Eisenbuds_conjecture}.
    However we will see that the formal smoothness assumption may be removed entirely if we
    have an appropriate morphism, see
    Corollary~\ref{ref:pushingsmoothings:cor}, and that the existence of
    smoothings depends only on the formal geometry of $X$ near the support of
    its subscheme, see
    Proposition~\ref{prop_smoothability_depends_only_on_sing_type}.

    \begin{corollary}\label{ref:pushingsmoothings:cor}
        Let $R$ be a finite scheme embedded in $X$ and smoothable in $X$.
        Let $Y$ be a separated scheme with a morphism $X\to Y$ which induces an isomorphism of $R$
           with its scheme-theoretic image $S \subseteq Y$. Then $R  \simeq S$
        is smoothable in $Y$.
    \end{corollary}
    \begin{proof}
        Let $\famil \subseteq X \times T\to T$ be an embedded smoothing of $X$ over
        a base $(T, t)$.  The morphism $X\to Y$ induces a morphism $\famil \to Y \times
        T$ which, over $t$, induces a closed embedding $R \subseteq Y$. Then
        we need to prove that $\famil \to Y \times T$ is a closed immersion.
        By Theorem~\ref{ref:goodbaseofsmoothing:thm} and
        Proposition~\ref{ref:smoothingcomponents:prop} we may reduce to the
        case when $R$ is irreducible.
        Then the claim follows from Lemma~\ref{ref:closedimmersions:lem}.
    \end{proof}

    Using Corollary~\ref{ref:pushingsmoothings:cor} we may strengthen
    Corollary~\ref{ref:smoothingeverytwhere:cor} a bit, obtaining a direct
    generalization of \cite[Proposition~2.1]{bubu2010}.
    \begin{corollary}
        Let $X$ be a finite type, separated scheme and $R \subseteq X$ be a finite
        subscheme, supported in the smooth locus of $X$. If $R$ is
        abstractly smoothable, then $R$ is smoothable in $X$.
    \end{corollary}

    \begin{proof}
        Let $X^{sm}$ be the smooth locus of $X$. By
        Theorem~\ref{ref:abstractvsembedded:thm} the scheme $R$ is smoothable in
        $X^{sm}$ and by Corollary~\ref{ref:pushingsmoothings:cor} is it also
        smoothable in~$X$.
    \end{proof}

    We now show that possibility of smoothing a given $R$ inside $X$
    depends only on $R$ and the formally local structure of $X$ near $R$. This
    is the strongest result in this direction we could hope for; it implies
    that smoothability depends only on Zariski-local or standard \'etale local
    neighbourhoods of $R$ in $X$.

    \begin{proposition}\label{prop_smoothability_depends_only_on_sing_type}
        Let $X$ be a separated scheme and $R \subset X$ be a finite scheme, supported at points $x_1, \ldots
        ,x_k$ of $X$. Then $R$ is smoothable in $X$ if
        and only if $R$ is smoothable in $\bigsqcup \Spec \hat{\OO}_{X, x_i}$.
    \end{proposition}
    \begin{proof}
        \def\gotn{\mathfrak{n}}%
        \def\YY{{\famil'}}%
%        Using Corollary~\ref{ref:smoothingcomponentsembedded:cor} we reduce
%        to the case of irreducible $R$. Let $\{x\} = \Supp R$.
%        We are to show that $R$ is smoothable in $\bigsqcup \Spec
%        \hat{\OO}_{X, x_i}$ if and
%        only if it is smoothable in $X$.
        The ``only if'' part follows from
        Corollary~\ref{ref:pushingsmoothings:cor} applied to the map
        $\bigsqcup\Spec
        \hat{\OO}_{X, x_i} \to X$. We prove the ``if'' part, so we
        assume that $R$ is smoothable in $X$.
        By Theorem~\ref{ref:goodbaseofsmoothing:thm} we may take
        a smoothing of $R$ over $T = \Spec A$ where $(A, \mm)$ is local and
        complete; this is a family $\famil  \subset X \times T$ cut out of
        $\OO_{X} \tensor_{\kk} A$ by an ideal sheaf $\mathcal{I}$.
        By Proposition~\ref{ref:smoothingcomponents:prop} we have $\famil =
        \bigsqcup \famil_i$ where $\famil_i$ is a smoothing of $R_{x_i}$.
        We now show that $\famil_i \to X \times T$ can be factorized as follows:
        \begin{equation}\label{eq:trimming}
            \famil_i \to \Spec \hat{\OO}_{X, x_i} \times T \to X \times T.
        \end{equation}
        Fix $i$. Let $x := x_i$, $\YY = \famil_i$ and $\gotn \subset \OO_X$ be the ideal sheaf
        of $y\in X$.
        Since $\YY \to T$ is finite, the algebra $H^0(\YY, \OO_\YY)$ is
        $\mm$-adically complete. Since $R$ is finite, say of degree $d$, we
        have $\gotn^d \OO_\YY \subset \mm \OO_\YY$.
        This means that each $\gotn \OO_\YY$-adic Cauchy sequence is also
        an $\mm \OO_\YY$-adic Cauchy sequence and hence has a unique limit in
        $\OO_\YY$.
        Thus the algebra $H^0(\YY, \OO_\YY)$ is complete in $\gotn \OO_\YY$-adic topology. By
        universal property of completion, the map $\YY \to X \times T$ factors
        through $\Spec \hat{\OO}_{X, y} \times T$. The map $\Spec
        \hat{\OO}_{X, x} \to X$ is separated, hence from~\eqref{eq:trimming}
        it follows that $\YY \to \Spec\hat{\OO}_{X, x} \times T$ is a closed
        immersion; this gives a deformation $\YY$ embedded into
        $\Spec \hat{\OO}_{X, x}$. Summing over all components we obtain the desired
        embedding $\famil \subset T \times \bigsqcup \Spec \hat{\OO}_{X, x_i}$.
    \end{proof}

    \begin{corollary}
        Let $X$ and $Y$ be two separated schemes and $x\in X$, $y\in Y$ be points with
        isomorphic completions of local rings; let
        $\varphi\colon \Spec\hat{\OO}_{X,x}\to \Spec\hat{\OO}_{Y,y}$
        be an isomorphism.

        Suppose that $R$ is a finite irreducible scheme
        with embeddings $i_X:R \to X$ and $i_Y:R\to Y$ such that $i_X(R)$,
        $i_Y(R)$ are supported at $x$, $y$ respectively. Suppose that
        $\varphi$ induces an isomorphism of $i_X(R)$ and $i_Y(R)$.
        Then $R =i_X(R)$ is smoothable in $X$ if and only if $R=i_Y(R)$ smoothable in $Y$.
    \end{corollary}
    \begin{proof}
        By Proposition~\ref{prop_smoothability_depends_only_on_sing_type} the
        scheme $i_X(R)$ is smoothable in $X$ if and only if it is smoothable in
        $\hat{\OO}_{X,x}$. By assumption $i_X(R) \subset \Spec
        \hat{\OO}_{X, x}$ is isomorphic to $i_Y(R) \subset \Spec
        \hat{\OO}_{Y, y}$ via $\varphi$. By
        Proposition~\ref{prop_smoothability_depends_only_on_sing_type} again,
        $i_Y(R)$ is smoothable in $\hat{\OO}_{Y, y}$ if and only if it is
        smoothable in $Y$.
    \end{proof}

\section{Comparing embedded smoothings and the Hilbert scheme}

    We now compare the abstract notion of smoothability of a finite scheme $R
    \subset X$ to the geometry of Hilbert scheme around $[R]\in \Hilbr{X}$.
    This enables direct investigation of smoothability: we prove that
    it is faithfully preserved under field extension, that for a
    given family the set of smoothable fibers is closed and, in
    Section~\ref{ssec:examplesnonsmoothable}, we give examples of nonsmoothable
    schemes.

    Fix a scheme $X$ of finite type over $\kk$ and such that the Hilbert
    scheme $\Hilbr{X}$ exists.
    \begin{proposition}\label{ref:embeddedvsHilbert:prop}
        Let $X$ be a scheme such that $\Hilbr{X}$ exists and $R \subset X$ be a finite subscheme of degree $r$. The
        following are equivalent
        \begin{enumerate}
            \item\label{it:1smooth} $R$ is smoothable in $X$,
            \item\label{it:2smooth} $[R]\in \Hilbsmr{X}$.
        \end{enumerate}
    \end{proposition}
    \begin{proof}
        To show \ref{it:1smooth}$\implies$\ref{it:2smooth}, pick an embedded smoothing of $R$ in $X$, 
          which is a family $\famil  \subset X \times T$ flat over an irreducible base $T$,
        such that a fiber over a $\kk$-rational point $t\in T$ is $\famil _{t}
        = R$.
        In particular, the degree of $\famil  \to T$ is $r$, 
    and hence it gives a map $\varphi\colon T\to \Hilbr{X}$. 
        The base $T$ is irreducible and the fiber of $\famil \to T$ over the generic point $\eta \in T$ is smooth.
        Thus the image of the generic point  $\varphi(\eta)$ is contained in
        $\Hilbzeror{X}$, and the image of any point of $T$ is contained in its
        closure $\Hilbsmr{X}$.  In particular, $\varphi(t) = [R] \in
        \Hilbsmr{X}$.

        To show \ref{it:2smooth}$\implies$\ref{it:1smooth} pick an irreducible
        component $T$ of $\Hilbsmr{X}$ containing $[R]$ and let $\famil
        \subset X \times T$ be the restriction of the universal family
        $\UU_r$ to $T$.  The map $f\colon \famil \to T$ is flat and finite.
        Since $\Hilbsmr{X} = \overline{\Hilbzeror{X}}$ by
        its definition~\ref{ref:smoothablecomponent:def},
        there exists an open dense $U$ such
        that $f\colon f^{-1}(U)\to U$ is smooth.  Hence in particular the fiber over
        the generic point of $T$ is smooth, so $\famil \to T$ gives an
        embedded smoothing of $R$.
    \end{proof}

    The following corollary reduces the questions of smoothability to
    schemes over $\kk = \kkbar$.
    \begin{corollary}\label{ref:basechangesmoothings:cor}
        Let $R$ be a finite scheme over $\kk$ and $\kk \subset \KK$ be a field extension. 
        Then $R$ is smoothable if and only if the $\KK$-scheme
        \[
            R_{\KK} = R\times \Spec \KK
        \]
        is smoothable.
    \end{corollary}
    \begin{proof}
        Suppose $R$ is smoothable and take its smoothing $(\famil , R)\to (T,
        t)$. Then $(\famil \times_{\kk} \KK, R_{\KK})\to (T\times_{\kk} \KK,
        t)$ is a smoothing of $R_{\KK}$.
%        then $\famil _t  \simeq R$.  The point $t$ is $\kk$-rational, so it
%        gives a $\KK$-rational point $t_{\KK} \in T_{\KK} = T \times \Spec
%        \KK$ just by a product of the composition $\Spec\KK\to \Spec\kk\to T$
%        and the identity $\Spec \KK \to \Spec \KK$.  Moreover $\famil
%        _{t_{\KK}} = \famil _t \times \Spec \KK = R_{\KK}$.
        Suppose now $R_{\KK}$ is smoothable as a scheme over $\KK$.
        Since $R$ is finite, we can embed $R$ into an affine space $\AA^N_{\kk}$.
        Since $\AA^N_{\kk}$ is smooth, by Theorem~\ref{ref:abstractvsembedded:thm} the scheme $R_{\KK}$ is smoothable
        in $\AA^N_{\KK} = \AA^N_{\kk} \times_{\kk} \KK$.
        By Proposition~\ref{ref:embeddedvsHilbert:prop} this means that the point $[R_{\KK}]$ lies in
        $\Hilbsmr{\AA^N_{\KK}/\KK}= \reduced{\Hilbsmr{\AA^N_{\kk}} \times \Spec \KK}$.
        The image of the projection of this point to $\Hilbsmr{\AA^N_{\kk}}$ is equal to $[R]$.
        Using Proposition~\ref{ref:embeddedvsHilbert:prop} again, we get that $R$ is smoothable in $X$.
    \end{proof}

    By the above comparison we can also translate known results about the Hilbert
    scheme to the language of smoothability.
    \begin{corollary}\label{ref:tmplowcodim:cor}
        Let $R \subset \mathbb{A}^2$ be a finite subscheme. Then $R$ is
        smoothable. Let $R' \subset \mathbb{A}^3$ be a finite Gorenstein
        subscheme. Then $R'$ is smoothable.
    \end{corollary}
    \begin{proof}
        From Theorem~\ref{ref:fogarty:thm} and
        Theorem~\ref{ref:kleppe:thm} it follows that $\Hilbr{\mathbb{A}^2} =
        \Hilbsmr{\mathbb{A}^2}$ and $\HilbGorr{\mathbb{A}^3} =
        \HilbGorsmr{\mathbb{A}^3}$ for all $r$, hence the result follows from
        Proposition~\ref{ref:embeddedvsHilbert:prop}.
    \end{proof}
    \begin{corollary}\label{ref:lowcodimensionsmoothability:cor}
        Let $(\DA, \mm, \kk)$ be a finite local algebra with $\dimk \mm/\mm^2
        \leq 2$. Then $\DA$ is smoothable. Let $(\DA', \mm', \kk)$ be a finite
        local Gorenstein algebra with $\dimk \mm'/\mm'^2
        \leq 3$. Then $\DA'$ is smoothable.
    \end{corollary}
    \begin{proof}
        Finite schemes $\Spec \DA$ and $\Spec \DA'$ are embeddable in
        $\mathbb{A}^2$ and $\mathbb{A}^3$, respectively, and the result
        follows from Corollary~\ref{ref:tmplowcodim:cor}.
    \end{proof}
    Proposition~\ref{ref:embeddedvsHilbert:prop} implies that smoothability is a closed property.
    \begin{proposition}\label{ref:smoothabilityisclosed:prop}
        Let $\pi\colon \ccX\to T$ be a (finite flat) family. Then the set
        \[
            T^{sm} := \{t\in T\ |\ \ccX_{t}\mbox{ smoothable}\}
        \]
        is closed in $T$.
    \end{proposition}
    \begin{proof}
        It is enough to find an open cover $\{U_i\}$ of $T$ such that $T^{sm}\cap
        U_i$ is closed in $U_i$. Let $r = \deg \pi$. For each point $x\in T$
        the fiber $\ccX_t$ embeds into
        $\AA^r_{\kappa(x)}$, so there is a neighbourhood $U_i$ of $x$ such
        that $\pi^{-1}(U_i)$ embeds into $\mathbb{A}^{r} \times U_i$. These embeddings induce maps
        $\varphi_i:U_i\to
        \Hilbr{\mathbb{A}^r}$ and $T^{sm} \cap U_i =
        \varphi_i^{-1}(\Hilbsmr{\mathbb{A}^r})$ are closed.
    \end{proof}

    \section{Smoothings over rational curves}\label{ssec:kollar}

    \newcommand{\projambient}{\mathbb{P}^n_{\kk}}%
    We now show that every finite smoothable scheme $R$ over an algebraically
    closed $\kk$ of
    characteristic zero has a smoothing over a $\mathbb{P}^1$.
    In fact it has
    such \emph{embedded} smoothings for all embeddings
    $R$ into $\projambient$ or other smooth, projective, rational variety.
    This is because $\Hilbsmr{\projambient}$ is a rational variety, hence it
    has enough rational curves. Note that in this section we use
    $\Hilbsmr{\projambient}$ rather than $\Hilbsmr{\AA^n}$, because we invoke
    Theorem~\ref{ref:kollarresult:thm} by Koll\'ar, which applies to
    projective (or proper) schemes.

    \begin{lemma}
        The variety $\Hilbsmr{\projambient}$ is rational.
        \label{ref:rationalsmoothable:lem}
    \end{lemma}
    \begin{proof}
        Recall that $\Hilbsmr{\projambient}$ is a closure of
        $((\projambient)^{\times r}\setminus\Delta)/\Sigma_r$, where $\Delta$
        is the sum of all diagonals $(x_i = x_j)_{i\neq j}\subset
        (\projambient)^{\times r}$ and $\Sigma_r$ acts on
        $(\projambient)^{\times r}$ by permutations. This already proves that
        it is uni-rational. The fact that is is rational is a result of
        Mattuck~\cite[Theorem, p.~764]{Mattuck__rationality}.
    \end{proof}

    The following is a deep result by Koll\'ar.
    \begin{theorem}
        Let $\kk = \kkbar$ and $\chark = 0$.
        Through any tuple of points of a proper, rationally chain connected
        variety $X$ over $\kk$ there is a rational curve.
        \label{ref:kollarresult:thm}
    \end{theorem}
    \begin{proof}
        By replacing $X$ with a resolution of singularities we may assume $X$
        is smooth. In that case $X$ is rationally connected
        by~\cite[Theorem~3.10]{kollar_book_rational_curves} and separably
        rationally connected by~\cite[Proposition~3.3]{kollar_book_rational_curves}. The result for smooth, separably
        rationally connected varieties
        is~\cite[Theorem~3.9]{kollar_book_rational_curves}.
    \end{proof}

    \begin{corollary}\label{ref:rationalcurvesonsmoothable:cor}
        Let $\kk = \kkbar$ and $\chark = 0$.
        Though any tuple of points on $\Hilbsmr{\projambient}$ there is a
        rational curve.
    \end{corollary}
    \begin{proof}
        The scheme $\Hilbsmr{\projambient}$ is a projective variety by
        Proposition~\ref{ref:smoothablecomponentdescription:prop} and
        Theorem~\ref{ref:representability:thm}. It is also rational by
        Lemma~\ref{ref:rationalsmoothable:lem}, so in particular rationally
        chain connected, so the claim follows from
        Theorem~\ref{ref:kollarresult:thm}.
    \end{proof}
    \begin{corollary}\label{ref:smoothingoverP1:cor}
        Let $\kk = \kkbar$ and $\chark = 0$.
        Every finite smoothable scheme $R$ over $\kk$ has a smoothing over
        $\mathbb{P}^1$.
    \end{corollary}
    \begin{proof}
        For $n$ large enough we have an embedding $R \subset \projambient$.
        Since $R$ is smoothable, by
        Theorem~\ref{thm_equivalence_of_abstract_and_embedded_smoothings} it
        is smoothable in $\projambient$, so it corresponds to a point $[R]\in
        \Hilbsmr{\projambient}$. Take any point $[R']\in\Hilbsmr{\projambient}$
        corresponding to a smooth $R' \subset \projambient$. By
        Corollary~\ref{ref:rationalcurvesonsmoothable:cor} there exists a rational
        curve $C\subset \Hilbsmr{\projambient}$ through $[R]$ and $[R']$. Its
        normalization $\tilde{C} \simeq \mathbb{P}^1_{\kk}$ comes with a morphism
        to
        $\Hilbsmr{\projambient}$. The pullback
        $\UU_{|\tilde{C}} \to \tilde{C}$ of the universal family via
        $\tilde{C}\to\Hilbsmr{X}$ is the required smoothing of $R$ as in
        Proposition~\ref{ref:embeddedvsHilbert:prop}.
    \end{proof}

    Corollary~\ref{ref:smoothingoverP1:cor} gives the following affine
    version, which is stronger version of
    Theorem~\ref{ref:goodbaseofsmoothing:thm}.
    \begin{corollary}
        Let $\kk = \kkbar$ and $\chark = 0$. Every smoothable $\kk$-scheme has a
        smoothing over $\Spec\kk[t]$.
    \end{corollary}
    \begin{proof}
        Restrict the smoothing given by
        Corollary~\ref{ref:smoothingoverP1:cor} to $\mathbb{A}^1 =
        \mathbb{P}^1 \setminus\{pt\}$.
    \end{proof}

    \begin{remark}
        Roggero and Lella
        proved~\cite[Theorem~C]{lella_roggero_Rational_Components}
        that each smooth component of each $\Hilbr{\projambient}$ is rational. As
        pointed in Problem~\ref{prob:nonrational}, no nonrational component of a
        Hilbert scheme of points $\Hilbr{\projambient}$ is known.
    \end{remark}

    \section{Examples of nonsmoothable finite schemes}\label{ssec:examplesnonsmoothable}

    A finite scheme is nonsmoothable if and only if one of its components is
    nonsmoothable by Theorem~\ref{thm_equivalence_of_abstract_and_embedded_smoothings}.
    Therefore below we consider only irreducible nonsmoothable schemes.
    Known examples of irreducible nonsmoothable schemes fall into two categories.  Both exploit the
    fact that $\Hilbsmr{\mathbb{A}^n}$ is quasi-projective and irreducible of
    dimension $rn$, see Proposition~\ref{ref:smoothablecomponentdescription:prop}.

    First, there are
    schemes with \emph{small tangent space}. Indeed, if a degree $r$ subscheme $R\subset
    \mathbb{A}^n$ has $\dim \Dtangspace{\Hilbr{\mathbb{A}^n}, [R]} < nr$, then
    $[R]\not\in \Hilbsmr{\mathbb{A}^n}$, so $R$ is nonsmoothable
    by~Proposition~\ref{ref:embeddedvsHilbert:prop}.

    \begin{example}[(1, 4, 3),
        {\cite{cartwright_erman_velasco_viray_Hilb8}}]\label{ex:143}
        In this example we consider elements of $\HH =
        \Hilbarged{8}{\mathbb{A}^4}$.
        Let $F_1, F_2, F_3\in \kdp[x_1, \ldots, x_4]$ be general quadrics.
        Let $\DS = \kk[\Dx_1, \ldots, \Dx_4]$, $I = \Ann(F_1, F_2, F_3)$ and
        $\DA = \DS/I$. Since $F_\bullet$ are general, we have $H_{\DA} = (1,
        4, 3)$ and $I$ is generated by $7$ quadrics.

        Let $R = \Spec \DA \subset \mathbb{A}^4$. We claim that
        \begin{equation}\label{eq:tmp143}
            \dim
            \Dtangspace{\HH, [R]} = 25 < 4\cdot 8.
        \end{equation}
        To prove~\eqref{eq:tmp143}, first note that $\Homthree{}{I}{\DA}$ is
        graded, concentrated in degrees $0, -1, -2$. The subspace
        $\Homthree{}{I}{\DA}_0$ corresponds to $\DS$-homomorphisms $I\to \DA$ induced
        by $I_2\to \DA_2$, so that $\dim \Homthree{}{I}{\DA}_0 = 7\cdot 3 = 21$.

        For every linear $\ell\in \kdp[x_1, \ldots ,x_4]$ we have
        a differentiation
        $\partial_{\ell}\in \Homthree{}{I}{\DA}_{-1}$. These differentiations
        span a space
        of dimension four. Therefore, $\dim\Dtangspace{\HH, [R]}
        \geq 25$.

        Let $F^{\circ}_1 = x_1x_3$, $F^{\circ}_2 = x_2x_4$, $F^{\circ}_3 = x_1x_4 -x_2x_3$. Let
        $R^{\circ} = \Spec\Apolar{F^{\circ}_\bullet}$. Then $\dim \Dtangspace{\HH, R^{\circ}} = 25$
        for every $\kk$, as proven
        in~\cite[Proposition~5.1]{cartwright_erman_velasco_viray_Hilb8}, so
        that $R^{\circ}$ is nonsmoothable. By semicontinuity,~\eqref{eq:tmp143} holds
        also for general $F_{\bullet}$.
        See~\cite[Theorem~1.3]{cartwright_erman_velasco_viray_Hilb8} for a precise
        necessary and sufficient condition for smoothability of
        $\Apolar{F_{\bullet}}$ in this case.

        Consider the open subset $V \subset \Grass(3, \DP_2)$ parameterizing
        triples $F_{\bullet}$ of quadrics with $H_{\Apolar{F_{\bullet}}} = (1,
        4, 3)$. Similarly to Proposition~\ref{ref:flatfamiliesforconstructible:prop}, it
        gives a map $\varphi\colon V\to \HH$ with $3\cdot 7$-dimensional image. Each
        subscheme in $\varphi(V)$ is supported at the origin of
        $\mathbb{A}^4$. Adding schemes with translated support, we obtain a
        $25$-dimensional family $\mathbb{A}^4 \times \varphi(V) \subset \HH$.
        By~\eqref{eq:tmp143} we conclude that the closure of this family is a
        component of $\HH$.
    \end{example}

    \begin{example}[$(1, d, e)$, $3 \leq e \leq \frac{(d-1)(d-2)}{6}+2$]
        Consider $\kk = \kkbar$ of characteristic zero. Fix $d \geq 4$ and $3\leq e \leq
        \frac{(d-1)(d-2)}{6}+2$. Consider a general tuple $F_{\bullet}$ of $e$ quadrics in
        $\kdp[x_1, \ldots ,x_d]$. Then $\Spec\Apolar{F_{\bullet}}$ has Hilbert
        function $(1, d, e)$ and is nonsmoothable
        by~\cite[Theorem~2]{Shafarevich_Deformations_of_1de}. In fact,
        such $\Spec \Apolar{F_{\bullet}}$ together with translations form
        an open subset of a component of $\Hilbarged{1+d+e}{\mathbb{A}^d}$.
        Note that for $d = 4$ the only possibility is $e = 3$ and we obtain
        the example $(1, 4, 3)$.
    \end{example}

    \begin{example}[(1, 6, 6, 1), Gorenstein,
        \cite{emsalem_iarrobino_small_tangent_space,
        jelisiejew_VSP}]\label{ex:1661}
        In this example we consider elements of $\HH =
        \HilbGorarged{14}{\mathbb{A}^6}$.
        Let $F\in \kdp[x_1, \ldots ,x_6]$ be a general polynomial of degree
        three, so that $H_{\Apolar{F}} = (1, 6, 6, 1)$. Since $\Apolar{F}
        \simeq \Apolar{F_3}$ by Corollary~\ref{ref:eliasrossi:cor}, we assume
        that $F=  F_3$ is homogeneous. Using generality once more, we assume
        that $I = \Ann(F)$ is generated by $15$ quadrics. Computing
        $H_{\Apolar{F_3}}(3) = 1$ by means of resolution, we see that there
        are exactly $6\cdot 15 - \binom{6+2}{3} + 1 =35$ linear syzygies.
        Let $\DS = \kk[\Dx_1, \ldots ,\Dx_6]$, $\DA = \DS/I$ and $R = \Spec
        \DA$.  We claim
        that
        \begin{equation}\label{eq:tmp1661}
            \dim \Dtangspace{\HH, [R]} = 76 <
            6\cdot 14.
        \end{equation}
        First, we prove the lower bound. The space $\Homthree{}{I}{\DA}$ is graded
        with $\dim\Homthree{}{I}{\DA}_{1} = 15$ and
        $\Homthree{\DS}{I}{\DA}_0 \subset \Homthree{\kk}{I_2}{\DA_2}$ is
        cut out of a $(15\cdot 6)$-dimensional space by $35$-linear syzygies
        so $\dim \Homthree{\DS}{I}{\DA}_0 \geq 55$.
        There is a $6$-dimensional space of partial derivatives, so $\dim
        \Homthree{\DS}{I}{\DA}_{-1} \geq 6$. Thus $\dim\Dtangspace{\HH, [R]}
        \geq 76$.

        As in the previous example, to prove~\eqref{eq:tmp1661} it is enough
        to find an example with $76$-dimensional tangent space.
        For $\chark = 2$ the scheme
        \[R^{\circ} = \Spec \Apolar{{x}_{1}
            {x}_{2} {x}_{3}+{x}_{1} \DPel{x_{4}}{2}+\DPel{x_{1}}{2} {x}_{5}+{x}_{2}
        {x}_{3} {x}_{5}+{x}_{2} {x}_{4} {x}_{6}+{x}_{3} {x}_{5} {x}_{6}+{x}_{2}
    \DPel{x_{6}}{2}}\]
        satisfies $\dim \Dtangspace{\HH, [R^{\circ}]} = 76$.
        Let $\chark \neq 2$ and
        \[
            R^{\circ} = \Spec\Apolar{x_1x_2x_4 - x_1\DPel{x_5}{2} + x_2\DPel{x_3}{2} + x_3x_5x_6 + x_4\DPel{x_6}{2}}.
        \]
        Let $I = I(R)$. A direct check shows that $H_{\DS/I^2} = (1, 6, 21,
        56, 6)$ for all $\kk$ with $\chark \neq 2$, compare~\cite[Lemma~23]{jelisiejew_VSP}. Then
        $\dim \Dtangspace{\HH, [R^{\circ}]} =
        76$, by Example~\ref{ex:tangentforGorenstein}.

        Consider the open subset $V \subset \DP_{\leq 3}$ parameterizing
        cubics with Hilbert function $(1, 6, 6, 1)$. By
        Proposition~\ref{ref:flatfamiliesforconstructible:prop} and
        Proposition~\ref{ref:apolarrepresentability:prop} we obtain an
        injective morphism from $V$ to a bundle of rank $14$ over
        $\HilbGorarged{14}{\AA^6}$, so also a morphism
        $\varphi\colon V\to \HilbGorarged{14}{\AA^6}$,
        whose fibers are at most $14$-dimensional. Hence $\dim \varphi(V) \geq
        \binom{9}{3} - 14 = 84 - 14 = 70$. All points in $\varphi(V)$
        correspond to subschemes supported at the origin. By adding isomorphic
        subschemes supported elsewhere, we obtain $76$-dimensional family
        $\varphi(V) \times \mathbb{A}^6 \subset \HilbGorarged{14}{\AA^6}$.
        By~\eqref{eq:tmp1661} the closure of this family forms a component.
    \end{example}

    The above examples are the only known components of $\Hilbr{\AA^n}$, which
    have dimension less than $rn$, see Problem~\ref{prob:smallcomponents} and
    Problem~\ref{prob:lowerbounddimension}.

    Second, and much more commonly, there are \emph{large
    families}. If $\famil \subset \mathbb{A}^n \times V\to V$ is an embedded
    family with distinct fibers and $\dim V > rn$, then $V\to
    \Hilbr{\mathbb{A}^n}$ is injective on points so the image has
    dimension greater that $rn = \dim \Hilbsmr{\mathbb{A}^n}$ and so it
    is not contained in the smoothable component. This idea first
    appeared in~\cite{iarrobino_reducibility} and was later expanded in~\cite{iarrobino_compressed_artin}.
    Proposition~\ref{ref:apolarflatness:prop} and
    Proposition~\ref{ref:flatfamiliesforconstructible:prop} enable us to produce such
    families using large loci of forms.

    \begin{example}[$\mathbf{(1, n, n, 1)}$, $\mathbf{n\geq
        8}$]\label{ex:1nn1largenonsmoothable}
        \def\univ{\mathcal{F}}%
        Let $\kk$ be an arbitrary field.
        Let $n\geq 8$ and $\DP = \kdp[x_1,  \ldots , x_n]$, then
        \begin{equation}\label{eq:tmpdim}
            \dim \DP_{\leq 3} = \binom{n+3}{3} > (2n+2) + n\cdot (2n+2)
        \end{equation}
        Let $V \subset \DP_{\leq 3}$ be the open set parameterizing
        polynomials $F\in \DP_{\leq 3}$ with maximal Hilbert function, in
        particular $V$ is irreducible.
        By Proposition~\ref{ref:flatfamiliesforconstructible:prop} we have
        a finitely flatly embedded family $\{(f, \Apolar{f})\} \to V$.
        Consequently, by Proposition~\ref{ref:apolarrepresentability:prop} we
        obtain an injective morphism from $V$ to a bundle of rank $2n+2$ over
        $\HilbGorarged{2n+2}{\mathbb{A}^n}$, so also a morphism
        $\varphi\colon V\to \HilbGorarged{2n+2}{\mathbb{A}^n}$,
        whose fibers are at most $(2n+2)$-dimensional.

        Therefore the image
        $\varphi(V)$ has dimension at least $\binom{n+3}{3} - (2n+2)$, which
        by~\eqref{eq:tmpdim} is greater than $\dim
        \HilbGorsmarged{2n+2}{\mathbb{A}^n}$, so $\varphi(V) \not\subset
        \HilbGorsmarged{2n+2}{\mathbb{A}^n}$. As far as we know, there are no
        known explicit examples of nonsmoothable finite subschemes in $\varphi(V)$.
    \end{example}

    \begin{remark}\label{ref:1nn1:rmk}
        Consider Gorenstein local algebras with Hilbert function $(1, n, n,
        1)$ over $\kk$ of characteristic zero.  Examples~\ref{ex:1661},~\ref{ex:1nn1largenonsmoothable} give
        nonsmoothable examples of those for all $n\geq 8$ or $n = 6$.
        In contrast,~\cite[Theorem~A]{cjn13}, which we reproduce as
        Theorem~\ref{ref:cjnmain:thm}, asserts that for $n\leq 5$ such
        algebras are smoothable. Bertone, Cioffi and Roggero prove
        that such algebras are also smoothable for $n = 7$,
        see~\cite{bertone_cioffi_roggero_division_algorithm}.
    \end{remark}
    \begin{example}[Gorenstein subschemes of $\mathbb{A}^4$ of large
        degree]\label{ex:nonsmoothableinA4}
        We follow Example~\ref{ex:1nn1largenonsmoothable}.
        Fix an arbitrary field $\kk$ and consider polynomials of degree nine in
        $\kdp[x_1, \ldots ,x_4]$. Their space is of dimension $715$ and the
        apolar algebra of a general polynomial has degree $140$, so that
        arguing as in Example~\ref{ex:1nn1largenonsmoothable} we obtain a
        $(715-140)$-dimensional locus inside
        $\HilbGorarged{140}{\mathbb{A}^4}$. Since $715 - 140 = 575 > 560$, a
        general element of
        this locus corresponds to a nonsmoothable algebra and
        $\HilbGorarged{140}{\mathbb{A}^4}$ is reducible. This Example
        follows easily from~\cite{iarrobino_compressed_artin}, as described
        in~\cite[Proposition~6.2]{bubu2010} over $\kk = \mathbb{C}$.
    \end{example}

    \begin{example}[Gorenstein subschemes of $\mathbb{A}^5$ of large
        degree]\label{ex:nonsmoothableinA5}
        Analogously to Example~\ref{ex:nonsmoothableinA4} we may consider
        polynomials of degree five in five variables. Their space is $252$
        dimensional, a general apolar algebra has degree $42$ and so we obtain
        $210$-dimensional locus. But $210 = 42\cdot 5$ and this locus does not
        contain any smooth schemes, so it cannot
        be dense inside $\HilbGorsmarged{42}{\AA^5}$. Thus, this locus does not lie inside
        the smoothable component and the scheme
        $\HilbGorarged{42}{\mathbb{A}^5}$ is
        reducible. This example appeared in~\cite[Proposition~6.2]{bubu2010}.
    \end{example}

    \begin{example}[subschemes of $\mathbb{A}^3$ of degree
        $96$]\label{ex:nonsmoothable96}
        Let $\kk$ be an arbitrary field.  The scheme
        $\Hilbarged{96}{\mathbb{A}^3}$ is reducible,
        as shown in~\cite{iarrobino_reducibility}.
        Namely, the locus of irreducible subschemes corresponding to local algebras
        with Hilbert function $(1, 3, 6, 10, 15, 21, 28, 12)$ has dimension
        $12\cdot 24 + 3 = 291 > 3\cdot 96$, so it is not contained in the
        smoothable component.
    \end{example}

    \begin{example}[subschemes of $\mathbb{A}^3$ of degree
        $78$]\label{ex:largeinA3}
        Let $\kk$ be an arbitrary field.  The scheme
        $\Hilbarged{78}{\mathbb{A}^3}$ is reducible,
        as shown in~\cite[Example~4.3]{iarrobino_compressed_artin}.
        Namely, the locus of irreducible subschemes corresponding to local algebras
        with Hilbert function $(1, 3, 6, 10, 15, 21, 17, 5)$ has dimension
        $235$, while the dimension of smoothable component is $3\cdot 78 =
        234$. We refer the reader to the aforementioned paper for details.
    \end{example}

    For $\kk = \kkbar$, existence of nonsmoothable subschemes of $\mathbb{A}^n_{\kk}$ of degree $r$ is
    equivalent to reducibility of $\Hilbr{\mathbb{A}^n_{\kk}}$. Consider the
    following condition:
    \begin{enumerate}
            \renewcommand{\labelenumi}{$(\star)$}
            \renewcommand{\theenumi}{$(\star)$}
        \item%[$(\star)$]
            \label{item_condition_on_smoothability_of_all}
            Every finite subscheme of $\mathbb{A}^n$ having degree $r$ is
            smoothable.
    \end{enumerate}
    Example~\ref{ex:143} and Example~\ref{ex:largeinA3} show
    that~\ref{item_condition_on_smoothability_of_all} does not hold for
    $n=3$ and $r\geq 78$ or $n\geq 4$ and $r\geq 8$; regardless of $\kk$.
    Cartwright et.al.~prove in~\cite{cartwright_erman_velasco_viray_Hilb8}, under the assumption
    $\chark \neq 2, 3$, that \ref{item_condition_on_smoothability_of_all} holds for $r\leq 7$ and all $n$ and
    for $r = 8$ and $n= 3$.
    Borges dos Santos et.al.~\cite{dosSantos}  prove, for $\chark = 0$, that
    \ref{item_condition_on_smoothability_of_all} holds also for $r
    = 9, 10$ and $n = 3$. Douvropoulos et.al.~\cite{DJNT} prove, for $\chark =
    0$, that \ref{item_condition_on_smoothability_of_all} holds for $r = 11$ and $n=3$. We summarize what is known
    it Table~\ref{tab:nonsmoothable}, where we take $\chark =
    0$.

    \newcommand{\OKsign}{\tikz[x=1em, y=1em]\fill(0,.35) -- (.25,0) -- (1,.7) --
    (.25,.15) -- cycle;}
    \newcommand{\NOsign}{no}
    \begin{table}[h]
        \centering
        \begin{tabular}{l l l l}
                & $n\leq 2$ & $n=3$ & $n\geq 4$\\
            \midrule
            $\phantom{108\leq}r\leq 7$ &  \OKsign & \OKsign & \OKsign\\
            $\phantom{0}8\leq r\leq 11$&  \OKsign & \OKsign & \NOsign\\
            $12\leq r\leq 77$ & \OKsign & ? & \NOsign\\
            $78 \leq r\phantom{\leq 100}$ & \OKsign & \NOsign & \NOsign
        \end{tabular}
        \caption{Is $\Hilbr{\mathbb{A}^n_{\kk}}$ irreducible (for $\chark = 0$)?}
        \label{tab:nonsmoothable}
    \end{table}
    A similar analysis is conducted for Gorenstein locus. Here the main positive
    results come from~\cite{cjn13}, where the authors, in characteristic $\neq 2,
    3$, prove that $\HilbGorr{\mathbb{A}^n}$ is irreducible for $r\leq13$ and
    arbitrary $n$ and also for $r = 14$ and $n\leq 5$, see Theorem~\ref{ref:cjnmain:thm}.
    The negative results stem from Example~\ref{ex:1661}, which gives a
    nonsmoothable degree $14$ finite Gorenstein subscheme of $\AA^6_{\kk}$ for all
    fields $\kk$.
    See
    Table~\ref{tab:nonsmoothableGorenstein} for a summary of what is known.
    \begin{table}[h]
        \centering
        \begin{tabular}{l l l l l}
                & $n\leq 3$ & $n=4$ & $n=5$ & $n\geq 6$\\
            \midrule
            $\phantom{1101\leq}r\leq 13$ &  \OKsign & \OKsign & \OKsign & \OKsign\\
            $\phantom{1101\leq}r = 14$                    &  \OKsign & \OKsign & \OKsign & \NOsign\\
            $\phantom{1}15\leq r\leq 41$           &  \OKsign & ?       & ?        &
            \NOsign\\
            $\phantom{1}42\leq r\leq 139$          &   \OKsign & ? & \NOsign   &
            \NOsign\\
            $140 \leq r$                 & \OKsign & \NOsign & \NOsign &
            \NOsign\\
        \end{tabular}
        \caption{Is $\HilbGorr{\mathbb{A}^n}$ irreducible (for $\chark \neq 2, 3$)?}
        \label{tab:nonsmoothableGorenstein}
    \end{table}

    \section{Example of smoothings: one-dimensional torus
    limits}\label{ssec:klimits}

    \newcommand{\VeryCompLocus}{\Hilbfunc_r^{\max} \AA^n}%
    \newcommand{\VeryCompLocusParams}[2]{\Hilbfunc_{#1}^{\max} \AA^{#2}}%
    \newcommand{\Gmult}{\mathbb{G}_{m}}%
    There are few classes of smoothings known, mainly because checking
    flatness of a finite family is subtle. In this section we present
    a class of smoothings coming from, equivalently, one-dimensional torus
    actions (from the point
    of view of affine geometry), cones over projective schemes (from the point of
    view of projective geometry) or initial ideals (from the algebraic point of view).
    We call these smoothings $\Gmult$-limits (Definition~\ref{ref:klimit:def}).
    We use them to analyze smoothability of very compressed algebras
    (Definition~\ref{ref:initialideal:def}). We also prove that there exists
    subschemes which are smoothable but are not $\Gmult$-limits of smooth
    schemes, see Example~\ref{ex:1771arenotklimits}.

    The theory of $\Gmult$-limits is classical, for the algebraic side
    see~\cite[Chapter~15]{Eisenbud}. The application to very compressed
    algebras first appeared in~\cite{DJNT}, while
    Example~\ref{ex:1771arenotklimits} was not published before.

    Let us introduce the necessary notions.
    \newcommand{\initial}[1]{\operatorname{in}(#1)}%
    \begin{definition}
        For an ideal $I$ of a polynomial ring $\DS$, its \emph{initial ideal} is the ideal
        spanned by top degree forms (with respect to the standard grading) of all elements of $I$. It is denoted by $\initial{I}
        \subset \DS$.
    \end{definition}
    Let $\AA^n = \Spec\DS = \Spec \kk[\Dx_1, \ldots ,\Dx_n]$. The torus
    $\Gmult$ acts on $\AA^n$ by \emph{dilation}
    \[
        \mu:\Gmult \times \AA^n \ni (t, (x_1, \ldots ,x_n)) \to (tx_1, \ldots
        ,tx_n).
    \]
    On the level of functions, $\mu^{\#}\colon\DS\to \DS[t^{\pm
    1}]$ is defined by $\mu^{\#}(\Dx_i) = t\Dx_i$.
    Let $\inv:\Gmult\to
    \Gmult$ be the inverse, then $\inv^{\#}\colon\kk[t^{\pm 1}]\to \kk[t^{\pm 1}]$
    is given by $\inv^{\#}(t) = t^{-1}$. By abuse of notation, let
    $\inv^{\#}\colon\DS[t^{\pm 1}] \to \DS[t^{\pm 1}]$ be the map $\inv\tensor
    \operatorname{id}$. Then $\inv^{\#}\mu^{\#}(\Dx_i) = t^{-1}\Dx_i$.
    \begin{lemma}\label{ref:dilation:lem}
        The dilation action of $\Gmult$ on $\AA^n$ induces a $\Gmult$-action
        on $\Hilb{\AA^n}$. The orbit of a finite subscheme $R \subset \AA^n$ is a (finite flat) family
        \begin{equation}
            \pi\colon  \Gmult \cdot [R] \subset \Gmult \times \AA^n \to \Gmult,
            \label{eq:orbit}
        \end{equation}
        given by the ideal $I(\Gmult[R]) = \inv^{\#}\mu^{\#}(I(R))$.
    \end{lemma}
    \begin{proof}
        \def\HH{\mathcal{H}}%
        Let $\HH := \Hilb{\AA^n}$ and $\UU \subset \HH \times \AA^n$ be the universal family.
        Before we prove that $\Gmult$ acts on $\HH$, let us describe
        its action point-wise. A $\kk$-point $t\in \Gmult$ induces an
        isomorphism $\mu(t)\colon\AA^n \to \AA^n$, so also an isomorphism
        $\mu(t)\colon\HH \times \AA^n \to \HH \times \AA^n$. Let $t\UU =
        \mu(t)(\UU)$. Then $t\UU \to \HH$ is a composition of the isomorphism
        $\mu(t)^{-1}$ and a flat map, hence it is flat. The family $t\UU \to \HH$
        induces a map $\mu_{\HH}(t)\colon \HH\to \HH$, which is the action of
        $t$. By uniqueness, all axioms of group action are satisfied for
        $\kk$-points. While this is enough to define the action of $\Gmult$,
        below we present this action abstractly to prove the description of
        $I(\Gmult[R])$.

        The action $\mu:\Gmult \times \AA^n \to \AA^n$ gives a diagram
        \begin{equation}
            \begin{tikzcd}
                \Gmult \times \AA^n\arrow{d}{\varphi}\arrow{r}{\inv} &
                \Gmult \times \AA^n\arrow{d}{\mu}\arrow{r}{\operatorname{pr}_1} &
                \Gmult\arrow[d]\\
                \AA^n\arrow[r, "="] &\AA^n\arrow[r] & \Spec \kk,
            \end{tikzcd}
            \label{eq:pullbackfromaction}
        \end{equation}
        where $\varphi = \mu \circ (\inv \times \operatorname{id})$.
        The right square of the diagram~\eqref{eq:pullbackfromaction} is isomorphic to the
        \emph{pullback} square via the isomorphism $\operatorname{id} \times
        \mu:\Gmult \times \AA^n \to
        \Gmult \times \AA^n$. Define $\UU'$ as the pullback of $\UU$ via
        $\varphi \times \operatorname{id}_{\HH}$. Then the following diagram
        consists of pullback squares.
        \begin{equation}
            \begin{tikzcd}
                \UU' \arrow[r]\arrow[d] & \Gmult \times \AA^n
                \times \HH
                \arrow[r] \arrow[d, "\varphi\times
                \operatorname{id}"] &
                \Gmult \times \HH\arrow[d]\\
                \UU \arrow[r] & \AA^n \times \HH \arrow[r] & \HH.
            \end{tikzcd}
            \label{eq:pullbackfromactionmain}
        \end{equation}
        In particular, $\UU'\to \Gmult \times \HH$ is a pullback of $\UU\to
        \HH$, so it is flat and it induces a map $\Gmult \times \HH \to
        \HH$, which is an action of $\Gmult$. For fixed $\kk$-point $t\in
        \Gmult$, the action of $t$ comes as pullback of upper row via $t\to
        \Gmult$ and so it is
        \begin{equation}
            \begin{tikzcd}
                \UU'_{|t} \arrow[r]\arrow[d] & \AA^n
                \times \HH
                \arrow[r] \arrow[d, "\varphi(t) \times \operatorname{id}"] &
                \HH\arrow[d]\\
                \UU \arrow[r] & \AA^n \times \HH \arrow[r] & \HH.
            \end{tikzcd}
            \label{eq:pullbackfromactionelement}
        \end{equation}
        Since $\varphi = \mu\circ(\inv \times \operatorname{id})$, we have
        $\varphi(t) = \mu(t^{-1})$, so $\UU'_{|t}$ is the pullback of $\UU$
        via $\mu(t^{-1})$. This is the same as $\mu(t)(\UU)$, so the two
        descriptions of the action of $\Gmult$ agree.
        The equality $I(\Gmult[R]) = \varphi^{\#}(I(R)) =
        \inv^{\#}\mu^{\#}(I(R))$ follows by construction of the action.
    \end{proof}

    For a one-parameter subgroup of a projective variety, we may always take a
    flat limit (see~\cite[Proposition~III.9.8]{hartshorne}). In the special case of
    Lemma~\ref{ref:dilation:lem} we have a flat limit in $\AA^n$.
    \begin{proposition}\label{ref:limitoneparameter:prop}
        Let $R \subset \AA^n = \Spec \DS$ be a finite $\kk$-scheme. The
        family~\eqref{eq:orbit} uniquely extends to an embedded (finite flat) family
        \[
            \pi\colon \ccX \subset \AA^1 \times \AA^n \to \AA^1,
        \]
        where $\AA^1 = \Gmult \cup \{0\}$. The ideal of $R_0 :=\pi^{-1}(0) \subset
        \AA^n$ is equal to $\initial{I(R)}$.
    \end{proposition}
    \begin{proof}
        \def\Ifam{\mathcal{I}}%
        By Lemma~\ref{ref:dilation:lem} the family~\eqref{eq:orbit} is given
        by ideal $\Ifam = \inv^{\#}\mu^{\#}(I(R)) \subset \DS[t^{\pm 1}]$. For an element $f\in I(R)$, write
        its decomposition into homogeneous summands as $f = f_0 +  \ldots +
        f_d$. Then $\inv^{\#}\mu^{\#}(f) = f_0 + t^{-1}f_1 +  \ldots +
        t^{-d}f_d$. Thus we have
        \begin{equation}\label{eq:element}
            t^df_0 + t^{d-1}f_1 +  \ldots + tf_{d-1} + f_d\in \inv^{\#}\mu^{\#}(I(R)).
        \end{equation}
        Define $\ccX$ by the ideal $\Ifam \cap \DS[t]$. Then clearly $\ccX$
        restricts to~\eqref{eq:orbit} on $\AA^1 \setminus \{0\}$ and
        \[
            \OX = \DS[t]/I(\ccX) \subset \DS[t^{\pm 1}]/\Ifam,
        \]
        so $\OX$ is a torsion-free $\OT$-module, hence $\ccX\to \AA^1$ is flat
        over $\AA^1$, see~\cite[Corollary~6.3]{Eisenbud}.
        Moreover, if a finite set of monomials spans
        $H^0(R, \OO_R)$ as $\kk$-vector space, then it also spans $\OX$ as a
        $\OT$-module, so $\ccX\to T$ is finite. Hence $\ccX\to T$ is a family.
        By~\eqref{eq:element}, for every $f\in I(R)$ we have $f_d\in
        I(\ccX_0)$, so $\ccX_{|0}$ is contained in $V(\initial{I(R)})$. But
        $V(\initial{I(R)})$ and $R$ have the same degree, so we must have
        $\ccX_{|0} = V(\initial{I(R)})$.
    \end{proof}

    \begin{definition}\label{ref:klimit:def}
        For finite $R \subset \AA^n$ the scheme $R_0 = V(\initial{I(R)}$ is
        called the $\Gmult$-limit of $R$.
    \end{definition}
    Note that $\initial{I(R)}$ is always a homogeneous ideal.

    \begin{proposition}\label{ref:klimits:prop}
        Let $R \subset \AA^n$ be a smoothable subscheme given by ideal $I$.
        Then its $\Gmult$-limit $R_0 \subset \AA^n$ is
        also smoothable.
    \end{proposition}
    \begin{proof}
        By Proposition~\ref{ref:limitoneparameter:prop} we have a family
        $\pi\colon \ccX \subset \AA^1 \times \AA^n \to \AA^1$
        with
        general fiber isomorphic to $R$ and special fiber $R_0$, so the
        smoothability of $R_0$ follows from
        Proposition~\ref{ref:smoothabilityisclosed:prop}.
%        The family $\pi$ is
%        invariant under the $\Gmult$-action $t\cdot (x_1, \ldots ,x_n, t') =
%        (tx_1, \ldots , tx_n, tt')$ and $0$ is a torus-fixed point, so $R_0$ is invariant under dilation
%        action; $I(R_0)$ is a homogeneous ideal.
    \end{proof}
    \newcommand{\projccX}{(\ccX\setminus\{0\})/\Gmult}%
    In the language of Proposition~\ref{ref:limitoneparameter:prop}, the family
    $\pi\colon \ccX \subset \AA^1 \times \AA^n$ is invariant under the dilation
    action of $\Gmult$ on $\AA^{n+1} = \AA^1 \times \AA^n$, so it is a cone over a
    projective scheme $\projccX \simeq R$ and $R_0$ is obtained as a
    section of the cone over this scheme with the cone over the hyperplane
    $V(t)$, where $t$ is the parameter on $\AA^1$.

%    This construction can be reinterpreted in geometric terms. Let $R = V(I)
%    \subset \DS$ be a finite subscheme of degree $r$. The torus
%    $\Gmult$ acts on $\AA^n$ by dilation
%    \[
%        \Gmult \times \AA^n \ni (t, (x_1, \ldots ,x_n)) \to (tx_1, \ldots
%        ,tx_n),
%    \]
%    so it acts on $\Hilbr{\AA^n}$. For a
%    point $[R]\in \Hilbr{\AA^n}$ we have an orbit map $\mu:\Gmult \to
%    \Hilbr{\AA^n}$ which uniquely extends to $\bar{\mu}:\Gmult \cup \{0\} \to
%    \Hilbr{\AA^n}$. The subscheme corresponding to $\bar{\mu}(0)$ is $R_0 = \Spec
%    \DS/\initial{I}$, see \cite[Chapter~15]{Eisenbud} for an algebraic
%    viewpoint. In particular, $R_0$ is a limit of schemes isomorphic to $R$,
%    so if $R$ is smoothable, then also $R_0$ is smoothable, by
%    Proposition~\ref{ref:smoothabilityisclosed:prop}.

    \newcommand{\Isheaf}{\mathcal{I}}%
    The initial ideal construction can be made relative. Indeed, the extension
    of $\initial{-}$ to ideals $I \subset A\tensor
    \DS$ is straightforward. In the general case of $\Isheaf \subset \OT
    \tensor \DS$ we construct $\initial{\Isheaf}$ locally on affine covering
    of $T$ and glue the construction. The gluing is possible, because an
    initial form of a section $s$ of $\Isheaf$ restricts either to initial form of
    restriction of $s$ or to zero.
    \begin{definition}\label{ref:initialfamily:def}
        \def\IX{\mathcal{I}_{\ccX}}%
        Let $\ccX \subset T \times \AA^n\to T$ be (finite flat) family over $T$
        given by ideal sheaf $\IX \subset \OT\tensor \DS$.
        Then its \emph{initial scheme} is $\ccX_0 = V(\initial{\IX})$.
    \end{definition}
    \begin{lemma}
        For every finite flat $\pi\colon \ccX\to T$, the initial scheme $\ccX_0$ constructed in
        Definition~\ref{ref:initialfamily:def} is finite over $T$.
    \end{lemma}
    \begin{proof}
        Locally on $T$ the sheaf $\pi_*\OX$ is an $\OT$-module spanned by
        finitely many fixed monomials. Then also $\pi_*\OO_{\ccX_0}$ is an
        $\OT$-module spanned by these monomials, so that $\ccX_0 \to T$ is
        finite.
    \end{proof}
    Even though $\ccX\to T$ is flat, the morphism $\ccX_0 \to T$ need
    not be flat.

    \begin{example}
        \def\IX{I_X}%
        Let $T = \Spec \kk[s]$.
        Consider $D = V(\Dx^2, \Dx\Dy, \Dy^3) \subset \mathbb{A}^2 = \Spec
        \kk[\Dx, \Dy]$ and
        \[
            \ccX = V(\Dx - s\Dy^2) \subset D \times T \subset \mathbb{A}^2
            \times T
        \]
        considered as a finite
        family $\pi\colon \ccX\to T$.
        For each $\lambda\in T$, the
        fiber $\ccX_\lambda$ is a degree three subscheme of $\mathbb{A}^2$, so
        $\pi_*\OX$ is a locally free sheaf of rank three; in particular $\pi$
        is flat. We have $\IX = (\Dx^2, \Dx\Dy, \Dy^3)\kk[\Dx, \Dy, s] \oplus (\Dx
        - s\Dy^2)\kk[s]$, so
        $\initial{\IX} = (\Dx^2, \Dx\Dy, \Dy^3)\kk[\Dx, \Dy, s] \oplus (s\Dy^2)\kk[s]$ and
        hence $V(\initial{\IX}) \to T$ is not flat near $s=0$.
    \end{example}

    \newcommand{\goodsubset}{U}%
    \newcommand{\II}{\mathcal{I}}%
    \newcommand{\MonoIdeals}{\mathcal{M}ono_r}%
    We now proceed to define an open subset of $\Hilbr{\AA^n}$ where the initial
    scheme of the universal family \emph{is} flat. Before we do it, we
    introduce very compressed subschemes.
    \begin{definition}\label{ref:initialideal:def}
        Choose $n$ and $r$. Let $\mathbb{A}^n = \Spec \DS$ and $\DmmS \subset
        \DS$ be the ideal of the origin. Consider subschemes of degree $r$
        given by ideals $I$ such that $\DmmS^{s+1} \subseteq I \subsetneq
        \DmmS^{s}$ for an integer $s$. We call such subschemes \emph{very compressed} and denote by
        \[
            \VeryCompLocus \subset \Hilbr{\AA^n}
        \]
        their family (with reduced structure).
    \end{definition}
    Clearly, $\VeryCompLocus  \simeq \Grass(a, \DmmS^s/\DmmS^{s+1})$ for
    appropriate $a$; in particular it is irreducible. The integer $s = s(n, r)$
    appearing in Definition~\ref{ref:initialideal:def} is uniquely determined by $n$ and $r$:
    \[
        s(n, r) = \min\left\{ i\ |\ \binom{n+i}{i} \geq r \right\}.
    \]
    Let $\AA^n = \Spec \DS$ and let $\MonoIdeals$ denote the set of
    monomial ideals $\lambda$ in $\DS$ which are finite of degree $r$ and
    satisfy $\DmmS^{s(n, r)+1} \subset \lambda \subsetneq \DmmS^{s(n, r)}$.
    For $\lambda\in \MonoIdeals$ consider the subset $U_{\lambda} \subset
    \Hilbr{\AA^n}$ consisting of subschemes $R
    \subset \AA^n$ such that $H^0(R, \OO_R)$ has a $\kk$-basis given by all
    monomials not in $\lambda$. Then $U_{\lambda}$ is an open subset of
    $\Hilbr{\AA^n}$. Let
    \[
        \goodsubset = \bigcup\left\{ U_{\lambda} \ |\ \lambda\in
    \MonoIdeals\right\}.
    \]
    Let $\UU \subset \goodsubset \times \AA^n$ be the restriction of the
    universal family to $\goodsubset$. Let $\UU_0 \subset \goodsubset \times
    \AA^n$ be its initial scheme.
    \begin{proposition}
        The initial scheme $\UU_0 \to \goodsubset$ is flat. Its fibers are
        given by initial ideals of the fibers of $\UU\to \goodsubset$.
    \end{proposition}
    Since $\UU_0\to \goodsubset$ is flat, we call it the \emph{initial
    family}.
    \begin{proof}
        \def\OOgood{\OO_{\goodsubset}}%
        \def\basis{\mathcal{B}}%
        Consider the ideal sheaf $\II_{\UU} \subset \OOgood \otimes \DS$ and
        pick a point $x\in \goodsubset$ and $\lambda$ such that $x\in
        U_{\lambda}$.
        It is enough to prove that the restriction of $\UU_0$ to
        $^{-1}(U_{\lambda})$ is flat.
        Let $\basis$ denote the set of monomials not in
        $\lambda$ and $\basis_s \subset \basis$ denote the set of elements of
        degree $s := s(n, r)$.  In $H^0(\UU_x, \OO_{\UU_x})$ the image of every monomial
        in $\lambda$ may be written as a combination of $\basis$. Hence, also
        in a neighbourhood $T$ of $x$, for every $m\in \lambda$ we have an
        element
        \begin{equation}\label{eq:tmpelement}
            m - \sum_{m_i\in \basis} a_im_i\in \II_{\UU}.
        \end{equation}
        In particular, if $m$ is a monomial of
        $\deg(m) > s = \max \deg(\basis)$ then $m \in \lambda$ and also
        $m$ is the initial form of~\eqref{eq:tmpelement}.
        If $\deg(m) \leq s$, then $\deg(m) =s$ by construction of $\lambda$,
        so
        \begin{equation}\label{eq:tmpelementin}
            m - \sum_{m_i\in \basis_s} a_im_i\in \II_{\UU}.
        \end{equation}
        Up to multiplying by $\OOgood$, these are the only equations of $\UU
        \subset U \times \AA^n$ near $x$, so
        near $x$ the sheaf $\OOgood$ is free with basis $B_{\lambda}$. Since
        $x$ is arbitrary, the map
        $\UU_{0} \to \goodsubset$ is flat. The claim about the fibers follows.
    \end{proof}
    See~\cite[Chapter~18]{Miller_Sturmfels} for details of the above
    construction of $\UU_0$.
    The finite flat family $\UU_0\to \goodsubset$ induces a mapping
    \[
        \varphi_r\colon\goodsubset \to \VeryCompLocus.
    \]
    Note that $\VeryCompLocus \subset \goodsubset$ and
    $(\varphi_r)_{|\VeryCompLocus} = \operatorname{id}$. Thus the map
    $\varphi_r$ is a \emph{retraction} onto $\VeryCompLocus$.
    The map $\varphi_r$ plays a key role in checking smoothability of very
    compressed schemes, as the following
    Proposition~\ref{ref:verycompressedsmoothability:prop} shows.
    \begin{proposition}\label{ref:verycompressedsmoothability:prop}
        We have $\VeryCompLocus \subset \Hilbsmr{\AA^n}$ if and only if
    $\varphi_r(\goodsubset \cap \Hilbsmr{\AA^n})$ surjects onto
    $\VeryCompLocus$.
    \end{proposition}
    \begin{proof}
        The map $\varphi_r$ maps every subscheme to its initial subscheme. If
        a subscheme is smoothable, also its initial subscheme is smoothable,
        by Proposition~\ref{ref:klimits:prop}.
        Hence $\varphi_r(\goodsubset \cap \Hilbsmr{\AA^n}) \subset \Hilbsmr{\AA^n}$.
        Therefore $\varphi_r(\goodsubset \cap \Hilbsmr{\AA^n}) =
        \Hilbsmr{\AA^n} \cap \VeryCompLocus$.
    \end{proof}

    Example~\ref{ex:nonsmoothable96} shows that
    $\VeryCompLocusParams{96}{3}\not\subset \Hilbfunc^{sm}_{96}\,(\AA^3)$ by
    dimensional reasons. We now show that $\VeryCompLocusParams{r}{3}
    \subset \Hilbsmr{\AA^3}$ for all $r < 96$. This result first appeared
    in~\cite{DJNT}.
    \begin{proposition}\label{ref:upto95:prop}
        Let $\chark = 0$.
        The family $\VeryCompLocusParams{r}{3}$ of very compressed ideals is
        contained in the smoothable component if and only if $r\leq 95$.
    \end{proposition}
    \begin{proof}
        The only if part follows from Example~\ref{ex:nonsmoothable96}.
        To prove the if part, suppose first that $\chark = 0$.
        It is enough to check that $\varphi_r$ is dominant for all $r\leq
        95$. All schemes of degree up to $7$ in $\AA^3$ are smoothable
        by~\cite{cartwright_erman_velasco_viray_Hilb8}, so it is enough to
        check for $8\leq r\leq 95$.  Pick a general tuple $R$ of $r$ points of
        $\AA^3$ over $\kk$. Then the tangent map
        \begin{equation*}
            \Dtangspace{\varphi_{r}}:
            \Dtangspace{\Hilbzeror{\AA^3},\, [R]}
            \to \Dtangspace{\VeryCompLocusParams{r}{3},\, \varphi_{r}([R])}
        \end{equation*}
        is surjective. This is verified by a direct computer calculation, see
        the Macaulay2 package \texttt{CombalggeomApprenticeshipsHilbert.m2}
        accompanying the arXiv version of~\cite{DJNT}.
        Then by~\cite[Theorem~17.11.1d, p.~83]{ega4-4} the morphism $\varphi_r$ is smooth at
        $[R]$, thus flat, thus open, and thus the claim.
    \end{proof}

    \begin{remark}[Comparison with the case of $8$ points in $\AA^4$]\label{ref:furtherwork:remark}
        For $r \geq 96$ the map $\varphi_r$ is not surjective by dimensional
        reasons.  Even though $\Dtangspace{\varphi_r}$ is not surjective, we conjecture
        that the maps $\Dtangspace{\varphi_{r}}$ are of maximal rank.
        This is no longer true for $\AA^4$: in fact $\Dtangspace{\varphi_{8}\,
        \AA^4}$ has $20$-dimensional image in the $21$-dimensional
        Grassmannian $\Grass(3,10)$, which accounts for the fact that there are
        nonsmoothable ideals of degree $8$ in $\AA^4$, as exhibited in
        Example~\ref{ex:143}.
    \end{remark}

    \begin{example}[smoothable schemes, which are not $\Gmult$-limits of smooth
        schemes]\label{ex:1771arenotklimits}
        \def\HH{\mathcal{H}}%
        \def\tmpparamspace{\mathcal{P}}%
        \def\PGL{\operatorname{PGL}}%
        \def\PGL{\operatorname{PGL}}%
        We sketch an example of a family of schemes which are smoothable and
        $\Gmult$-invariant but whose general member is not a $\Gmult$-limit of
        a smooth scheme.

        Assume $\kk = \kkbar$.
        Consider a subset $\HH \subset \HilbGorarged{18}{\AA^7}$ consisting
        of subschemes $R \subset \AA^7$ such that $R$ is Gorenstein, the ideal $I(R)$ graded
        (hence $R$ is irreducible),
        and $H_{H^0(R, \OO_R)} = (1, 7, 7, 1)$.
        Elements of $\HH$ are smoothable by a result of Bertone, Cioffi and
        Roggero, see Remark~\ref{ref:1nn1:rmk}.
        Each $R\in \HH$ has a unique up to scaling cubic dual generator in
        $\kdp[x_1, \ldots ,x_7]$ and a general cubic $F$ in $\kdp[x_1, \ldots
        ,x_7]$ corresponds to a $\Spec\Apolar{F} \in \HH$, so $\dim \HH =
        \binom{7+3-1}{3} - 1 = 83$.

        Suppose that an element $R\in \HH$ is a $\Gmult$-limit of a smooth
        scheme $R^{\circ} \subset \AA^7$. The family $\ccX \subset \AA^7 \times
        \AA^1$ gives a projective scheme $\projccX \subset \PP^7$,
        abstractly isomorphic to $R^\circ$. The scheme $R$ is
        Gorenstein and is a hyperplane section of the cone over
        $\projccX$, so the scheme $\projccX$ is arithmetically
        Gorenstein~\cite[Chapter~10]{hartshorne_deformation_theory}.
        Moreover, $\projccX$ spans $\PP^7$.
        Denote by $\tmpparamspace$ the variety of ordered arithmetically Gorenstein
        tuples of points in $\PP^7$ which span $\PP^7$, so that
        \[
            \tmpparamspace \subset (\PP^7)^{18}/\PGL_{7}.
        \]
        We have $\dim \tmpparamspace = \binom{8}{2} = 28$ by the results of Coble and Dolgachev-Ortland,
        see~\cite[Corollary~8.4]{Eisenbud__Popescu__Gale_geometry}.
        Each element of $\tmpparamspace$ gives, by intersecting with a
        hyperplane, an element of $\HH$ defined up to $\PGL_7$-action.
        But $\dim \HH/\PGL_7 = 83 - 48 > 28$, so a general point of $\HH$ is
        not obtained this way.
    \end{example}

    \begin{partwithabstract}{Applications}
        \label{part:applications}

        In this part prove that the Gorenstein locus of the
        Hilbert scheme -- the open subscheme containing all finite Gorenstein
        subschemes -- is irreducible for small degrees, see
        Theorem~\ref{ref:cjnmain:thm}. We describe the smallest
        case when it is reducible, see Theorem~\ref{ref:14pointsmain:thm}. These results about the Gorenstein locus first
        appeared in~\cite{cjn13, jelisiejew_VSP}. We also bound the dimension
        of the punctual Gorenstein Hilbert scheme, which parameterizes
        irreducible subschemes supported at a fixed point, see
        Theorem~\ref{ref:expecteddim:thm}. This result first appeared
        in~\cite{Michalek}.

        Theorem~\ref{ref:cjnmain:thm} and
        Theorem~\ref{ref:expecteddim:thm} are motivated by applications to
        secant varieties and constructing $r$-regular maps respectively, as
        explained in Section~\ref{ssec:introsecants} and
        Section~\ref{ssec:introkregularity}.
    \end{partwithabstract}

    \chapter{Gorenstein loci for small number of points}\label{sec:Gorensteinloci}

    In this section we discuss smoothability of Gorenstein schemes, with the
    aim of proving the following main theorem of~\cite{cjn13}.
    \begin{theorem}[Irreducibility up to 13 points]
        Let $\kk$ be a field of characteristic $\neq 2, 3$.
        \label{ref:cjnmain:thm}
        Let $R$ be an finite Gorenstein scheme of degree at most $14$.  Then
        either $R$ is smoothable or it corresponds to a local algebra $(\DA,
        \mm, \kk)$ with $H_A = (1, 6, 6, 1)$. In particular, if $R$ has degree
        at most $13$, then $R$ is smoothable.
    \end{theorem}
    This theorem will be proved along a series of partial results.
    By Corollary~\ref{ref:basechangesmoothings:cor} we reduce to $\kk =
    \kkbar$.
    The proof goes by induction on the degree. Under the inductive assumption,
    all reducible schemes are smoothable
    (Corollary~\ref{ref:smoothingcomponentsresult:cor}). By
    Proposition~\ref{ref:smoothabilityisclosed:prop} also limits of reducible
    schemes are smoothable. Proving that a given scheme is a limit of
    reducible ones is a key ingredient in our approach. Accordingly, we define
    cleavable schemes.
    \begin{definition}\label{ref:cleavable:def}
        A finite subscheme $R$ is \emph{cleavable} (or \emph{limit-reducible})
        if there exists a finite flat family $\ccX \to T$ over an irreducible
        $T$, with a special fiber
        isomorphic to $R$ and general fiber reducible. Each such family is called a
        \emph{cleaving} of $R$.
    \end{definition}
    The name \emph{limit-reducible} is introduced in~\cite{cjn13}, while
    \emph{cleavable} is used in~\cite{nisiabu_jabu_kleppe_teitler_direct_sums}. If a cleavable $R$ is
    embedded into $\amb$, then, after changing $T$, we may assume that the cleaving
    $\ccX$ is embedded into $\amb$, see the argument of
    Theorem~\ref{ref:goodbaseofsmoothing:thm}.
    In fact, the families $\ccX\to T$ constructed below are all embedded into
    affine spaces.

    \newcommand{\cubicspace}{\mathbb{P}\left(\Sym^3 \kk^6\right)}%
    \newcommand{\Hilbshort}{\mathcal{H}}%
    \newcommand{\Hilbshortzero}{\Hilbshort_{gen}}%
    \newcommand{\Hilbshortother}{\Hilbshort_{1661}}%
    \newcommand{\exceptional}{\Hilbshort^{gr}_{1661}}%
    \newcommand{\DIR}{D_{IR}}%
    \newcommand{\DVap}{D_{V-ap}}%
    \newcommand{\VV}{\AA^6}%
    \paragraph{Nonsmoothable component for 14 points.} For degree $14$, there are nonsmoothable finite schemes $R$, which are
    irreducible and correspond to algebras with Hilbert function $(1, 6, 6,
    1)$. Each such subscheme can be embedded into $\mathbb{A}^6$.
    As proven in Example~\ref{ex:1661} using relative Macaulay inverse systems, these nonsmoothable schemes form a
    component of $\HilbGorarged{14}{\AA^6}$.
    Denote this component by $\Hilbshortother$ and the smoothable component by
    $\Hilbshortzero$, so that topologically
    \[
        \HilbGorarged{14}{\AA^6} = \Hilbshortzero \cup \Hilbshortother.
    \]
    Let $\Hilbshort = \HilbGorarged{14}{\AA^6}$ and introduce a
    scheme structure on $\Hilbshortother$ by $\Hilbshortother =
    \overline{\Hilbshort \setminus \Hilbshortzero}$.
    Under this definition it is not clear whether $\Hilbshortother$ is reduced
    or smooth. For $\chark = 0$ we show that indeed it is (we do not say
    anything about the reducedness or smoothness of $\Hilbshortzero$).
    Moreover, we describe $\Hilbshortother$
    and explicitly find the intersection $\Hilbshortzero\cap
    \Hilbshortother$ of the two components. Equivalently, we find necessary and
    sufficient conditions for smoothability of finite Gorenstein schemes of
    degree $14$. Such condition are rarely known, the only other case
    is~\cite{erman_velasco_syzygetic_smoothability}.
    We follow~\cite{jelisiejew_VSP}.

    \newcommand{\cubicspacereg}{\mathbb{P}\left(\Sym^3 \kk^6\right)_{1661}}%
    Let $\exceptional \subset \Hilbshortother$ be the set corresponding to
    $R$ invariant under the dilation
    action of $\Gmult$,~i.e.,~such that $I(R)$ is homogeneous. Then
    $\exceptional \subset \Hilbshortother$ is a closed subset and we endow it
    with a reduced scheme structure.

    \newcommand{\Gtwosix}{\Grass(2, \kk^6)}%
    Let $\DIR \subset \cubicspace$ denote the Iliev-Ranestad divisor in the space of cubic
    fourfolds.  This divisor consists of cubics corresponding to
    finite schemes which are sections of the cone over $\Gtwosix$ in
    the Pl\"ucker embedding; see Section~\ref{ssec:14pointsproof} for a precise
    definition.

    Let $\cubicspacereg \subset \cubicspace$ be the open subset of cubics $F$
    such that $\dimk \DS_1\hook F = 6$. Geometrically, these are $F$ such that
    $V(F) \subset \mathbb{P}^5$ is not a cone.

%    Theorem~\ref{ref:14pointsmain:thm} describes the component $\Hilbshortother$.
    \begin{theorem}\label{ref:14pointsmain:thm} Assume $\chark = 0$.
		With notation as above, we have the description of $\Hilbshortother$.
        \begin{enumerate}
            \item\label{it:smoothness} The component $\Hilbshortother$ is smooth
                (hence reduced) and connected.
            \item\label{it:retraction} There is an ``associated-graded-algebra'' morphism
                \[\pi\colon \Hilbshortother \to \exceptional\]
                which makes $\Hilbshortother$ the total space of a rank $21$ vector
                bundle over $\exceptional$.
            \item\label{it:exceptional} The scheme $\exceptional$ is canonically
                isomorphic to $\cubicspacereg$.
            \item\label{it:intersection} The set theoretic intersection $\Hilbshortzero \cap \Hilbshortother$
                is a prime divisor inside $\Hilbshortother$ and it is equal to
                $\pi^{-1}(\DIR)$, where ${\DIR\subset \exceptional \subset
                \cubicspace}$ is the restriction of the Iliev-Ranestad divisor.
                We get the following diagram of vector bundles:

                \begin{tikzcd}
                    \Hilbshortzero \cap \Hilbshortother   \arrow[d] &
                        \subset &
                        \Hilbshortother \arrow[d]\\
                    \DIR  & \subset&\cubicspacereg
                \end{tikzcd}
        \end{enumerate}
    \end{theorem}
    The most difficult steps of the proof are reducedness of $\Hilbshortother$
    and~description of the intersection.
    The map $\pi$ is defined at the level of points as follows. We take $[R]\in
    \Hilbshortother$. After translation, its support becomes $0\in \VV$. Then we
    replace $H^0(R, \OO_R)$ by its associated graded algebra which is also
    Gorenstein by Corollary~\ref{ref:eliasrossi:cor}. We take $\pi([R])$ to be the point
    corresponding to $\Spec \gr H^0(R, \OO_R)$ supported at the origin of
    $\VV$.

    The identification of $\exceptional$ with $\cubicspacereg$ in
    Point~\ref{it:exceptional} is done canonically using Macaulay's
    inverse systems, as in the argument of Example~\ref{ex:1661}. Note that
    the complement of $\cubicspacereg$  has codimension greater than one,
    hence divisors on $\cubicspace$ and $\exceptional$ are identified via
    restriction and closure.

    \section{Ray families}\label{ssec:rayfamilies}

    To prove that a finite scheme $R$ is smoothable, we need to find a
    family with special fiber $R$ and general fiber smoothable. We find such
    families over $\mathbb{A}^1$. We are interested primarily in the case of
    irreducible $R$, as smoothability is checked on components by
    Theorem~\ref{thm_equivalence_of_abstract_and_embedded_smoothings}.
    Recall that in a \emph{family} is by assumption finite and flat
    (Definition~\ref{ref:finiteflatfamily:def}). For finite morphisms, which
    are not necessarily flat, we use the name \emph{deformation}.

    \newcommand{\pivot}{\Dx^{\nu} - \partial}%
    \newcommand{\pivotupper}{\Dx^{\nu} - t\Dx^{\nu-1} - \partial}%
    \newcommand{\pivotlower}{\Dx^{\nu} - t\Dx - \partial}%
    Suppose that $R \subset \mathbb{A}^n$ is irreducible, supported at the
    origin and that $C \subset \mathbb{A}^n$ is curve
    intersecting $R$ and \emph{smooth} at the intersection point. Denote by $I(R)$, $I(C)$ their ideals in $\mathbb{A}^n$.
    Let $H = V(\Dx)$ be a hyperplane intersecting $C$ transversely. Then
    $\Dx_{|C}$ is a local parameter on $C$ at the origin.  Consequently, $R\cap C
    \subset C$ is cut out of $C$ by $\Dx^\nu$ for some $\nu\geq 1$. Take a
    lifting
    $\Dx^{\nu}$ to an element $\pivot\in I(R)$.
    \begin{lemma}\label{ref:tmpidealprinc:lem}
        In the setup above, $R$ is cut out of $R\cup C$ by a single equation
        $\pivot$.
    \end{lemma}
    \begin{proof}
        By assumption, $(\pivot) + I(C) = I(R\cap C) = I(R) + I(C)$.
        Intersecting both sides with $I(R)$, we get $(\pivot) + I(C)\cap I(R)
        = I(R)$, hence the claim.
    \end{proof}
    In light of Lemma~\ref{ref:tmpidealprinc:lem} above, we try to deform $R$ inside $R\cup C$ by deforming
    $\pivot$. We would like a general fiber of the deformation to be
    reducible, so a natural deformation over $\kk[t]$ is given by
    \[
        (\Dx^{\nu} - t\Dx^{s} - \partial = 0) \subset (R\cup C) \times \Spec \kk[t]
    \] for a chosen $s < \nu$. The
    restriction of this deformation to $C \times \Spec \kk[t]$ is flat, given by $\Dx^{\nu} - t\Dx^s$.
    We will see that the deformation itself is flat provided that $R\cap C$ is
    large enough. Intuitively,
    when $R\cap C \subset C$ is large, we may peel a point off $R$ along $C$.
    We illustrate this in the following
    Proposition~\ref{ref:cleavinggeometric}. Let $H^{\nu-1} = V(\Dx^{\nu-1})$ be
    a thick hyperplane.

    \begin{proposition}\label{ref:cleavinggeometric}
        In the above setup, assume that $R \subset C \cup H^{\nu-1}$. Then
        $R$ is cleavable.
    \end{proposition}

    \begin{proof}
        \def\exten#1{#1 \times \AA^1}
        \def\Cext{C\times{\mathbb{A}^1}}%
        \def\Rext{R\times{\mathbb{A}^1}}%
        \def\IRext{I_R}%
        \def\ICext{I_C}%
        \def\IV{I_V}
        For brevity denote $D = R \cup C$.
        Consider the deformation
        \begin{equation}\label{eq:rayfamily}
            V(\pivotupper) \subset D \times \AA^1,
        \end{equation}
        with $t$ being the local parameter on $\AA^1$.
        Let $\ICext, \IRext \subset \kk[D \times\AA^1]$ be the ideals of
        $\Cext$ and $\Rext$, respectively, so $\ICext \cap \IRext = 0$. Let
        $\IV = (\pivotupper) \subset \kk[D \times \AA^1]$. By the assumption, we have $(\Dx^{\nu-1})\cap
        \ICext \subset \IRext$.
        Since $H = (\Dx)$ is transversal to $C$, we have $\IV\cap \ICext =
        \IV\cdot \ICext$.
        Consequently, we obtain
        \begin{equation}\label{eq:criteq}
            \IV \cap \ICext = \IV\cdot \ICext = (\Dx^{\nu} -
            t\Dx^{\nu-1} - \partial)\cdot \ICext \subset
            (\Dx^{\nu}-\partial)\cdot \ICext +
            (\Dx^{\nu-1})\cdot \ICext \subset \IRext \cap \ICext = 0.
        \end{equation}
%        Denote for brevity
%        $\exten{X} := X \times \mathbb{A}^1$.

        To prove flatness of Deformation~\eqref{eq:rayfamily} it is enough to
        prove that every polynomial $f \in \kk[t]$ is not a zero-divisor in the coordinate ring of
        $V = V(\pivotupper)$, see~\cite[Corollary~6.3]{Eisenbud}. Suppose there is an $f\in
        \kk[t]$ and a function $g\in \kk[V]$ such that $fg = 0$
        in $\kk[V]$.

        Let us restrict to $C$, i.e., consider the deformation $V\cap (C \times
        \AA^1)$. It is given by the equation $\Dx^\nu -t\Dx^{\nu-1}$ thus, it is
        flat over $\kk[t]$. Therefore, $f$ is not a zero-divisor, hence, $g$ restricts to
        zero on $V\cap (C \times \AA^1)$. Therefore, $g$ lies in
        $(\ICext+\IV)/\IV \subset \kk[V]$. By Equation~\eqref{eq:criteq} we
        have naturally
        \begin{equation}
            (\ICext+\IV)/\IV \simeq \ICext/(\IV\cap \ICext) = \ICext \subset \kk[D \times \AA^1],
        \end{equation}
        so $g$ is an element of a flat $\kk[t]$-module $\kk[D\times \AA^1]$.
        Since $fg = 0$, it follows that $g = 0$, which concludes proof of
        flatness; therefore~\eqref{eq:rayfamily} is a \emph{family}.
        The fiber of the family~\eqref{eq:rayfamily} over $t\neq 0$ is
        supported on at least two points: the origin and $(t, 0,  \ldots ,0)$,
        thus, reducible. Therefore, $R$ is cleavable.
    \end{proof}

    We now formally define ray deformations.
    \begin{definition}
        Let $R \subset \mathbb{A}^n$ be an irreducible finite scheme supported
        at the origin, $C \subset \mathbb{A}^n$ be a curve smooth at the
        origin and $H = V(\Dx) \subset \mathbb{A}^n$ be transversal to $C$, so
        that $R\cap C=V(\Dx^\nu) \cap C$ for an element $\nu\geq 1$.
        A \emph{ray decomposition} of $R$ is a subscheme $D \subset
        \mathbb{A}^n$ such that $C\cup R
        \subset D$
        together with a lift of $\Dx^{\nu}$ to an element $\pivot \in I(R)
        \subset \kk[\mathbb{A}^n]$ such
        that
        \[
            R = D \cap V(\pivot).
        \]
        The associated \emph{lower ray deformation} is $\ccX = V(\pivotlower)
        \subset D \times \Spec \kk[t]$. The associated \emph{upper ray
        deformation}
        is $\ccX = V(\pivotupper) \subset D\times \Spec \kk[t]$.
    \end{definition}
    If $\ccX$ is an upper ray deformation, then $\ccX \cap (C\times \Spec \kk[t]) =
    V(\Dx^{\nu} - t\Dx^{\nu-1})$, so the support of any fiber over $\kk^*$ is
    reducible. Therefore if the upper deformation is flat in a neighborhood of
    $0\in \kk$, then $R$ is
    cleavable; similarly for lower ray deformation.
    Proposition~\ref{ref:cleavinggeometric} may be rephrased as: if $D = C \cup
    R$ and $R \subset D\cup H^{\nu-1}$, then the upper ray deformation is flat. The
    extra flexibility in choosing $D$ is used in
    Section~\ref{ssec:rayfromMacaulay}.

    Flatness of ray deformations is, in general, a delicate issue. We exhibit more
    examples of flat ray deformations in
    Section~\ref{ssec:rayfromMacaulay}, where we consider ray families associated to polynomials.
    Below we give an algebraic version of
    Proposition~\ref{ref:cleavinggeometric} and a special case of Gorenstein
    algebras.
    \begin{corollary}\label{ref:algebraiccleaving:cor}
        Let $R \subset \AA^n$ be a finite scheme supported at the
        origin. Let $I = I(R)$ be its ideal. Choose coordinates
        $\alpha_1,\alpha_2,
        \dotsc ,\alpha_n$ on $\AA^n$.
        Assume that $b$ is such that
        \(
            \alpha_1^b\cdot \alpha_j \in I
        \)
        for all $j\neq 1$.
        Assume moreover that $\alpha_1^b\notin I + (\alpha_2,\alpha_3,\dotsc,\alpha_n)$.
        Then $R$ is cleavable.
    \end{corollary}
    \begin{proof}
        This follows from Proposition~\ref{ref:cleavinggeometric} above if we take $C =
        V(\alpha_2,\alpha_3,\dotsc,\alpha_n)$, $H = (\alpha_1)$. Then $\nu$ is defined by $R\cap C =
        (\alpha_1^\nu)$ and by assumption $\nu > b$, so that $R \subset C
        \cup H^{\nu-1}$.
    \end{proof}

    The criterion of Corollary~\ref{ref:algebraiccleaving:cor} has a
    convenient formulation in terms of inverse systems (defined in
    Chapter~\ref{sec:macaulaysinversesystems}). Recall that $\DP =
    \kdp[x_1, \ldots ,x_n] \subset \Homthree{\kk}{\DS}{\kk}$ is an
    $\DS$-module by the contraction action, see
    Definition~\ref{ref:contraction:propdef}.
    \begin{corollary}\label{ref:stretchedhavedegenerations:cor}
        Let $R = \Spec\Apolar{f} \subset \mathbb{A}^n$, where
        $f = \DPel{x_1}{d} + g\in \DP$ is such
        that $\Dx_1^c\hook g = 0$ for some $c$ satisfying $2c\leq d$. Then
        $R$ is cleavable.
    \end{corollary}
    \begin{proof}
        Let $\Spec\Apolar{f} \cap V(\Dx_2, \ldots ,\Dx_n)$ be defined by
        $\Dx_1^{\nu}$ and $\Dx_1^{\nu} - \partial$ be a lift of $\Dx_1^{\nu}$ to
        $\Ann(f)$.
        \def\Dord{\nu}%
        \def\Ddegf{d}%
        Since $\Dx_1^{\nu} - \partial\in \Ann(f)$, we have $\partial\hook g = \partial\hook f =
        \Dx_1^{\Dord}\hook f = \DPel{x_1}{\Ddegf -\Dord} + \Dx_1^\Dord \hook g$. Then
        $\Dx_1^{\Ddegf -\Dord}(\partial\hook
        g) = \Dx_1^{\Ddegf -\Dord}\hook \DPel{x_1}{\Ddegf -\Dord} + \Dx_1^\Ddegf \hook g = 1$, thus
        $\Dx_1^{\Ddegf -\Dord}\hook g\neq 0$.
        It follows that $\Ddegf  - \Dord \leq c-1$, so $\Dord \geq \Ddegf  - c
        + 1\geq c + 1$. The assumptions of
        Corollary~\ref{ref:algebraiccleaving:cor} are satisfied with $b = \nu
        -1$.
    \end{proof}

    \begin{corollary}\label{ref:squareshavedegenerations:cor}
        Let $\kk = \kkbar$ and $\chark \neq 2$.
        Let $R = \Spec \DA \subset \mathbb{A}^n$, where $\DA$ is Gorenstein of
        socle degree $d$ and such that $\Delta_{d-2}\neq 0$,
        where $\Delta_{\bullet}$ is the symmetric decomposition of Hilbert
        function of $\DA$.
        Then $R$ is cleavable.
    \end{corollary}
    \begin{proof}
        By Proposition~\ref{ref:squares:prop} we have $R  \simeq \Apolar{f}$,
        where $f = g + \DPel{x_n}{2}$ for $g\in \kdp[x_1, \ldots ,x_{n-1}]$. Thus
        $\Dx_n\hook g = 0$, so $\Spec\Apolar{f}$ is cleavable by
        Corollary~\ref{ref:stretchedhavedegenerations:cor} with $c = 1$ and
        $d=2$. In fact, $\Spec \Apolar{f}$ is a limit of subschemes isomorphic
        to $\Spec \Apolar{g} \sqcup \Spec \kk$.
    \end{proof}

    \begin{example}\label{ref:quarticlimitreducible:example}
        Let $\chark\neq 2,3$ and $f\in \kdp[x_1, x_2, x_3, x_4]$ be a
        polynomial, $\deg(f) = 4$. Suppose
        that the leading form $f_4$ of $f$ is written as $f_4 = \DPel{x_1}{4} + g_4$
        where $g_4\in \kdp[x_2, x_3, x_4]$.
        By Proposition~\ref{ref:topdegreetwist:lem} we may nonlinearly change
        coordinates so that $f = \DPel{x_1}{4} + g$, where $\Dx_1^2\hook g = 0$. By
        Corollary~\ref{ref:stretchedhavedegenerations:cor} we see that
        $\Apolar{f}$ is cleavable.
    \end{example}

    \begin{example}\label{ref:stretched:example}
        \def\Ddegf{d}%
        \def\Dord{\nu}%
        Suppose that an finite local Gorenstein algebra $A$ has Hilbert function
        $H_A = (1, H(1),\dots, H(c), 1,\dots, 1)$ and socle degree $\Ddegf  \geq 2c$.
        By Example~\ref{ref:standardformofstretched:ex} we have an isomorphism
        $A  \simeq \Apolar{\DPel{x_1}{\Ddegf } + g}$, where $\Dx_1^{c} \hook g =
        0$ and $\deg g\leq c+1$. By
        Corollary~\ref{ref:stretchedhavedegenerations:cor} we
        obtain a upper ray (flat) family
        \begin{equation}\label{eq:exampledegeneration}
            \kk[t] \to \frac{\DS[t]}{(\Dx_1^{\Dord } - t\Dx_1^{\Dord -1} - q) + J},
        \end{equation}
        where $J = \Ann(\DPel{x_1}{\Ddegf } + g) \cap (\Dx_2, \ldots ,\Dx_n)$.
        Thus $A$ is cleavable.
        Take $\lambda\neq 0$.
        The fiber over $t = \lambda$ is
        supported at $(0, 0,\dots, 0)$ and at $(\lambda, 0,\dots, 0)$ and
        the ideal defining this fiber near $(0, 0,\dots, 0)$ is $I_0 = (\lambda\Dx_1^{\Dord -1} - q)
        + J$.
        From the proof of Corollary~\ref{ref:stretchedhavedegenerations:cor} it follows that
        $\Dx_1^{\Dord -1} \hook g = 0$. Then the ideal $I_0$ lies in the
        annihilator of $\lambda^{-1} \DPel{x_{1}}{{\Ddegf -1}} + g$.
        Since $\sigma\hook (\DPel{x_1}{{\Ddegf}} + g) = \sigma \hook(\lambda^{-1} \DPel{x_1}{{\Ddegf - 1}} +
        g)$ for every $\sigma\in (\Dx_2, \ldots ,\Dx_n)$, the
        apolar algebra of $\lambda^{-1} \DPel{x_1}{{\Ddegf -1}} + g$ has Hilbert function
        $(1, H_1,\dots, H_c, 1,\dots, 1)$ and socle degree $\Ddegf -1$. Then
        $\dimk \Apolar{\DPel{x_1}{{\Ddegf -1}} + g} = \dimk \Apolar{\lambda^{-1}\DPel{x_1}{{\Ddegf }} + g} -
        1$. Thus the fiber is a union of a point and
        $\Spec\Apolar{\lambda^{-1}\DPel{x_1}{{\Ddegf }} + g}$,
        i.e.~the family~\eqref{eq:exampledegeneration} peels one point off
        $\Spec A$.
    \end{example}

    \section{Ray families from Macaulay's inverse
    systems.}\label{ssec:rayfromMacaulay}
    While Proposition~\ref{ref:cleavinggeometric} is important for
    applications to smoothability of Gorenstein algebras, its assumptions are
    often not satisfied, especially when the socle degree is small.
    \newcommand{\Dspecvar}{x}%
    Below we present another source of
    flat ray families, using Macaulay's inverse systems.
    We follow~\cite[Chapter~5]{cjn13}.

    The (divided power) polynomial ring $\DP$ is defined in Definition~\ref{ref:contraction:propdef}.
    Let $\DP[\Dspecvar]$ be a
    (divided power) polynomial ring obtained by adjoining a new
    variable $\Dspecvar$ to $\DP$.
    Let $\Dx$ be an element dual to $x$, so that $P[x]$ and $\DPut{T}{T}:= \DS
    [\Dx]$ are dual.
    \begin{definition}\label{ref:raysum:def}
        Let $d\geq 2$ be an integer.
        For a nonzero polynomial $f\in \DP$ and $\partial\in \DmmS$ such that
        $\partial\hook f \neq 0$ the \emph{ray sum of $f$ with respect to
        $\partial$} is the polynomial
        \[
            \sum_{i\geq 0} \DPel{\Dspecvar}{di}\partial^i\hook f = f + \DPel{\Dspecvar}{d}\partial \hook f +
            \DPel{\Dspecvar}{2d}\partial^2\hook f +  \ldots \in
            \DP[\Dspecvar].
        \]
    \end{definition}

\newcommand{\annn}[2]{\Ann_{#1}(#2)}%
The following proposition shows that a ray sum induces an explicit ray
decomposition.

\begin{proposition}\label{ref:raysumideal:prop}
    Let $g$ be the $d$-th ray sum of $f$ with respect to $\partial$.
    The annihilator of $g$ in $\DT$ is
    given by the formula
    \begin{equation}\label{eq:anndecomposition}
        \annn{\DT}{g} = \annn{\DS}{f} + \pp{\sum_{i=1}^{d-1} \kk\Dx^i}
        \annn{\DS}{\partial\hook f} + (\Dx^{d} - \partial)\DT,
    \end{equation}
    where the sum denotes the sum of $\kk$-vector spaces. In particular,
    the ideal $\annn{\DT}{g} \subset \DT$ is generated by $\annn{\DS}{f}$,
    $\Dx\annn{\DS}{\partial\hook f}$ and $\Dx^d - \partial$.
    The formula~\eqref{eq:anndecomposition} induces a ray decomposition of
    $R = \Spec \Apolar{g}$ in $\mathbb{A}^n = \Spec \DT$, with $H = V(\Dx)$, $C = V(\DmmS)$ and
    $D = V(\annn{\DS}{f}\DT +
     \Dx\annn{\DS}{\partial\hook f}\DT)$.
\end{proposition}

\begin{proof}
    It is straightforward to see that the right hand side of Equation
    \eqref{eq:anndecomposition} lies in $\annn{\DT}{g}$.
    Let us
    take any $\partial'\in \annn{\DT}{g}$. Reducing the powers of $\Dx$ using
    $\Dx^{d} - \partial$ we write
    \[
        \partial' = \sigma_0 + \sigma_1\Dx + \dots +
        \sigma_{d-1}\Dx^{d-1},
    \]
    where $\sigma_{\bullet}$ do not contain $\Dx$.
    Then
    \[
        0 = \partial'\hook g = \sigma_0 \hook f + \Dspecvar \sigma_{d-1}\partial \hook f +
        \DPel{\Dspecvar}{2} \sigma_{d-2}\partial\hook f + \dots +
        \DPel{\Dspecvar}{d-1}\sigma_1\partial \hook f.
    \]
    We see that $\sigma_0\in \annn{\DS}{f}$ and $\sigma_i \in
    \annn{\DS}{\partial\hook f}$ for $i \geq 1$, so the equality
    is proved.
    Since $\partial\hook f \neq 0$, we have $C \cup R\subset D$, so
    that indeed we obtain a ray decomposition.
\end{proof}

\begin{remark}\label{ref:Hilbfunccouting:rmk}
    It is not hard to compute the Hilbert function of the apolar algebra of
    a ray sum in some special cases. We mention one
    such case below.
    Let $f\in \DP$ be a polynomial satisfying $f_2 = f_1 = f_0 = 0$ and
    $\partial\in \DmmS^2$ be such that  $\partial \hook f = \ell$ is a linear
    form, so that $\partial^2\hook f = 0$. Let $A = \Apolar{f}$ and $B =
    \Apolar{f + \DPel{x}{2}\ell}$. The only different values of $H_A$ and $H_B$
    are $H_B(i) = H_A(i) + 1$ for $i=1, 2$. The $f_2 = f_1 = f_0 = 0$
    assumption is needed to ensure that the degrees of $\partial\hook f$ and
    $\partial\hook (f + \DPel{x}{2}\ell)$ are equal for all $\partial$ not
    annihilating $f$.
\end{remark}

We now prove that the ray families coming from ray sums are flat.
The proof is technical, so we stick to the algebraic language. We first produce
a suitable flatness criterion.
\begin{proposition}\label{ref:flatelementary:prop}
    Let $\kk = \kkbar$.
    Suppose that $S$ is a $\kk$-module (in applications of this proposition,
    $\DS$ will be the
    polynomial ring, as before) and $I \subseteq S[t]$ is a $\kk[t]$-submodule. Let $I_0 := I\cap S$.
    If for every $\lambda\in \kk$ we have
    \[(t-\lambda)\cap I \subseteq (t-\lambda)I + I_0[t],\] then $S[t]/I$ is a
    flat $\kk[t]$-module.
\end{proposition}

\begin{proof}
The ring $\kk[t]$ is a principal ideal domain, thus a $\kk[t]$-module is flat if and only
if it is torsion-free, see \cite[Corollary~6.3]{Eisenbud}.
Since $\kk = \kkbar$, every polynomial in $\kk[t]$ decomposes into linear
factors. To prove that $M = S[t]/I$ is
torsion-free it is enough to show that $t-\lambda$ are nonzerodivisors on
$M$,~i.e.~that $(t-\lambda)x\in I$ implies $x\in I$ for all $x\in S[t]$,
$\lambda\in \kk$.

    Fix $\lambda\in \kk$ and suppose that $x\in
    S[t]$ is such that $(t-\lambda)x \in I$. Then by assumption $(t-\lambda)x\in
    (t-\lambda)I + I_0[t]$, so that $(t-\lambda)(x-i) \in I_0[t]$ for some $i\in I$.
    Since $S[t]/I_0[t]   \simeq S/I_0[t]$ is a free $\kk[t]$-module, we have
    $x-i\in I_0[t]\subseteq I$ and so $x\in I$.
\end{proof}

\begin{remark}\label{ref:flatnessremovet:remark}
    Let $\DS$ be a ring and $I \subset \DS[t]$ be an ideal, generated by
    $i_1, \ldots ,i_r$. To check the inclusion
    which is the assumption
    of Proposition~\ref{ref:flatelementary:prop}, it is enough to
    check that $s\in (t-\lambda)\cap I$ implies $s\in (t-\lambda)I + I_0[t]$
    for all $s
    = s_1 i_1 +  \ldots  + s_r i_r$, \emph{where $s_i\in S$}.

    Indeed, take an arbitrary element $s\in I$ and write $s = t_1 i_1 +  \ldots  + t_r i_r$, where $t_1, \ldots ,t_r\in
    S[t]$. Dividing $t_i$ by $t-\lambda$ we obtain $s = s_1 i_1 +  \ldots  +
    s_r i_r + (t-\lambda)i$, where $i\in I$ and $s_i\in S$. Denote $s' = s_1
    i_1 +  \ldots  + s_r i_r$, then
    $s\in (t-\lambda)\cap I$ if and only if $s'\in (t-\lambda)\cap I$ and $s\in (t-\lambda)I +
    I_0[t]$ if and only if $s'\in (t-\lambda)I + I_0[t]$.
\end{remark}

\begin{lemma}\label{ref:decompositionhomog:lem}
    Let $B$ be a ring. Consider
    a ring $R = B[\Dx]$ graded by the degree of $\Dx$.
    Let $d$ be a natural number and $J\subseteq R$ be a homogeneous ideal
    generated in degrees less or equal to $d$.
    Let $\partial\in B[\Dx]$ be a (non necessarily homogeneous) element of degree strictly
    less than $d$ and such that for every $b\in B$ satisfying $b\Dx^{d}\in J$,
    we have $b\partial\in J$.
    Then for every $r\in R$ the condition
    \[r(\Dx^d - \partial)\in J\ \ \mbox{implies}\ \
        r\Dx^{d}
    \in J \ \ \mbox{and}\ \ r\partial\in J.\]
\end{lemma}
\begin{proof}
    We apply induction with respect to degree of $r$, the base
    case being $r = 0$.
    Write
    \[r = \sum_{i=0}^{m} r_i \Dx^i,\quad\mbox{where}\quad r_i\in B.\]
    The leading form of $r(\Dx^d - \partial)$ is $r_m\Dx^{m+d}$ and it lies in $J$. Since $J$
    is homogeneous and generated in degree at most $d$, we have $r_m \Dx^d\in J$. Then $r_m
    \partial\in J$ by assumption, so that $\hat{r} := r - r_m\Dx^{m}$ satisfies
    $\hat{r}(\Dx^{d} - \partial)\in J$. By induction we have
    $\hat{r}\Dx^d,\,\hat{r}\partial \in J$, then also $r\Dx^d,\,r\partial\in J$.
\end{proof}

\begin{proposition}[flatness of ray families]\label{ref:raysumflatness:prop}
    \def\DTpoly{\DT}%
    \def\DJpoly{J}%
    Let $g$ be the $d$-th ray sum with respect to $f$ and $\partial$. Then the
    corresponding upper and lower ray families are flat. Recall that these
    families are explicitly given as
    \begin{equation}\label{eq:upperrayfamily}
        \kk[t] \to \frac{\DTpoly[t]}{\DJpoly[t] + (\Dx^{d} - t\Dx^{d-1} -
    \partial)\DTpoly[t]}\quad
        \mbox{ (upper ray deformation),}
    \end{equation}
    \begin{equation}\label{eq:lowerrayfamily}
        \kk[t] \to \frac{\DTpoly[t]}{\DJpoly[t] + (\Dx^{d} -t\Dx - \partial)\DTpoly[t]}\quad\quad
        \mbox{ (lower ray deformation),}
    \end{equation}
    where $J$ is
    defined in Proposition~\ref{ref:raysumideal:prop}.
\end{proposition}

\begin{proof}
    \def\DTpoly{\DT}%
    It is enough to prove flatness after tensoring with $\kkbar$, so we may
    assume $\kk = \kkbar$.
    We start by proving the flatness of
        Deformation~\eqref{eq:lowerrayfamily}.
    We use Proposition~\ref{ref:flatelementary:prop}.
%    To simplify
%    notation let $J := J_{poly}$.
    Denote by $\DPut{tmpII}{\mathfrak{I}} \subset \DTpoly[t]$ the ideal defining the deformation and
    suppose that some $z\in \DtmpII$ lies in $(t-\lambda)$ for some
    $\lambda\in \kk$.
    Write $z$ as $i + i_2\DPut{spec}{\pp{\Dx^{d} - t\Dx -
\partial}}$, where $i\in J[t]$, $i_2\in \DTpoly[t]$, and note that by
Remark~\ref{ref:flatnessremovet:remark} we may assume $i\in J$, $i_2\in
\DTpoly$. Since $z\in (t-\lambda)$, we have that
$i + i_2\DPut{spec}{(\Dx^d - \lambda\Dx - \partial)} = 0$, so
\[
    i_2\Dspec = -i\in J.
\]
    By Proposition~\ref{ref:raysumideal:prop} the ideal $J$ is homogeneous with
respect to the grading by $\Dx$. More precisely it is equal to $J_0 + J_1\Dx$,
where $J_0 = \annn{\DS}{f}\DT,\ J_1 = \annn{\DS}{\partial\hook f}\DT$ are generated by elements not containing $\Dx$, so that $J$
is generated by elements of $\alpha$-degree at most one. We now check the
assumptions of Lemma~\ref{ref:decompositionhomog:lem}. Note that $\partial J
\subseteq J_0$ by definition of $J$. If $r\in \DTpoly$ is
such that $r\Dx^d\in J$, then $r\in J_1$, so that $r(\lambda\Dx + \partial)\in
\Dx J_1 + J_0 \subseteq J$. Therefore the assumptions are
satisfied and the Lemma shows that
    $i_2\Dx^d\in J$. Then $i_2\Dx\in J$, thus
    $i_2(\Dx^d - t\Dx)\in J[t] \subseteq(\DtmpII \cap \DTpoly)[t]$. Since $i_2 \partial\in
    \DtmpII \cap \DTpoly$ by definition, this
    implies that $i + i_2(\Dx^d - t\Dx - \partial)\in J[t] \subseteq (\DtmpII\cap
    \DTpoly)[t]$. Now the flatness follows from
    Proposition~\ref{ref:flatelementary:prop}.

    The same proof works equally well for upper ray deformation: one should just
    replace $\Dx$  by $\Dx^{d-1}$ in appropriate places of the proof. For this reason we
    leave the case of Deformation~\eqref{eq:upperrayfamily} to the reader.
\end{proof}

\begin{proposition}\label{ref:fibersofray:prop}
    Let us keep the notation of Proposition \ref{ref:raysumflatness:prop} and
    additionally assume $\kk = \kkbar$.
    Then the fibers of Families~\eqref{eq:upperrayfamily}
    and~\eqref{eq:lowerrayfamily} over $t-\lambda$ are
    reducible for every $\lambda\in \kk^*$.

    Suppose moreover that $\partial^2\hook f = 0$ and the characteristic of $\kk$ does
    not divide $d-1$. Then the fiber of the Family~\eqref{eq:lowerrayfamily} over $t-\lambda$ is
    isomorphic to \[\Spec \Apolar{f} \sqcup \left(\Spec \Apolar{\partial
    f}\right)^{\sqcup d-1}.\]
\end{proposition}

\begin{proof}
    \def\Dan#1{\annn{\DS}{#1}}%
    For both families the support of the fiber over $t - \lambda$ contains the
    origin. The support of the fiber of Family~\eqref{eq:upperrayfamily} contains
    furthermore a point with $\alpha = \lambda$ and other coordinates equal to
    zero. The support of the fiber of Family~\eqref{eq:lowerrayfamily} contains a
    point with $\alpha = \omega$, where $\omega^{d-1} = \lambda$.

    Now let us concentrate on Family~\eqref{eq:lowerrayfamily} and on the case
    $\partial^2\hook f = 0$.
    The support of the fiber over $t-\lambda$ is $(0,\dots,0,0)$ and
    $(0, \dots, 0, \omega)$, where $\omega^{d-1} = \lambda$ are $(d-1)$-th roots of
    $\lambda$, which are pairwise different because of the characteristic assumption.
    We will analyse the support point by point.
    By assumption $\partial\in \Dan{\partial\hook f}$, so that $\alpha\cdot \partial\in J$, thus $\alpha^{d+1} -
    \lambda\cdot \alpha^2$ is in the ideal $I \subset \DT$ of the fiber over $t = \lambda$.

    \def\DTpoly{\DT}%
    \def\Ilocal{I_{(0, \ldots ,0)}}%
    Near $(0,0,\dots,0)$ the element $\alpha^{d-1} - \lambda$ is invertible, so
    $\alpha^2$ is in the localisation $\Ilocal$, thus $\alpha +
    \lambda^{-1}\partial$ lies in $\Ilocal$.
    Now we check that $\Ilocal$ is generated by $\Dan{f} +
    (\alpha +
    \lambda^{-1}\partial)\DTpoly$. Explicitly, one should check that
    \[
        \pp{\Dan{f} + (\alpha + \lambda^{-1}\partial)\DTpoly}_{(0, \ldots ,0)}
        = \pp{\Dan{f} + (\Dx^{d}
    -\lambda\Dx - \partial)\DTpoly}_{(0,  \ldots ,0)}.
    \]
    Then the stalk of the fiber at $(0,
    \ldots , 0)$ is isomorphic to $\Apolar{f}$.

    Near $(0, 0,\dots, 0, \omega)$ the elements $\alpha$ and
    $\frac{\alpha^{k+1} - \lambda\cdot \alpha^2}{\alpha - \omega}$
    are invertible, so
    $\Dan{\partial\hook f}$ and $\alpha - \omega$ are in the localisation
    $I_{(0, \ldots 0, \omega)}$. This, along with
    the other inclusion, proves
    that this localisation  is generated by $\Dan{\partial\hook f}$ and $\alpha -
    \omega$ and thus the stalk of the fiber is isomorphic to $\Apolar{\partial f}$.
\end{proof}

\section{Tangent preserving ray families}\label{subsec:tangentpreserving}

A ray family gives a morphism from $\mathbb{A}^1 = \Spec \kk[t]$ to an appropriate Hilbert
scheme $\Hilbr{\mathbb{A}^n}$. In this section we prove that in some cases the
dimension of the tangent space to $\Hilbr{\mathbb{A}^n}$ is constant along
the image.  We use it to prove that certain points of $\Hilbr{\AA^n}$ are
smooth without the need for computer aided computations; such a result was only
obtained in~\cite{Shafarevich_Deformations_of_1de} and~\cite{cjn13}.
The complexity of calculating the tangent space is an obstacle to a
direct analysis of $\Hilbr{\mathbb{A}^n}$ for $r\gg 0$, see \cite{Huibregtse_elementary}.
The most important results here are Theorem \ref{ref:nonobstructedconds:thm}
together with Corollary \ref{ref:CIarenonobstructed:cor}; see examples below
Corollary \ref{ref:CIarenonobstructed:cor} for applications. This section
first appeared in~\cite{cjn13}.

Recall from Example~\ref{ex:tangentforGorenstein} that for a $\kk$-point
$[R]\in \Hilbr{\Spec \DS}$ corresponding to a
Gorenstein scheme $R = \Spec \DS/I$ the dimension of the
tangent space $\Dtangspace{[R]}$ is $\dimk \DS/I^2 - \dimk \DS/I$.
\begin{definition}\label{ref:unobstructed:def}
    A finite smoothable subscheme $R \subset \mathbb{A}^n$ of degree $r$ is \emph{unobstructed} if
    the corresponding point $[R]\in \Hilbsmr{\AA^{n}}$ is smooth, that is, if
    $\dimk\Dtangspace{[R]} = rn$.
\end{definition}
Note that being unobstructed does not depend on the embedding of $R$, by
Theorem~\ref{thm_equivalence_of_abstract_and_embedded_smoothings} and
Proposition~\ref{ref:invarianceoftangentspace:prop}. Thus we will freely speak
about unobstructed finite schemes and finite algebras. By abuse of language,
we will also say that an $f\in \DP$ is \emph{unobstructed} if $\Apolar{f}$ is.
We prefer the word ``unobstructed'' to ``smooth'' as the latter is
ambiguous: it might refer to smoothness of $R$ as a finite scheme.
We will use unobstructed schemes to prove smoothability, employing the
following observation.
\begin{lemma}\label{ref:unobstructed:lem}
    Let $\famil \subset \Hilbr{\AA^n}$ be an irreducible subset containing an
    unobstructed point. Then $\famil \subset \Hilbsmr{\AA^n}$.
\end{lemma}
\begin{proof}
    Let $U \subset \Hilbsmr{\AA^n}$ be the smooth locus of the smoothable
    component. Then $U$ is open in
    $\Hilbr{\AA^n}$. An unobstructed point lies in $U$, hence
    $U \cap \famil \subset \famil$ is open and
    non-empty, thus dense. Therefore, $\famil \subset \overline{U} \subset
    \Hilbsmr{\AA^n}$.
\end{proof}

The key idea of the section is enclosed in the following
technical Lemma~\ref{ref:tangentflatcondition:lem}, which gives necessary and
sufficient conditions for flatness of a thickening of a ray family. Regretfully, it lacks
geometric motivation and in fact the geometry behind it is unclear, apart from
the special case of complete intersections, which we discuss in
Corollary~\ref{ref:CIarenonobstructed:cor}.

\begin{lemma}\label{ref:tangentflatcondition:lem}
    Let $\kk = \kkbar$ and $d\geq 2$. Let $g$ be the $d$-th ray sum of $f\in \DP$ with respect to
    $\partial\in \DS$ such that $\partial^2 \hook f = 0$.
    Denote $I := \annn{\DS}{f}$ and $J := \annn{\DS}{\partial\hook f}$.
%    Take $\DPut{T}{T} = S[[\Dx]]$ to be the ring dual to $\DP[x]$ and
    Let
    \[\DPut{II}{\mathfrak{I}} := \pp{I + J\Dx +
        \DPut{spec}{(\Dx^{d} - t\Dx - \partial)}}\cdot \DT[t]\] be the ideal in $\DT[t]$ defining the
    associated lower ray family, see Proposition \ref{ref:raysumflatness:prop}.
    Then the morphism $\kk[t] \to \DT[t]/\DII^2$ is flat if and only if $(I^2 :
    \partial) \cap I \cap J^2 \subseteq I\cdot J$.
\end{lemma}

\begin{proof}

    We begin with the ``if'' implication. To prove flatness we will use
    Proposition~\ref{ref:flatelementary:prop}.
    Take an element $i\in \DII^2\cap (t-\lambda)$. We want to prove that $i\in
    \DII^2 (t-\lambda) + \DPut{zerocomp}{\DII_0[t]}$, where $\Dzerocomp =
    \DII^2 \cap \DT$. Let $\DPut{JJ}{\mathcal{J}} := (I + J\Dx)\DT$.
    Subtracting a suitable element of $\DII^2(t-\lambda)$ we
    may assume that
    \[i = i_1 + i_2\Dspec + i_3\Dspec^2,\] where $i_1\in
    \DJJ^2$, $i_2\in \DJJ$ and $i_3\in \DT$.
    We will in fact show that $i\in \DII^2(t-\lambda) + \DJJ^2[t]$.

    To simplify our notation, let $\DPut{ss}{\sigma} = \Dx^d - \lambda\Dx -
    \partial$. Note that $J\Dss\subseteq \DJJ$.
    We have $i_1 + i_2\Dss + i_3\Dss^2 = 0$. Let $j_3 := i_3\Dss$.
    We want to apply Lemma~\ref{ref:decompositionhomog:lem}, below we check
    its assumptions.
    The ideal $\DJJ$ is homogeneous with
    respect to $\Dx$, generated in degrees less than $d$. Let $s\in
    \DT$ be an element satisfying $s\Dx^{d}\in \DJJ$.
    Then $s\in J$, which implies $s(\lambda \Dx + \partial)\in \DJJ$.
    By Lemma~\ref{ref:decompositionhomog:lem} and $i_3\Dss^2 = j_3\Dss\in \DJJ$ we
    obtain
    $j_3\Dx^{d} \in \DJJ$, i.e.~$i_3 \Dss \Dx^d\in \DJJ$.
    Applying the same
    argument to $i_3\Dx^d$ we obtain $i_3\Dx^{2d}\in \DJJ$, therefore $i_3\in J\DT$.
    Then
    \[
        i_3 \Dspec^2 - i_3 \Dss \Dspec = i_3\Dx (t - \lambda) \Dspec \in
        \DJJ(t-\lambda) \Dspec \subseteq \DII^2(t-\lambda).
    \]
    Subtracting this element from $i$ and substituting $i_2 := i_2 + i_3 \Dss$
    we may assume $i_3 = 0$.
    We obtain
    \begin{equation}\label{eq:flatnesslemma}
        0 = i_1 + i_2\Dss = i_1 + i_2(\Dx^d - \lambda\Dx - \partial).
    \end{equation}
    Let $i_2 = j_2 + v_2\Dx$, where $j_2\in \DS$, i.e.~it does not contain
    $\Dx$. Since $i_2\in \DJJ$, we have $j_2\in I$. As before, we have $v_2\Dx (\Dspec - \Dss) = v_2\Dx^2(t-\lambda)\in
    \DII^2(t-\lambda)$, so that we may assume $v_2 = 0$.

    Comparing the top $\Dx$-degree terms of \eqref{eq:flatnesslemma} we see
    that $j_2\in J^2$.
    In equation~\eqref{eq:flatnesslemma}, comparing the terms not
    containing $\Dx$, we deduce that
    $j_2\partial\in I^2$, thus $j_2\in (I^2:\partial)$. Jointly, $j_2\in I\cap
    J^2\cap (I^2:\partial)$, thus $j_2\in IJ$ by assumption.
    But then $j_2\Dx \in \DJJ^2$, thus $j_2\Dspec\in \DJJ^2[t]$ and since
    $i_1\in \DJJ^2$, the
    element $i$ lies in $\DJJ^2[t] \subseteq \Dzerocomp$. Thus the assumptions of
    Proposition~\ref{ref:flatelementary:prop} are satisfied and the
    $\kk[t]$-module
    $\DT[t]/\DII^2$ is flat.

    The ``only if'' implication is easier: one takes $i_2\in I\cap
    J^2\cap (I^2:\partial)$ such that $i_2\not\in IJ$. On one hand, the element $j:=i_2(\Dx^d -
    \partial)$ lies in $\DJJ^2$ and we get that $i_2\Dspec - j = ti_2\Dx\in
    \DII^2$. On the other hand if $i_2\Dx\in \DII^2$, then $i_2\Dx\in (\DII^2 + (t))
    \cap \DT = (\DJJ + (\Dx^d - \partial))^2$, which is not the
    case.
\end{proof}

\begin{remark}\label{ref:tangentfibers:rmk}
    \def\tansp#1{\tan(#1)}%
    \def\Dan#1{\annn{\DS}{#1}}%
    Let us keep the notation of Lemma~\ref{ref:tangentflatcondition:lem}. Fix
    $\lambda\in \kk^*$ and suppose that the characteristic of
    $\kk$ does not divide $d-1$.
    The supports of the fibers of $\DS[t]/\DII$ and
    $\DS[t]/\DII^2$ over $t = \lambda$ are finite and equal.
    In particular, from Proposition \ref{ref:fibersofray:prop} it follows that
    the dimension of the fiber of $\DII/\DII^2$ over $t-\lambda$ is equal to
    $\tansp{f} + (d-1)\tansp{\partial\hook f}$, where $\tansp{h} = \dimk \Dan{h}/\Dan{h}^2$ is the dimension of
    the tangent space to the point of the Hilbert scheme corresponding to
    $\Spec \DS/\Dan{h}$, see Example~\ref{ex:tangentforGorenstein}.
\end{remark}

\begin{theorem}\label{ref:nonobstructedconds:thm}
    Let $\kk = \kkbar$.
    Suppose that a polynomial $f\in \DP$ corresponds to an
    unobstructed (see Definition~\ref{ref:unobstructed:def}) algebra
    $\Apolar{f}$. Let $\partial\in \DS$ be such that $\partial^2\hook f = 0$
    and the algebra $\Apolar{\partial\hook f}$ is smoothable and unobstructed.
    The following are equivalent:
    \begin{enumerate}
        \item[1.]\label{it:somenonob} the $d$-th ray sum of $f$ with respect to
            $\partial$ is unobstructed for some $d$ such that $2\leq d \leq
            \chark$ (or $2\leq d$ if $\chark = 0$).
        \item[1a.]\label{it:allnonob} the $d$-th ray sum of $f$ with respect to
            $\partial$ is unobstructed for all $d$ such that $2\leq d \leq \chark
            $ (or $2\leq d$ if $\chark = 0$).
        \item[2.]\label{it:tgflat} The $\kk[t]$-module  $\DPut{defid}{\DII}/\Ddefid^2$ is flat,
            where $\Ddefid$ is the ideal defining the lower ray family of the
            $d$-th ray sum for some $2\leq d \leq \chark$ (or $2\leq d$ if
            $\chark = 0$).
        \item[2a.]\label{it:tgflatevery} The $\kk[t]$-module $\DPut{defid}{\DII}/\Ddefid^2$ is flat,
            where $\Ddefid$ is the ideal defining the lower ray family of the
            $d$-th ray sum for every $2\leq d \leq  \chark$ (or $2\leq d$ if
            $\chark = 0$).
        \item[3.]\label{it:quoflat} The family $\kk[t] \to \DS[t]/\Ddefid^2$ is flat,
            where $\Ddefid$ is the ideal defining the lower ray family of the
            $d$-th ray sum for some $2\leq d \leq \chark$ (or $2\leq d$ if
            $\chark = 0$).
        \item[3a.]\label{it:quoflatevery} The family $\kk[t] \to \DS[t]/\Ddefid^2$ is flat,
            where $\Ddefid$ is the ideal defining the lower ray family of the
            $d$-th ray sum for every $2\leq d \leq \chark$ (or $2\leq d$ if
            $\chark = 0$).
        \item[4.] The following inclusion (equivalent to equality) of ideals in
            $\DS$
            holds: $I\cap J^2 \cap (I^2:\partial) \subseteq I\cdot J$, where
            $I = \annn{\DS}{f}$ and $J = \annn{\DS}{\partial\hook f}$.
    \end{enumerate}
\end{theorem}

\begin{proof}
    It is straightforward to check that the inclusion $I\cdot J\subseteq I\cap J^2 \cap
    (I^2:\partial)$ in Point 4 always holds,
    thus the other inclusion is
    equivalent to equality.\\
    3. $\iff$ 4. $\iff$ 3a.
    The equivalence of Point 3 and
    Point 4 follows from Lemma
    \ref{ref:tangentflatcondition:lem}. Since Point 4 is independent of $d$,
    the equivalence of Point 4 and Point 3a also follows.

    2. $\iff$ 3. and 2a. $\iff$ 3a.
    We have an exact sequence of
    $\kk[t]$-modules
    \[0\to \Ddefid/\Ddefid^2 \to \DS[t]/\Ddefid^2 \to \DS[t]/\Ddefid
\to 0.\] Since $\DS[t]/\Ddefid$ is a flat
    $\kk[t]$-module by Proposition \ref{ref:raysumflatness:prop}, we see from
    the long exact sequence of $\operatorname{Tor}$ that
    $\Ddefid/\Ddefid^2$ is flat if and only if $\DS[t]/\Ddefid^2$ is flat.

    1. $\iff$ 2. and 1a. $\iff$ 2a.
    By assumption, $\chark$ does not divide $d-1$.
    Let $g\in P[x]$ be the $d$-th ray sum of $f$ with
    respect to $\partial$. We may consider
    $\Apolar{g}$, $\Apolar{f}$, $\Apolar{\partial\hook f}$ as quotients of a
    polynomial ring $\DT$, corresponding to points of the Hilbert scheme.
    Assume 2. (resp.~2a.). The dimension of the tangent space at $\Apolar{g}$ is $\dimk
    \Ddefid/\Ddefid^2 \tensor \kk[t]/t = \dimk \Ddefid/(\Ddefid^2 + (t))$. By Remark~\ref{ref:tangentfibers:rmk} it is
    equal to the sum of the dimension of the tangent space at $\Apolar{f}$ and
    $(d-1)$ times the dimension of the tangent space to $\Apolar{\partial\hook f}$. Since both
    algebras are smoothable and unobstructed we conclude that $\Apolar{g}$ is also
    unobstructed. On the other hand, assuming 1.~(resp.~1a.), we have
    $\Apolar{g}$ is unobstructed, so
    $\Ddefid/\Ddefid^2$ is a finite $\kk[t]$-module such that the degree of
    the fiber $\Ddefid/\Ddefid^2\tensor \kk[t]/\mathfrak{m}$ does not depend on the
    choice of the maximal ideal $\mathfrak{m} \subseteq \kk[t]$. Then
    $\Ddefid/\Ddefid^2$ is flat by \cite[Exercise~II.5.8]{hartshorne} or
    \cite[Theorem~III.9.9]{hartshorne} applied to the associated sheaf.
\end{proof}

\begin{remark}
    The condition from Point 4 of Theorem
    \ref{ref:nonobstructedconds:thm} seems very technical. It is
    enlightening to look at the images of $(I^2:\partial)\cap I$ and $I\cdot
    J$ in $I/I^2$.
    The image of $(I^2:\partial)\cap I$ is the annihilator of $\partial$ in
    $I/I^2$. This annihilator clearly contains $(I:\partial)\cdot I/I^2 =
    J\cdot I/I^2$. This shows that if the $S/I$-module $I/I^2$ is ``nice'', for
    example free, we should have an equality $(I^2:\partial)\cap I = I\cdot J$.
    More generally this equality is connected to the syzygies of
    $I/I^2$.
\end{remark}

In the remainder of this subsection we will prove that in several situations
the conditions of Theorem~\ref{ref:nonobstructedconds:thm} are satisfied.

\begin{corollary}\label{ref:CIarenonobstructed:cor}
    We keep the notation and assumptions of
    Theorem~\ref{ref:nonobstructedconds:thm}. Suppose further
    that the algebra $\DS/I = \Apolar{f}$ is a complete intersection. Then the equivalent
    conditions of Theorem~\ref{ref:nonobstructedconds:thm} are satisfied.
\end{corollary}

\begin{proof}
    \def\Dan#1{\annn{\DS}{#1}}%
    Since $\DS/I$ is a complete intersection, it is
    unobstructed by
    Theorem~\ref{ref:hunekeulrich:thm}. Moreover, the $\DS/I$-module $I/I^2$ is
    free, see e.g. \cite[Theorem~16.2]{Matsumura_CommRing} and the discussion above it or
    \cite[Exercise~17.12a]{Eisenbud}. This implies that
    \[(I^2 : \partial) \cap I = (I : \partial)I
    = JI,\] because $J = \Dan{\partial\hook f} = \{ s\in \DS\ |\ s
    \partial\hook f = 0\} = (\Dan{f} : \partial) = (I : \partial)$. Thus the
    condition from Point~4 of Theorem
    \ref{ref:nonobstructedconds:thm} is satisfied.
\end{proof}

\begin{example}[{(1, 4, 5, 3, 1)}]\label{ref:14531case:example}
    Let $\kk = \kkbar$ and $\chark \neq 2$.

    If $A = \DS/I$ is a complete intersection, then it is
    unobstructed
    by Theorem~\ref{ref:hunekeulrich:thm}. The apolar algebras
    of monomials are complete intersections, therefore the assumptions of
    Theorem~\ref{ref:nonobstructedconds:thm} are satisfied e.g.~for $f
    =\DPel{x_1}{2}\DPel{x_2}{2}x_3$ and $\partial = \Dx_2^2$. Now
    Corollary~\ref{ref:CIarenonobstructed:cor} implies that the equivalent
    conditions of the Theorem are also satisfied, thus
    \[
        \DPel{x_1}{2}\DPel{x_2}{2}x_3 + \DPel{x_4}{{d}}\DPel{x_1}{2}x_3 = \pp{\DPel{x_2}{2}x_3}\pp{\DPel{x_1}{2} +
    \DPel{x_4}{d}}
    \]
    is unobstructed for every $d\geq
    2$, provided $\chark = 0$ or $d \leq \chark$.
    Similarly, $\DPel{x_1}{2}x_2x_3 + \DPel{x_4}{2}x_1$ is unobstructed
    and has Hilbert function $(1, 4, 5, 3, 1)$.
\end{example}

\begin{example}[{(1, 4, 4, 1)}]\label{ref:1441:example}
    \def\Dan#1{\annn{\DS}{#1}}%
    Let $\kk = \kkbar$ and $\chark \neq 2$.

    Let $f = \pp{\DPel{x_1}{2} + \DPel{x_2}{2}}x_3$, then $\Dan{f} = (\Dx_1^2 - \Dx_2^2,
    \Dx_1\Dx_2, \Dx_3^2)$ is a complete intersection. Take $\partial =
    \Dx_1\Dx_3$, then $\partial\hook f = x_1$ and
    $\partial^2\hook f = 0$, thus
    \[
        f + \DPel{x_4}{2}\partial\hook f = \DPel{x_1}{2}x_3 + \DPel{x_2}{2}x_3 + \DPel{x_4}{2}x_1
    \]
    is unobstructed. Note that, by
    Remark~\ref{ref:Hilbfunccouting:rmk} or by a direct computation, the apolar algebra of this
    polynomial has Hilbert function $(1, 4, 4, 1)$.
\end{example}

Below in Proposition~\ref{ref:unobstructeddoubleray:prop} we use a composition
of ray families, in particular to produce an example of a smoothable subscheme $R \subset
\mathbb{A}^5$ corresponding to a local Gorenstein algebra $\DA$ with $H_{\DA} = (1, 5, 5,
1)$ and such that $R$ is unobstructed.
Such an example was first obtained independently in~\cite{jelisiejew_1551}
and~\cite{bertone_cioffi_roggero_division_algorithm}.
\begin{proposition}\label{ref:unobstructeddoubleray:prop}
    Let $f\in \DP$ be such that $\Apolar{f}$ is a complete
    intersection.

    Let $d$ be a natural number. Suppose that $\kk = \kkbar$ and $\chark = 0$ or $d\leq \chark$.
    Take $\partial\in \DPut{Sf}{\DS}$ such that $\partial^2\hook f = 0$ and
    $\Apolar{\partial\hook f}$ is also a complete intersection.
    Let $g\in \DP[y]$ be the $d$-th ray sum $f$ with respect to $\partial$, so
    that
    $g = f + \DPel{y}{d} \partial\hook f$.

    Suppose that $\deg \partial\hook f > 0$.
    Let $\beta$ be the variable dual to $y$ and $\sigma\in \DSf$ be such that
    $\sigma\hook (\partial\hook f) = 1$. Take $\varphi := \sigma\beta\in
    \DPut{Sg}{\DT} = \DS[\beta]$.
    Let $h$ be any ray sum of $g$ with respect to $\varphi$. Explicitly
    \[
        h = f + \DPel{y}{d} \partial\hook f + \DPel{z}{m}\DPel{y}{d-1}\mbox{ for some }m\geq 2.
    \]
    Then the algebra $\Apolar{h}$ is
    unobstructed.
\end{proposition}

\begin{proof}
    First note that $\varphi\hook g = y^{d-1}$ and so $\varphi^2\hook g =
    \sigma\hook y^{d-2} = 0$, since $\sigma\in
    \mathfrak{m}_{\DSf}$. Therefore indeed $h$ has the presented form.
    \def\DmmV{\mathfrak{m}_{\DSf}}%

    From Corollary \ref{ref:CIarenonobstructed:cor} it follows that
    $\Apolar{g}$ is unobstructed. Since $\varphi\hook g = y^{d-1}$,
    the algebra $\Apolar{\varphi\hook g}$ is unobstructed as well. Now by
    Theorem \ref{ref:nonobstructedconds:thm} it remains to prove that
    \begin{equation}\label{eq:maincontainment}
        (I_g^2:\varphi) \cap I_g \cap J_g^2 \subseteq I_g J_g,
    \end{equation}
    where
    $\DPut{Ig}{I_g} =
    \annn{\DSg}{g}, \DPut{Jg}{J_g} = \annn{\DSg}{\varphi\hook g}$.
    The rest of the proof is a technical verification of this claim.
    Denote $\DPut{If}{I_f} := \annn{\DSf}{f}$ and $\DPut{Jf}{J_{f}} := \annn{\DSf}{\partial\hook f}$;
    note that we take annihilators in $\DSf$.
    By Proposition \ref{ref:raysumideal:prop} we have $\DIg = \DIf\DT +
    \beta\DJf\DT + \DPut{spec}{(\beta^{d} - \partial)}\DT$.
    Consider $\gamma\in \DSg$ lying in $(\DIg^2 : \varphi) \cap \DIg \cap
    \DJg^2$. Write $\gamma = \gamma_0 + \gamma_1 \beta + \gamma_2 \beta^2 +
    \dots$ where $\gamma_i\in \DSf$, so they do not contain $\beta$. We will
    prove that $\gamma\in \DIg\DJg$.

    First, since $\Dspec^2 \in \DIg\DJg$ we may reduce powers of $\beta$ in $\gamma$ using this
    element and so we assume $\gamma_{i} = 0$ for $i\geq 2d$.
    Let us take $i < 2d$. Since $\gamma\in \DJg^2 =
    \pp{\annn{\DSg}{y^{d-1}}}^2 = \pp{\DmmV, \beta^d}^2$ we see that $\gamma_i\in
    \DmmV \subseteq \DJg$. For $i > d$ we have $\beta^i \in \DIg$, so
    that $\gamma_i \beta^i \in \DJg\DIg$ and we may
    assume $\gamma_i = 0$.
    Moreover, $\beta^d \gamma_d - \partial \gamma_d \in \DIg\DJg$ so we may also
    assume $\gamma_d = 0$, obtaining
    \[\gamma = \gamma_0 + \dots + \gamma_{d-1} \beta^{d-1}.\]
    From the explicit description of $\DIg$ in
    Proposition~\ref{ref:raysumideal:prop} it follows that $\gamma_i\in \DJf$
    for all $i$.

    Let $M = \DIg^{2} \cap \varphi\DT = \DIg^2 \cap \DJf\beta\DT$. Then for
    $\gamma$ as above we have $\gamma \varphi\in M$, so we will analyse the
    module $M$.
    Recall that
    \begin{equation}\label{eq:scarydecomposition}
        \DIg^2 = \DIf^2\cdot \DT + \beta \DIf \DJf\cdot \DT + \beta^2 \DJf^2\cdot \DT +
        \Dspec\DIf\cdot \DT + \Dspec\beta\DJf \cdot \DT + \Dspec^2\cdot \DT.
    \end{equation}
    We claim that
    \begin{equation}\label{eq:contains}
        M \subseteq \DIf^2\cdot \DT + \beta\DIf \DJf\cdot \DT + \beta^2
        \DJf^2\cdot \DT + \Dspec\beta\DJf\cdot \DT.
    \end{equation}
    We have $\DIg^2 \subseteq
    \DJf \cdot \DT + \Dspec^2\cdot\DT$, so
    if an element of $\DIg^2$ lies in
    $\DJf\cdot\DT$, then its coefficient standing next to $\Dspec^2$ in Presentation
    \eqref{eq:scarydecomposition} is an element of $\DJf$ by
    Lemma~\ref{ref:decompositionhomog:lem}.
    Since $\DJf \cdot
    \Dspec \subseteq \DIf + \beta\DJf$, we may ignore the term $\Dspec^2$:
    \begin{equation}\label{eq:lessscdec}
        M \subseteq \DIf^2\cdot \DT + \beta \DIf \DJf\cdot \DT + \beta^2 \DJf^2\cdot \DT +
        \Dspec\DIf\cdot \DT +  \Dspec\beta\DJf\cdot \DT.
    \end{equation}
    Choose an element of $M$ and let $i\in \DIf\cdot\DT$ be the coefficient of this
    element standing next to $\Dspec$. Since $\DIf\DT \cap \beta \DT \subseteq
    \DJf\DT$ we may assume that $i$ does not contain $\beta$, i.e. $i\in
    \DIf$.
    Now, if an element of the right hand side of \eqref{eq:lessscdec} lies in
    $\beta\cdot\DT$, then the coefficient $i$ satisfies
    $i\cdot \partial\in \DIf^2$, so that $i\in (\DIf^2 : \partial)$. Since
    $\DIf$ is a complete intersection ideal the $\DS/\DIf$-module
    $\DIf/\DIf^2$ is free, see Corollary~\ref{ref:CIarenonobstructed:cor} for
    references. Then we have $(\DIf^2: \partial) =
    (\DIf:\partial)\DIf$ and $i\in (\DIf:\partial)\DIf = \DIf\DJf$. Then
    $i\cdot \Dspec \subseteq \DIf^2 + \beta\cdot \DIf\cdot \DJf$ and so the
    Inclusion \eqref{eq:contains} is proved. We come back to the proof of
    proposition.

    From Lemma~\ref{ref:decompositionhomog:lem} applied to the ideal
    $\DJf^2\DT$ and the element $\beta\Dspec$ and the fact that $\beta\partial\DJf^2
    \subseteq I_g^2$ we compute
    that $M\cap \{ \delta\ |\ \deg_{\beta} \delta \leq d\}$ is
    a subset of $\DIf^2\cdot\DT + \beta\cdot \DIf \DJf\cdot\DT + \beta^2
    \DJf^2\cdot \DT$. Then $\gamma \varphi = \gamma \beta\sigma$ lies in this set, so that
    $\gamma_0 \in (\DIf\DJf : \sigma)$ and $\gamma_{n} \in (\DJf^2 : \sigma)$
    for $n > 1$. Since $\Apolar{f}$ and $\Apolar{\partial\hook f}$ are
    complete intersections, we have
    $\gamma_0 \in \DIf\DmmV$ and $\gamma_i \in
    \DJf\DmmV$ for $i \geq 1$.
    It follows that $\gamma\in \DIg\DmmV \subseteq \DIg\DJg$.
\end{proof}

\begin{example}[{(1, 5, 5, 1)}]\label{ref:1551:example}
    \def\Dan#1{\annn{\DS}{#1}}%
    Let $\kk = \kkbar$ and $\chark \neq 2$.

    Let $f\in P$ be a polynomial such that $A = \Apolar{f}$ is a complete
    intersection. Take
    $\partial$ such that $\partial\hook f = x_1$ and $\partial^2\hook f = 0$.
    Then the apolar algebra of $f + \DPel{y_1}{d}x_1 +
    \DPel{y_{2}}{m}\DPel{y_1}{d-1}$ is unobstructed
    for every $d, m\geq 2$ (less or equal to $\chark$ if the characteristic is non-zero). In
    particular
    \[
        g = f + \DPel{y_1}{2}x_1 + \DPel{y_2}{2}y_1
    \]
    is unobstructed.
%    The apolar algebra of $g$ has Hilbert function $H_{B}(t) =
%    H_{A}(t)$ for $t\neq 1, 2$ and $H_B(t) = H_A(t) + 2$ for $t\in \{1,
%    2\}$.

    Continuing Example \ref{ref:1441:example}, if $f =
    \DPel{x_1}{2}x_3 + \DPel{x_2}{2}x_3$, then $\DPel{x_1}{2}x_3 + \DPel{x_2}{2}x_3 + \DPel{x_4}{2}x_1 + \DPel{x_5}{2}x_4$ is
    unobstructed. The apolar algebra of this polynomial has Hilbert function
    $(1, 5, 5, 1)$.

    Let
    $g = \DPel{x_1}{2}x_3 + \DPel{x_2}{2}x_3 + \DPel{x_4}{2}x_1$,
    then $\DPel{x_1}{2}x_3 + \DPel{x_2}{2}x_3 +
    \DPel{x_4}{2}x_1 + \DPel{x_5}{2}x_4$ is a ray sum of $g$ with respect to $\partial =
    \Dx_4\Dx_1$. Let $I := \Dan{g}$ and $J := (I : \partial)$.
    In contrast with Corollary~\ref{ref:CIarenonobstructed:cor} and Example~\ref{ref:1441:example} one may
    check that all three terms $I$, $J^2$ and $(I^2 : \partial)$ are necessary to
    obtain equality in the inclusion \eqref{eq:maincontainment} for $g$ and $\partial$, i.e.~no two
    ideals of $I$, $J^2$, $(I^2 : \partial)$ have intersection equal to $IJ$;
    we need to intersect all three of them, to obtain $IJ$.
    %$\Dx_4^3$ is the only element requiring cutting with J_g^2.
\end{example}

\begin{example}[{(1, 4, 4, 3, 1, 1)}]\label{ref:144311:example}
    Let $\kk = \kkbar$ and $\chark \neq 2$.

    Let $f = \DPel{x_1}{5} + \DPel{x_2}{4}$. Then the annihilator of $f$ in $\kk[\Dx_1, \Dx_2]$
    is a complete intersection, and this is true for every $f\in \kdp[x_1, x_2]$. Let
    $g = f + \DPel{x_3}{2}\DPel{x_1}{2}$ be the second ray sum of $f$ with respect to
    $\Dx_1^3$ and $h = g + \DPel{x_4}{2}x_3$ be the second ray sum of $g$ with
    respect to $\Dx_3\Dx_1^2$.
    Then the apolar algebra of
    \[h = \DPel{x_1}{5} + \DPel{x_2}{4} + \DPel{x_3}{2}\DPel{x_1}{2} + \DPel{x_4}{2}x_3\] is
    unobstructed. It has Hilbert function $(1, 4, 4, 3, 1, 1)$.
    %Computed explicitly --- works.
\end{example}

\begin{remark}
    The assumption $\deg \partial\hook f > 0$ in
    Proposition~\ref{ref:unobstructeddoubleray:prop} is necessary:
    the polynomial $h = x_1x_2x_3 + \DPel{x_4}{2} + \DPel{x_5}{2}x_4$ is not
    unobstructed, since it has
    degree $12$ and tangent space dimension $67 > 12\cdot 5$ over
    $\kk = \mathbb{Q}$. The polynomial $g$ is the fourth ray sum of $x_1 x_2 x_3$
    with respect to $\Dx_1\Dx_2\Dx_3$ and $h$ is the second
    ray sum of $g = x_1 x_2 x_3 + \DPel{x_4}{2}$ with respect to $\Dx_4$, thus this
    example satisfies the assumptions of
    Proposition~\ref{ref:unobstructeddoubleray:prop} except for $\deg
    \partial\hook f > 0$. Note that in this case $\Dx_4^2\hook g \neq 0$.
\end{remark}

\section{Proof of Theorem~\ref{ref:cjnmain:thm} --- preliminaries}

This section is the starting point of the proof of Theorem~\ref{ref:cjnmain:thm} ---
irreducibility of the Gorenstein locus for small degrees. It
contains the necessary preliminaries and it is of limited interest of its own.
We employ Macaulay's inverse systems, as described in
Chapter~\ref{sec:macaulaysinversesystems}, and in particular
the symmetric decomposition $\Dhdvect{\bullet}$ of the Hilbert function, see
Section~\ref{ssec:HFininversesystems}, Lemma~\ref{ref:sumsofdhdaraOseqence:lem}
and the standard form of the dual generator, see Section~\ref{ssec:standardforms}.

Recall from Proposition~\ref{ref:flatfamiliesforconstructible:prop} that for a
constructible $V \subset \DP_{\leq d}$ with $\dimk\Apolar{f} = r$ for all
$f\in V$, we have an associated morphism $V\to \Hilbr{\Spec \DS}$.
Consider $f\in \DP_{\leq d}$.
The apolar algebra of $f$ has degree at most $s$ if and only if the
space $\DS_{\leq d} f$ has dimension at most $s$. In coordinates, this is a
rank $\leq s$ condition, so it is closed and we obtain the following
Remark~\ref{ref:semicontinuity:rmk}.
\begin{remark}\label{ref:semicontinuity:rmk}
    Let $d$ be a positive integer and $V \subseteq \DP_{\leq d}$ be a constructible subset. Then
    the set $U$, consisting of $f\in V$ such that the apolar algebra of $f$
    has the maximal degree (among the elements of $V$), is open in $V$. In
    particular, if $V$ is irreducible then $U$ is also irreducible.
\end{remark}

\begin{example}\label{ref:semicondegthree:example}
    Let $\DP_{\geq 4} = \kdp[x_1, \ldots
    ,x_n]_{\geq 4}$.
    Suppose that the set $V \subseteq \DP_{\geq 4}$ parameterizing algebras
    with fixed Hilbert function $H$ is irreducible. Then also the set $W$ of
    polynomials $f\in \DP$ such that $f_{\geq 4}\in V$ is irreducible. Let
    $e:= H(1)$ and suppose that the symmetric decomposition of $H$, see
    Definition~\ref{ref:symmetricdecomposition:def},  has zero
    rows $\Dhdvect{d-3} = (0, 0, 0, 0)$ and $\Dhdvect{d-2} = (0, 0, 0)$, where
    $d = \max\{i\mid H(i)\neq 0\}$.
    We claim that general element of $W$ corresponds to an algebra $B$ with Hilbert
    function
    \[
        H_{max} = H + (0, n-e, n-e, 0).
    \]
    Indeed, since we only vary the degree three part of the polynomial,
    the function $H_B$ has the form $H + (0, a, a, 0) + (0, b, 0)$ for some
    $a, b$ such that $a + b \leq n - e$. Therefore algebras with Hilbert
    function $H_{max}$ are precisely the algebras of maximal possible degree.
    Since $H_{max}$ is attained for $f_{\geq 4} + \DPel{x_{e+1}}{3} +
    \ldots  + \DPel{x_n}{3}$, the claim follows from
    Remark~\ref{ref:semicontinuity:rmk}.
\end{example}

We now state a number of lemmas concerning the Hilbert function $H_A$ of a local
Gorenstein algebra $\DA$. These lemmas are used in the proof and themselves
are probably of little interest other than an exercise in properties of the symmetric
decomposition $\Dhdvect{\bullet}$ of $H_A$.

\begin{lemma}\label{ref:hilbertfunc:lem}
    Suppose that $(\DA, \mm, \kk)$ is a finite local Gorenstein algebra of socle degree $d\geq 3$ such that
    $\Dhdvect{A, d-2} = (0, 0, 0)$. Then $\deg A \geq 2\left(H_A(1) + 1\right)$.
    Furthermore, equality occurs if and only if $d = 3$.
\end{lemma}

\begin{proof}
    Consider the symmetric decomposition $\Dhdvect{\bullet} = \Dhdvect{A,
    \bullet}$ of $H_A$.
    From symmetry, see Definition~\ref{ref:symmetricdecomposition:def},  we have $\sum_j \Dhd{0}{j} \geq 2 + 2\Dhd{0}{1}$ with
    equality only if $\Dhdvect{0}$ has no terms between $1$ and $d-1$~i.e.~when $d = 3$.
    Similarly $\sum_j \Dhd{i}{j}\geq 2\Dhd{i}{1}$ for all $1 \leq i < d-2$.
    Summing these inequalities we obtain
    \[
        \deg A = \sum_{i<d-2} \sum_j \Dhd{i}{j} \geq 2 + \sum_{i<d-2} 2\Dhd{i}{1} = 2
        + 2H_A(1).\qedhere
    \]
\end{proof}

\begin{lemma}\label{ref:trikofHilbFunc:lem}
    Let $(\DA, \mm, \kk)$ be a finite local Gorenstein algebra of degree at most $14$. Suppose
    that $4\leq H_A(1) \leq 5$. Then $H_A(2)\leq 5$.
\end{lemma}

\begin{proof}
    Let $d$ be the socle degree of $A$.
    Suppose $H_A(2) \geq 6$. Then $H_{A}(3) + H_{A}(4) + \dots \leq 3$, thus
    $d\in \{3, 4, 5\}$. The cases $d = 3$ and $d = 5$ immediately lead to
    contradiction --- it is impossible to get the required symmetric
    decomposition. We will consider the case $d = 4$. In this case $H_A = (1,
    *, *, *, 1)$ and its symmetric decomposition is $(1, e, q, e, 1) + (0, m,
    m, 0) + (0, t, 0)$.
    Then $e = H_A(3) \leq 14 - 2 - 4 - 6 = 2$.
    Since $H_A(1) < H_A(2)$ by assumption, we have $e < q$. This can
    only happen if $e = 2$ and $q = 3$. But then $14\geq \deg A = 9 + 2m + t$,
    thus $m\leq 2$ and $H_A(2) = m + q \leq 5$. A contradiction.
\end{proof}

\begin{lemma}\label{ref:14341notexists:lem}
    There does not exist a finite local Gorenstein algebra $(\DA, \mm, \kk)$ with Hilbert
    function \[(1, 4, 3, 4, 1, \ldots , 1).\]
\end{lemma}

\begin{proof}
    \def\Ddegf{d}
    \def\Dan#1{\Ann(#1)}%
    See \cite[pp.~99-100]{iarrobino_associated_graded} for the proof or
    \cite[Lemma~5.3]{CJNPoincare} for a generalization. We provide a sketch for
    completeness.
    Suppose such an algebra $A$ exists and fix its dual
    generator $f\in \kk[x_1, \ldots, x_4]_\Ddegf$ in the standard form (Definition~\ref{ref:standardform:def}). Let $I =
    \Dan{f}$.
    The proof relies on two observations. First, the leading term of $f$ is, up
    to a constant, equal to $\DPel{x_1}{\Ddegf}$ and in fact we may take $f =
    \DPel{x_1}{\Ddegf} +
    f_{\leq 4}$. Moreover, analysing the symmetric decomposition directly, we
    have $\Dhdvect{A, d-2} = \Dhdvect{A, d-3} = 0$. Using
    Proposition~\ref{ref:degreecomparison:prop}, we derive that the
    Hilbert functions of  $\Apolar{\DPel{x_1}{d} + f_4}$ and $\Apolar{f}$ are equal. Second,
    $h(3) = 4 = 3^{\langle 2\rangle} = h(2)^{\langle 2\rangle}$ is the maximal growth, so arguing similarly as in
    Lemma~\ref{ref:P1gotzmann:lem} we may assume that the degree
    two part $I_2$ of the ideal of $\gr A$ is equal to $((\Dx_3,
    \Dx_4)\DS)_2$. Then any derivative of $\Dx_3\hook f_4$ is a derivative of
    $\DPel{x_1}{d}$, i.e., a power of $x_1$. It follows that $\Dx_3\hook f_4$ itself is a
    power of $x_1$; similarly $\Dx_4\hook f_4$ is a power of $x_1$.
    It follows that $f_4\in \DPel{x_1}{3}\cdot \kk[x_3,x_4] + \kk[x_1,
    x_2]$, but then $f_4$ is annihilated by a
    linear form, which contradicts the fact that $f$ is in
    the standard form.
\end{proof}

The following lemmas essentially deal with the cleavability
(Definition~\ref{ref:cleavable:def}) in the case
$(1, 4, 4, 3, 1, 1)$. Here the method is straightforward, but the cost
is that the proof is broken into several cases and quite long.

\begin{lemma}\label{ref:144311Hilbfunc:lem}
    \def\Dan#1{\Ann(#1)}%
    Let $f = \DPel{x_1}{5} + f_4\in \DP$ be a polynomial such that $H_{\Apolar{f}}(2) <
    H_{\Apolar{f_4}}(2)$.
    Let $\DPut{tmpQ}{\mathcal{Q}} = \DS_2 \cap
    \Dan{\DPel{x_1}{5}} \subseteq
    \DS_2$. Then $\DPel{x_1}{2}\in \DtmpQ f_4$ and $\Dan{f_4}_{2} \subseteq \DtmpQ$.
\end{lemma}

\begin{proof}
    \def\Dan#1{\Ann(#1)}%
    Note that $\dim \DtmpQ f_4 \geq \dim \DS_2 f_4 - 1 = H_{\Apolar{f_4}}(2) -
    1$. If $\Dan{f_4}_{2} \not\subseteq \DtmpQ$, then there is a $q\in \DtmpQ$
    such that $\Dx_1^2 - q\in \Dan{f_4}$. Then $\DtmpQ f_4 = \DS_2 f_4$ and
    so we obtain $H_{\Apolar{f}}(2) = H_{\Apolar{f_4}}(2)$, which is a contradiction.
    Suppose that $\DPel{x_1}{2}\not\in \DtmpQ f_4$. Then the degree
    two partials of $f$ contain a direct sum of $\kk \DPel{x_1}{2}$ and $\DtmpQ f_4$,
    thus they are at least $H_{\Apolar{f_4}}(2)$-dimensional, so that
    $H_{\Apolar{f}}(2)\geq H_{\Apolar{f_4}}(2)$, a
    contradiction.
\end{proof}

\begin{lemma}\label{ref:144311caseCI:lem}
    \def\Dan#1{\Ann(#1)}%
    Let $f = \DPel{x_1}{5} + f_4\in \DP$ be a polynomial such that $H_{\Apolar{f}} = (1, 3, 3, 3, 1, 1)$
    and $H_{\Apolar{f_4}} = (1, 3, 4, 3, 1)$. Suppose that $\Dx_1^3\hook f_4 =
    0$ and that $\pp{\Dan{f_4}}_2$ defines a complete intersection. Then
    $\Apolar{f_4}$ and $\Apolar{f}$ are complete intersections.
\end{lemma}

\begin{proof}
    \def\Dan#1{\Ann(#1)}%
    Let $I := \Dan{f_4}$.
    First we will prove that $\Dan{f_4} = (q_1, q_2, c)$, where $\langle q_1,
    q_2\rangle = I_2$ and $c\in I_3$. Then $\Apolar{f_4}$ is a complete intersection.
    By assumption, $q_1, q_2$ form a regular sequence. Thus there are no syzygies of
    degree at most three in the minimal resolution of $\Apolar{f_4}$. By
    the symmetry of the minimal resolution, see
    \cite[Corollary~21.16]{Eisenbud},
    there are no generators of degree at least four in the minimal generating
    set of $I$. Thus $I$ is generated in degree two and three. But
    $H_{\DS/(q_1, q_2)}(3) = 4 = H_{\DS/I}(3) + 1$, thus there is a
    cubic $c$, such that $I_3 = \kk c\oplus (q_1, q_2)_3$, then $(q_1, q_2, c) = I$, thus $\Apolar{f_4} = \DS/I$
    is a complete intersection.

    Let $\DPut{anntwo}{\mathcal{Q}}:= \Dan{\DPel{x_1}{5}} \cap S_2 \subseteq S_2$.
    By Lemma~\ref{ref:144311Hilbfunc:lem} we have $q_1, q_2\in \Danntwo$, so
    that $\Dx_1^3\in I \setminus (q_1, q_2)$, then $I = (q_1, q_2, \Dx_1^3)$.
    Moreover, by the same lemma, there exists $\sigma\in \Danntwo$ such that $\sigma\hook f_4 =
     \DPel{x_1}{2}$.

    Now we prove that $\Apolar{f}$ is a complete intersection.
    Let $J := (q_1, q_2, \Dx_1^3 - \sigma) \subseteq \Dan{f}$.
    We will prove that $\DS/J$ is a complete intersection.
    Since $q_1$, $q_2$, $\Dx_1^3$ is a
    regular sequence, the scheme $\Spec\DS/(q_1, q_2)$ is a cone over a scheme of
    dimension zero and $\Dx_1^3$ does not
    vanish identically on any of its components. Since $\sigma$ has degree two, $\Dx_1^3 - \sigma$
    also does not vanish identically on any of the components of $\Spec
    \DS/(q_1, q_2)$, thus $\Spec \DS/J$ has dimension zero,
    so it is a complete intersection (see also \cite[Corollary~2.4,
    Remark~2.5]{Valabrega_FormRings}).
    Then the quotient by $J$ has degree at most
    $\deg(q_1)\deg(q_2)\deg(\Dx_1^3 - \sigma) = 12 = \dimk \DS/\Dan{f}$. Since
    $J \subseteq \Dan{f}$, we have $\Dan{f} = J$ and
    $\Apolar{f}$ is a complete intersection.
\end{proof}

\begin{lemma}\label{ref:144311casenotCI:lem}
    \def\Dan#1{\Ann(#1)}%
    Let $f = \DPel{x_1}{5} + f_4 + g\in \DP$, where $\deg g\leq 3$, be a polynomial such that
    $H_{\Apolar{f_{\geq 4}}} = (1, 3, 3, 3, 1, 1)$
    and $H_{\Apolar{f_4}} = (1, 3, 4, 3, 1)$. Suppose that $\Dx_1^3\hook f_4 =
    0$ and that $\pp{\Dan{f_4}}_2$ does not define a complete intersection.
    Then $\Apolar{f}$ is cleavable.
\end{lemma}

\begin{proof}
    \def\spann#1{\langle #1 \rangle}%
    \def\Dan#1{\Ann(#1)}%
    Let $\langle q_1, q_2\rangle = \pp{\Dan{f_4}}_2$. Since $q_1, q_2$ do not
    form a regular sequence, we have, after a linear transformation $\varphi$, two
    possibilities: $q_1 = \Dx_1\Dx_2$ and $q_2 = \Dx_1\Dx_3$ or $q_1 =
    \Dx_1^2$ and $q_2 = \Dx_1\Dx_2$. Let $\beta$ be the image of $\Dx_1$ under
    $\varphi$, so that $\beta^3\hook f_4 = 0$.

    Suppose first that $q_1 = \Dx_1\Dx_2$ and $q_2 = \Dx_1\Dx_3$. If $\beta$
    is up to constant equal to $\Dx_1$, then $\Dx_1\Dx_2, \Dx_1\Dx_3,
    \Dx_1^3\in \Dan{f_4}$, so that $\Dx_1^2$ is in the socle of
    $\Apolar{f_4}$, a contradiction. Thus we may assume, after another change
    of variables, that $\beta = \Dx_2$, $q_1 = \Dx_1\Dx_2$ and $q_2 = \Dx_1\Dx_3$.
    Then $f = \DPel{x_2}{5} + f_4 + \hat{g} = \DPel{x_2}{5} + \DPel{x_1}{4} + \hat{h} + \hat{g}$,
    where $\hat{h}\in \kdp[x_2, x_3]$ and $\deg(\hat{g})\leq 3$. Then by
    Lemma~\ref{ref:topdegreetwist:lem} we may assume that $\Dx_1^2\hook
    (f-\DPel{x_1}{4}) =
    0$, so $\Apolar{f}$ is cleavable by
    Corollary~\ref{ref:stretchedhavedegenerations:cor}.

    Suppose now that $q_1 = \Dx_1^2$ and $q_2 = \Dx_1\Dx_2$. If
    $\beta$ is not a linear combination of $\Dx_1, \Dx_2$, then we may assume $\beta
    = \Dx_3$. Let $m$ in $f_4$ be any term divisible by $x_1$. Since $q_1,
    q_2\in \Dan{f_4}$, we see that $m =\lambda x_1\DPel{x_3}{3}$ for some $\lambda\in
    \kk$. But since $\beta^3\in
    \Dan{f_4}$, we have $m = 0$. Thus $f_4$ does not contain $x_1$, so
    $H_{\Apolar{f_4}}(1) < 3$, a contradiction. Thus $\beta\in \langle \Dx_1,
    \Dx_2\rangle$. Suppose $\beta = \lambda\Dx_1$ for
    some $\lambda\in \kk^*$.
    Applying Lemma~\ref{ref:144311Hilbfunc:lem} to $f_{\geq 4}$ we see
    that $\DPel{x_1}{2}$ is a derivative of $f_4$, so $\beta^2\hook f_4\neq 0$, but
    $\beta^2\hook f_4 = \lambda^2q_1\hook f_4 = 0$, a contradiction. Thus
    $\beta = \lambda_1 \Dx_1 + \lambda_2 \Dx_2$ and changing $\Dx_2$ we
    may assume that $\beta = \Dx_2$. This substitution does not change $\langle \Dx_1^2,
    \Dx_1\Dx_2\rangle$. Now we directly check that $f_4 =
    \kappa_1 x_1\DPel{x_3}{3} + \kappa_2 \DPel{x_2}{2}\DPel{x_3}{2} +
    \kappa_3 x_2\DPel{x_3}{3} + \kappa_4
    \DPel{x_3}{4}$, for some $\kappa_{\bullet}\in \kk$. Since $x_1$ is a derivative
    of $f$, we have $\kappa_1\neq 0$. Then a non-zero element
    $\kappa_2\Dx_1\Dx_3 - \kappa_1\Dx_2^2$ annihilates $f_4$. A contradiction
    with $H_{\Apolar{f_4}}(2) = 4$.
\end{proof}

\begin{lemma}\label{ref:144311addingpartial:lem}
    \def\DC{C}%
    Let a quartic $f_4\in \DP$ be such that $H_{\Apolar{f_4}} = (1, 3, 3, 3, 1)$ and $\Dx_1^3\hook f_4 =
    0$. Let $\DC = \Apolar{\DPel{x_1}{5} + f_4}$, then $H_{\DC}(2) \geq 4$.
\end{lemma}

\begin{proof}
    \def\DC{C}%
    \def\spann#1{\langle #1 \rangle}%
    \def\Dan#1{\Ann(#1)}%
    Let $\DPut{anntwo}{\mathcal{Q}} = \Dan{\DPel{x_1}{5}}_2 \subseteq \DS_2$.
    Let $I$ denote the apolar ideal of $f_4$.
    By Proposition~\ref{ref:thirdsecant:prop} we see that $I$ is minimally
    generated by three elements of degree two and two elements of degree four.
    In particular, there are no cubics in the generating set.
    Since $\Dx_1^3\in I_3$, there is an element in $\sigma\in I_2$ such that
    $\sigma = \Dx_1^2 - q$, where $q\in \Danntwo$. Therefore $\Danntwo\hook
    f_4 = \DS_2\hook f_4$.
    Moreover, $\sigma$ does not annihilate $\DPel{x_1}{2}$, so that $\DPel{x_1}{2}$ is not a
    partial of $f_4$.
    We see that $\DPel{x_1}{2}$ and $\Danntwo\hook f_4$ are leading forms of partials of
    $\DPel{x_1}{5} + f_4$, thus
    \[H_{\DC}(2) \geq 1 + \dim(\Danntwo\hook f_4) = 1 +
        \dim(\DS_2\hook f_4) = 1 +
    H_{\Apolar{f_4}}(2) = 4.\qedhere\]
\end{proof}

\begin{remark}
    \def\DC{C}%
    In the setting of Lemma~\ref{ref:144311addingpartial:lem}, it is not hard
    to deduce that $H_{\DC} = (1, 3, 4, 3, 1, 1)$ by
    analysing the possible symmetric decompositions. We do not need this
    stronger statement, so we omit the proof.
\end{remark}

\begin{lemma}\label{ref:144311final:lem}
    Let $\chark \neq 2$. Let $(\DA, \mm, \kk)$ be a finite
    local Gorenstein algebra with Hilbert function $(1, 4, 4, 3, 1, 1)$. Then
    $A$ is cleavable.
\end{lemma}

\begin{proof}
    \def\Dan#1{\Ann(#1)}%
    \def\DC{C}%
    Using Corollary~\ref{ref:basechangesmoothings:cor} we may assume $\kk =
    \kkbar$.
    Let $d = 5$ be the socle degree of $A$.  If $\Dhdvect{A, d-2}\neq (0,
    0, 0)$ then $A$ is cleavable by
    Corollary~\ref{ref:squareshavedegenerations:cor}, so we assume $\Dhdvect{A, d-2} =
    (0, 0, 0)$.
    The only possible symmetric decomposition of the Hilbert function $H_A$ with $\Dhdvect{A,
    d-2} = (0, 0, 0)$ is
    \begin{equation}\label{eq:hfdecompositionprim}
        (1, 4, 4, 3, 1, 1) = (1, 1, 1, 1, 1, 1) + (0, 2, 2, 2, 0) + (0, 1, 1,
        0).
    \end{equation}
    Let us take a dual generator $f$ of $A$. We assume that $f$ is in
    the standard form: $f = \DPel{x_1}{5} + f_4 + g$, where $f_4\in \kdp[x_1,
    x_2, x_3]$ and $\deg g\leq 3$.
    By Lemma~\ref{ref:topdegreetwist:lem} we assume that $\Dx_1^3\hook f_4 = 0$.
    Let $\DC = \Apolar{\DPel{x_1}{5} + f_4}$, then
        $H_{\DC} = (1, 3, 3, 3, 1, 1)$ by
        Proposition~\ref{ref:degreecomparison:prop} and
        Equation~\eqref{eq:hfdecompositionprim}. We analyse
    the possible Hilbert functions of $B = \Apolar{f_4}$.
    Suppose first that
    $H_{B}(1) \leq 2$. Since $H_{\DC}(1) = 3$, we have $H_B(1) = 2$ and, up to
    coordinate change, we have $f_4\in \kdp[x_2, x_3]$. Then by
    Lemma~\ref{ref:topdegreetwist:lem} we may further assume that $\Dx_1^2\hook (f
    - \DPel{x_1}{5}) = 0$. Then Proposition~\ref{ref:stretchedhavedegenerations:cor}
    asserts that $A = \Apolar{f}$ is cleavable.

    Suppose now that $H_B(1) = 3$. Since $\DPel{x_1}{5}$ is annihilated by a codimension
    one space of quadrics, we have $H_B(2) \leq H_A(2) + 1$, so there are two
    possibilities: $H_B = (1, 3, 3, 3, 1)$ or $H_B = (1,
    3, 4, 3, 1)$. By Lemma~\ref{ref:144311addingpartial:lem} the case
$H_B = (1, 3, 3, 3, 1)$ is not possible, so that $H_B = (1, 3, 4, 3, 1)$. Now by Lemma~\ref{ref:144311casenotCI:lem} we may
    consider only the case when $\pp{\Dan{f_4}}_2$ is a complete
    intersection, then by Lemma~\ref{ref:144311caseCI:lem} we have that
    $\DC$ is a complete intersection. In this case we
    will actually prove that $A$ is smoothable.

    By Example~\ref{ref:semicondegthree:example} the set $W$ of
    polynomials $f$ with fixed leading polynomial $f_{\geq 4}$
    and Hilbert function $H_{\Apolar{f}} = (1, 4, 4, 3, 1, 1)$ is irreducible. Consider
        the apolar algebra $B$ of the polynomial $\DPel{x_1}{5} + f_4 +
        \DPel{x_4}{2}x_1\in W$. Since $\Dx_1^{3} \hook f_4 = 0$, this
        polynomial is a ray sum (Definition~\ref{ref:raysum:def}). By Proposition~\ref{ref:fibersofray:prop},
        the scheme $\Spec B$ is
        the limit of smoothable schemes
        \[
            \Spec\Apolar{\DPel{x_1}{5} + f_4} \sqcup \Spec\Apolar{x_1},
        \]
        thus it is smoothable.  By Corollary~\ref{ref:CIarenonobstructed:cor}
        the scheme $\Spec B$ is unobstructed. By Lemma~\ref{ref:unobstructed:lem}, the apolar algebra of every element
        of $W$ is smoothable; in particular $A$ is smoothable.
\end{proof}

\section{Proof of Theorem~\ref{ref:cjnmain:thm} --- smoothability results}

In this section we prove that all Gorenstein algebras of degree at most $14$
are smoothable, with the exception of local algebras with Hilbert function $(1, 6, 6, 1)$.
As in the previous section, our pivotal tool are Macaulay's inverse systems,
see Chapter~\ref{sec:macaulaysinversesystems}, and in particular the symmetric
decomposition $\Dhdvect{\bullet}$ of the Hilbert function, see
Section~\ref{ssec:HFininversesystems},
Lemma~\ref{ref:sumsofdhdaraOseqence:lem} and the standard form of the dual
generator, see Section~\ref{ssec:standardforms}.

\begin{proposition}
    Let $\chark \neq 2$. Let $(\DA, \mm, \kk)$ be a finite local Gorenstein algebra of
    socle degree at most two. Then $\DA$ is smoothable.
\end{proposition}
\begin{proof}
    Using Corollary~\ref{ref:basechangesmoothings:cor} we assume
    $\kk = \kkbar$.
    If the socle degree is less than two, then $\DA =
    \Apolar{x_1}=\kk[\varepsilon]/\varepsilon^2$ or $\DA = \Apolar{1} =\kk$,
    so $\DA$ is smoothable. If $\DA$ has socle degree two, then $H_{\DA} = (1,
    n, 1)$ for some $n$ and $\DA  \simeq \Apolar{q}$, where $q\in \kdp[x_1,
    \ldots ,x_n]$ is a full rank quadric. Then $q$ is diagonalizable and
    $\DA$ is smoothable by a repeated use of
    Corollary~\ref{ref:squareshavedegenerations:cor}.
\end{proof}

\begin{proposition}\label{ref:mainthmthree:prop}
    Let $\chark \neq 2$.
    Let $(\DA, \mm, \kk)$ be a finite local Gorenstein algebra of socle degree three and
    ${H_A(2)\leq 5}$. Then $\DA$ is smoothable.
\end{proposition}

\begin{proof}
    Using Corollary~\ref{ref:basechangesmoothings:cor} we assume
    $\kk = \kkbar$.
    Suppose that the Hilbert function of $A$ is $(1, n, e, 1)$.
    By Proposition \ref{ref:squares:prop} the dual generator of $A$ may
    be put in the form $f + \DPel{x_{e+1}}{2} + \dots + \DPel{x_n}{2}$, where $f\in
    \kk[x_1,\dots,x_e]$. If $e < n$, then by Corollary~\ref{ref:squareshavedegenerations:cor} the
    scheme $\Spec \DA$ is cleavable; it is a limit of schemes of the
    form
    \[\Spec \Apolar{f} \sqcup (\Spec\kk)^{\sqcup n-e}.\]
    Thus it is smoothable if and only if $B = \Apolar{f}$ is. We have reduced
    to the case $n = e$.

    Let $e:=H_A(2)$, then $H_B = (1, e, e, 1)$.
    If $H_B(1) = e \leq 3$ then $B$ is smoothable by
    Corollary~\ref{ref:lowcodimensionsmoothability:cor}. It remains to
    consider $e=4, 5$.
    The set of points corresponding to algebras with Hilbert function $(1, e,
    e, 1)$ is irreducible in $\HilbGorarged{e}{\mathbb{A}^{2e+2}}$ by the
    argument given in Example~\ref{ex:1nn1largenonsmoothable}.
    By Lemma~\ref{ref:unobstructed:lem}, it is enough to find an unobstructed point in this set.
    The cases $e = 4$ and
    $e = 5$ are considered in Example~\ref{ref:1441:example} and Example~\ref{ref:1551:example} respectively.
\end{proof}

\begin{remark}
    The claim of Proposition~\ref{ref:mainthmthree:prop} holds true if we
    replace the assumption ${H_A(2) \leq 5}$ by $H_A(2) = 7$, thanks to the
    smoothability of finite local Gorenstein algebras with Hilbert function
    $(1, 7, 7, 1)$, see~\cite{bertone_cioffi_roggero_division_algorithm}. We will not use this
    result.
\end{remark}

\begin{lemma}\label{ref:14521case:lem}
    Let $\chark \neq 2$.
    Let $(\DA, \mm, \kk)$ be a finite local Gorenstein algebra with Hilbert function $H_A$
    beginning with $H_A(0) = 1$, $H_A(1) = 4$, $H_A(2) = 5$, $H_A(3) \leq 2$.
    Then $A$ is smoothable.
\end{lemma}

\begin{proof}
    Using Corollary~\ref{ref:basechangesmoothings:cor} we may assume $\kk =
    \kkbar$.
    Let $f$ be a dual generator of $A$ in the standard form. From
    Macaulay's Growth Theorem~\ref{ref:MacaulayGrowth:thm} it follows that $H_A(m) \leq 2$ for all $m\geq 3$,
    so that $H_A = (1, 4, 5, 2, 2,  \ldots , 2, 1,  \ldots , 1)$. Let $d$ be
    the socle degree of $A$.

    Let $\Dhdvect{A, d-2} = (0, q, 0)$ be the $(d-2)$-nd row of the symmetric
    decomposition of $H_A$. If $q > 0$, then by
    Corollary~\ref{ref:squareshavedegenerations:cor} the scheme $\Spec A$ is
    cleavable; it
    is a limit of schemes of the form $\Spec B \sqcup \Spec\kk$, such that $H_{B}(1) =
    H_A(1) - 1 = 3$. Then $\Spec B$ is smoothable by Corollary~\ref{ref:lowcodimensionsmoothability:cor}.
    Then $\Spec A$ is also smoothable. In the following we assume that $q = 0$.

    We claim that $\DPut{ffour}{f_{\geq 4}} \in \kdp[x_1, x_2]$. Indeed, the symmetric
    decomposition of the Hilbert function is either $(1, 1,  \ldots , 1) + (0,
    1,  \ldots , 1, 0) + (0, 0, 1, 0, 0) + (0, 2, 2, 0)$ or $(1, 2,  \ldots ,
    2, 1) + (0, 0, 1, 0, 0) + (0, 2, 2, 0)$. In particular $\sum_{i\geq 3} \Dhd{i}{1} = 2$, so that $\Dffour \in \kdp[x_1,
    x_2]$ and $H_{\Apolar{\Dffour}}(1) = 2$. Hence, the form $x_1$ is a
    derivative of $\Dffour $, i.e.,~there exist
    a $\partial\in \DS$ such that $\partial\hook \Dffour  = x_1$. Then we may
    assume $\partial\in \DmmS^3$, so $\partial^2\hook f = 0$.

    Let us fix $\Dffour $ and consider the set of all polynomials of the
    form $h = \Dffour  + g$, where $g\in \kdp[x_1, x_2, x_3, x_4]$ has degree at
    most three. By Example~\ref{ref:semicondegthree:example} the apolar
    algebra of a general such polynomial will have
    Hilbert function $H_A$. The set of polynomials $h$ with fixed
    $h_{\geq 4} = \Dffour $, such that $H_{\Apolar{h}} = H_A$, is irreducible.
    This set contains $h := \Dffour  + \DPel{x_3}{2}x_1 + \DPel{x_4}{2}x_3$.
    Since
    $\Apolar{\Dffour }$ is a complete intersection, it
    follows from Example~\ref{ref:1551:example} that $\Spec \Apolar{h}$ is unobstructed.
    The claim follows from Lemma~\ref{ref:unobstructed:def}.
\end{proof}

The following Theorem~\ref{ref:mainthmstretchedfive:thm} generalizes numerous
earlier smoothability results on stretched (by Sally, see~\cite{SallyStretchedGorenstein}),
 $2$-stretched (by Casnati and Notari, see \cite{CN2stretched}) and almost-stretched (by Elias and Valla, see
\cite{EliasVallaAlmostStretched}) algebras. It is important to understand
that, in contrast with the mentioned papers, it avoids a full classification of
algebras. In the course of the proof it gives some partial classification.
To the author's knowledge, this is the strongest result on smoothability of
finite Gorenstein schemes, with no restrictions on the degree.

\begin{theorem}\label{ref:mainthmstretchedfive:thm}
    Let $\chark \neq 2$.
    Let $(\DA, \mm, \kk)$ be a finite local Gorenstein algebra with Hilbert function $H_A$ satisfying
    $H_A(2) \leq 5$ and $H_{A}(3)\leq 2$. Then $A$ is smoothable.
\end{theorem}

\begin{proof}
    Using Corollary~\ref{ref:basechangesmoothings:cor} we assume
    $\kk = \kkbar$.
    We proceed by induction on $\deg A$, the case $\deg A = 1$ being trivial.
    If $A$ has socle degree three, then the result follows from
    Proposition~\ref{ref:mainthmthree:prop}. Suppose that $A$ has socle degree
    $d\geq 4$.

    Let $f$ be a dual generator of $A$ in the standard form.
    If the symmetric decomposition of
    $H_A$ has a term $\Dhdvect{d-2} = (0, q, 0)$ with $q\neq 0$, then
    Corollary~\ref{ref:squareshavedegenerations:cor} implies that $\Spec A$ is a limit of
    schemes of the form $\Spec B \sqcup \Spec\kk$, where $B$ satisfies the assumptions
    $H_B(2) \leq 5$ and $H_B(2) \leq 2$ on the Hilbert function. Then $B$ is
    smoothable by induction, so also $A$ is smoothable. Further in the proof
    we assume that $\Dhdvect{A, d-2} = (0, 0, 0)$.

    We now investigate the symmetric decomposition of the Hilbert
    function $H_A$ of the algebra~$A$. Macaulay's Growth
    Theorem~\ref{ref:MacaulayGrowth:thm} asserts that $H_A =
    (1, n, m, 2, 2, \dots, 2, 1, \dots, 1)$, where the number of ``$2$'' is
    possibly zero. If follows that the possible symmetric decompositions
    of the Hilbert function are
    \begin{enumerate}
        \item $(1, 2, 2,  \ldots , 2, 1) + (0, 0, 1, 0, 0) + (0, n-3, n-3, 0)$,
        \item $(1, 1, 1 \ldots , 1, 1) + (0, 1, 1, \ldots , 1, 0) + (0, 0, 1,
            0, 0) + (0, n-3, n-3, 0)$,
        \item $(1, 1, 1 \ldots , 1, 1) + (0, 1, 2, 1, 0) + (0, n-3, n-3, 0)$,
        \item $(1,  \ldots , 1) + (0, n-1, n-1, 0)$,
        \item $(1, 2, \ldots ,2, 1) + (0, n-2, n-2, 0)$,
        \item $(1,  \ldots , 1) + (0,1,  \ldots , 1, 0) + (0, n-2, n-2, 0)$,
    \end{enumerate}
    and that the decomposition is uniquely determined by the Hilbert function.
    In all cases we have $H_A(1)\leq H_A(2)\leq 5$, so $f\in
    \kdp[x_1, \ldots ,x_5]$.
    Let us analyse the first three cases. In each of them we have $H_A(2) = H_A(1) + 1$. If $H_A(1) \leq
    3$, then $A$ is smoothable by
    Corollary~\ref{ref:lowcodimensionsmoothability:cor}. Suppose $H_A(1) \geq
    4$. Since $H_A(2) \leq 5$,
    we have $H_A(2) = 5$ and $H_A(1) = 4$. In this case the result follows from
    Lemma~\ref{ref:14521case:lem} above.

    It remains to analyse the three remaining cases. The proof is similar to
    the proof of Lemma~\ref{ref:14521case:lem}, however here it
    essentially depends on induction.
    Let $\DPut{ffour}{f_{\geq 4}}$ be the sum of homogeneous components of $f$
    that have
    degree at least four. Since $f$ is in the standard form, we have
    $\Dffour\in \kdp[x_1, x_2]$.
    By Proposition~\ref{ref:degreecomparison:prop},
    the decomposition of the Hilbert function $\Apolar{\Dffour}$ is
    one of the decompositions $(1, \ldots ,1)$, $(1,2 \ldots ,2,1)$, $(1, \ldots ,1) + (0, 1,  \ldots
    , 1, 0)$, depending on the decomposition of the Hilbert function of
    $\Apolar{f}$.

    Let us fix a vector $\hat{h} = (1, 2, 2, 2,  \ldots , 2, 1, 1, \ldots , 1)
    $ and take the sets
    \[
        V_1 := \left\{ f\in \kk[x_1, x_2]\ |\ H_{\Apolar{f}} = \hat{h}\right\}\mbox{
        and } V_2 := \left\{ f\in \kk[x_1, \ldots ,x_n]\ |\ f_{\geq 4}\in V_1 \right\}.
    \]
    By Proposition~\ref{ref:irreducibleintwovariables:prop} the set $V_1$ is
    irreducible and thus $V_2$ is also irreducible.
    The Hilbert function of the apolar algebra of a general member of $V_2$ is,
    by Example~\ref{ref:semicondegthree:example}, equal to $H_A$. It remains
    to show that the apolar algebra of this general member is
    smoothable.

    Proposition~\ref{ref:irreducibleintwovariables:prop} implies that the general
    member of $V_2$ has, after a nonlinear change of coordinates, the form
    $f = \DPel{x_1}{{d}} + \DPel{x_2}{{d_2}} + g$ for some $g$ of
    degree at most three.  Using Lemma \ref{ref:topdegreetwist:lem} we may
    assume, after another
    nonlinear change of coordinates, that $\Dx_1^2\hook g = 0$.

    Let $B := \Apolar{\DPel{x_1}{{d}} + \DPel{x_2}{{d_2}} + g}$. We will show that $B$ is
    smoothable.
    Since $d \geq 4 = 2\cdot 2$
    Proposition~\ref{ref:stretchedhavedegenerations:cor} shows that $B$
    is cleavable. Analysing the fibers of
    the resulting deformation, as
    in Example~\ref{ref:stretched:example}, we see that they have the form
    $\Spec (B' \times \kk)$, where $B' = \Apolar{h}$ and $h =
    \lambda^{-1}\DPel{x_1}{{d-1}} + \DPel{x_2}{{d_2}} + g$.
    Then $H_{B'}(3) = H_{\Apolar{h_{\geq 4}}}(3) \leq 2$. Moreover,
    $h\in \kdp[x_1, \ldots ,x_5]$, so that $H_{B'}(1) \leq 5$. Now
    analysing the possible symmetric decompositions of $H_{B'}$, which are
    listed above, we see that $H_{B'}(2)\leq H_{B'}(1) = 5$.
    It follows from induction on the degree that $B'$ is smoothable, thus
    $B'\times \kk$ and $B$ are smoothable.
\end{proof}

\begin{proposition}\label{ref:mainthmsfour:prop}
    Let $\chark \neq 2$.
    Let $(\DA, \mm, \kk)$ be a finite local Gorenstein algebra of socle degree four satisfying $\deg
    A\leq 14$. Then $A$ is smoothable.
\end{proposition}

\begin{proof}
    Using Corollary~\ref{ref:basechangesmoothings:cor} we assume
    $\kk = \kkbar$.
    We proceed by induction on the degree of $A$. By
    Proposition~\ref{ref:mainthmthree:prop} we may assume that all algebras
    of socle degree \emph{at most} four and degree less than $\deg A$ are
    smoothable.

    If $\Dhdvect{A, d-2} = (0, q, 0)$ with $q\neq 0$, then by
    Corollary~\ref{ref:squareshavedegenerations:cor} the scheme $\Spec A$ is a limit of
    schemes of the form $\Spec A'
    \sqcup \Spec\kk$, where $A'$ has socle degree four. Hence $A$ is
    smoothable. Therefore we assume $q = 0$. Then $H_A(1) \leq 5$ by Lemma
    \ref{ref:hilbertfunc:lem}. Moreover, we assume $H_A(1) \geq 4$ since
    otherwise $A$ is smoothable by
    Corollary~\ref{ref:lowcodimensionsmoothability:cor}.

    The symmetric
    decomposition of $H_A$ is $(1, n, m, n, 1) + (0, p, p, 0)$ for some $n, m,
    p$.
    Clearly, $n\leq 5$. A unimodality result by Stanley, see
    \cite[p.~67]{StanleyCombinatoricsAndCommutative}, asserts that
    $n\leq m$. Since $\deg A\leq 14$, we have $n\leq 4$ and $H_A(2) \leq H_A(1)\leq 5$.
    We have four cases: $n = 1$, $2$, $3$,
    $4$ and five possible shapes of Hilbert functions: $H_A = (1, *, *, 1, 1)$,
    $H_A = (1, *, *, 2, 1)$, $H_A = (1, 4, 4, 3, 1)$, $H_A = (1, 4, 4, 4,
    1)$, $H_A = (1, 4, 5, 3, 1)$.

    The conclusion in the first two cases follows from Theorem
    \ref{ref:mainthmstretchedfive:thm}.
    In the remaining cases we first look for a suitable irreducible set of
    dual generators parameterizing algebras with prescribed $H_A$.
    We examine the case $H_A = (1, 4, 4, 3,
    1)$.
    Consider the locus of $f\in \DP = \kdp[x_1, x_2, x_3, x_4]$ in the
    standard form that are generators of algebras with
    Hilbert function $H_A$. We claim that this locus is irreducible. Since the leading form $f_4$ of
    $f$ from this locus has Hilbert function $(1, 3, 3, 3, 1)$, the locus of possible
    leading forms is irreducible by Proposition~\ref{ref:thirdsecant:prop}.
    Then the irreducibility follows from
    Example~\ref{ref:semicondegthree:example}. The irreducibility in the cases
    $H_A = (1, 4, 4, 4, 1)$ and $H_A = (1, 4, 5, 3, 1)$ follows similarly from
    Proposition~\ref{ref:fourthsecant:prop} together with
    Example~\ref{ref:semicondegthree:example}.
    In the first two cases we see that $f_4$ is a sum of powers of variables,
    then Example~\ref{ref:quarticlimitreducible:example} shows
    that the apolar algebra $A$ of a general $f$ is cleavable. More
    precisely,
    $\Spec A$ is limit of schemes of the form $\Spec A' \sqcup \Spec\kk$, where $A'$ has socle
    degree at most four (compare Example~\ref{ref:stretched:example}). Then
    $\Spec A$ is smoothable.
    In the last case Example~\ref{ref:14531case:example} gives an unobstructed
    algebra in this irreducible set.
    By Lemma~\ref{ref:unobstructed:lem} this completes the proof.
\end{proof}

\begin{proof}[Proof of Theorem~\ref{ref:cjnmain:thm}]
    Let $A = H^0(R, \OO_R)$.
    By Theorem~\ref{thm_equivalence_of_abstract_and_embedded_smoothings} and
    Corollary~\ref{ref:basechangesmoothings:cor} it is enough to consider
    local algebras over $\kk = \kkbar$, each such algebra has residue field
    $\kk$.
    We do induction on the degree of $\DA$.
    \def\Dh#1{H(#1)}%
    Let $(\DA, \mm, \kk)$ be a local algebra of degree at most $14$ and of socle degree $d$.
    By $H$ we denote the Hilbert function of $A$. By induction it is enough to
    prove that $\Spec\DA$ is cleavable. Suppose it is not so.
%    As mentioned in Section~\ref{sss:smoothability} it is enough to prove
%    $A$ is cleavable. Suppose it is not so.
    By Corollary~\ref{ref:squareshavedegenerations:cor} we
    have $\Dhdvect{A, d-2} = (0, 0, 0)$. Then by Lemma~\ref{ref:hilbertfunc:lem} we see that either $H = (1,
    6, 6, 1)$ or $\Dh{1} \leq 5$. It is enough to consider $\Dh{1}\leq 5$.
    If $d = 3$ then $\Dh{2} \leq \Dh{1} \leq 5$, so by Proposition
    \ref{ref:mainthmthree:prop} we assume $d > 3$. By Proposition
    \ref{ref:mainthmsfour:prop} it follows that
    we may consider only $d\geq 5$.

    If $\Dh{1}\leq 3$ then $A$ is smoothable by
    Corollary~\ref{ref:lowcodimensionsmoothability:cor}, thus we
    assume $\Dh{1} \geq 4$. By Lemma \ref{ref:trikofHilbFunc:lem} we
    see that $\Dh{2} \leq 5$. Then by Theorem~\ref{ref:mainthmstretchedfive:thm} we reduce to the
    case $\Dh{3} \geq 3$. By Macaulay's Growth Theorem we have $\Dh{2} \geq 3$.
    Then $\sum_{i>3} \Dh{i} \leq 14 - 11$, so we are left with several
    possibilities: $H = (1, 4, 3, 3, 1, 1, 1)$, $H = (1, 4, 3, 3, 2, 1)$ or
    $H = (1, *, *, *, 1, 1)$.
    In the first two cases it follows from the symmetric decomposition that
    $\Dhdvect{A, d-2} \neq (0, 0, 0)$ which is a contradiction. We examine the
    last case.
    By Lemma~\ref{ref:14341notexists:lem} there does not exist an algebra with Hilbert function $(1, 4, 3,
    4, 1, 1)$. Thus the only possibilities are $(1, 4, 3, 3, 1, 1)$,
    $(1, 5, 3, 3, 1, 1)$ and $(1, 4, 4, 3, 1, 1)$. Once more, it is
    checked directly
    that in the first two cases $\Dhdvect{A, d-2} \neq (0, 0, 0)$. The
    last case is the content of Lemma~\ref{ref:144311final:lem}.
\end{proof}

\section{Proof of Theorem~\ref{ref:14pointsmain:thm} --- the nonsmoothable
component {(1,6,6,1)}}\label{ssec:14pointsproof}

\renewcommand{\VV}{V}%
In this section we make the following global assumption.
This assumption is used only once, when referring to the work of Iliev-Ranestad
in Proposition~\ref{ref:dimensionofiliev:prop}.
\begin{assumption}
    The field $\kk$ has characteristic zero.
\end{assumption}

We take a six-dimensional $\kk$-vector space $\VV$ and endow it with an
affine space structure given by $H^0(\VV, \OO_{\VV}) = \Sym \Ddual{V}$. We
prefer $V$ to $\mathbb{A}^6$, since the proofs are more transparent when done
in a coordinate free manner. We endow $\Sym \VV$ with a \emph{divided power}
polynomial ring structure, as in Definition~\ref{ref:dividedpows:def} and
consider the action of $\Sym \Ddual{\VV}$ on $\Sym \VV$, as in
Definition~\ref{ref:contraction:propdef}. The symbol $\Sym \VV$ is an abuse of
notation --- one would expect this to be a polynomial ring --- however the
ring structure will be used only in
Lemma~\ref{ref:quartics:lem}, so we keep this intuitive notation. In
characteristic zero $\Sym \VV$, with its divided power structure, is
isomorphic to a polynomial ring, see Proposition~\ref{ref:divpowerispoly:prop}.

We begin with constructing the subset of $\Hilbshort = \HilbGorarged{14}{\VV}$,
which, as we prove later, is the intersection $\Hilbshortzero \cap
\Hilbshortother$. The key ingredient is the Iliev-Ranestad divisor, introduced
in~\cite{Ranestad_Iliev__VSPcubic, Ranestad_Iliev__VSPcubic_addendum}.

\newcommand{\wedgetwo}{\Lambda^2 \VV}%
\newcommand{\wedgetwodual}{\Lambda^2 \Ddual{\VV}}%
\newcommand{\cone}[1]{\operatorname{cone}(#1)}%
\newcommand{\GLV}{\operatorname{GL}(\VV)}%
\renewcommand{\Gtwosix}{\Grass(2, \VV)}%
\renewcommand{\cubicspace}{\mathbb{P}\left(\Sym^3 \VV\right)}%
\renewcommand{\cubicspacereg}{\mathbb{P}\left(\Sym^3 \VV\right)_{1661}}%

\paragraph{The Iliev-Ranestad divisor.} The Grassmannian $\Gtwosix\subset \mathbb{P}(\wedgetwo)$ is
    non-degenerate, arithmetically Gorenstein and of degree $14$. A general
    $\mathbb{P}^5 = \mathbb{P}W$ does not intersect it. For such $\mathbb{P}W$
    the cone
    \[
        R = W \cap \cone{\Gtwosix}\subset \wedgetwo
    \]
    is a finite standard
    graded Gorenstein scheme $R$ supported at the origin.
    For a general $\mathbb{P}^6$ containing general such $\mathbb{P}W$, the
    intersection $\mathbb{P}^6 \cap \Gtwosix$ is a set of $14$ points and $R$
    is a hyperplane section of the cone over these points, thus $R$ is
    smoothable.
    \newcommand{\quotient}{\mathbin{\!/\mkern-4mu/\!}}%
    Since $R$ spans $\VV$ and is of degree $14$, one checks, using
    the symmetry of the Hilbert function, that $R$ has Hilbert
    series $(1, 6, 6, 1)$. Therefore $R =
    \Apolar{F}$ for a cubic $F\in \Sym^3 W^* \simeq \Sym^3 \kk^6$,
    unique up to scalars and $R$
    gives a well defined element $[F_R]\in \cubicspace \quotient \GLV$.
    Denote by $\DPut{DIRmod}{\mathcal{R}}$ the set of such $[F_R]$ obtained
    from all admissible $\mathbb{P}^5 = \mathbb{P}W$ and by $\DIR \subset \cubicspace$
    the closure of the preimage of $\DDIRmod$. The subvariety $\DIR$ it is
    called the \emph{Iliev-Ranestad} divisor, see~\cite{Ranestad_Voisin__VSP}.
    By Proposition~\ref{ref:flatfamiliesforconstructible:prop} we obtain a map
    $\varphi\colon \DIR\to \Hilbshort$, whose image is set-theoretically contained
    in $\Hilbshortzero \cap \exceptional$.

    \begin{proposition}\label{ref:dimensionofiliev:prop}
        The closure of $\varphi(\DIR) \subset \exceptional$ has dimension
        $54$, hence is a divisor in $\exceptional$.
    \end{proposition}
    \begin{proof}
        For $\kk = \mathbb{C}$ it is proven in~\cite[Lemma~1.4]{Ranestad_Iliev__VSPcubic} that
        $\DDIRmod$ is a divisor in the moduli space of cubic fourfolds, hence
        $\dim \DDIRmod = 19$ and $\dim \varphi(\DIR) = \dim \DIR = 19 + 35 =
        54$. Since $\exceptional$ has dimension $55$, the claim follows in the
        case $\kk = \mathbb{C}$. The claim follows for $\kk = \mathbb{Q}$ and
        then for all other fields of characteristic zero by base change,
        see~\cite[(5)~pg.~112]{fantechi_et_al_fundamental_ag}.
    \end{proof}
    \begin{remark}
        It is proven in \cite[Lemma~2.7]{Ranestad_Voisin__VSP}
        for $\kk = \mathbb{C}$ that
        $F$ lies in $\DIR$ if and only if it is apolar to a quartic scroll.
    \end{remark}

    \paragraph{Prerequisites.}
    We will now rigorously prove several claims which together lead to the
    proof of Theorem~\ref{ref:14pointsmain:thm}. Our
    approach is partially based on the natural method
    of~\cite{cartwright_erman_velasco_viray_Hilb8}.
    Additional (crucial) steps are proving that $\Hilbshort \setminus
    \Hilbshortzero$ is smooth and that $\Hilbshortzero \cap \Hilbshortother$
    is irreducible.

    In the first two steps we use the following abstract observation.
    \begin{lemma}\label{ref:abstract:lem}
        Let $X$ and $Y$ be reduced, finite type schemes over $\kk$. Let
        $X\to Y$, $Y\to X$ be two morphisms, which are bijective on closed
        points. If the composition $X\to Y\to X$ is equal to
        identity, then $X\to Y$ is an isomorphism.
    \end{lemma}
    \begin{proof}
        Denote the morphisms by $i\colon X\to Y$ and $\pi\colon Y\to X$.
        The scheme-theoretical image of $i$ contains all closed points, hence
        is the whole $Y$. Therefore the pullback of functions via $i$ is
        injective. It is also surjective, since the pullback via composition
        $\pi\circ i$ is the identity. Hence $i$ is an isomorphism.
    \end{proof}
     In our setting, $X$ is a subset of the Hilbert scheme, $Y$ a
     subspace of polynomials and the maps are constructed using relative
     Macaulay's inverse systems.

     Below we precisely explain the freedom of choice of a dual generator of
     an algebra with Hilbert function $(1, n, n, 1)$.
    \begin{remark}\label{ref:uniquenessofcubics:rem}
        Let $F\in \DP$ be a polynomial of degree three such that $H_{\Apolar{F}} = (1, n,
        n, 1)$ where $n = \dim \DP_1$; in other words the Hilbert function is maximal.
        The ideal $\Ann{F}$ and hence $\Apolar{F}$ only partially depends
        on the lower homogeneous components of $F$. To see this,
        write explicitly the $\DS$-module $\DS\hook F$.
        \begin{align*}
            \DS\hook F &= \sspan{F,\ \{\Dx_i\hook F\ |\ i\}, \ \{\Dx_i\Dx_j\hook
            F\ |\ i, j\},\   \left\{ \Dx_i\Dx_j\Dx_k\hook F\ |\ i, j, k \right\}}
            =
            \\ &=
            \sspan{F,\ \{\Dx_i\hook F\ |\ i\}} \oplus \DP_{\leq 1} =
            \sspan{F_3 + F_2}\oplus \sspan{\Dx_i\hook F_3\ |\ i} \oplus
            \DP_{\leq 1}.
        \end{align*}
        Therefore $\DS\hook F$, as a submodule of $\DP$, is uniquely defined by giving $F_3$ and the
        class $[F_2
        \mod \sspan{\Dx_i\hook F_3\ |\ i}]$, up to multiplication by a constant.
    \end{remark}

    \paragraph{Identification of $\exceptional$ with an open subset of
    $\cubicspace$.}
    \begin{claim}\label{claim:exceptional}
        The map $\varphi\colon \cubicspacereg \to \exceptional$ is an isomorphism.
    \end{claim}
    \begin{proof}
        Let $[R]\in \exceptional$. Each fiber of the universal family
        over $\exceptional$ is $\kk^*$-invariant thus the whole family is
        $\kk^*$-invariant. By Local Description of Families and especially
        Remark~\ref{ref:localdescriptiongraded:rmk}, near
        $[R]$ this family has the form $\Spec \Apolar{F}
        \to \Spec B$ for some $F\in B\tensor \Sym^3 \VV$, so that
        $[F]$ gives a morphism $\Spec B \to
        \cubicspacereg$ which is locally an inverse to $\cubicspacereg \to
        \exceptional$. The claim follows from Lemma~\ref{ref:abstract:lem}.
    \end{proof}
    We will abuse the notation and speak
    about elements of $\cubicspacereg$ being smoothable~etc.
    We will also identify $\DIR$ with $\varphi(\DIR\cap \cubicspacereg)$.
    Note that the codimension of complement of $\cubicspacereg \subset \cubicspace$ is
    greater than one, so divisors on these spaces are identified.

    \newcommand{\Hilbotherred}{\reduced{\Hilbshortother}}%
    \paragraph{The bundle $\Hilbotherred \to \exceptional$.}
    We now show how the questions
    about $\Hilbshortother$ reduce to the questions about $\exceptional$. Note that
    we will work on the reduced scheme $\Hilbotherred$, which eventually turns out
    to be equal to $\Hilbshortother$.
    \begin{claim}\label{claim:vectorbundle}
        The scheme $\Hilbotherred$ is a rank $21$ vector bundle over $\exceptional$ via a
        map
        \[\pi\colon \Hilbotherred \to \exceptional.\] On the level of
        points $\pi$ maps $[R]$ to $\Spec \gr H^0(R, \OO_R)$ supported at the origin
        of $\VV$. The schemes corresponding to points in the same fiber of
        $\pi$ are isomorphic.
    \end{claim}
    \newcommand{\tr}[1]{\operatorname{tr}(#1)}%
    \begin{proof}
        \def\ap#1{\pi_{#1}}%
        \newcommand{\supp}{\operatorname{supp}}%
        \def\Funiv{\mathcal{F}}%
        \def\polyspacereg{\Sym^{\leq 3}_{max} \VV}%
        \def\cubiccone{\Sym^{3}_{max} \VV}%
        First we recall the support map, as defined
        in~\cite[Section~5A]{cartwright_erman_velasco_viray_Hilb8}.
    Consider the universal family $\UU\to \Hilbshort$, which is flat.
    The multiplication by $\Ddual{V}$ on $\OO_{\UU}$ is
    $\OO_{\Hilbshort}$-linear. The relative trace of such multiplication
    defines a map
    $\Ddual{V}\to H^0(\Hilbshort, \OO_{\Hilbshort})$, thus a morphism
    $\Hilbshort\to \VV$. We restrict this morphism to $\Hilbshortother \to
    \VV$ and compose it with multiplication by
    $\frac{1}{14}$ on $\VV$ to obtain a map denoted $\supp$. If $[R]\in
    \Hilbshortother$ corresponds to a scheme supported at
    $v\in \VV$, then for every $v^*\in \Ddual{V}$ the multiplication by $v^* -
    v^*(v)$ is nilpotent on $R$, hence traceless. Thus on $R$, we have $\tr{v^*} = \tr{v^*(v)}
    = 14v^*(v)$ and $\supp([R]) = v$ as expected.

    The support morphism $\supp:\Hilbshortother \to \VV$ is $(\VV, +)$
    equivariant, thus it is a trivial vector bundle:
    \[\Hilbshortother  \simeq \VV \times \supp^{-1}(0).\]
    \def\Hilbshortfiber{\Hilbshortother^0}%
    Restrict $\supp$ to $\Hilbotherred$ and consider the fiber
    $\Hilbshortfiber := \supp^{-1}(0)$. Since $\Hilbotherred$
    is reduced, also $\Hilbshortfiber$ is reduced.
    We will now use this in an essential way. By Local
    Description~\ref{ref:localdescriptionGorenstein:cor}, the universal
    family over this scheme locally has the form $\ap{F}:\Spec \Apolar{F} \to
    \Spec B$ for some $F\in B\tensor V$. For every $p\in \Spec B$ we
    have $\deg F(p)\leq 3$ and $B$ is reduced, so $\deg F\leq 3$. Let $F_3$ be the leading form. The fibers of
    $\gr{\ap{F}}:\Spec \Apolar{F_3}\to \Spec B$ and $\ap{F}$ are isomorphic. Since $B$ is
    reduced, by Proposition~\ref{ref:apolarflatness:prop} the family
    $\gr{\ap{F}}$ is also flat and gives a morphism
    $\Spec B \to \exceptional$. These morphisms glue to give a morphism
    \[
        \gr:\Hilbshortfiber \to \exceptional.
    \]
    Now we prove that $\gr$ makes $\Hilbshortfiber$ a vector bundle
    over $\exceptional$ of rank $15$.
%    Fix a point in $\exceptional$. After restricting to its neighborhood
%    $T = \Spec B$, by Local Description~\ref{ref:localdescriptiongraded:rmk}, we
%    have a cubic $F\in B\tensor \left(\Sym^3 V\right)_{1661}$ such that the
%    universal family over $T$ is $\Spec \Apolar{F} \to T$.
%    The annihilator $I = \Ann(F) \subset B\tensor \Sym \Ddual{V}$ is generated
%    by $15$ quadrics $\theta_i$ and higher other elements. Pick $15$ quadrics
%    quadrics $q_i$ dual to $\theta_i$, then there is a finitely flatly
%    embedded apolar family $\Spec \Apolar{F + \lambda_i q_i} \to T \times
%    \mathbb{A}^{15}$ where $\lambda_i$ are coordinates on $\mathbb{A}^{15}$
%    and we obtain a morphism $T \times \mathbb{A}^{15} \to \gr^{-1}(T)$
%    bijective on points. By Local
%    Description~\ref{ref:localdescriptionGorenstein:cor} we may construct its
%    inverse, thus providing a trivialization $\gr^{-1}(T)  \simeq T \times
%    \mathbb{A}^{15}$.

    \def\EEker{\mathcal{K}}%
    \def\EE{\mathcal{E}}%
    \def\PEE{\mathbb{P}\EE}%
    \def\SSplus{\SS^{+}}%
    Let $U = \polyspacereg := \cubiccone + \Sym^{\leq 2} \VV$ be the space
    of degree three polynomials with apolar algebras of degree $14$. By
    Proposition~\ref{ref:flatfamiliesforconstructible:prop} we have a morphism
    $\varphi\colon U \to \Hilbshortfiber$ which is a surjection on points.
    This surjection comes from a flatly embedded apolar family $\Spec
    \Apolar{\Funiv}\to U$, where $\Funiv\in \Gamma(U)\tensor \polyspacereg$ is a
    universal cubic.
    For a point $u\in U$, we have $\gr \circ \varphi(\Funiv(u)) =
    [\Funiv_3(u)]$, so
    $U$ becomes a trivial vector bundle of rank $1 + 6 +
    \binom{7}{2} = 28$ over the cone $\cubiccone$ over $\cubicspacereg$.

    We will prove that $\Hilbshortfiber$ is a projectivisation of a quotient bundle of this bundle.
    Take a subbundle $\EEker$ of $U$ whose fiber over $F_3\in \cubiccone$
    is $(\Sym^{\geq 1} \Ddual{V})\hook F_3$.
    The apolar algebra depends only on class of element modulo $\EEker$
    by Remark~\ref{ref:uniquenessofcubics:rem}, so
    that the family $\Spec\Apolar\Funiv$ over $U$
    is pulled back from the quotient bundle $U/\EEker$, which we denote by
    $\EE$. Hence also the associated morphism $U\to \Hilbshortfiber$ factors
    as $U\to \EE \to \Hilbshortfiber$.
    Finally we may projectivize these bundles:
    we replace the
    polynomials in $\EE$ by their classes, obtaining a bundle over
    $\cubicspacereg$ which we denote, abusing notation, by $\PEE$. The
    morphism $\EE\to \Hilbshortfiber$
    factors as $\EE\to \PEE \to \Hilbshortfiber$ and we obtain
    \[
        \bar{\varphi}:\PEE\to \Hilbshortfiber,
    \]
    which is bijective on points.

    By the Local Description~\ref{ref:localdescriptionGorenstein:cor}, for every $[R]\in \Hilbshortfiber$ we
    have a neighborhood $U$ so that the universal family is $\Spec
    \Apolar{F}\to U$ for $F\in H^0(U, \OO_U)\tensor \polyspacereg$. Then $F$ gives a map $U\to
    \polyspacereg$, thus $U\to \PEE$. This is a local inverse of
    $\bar{\varphi}$. Hence by Lemma~\ref{ref:abstract:lem} the variety $\Hilbshortfiber$ is isomorphic to the bundle
    $\PEE$ over $\exceptional$.
    To prove Claim~\ref{claim:vectorbundle} we define $\pi$ to be the
    composition of projection $V \times \Hilbshortfiber \to \Hilbshortfiber$
    and $\gr$. Since the former is a \emph{trivial} vector bundle and the
    latter is a vector bundle the composition is a vector bundle as well.

    Finally note that $\pi([R])$ is isomorphic to the scheme $\Spec \gr
    H^0(R, \OO_R)$, which in turn is (abstractly) isomorphic to $R$
    by~Corollary~\ref{ref:eliasrossi:cor}. Hence all the schemes corresponding to points
    in the same fiber are isomorphic.
\end{proof}

\begin{corollary}\label{ref:divisorinlocus:cor}
    The locus $\Hilbshortother \cap \Hilbshortzero \subset \Hilbshortother$
    contains a divisor, which is equal to $\pi^{-1}(\DIR)$, where $\DIR
    \subset \cubicspacereg$ is the restriction of the Iliev-Ranestad divisor.
\end{corollary}
\begin{proof}
    By its construction, the divisor $\DIR \subset \cubicspacereg  \simeq \exceptional$
    parameterizes smoothable schemes. By Claim~\ref{claim:vectorbundle} all
    schemes in $\pi^{-1}(\DIR)$ are smoothable, hence $\pi^{-1}(\DIR)$ is contained in
    $\Hilbshortother\cap \Hilbshortzero$. Again by Claim~\ref{claim:vectorbundle} this preimage is divisorial in
    $\Hilbshortother$.
\end{proof}

    \paragraph{The scheme $\Hilbshortother\setminus\Hilbshortzero$ is
    smooth, so $\Hilbshortother$ is reduced.}
    \newcommand{\Rzero}{R_0}%
    \newcommand{\Izero}{I}%

    Let $R \subset V$ be a finite irreducible Gorenstein scheme with
    Hilbert function $(1, 6, 6, 1)$. Let $\DS := H^0(V, \OO_V) = \Sym V^*$ and
    $A = H^0(R, \OO_R)$, then $A = \DS/I$. The tangent space to $\Hilbshort$
    at $[R]$ is isomorphic to $\Homthree{\DS/I}{I/I^2}{\DS/I}$. Since $R$ is
    Gorenstein, this space is dual to $I/I^2$, see
    Example~\ref{ex:tangentforGorenstein}. Note that
    $R$ is isomorphic to $\Rzero = \Spec \gr A$ and $[\Rzero]\in \exceptional$.

    \begin{claim}\label{claim:singularlocus}
        $\Hilbshortother \cap \Sing \Hilbshort = \Hilbshortother \cap
        \Hilbshortzero = \pi^{-1}(\DIR)$ as sets. Therefore $\Hilbshortother$ is reduced.
        Moreover $\Hilbshortother \cap \Hilbshortzero \subset \Hilbshortother$ is
        a prime divisor.
    \end{claim}

    Being a singular point of $\Hilbshort$ and lying in
    $\Hilbshortzero$ are both independent of the embedding of a finite scheme
    by Proposition~\ref{ref:invarianceoftangentspace:prop} and
    Theorem~\ref{thm_equivalence_of_abstract_and_embedded_smoothings}
    respectively.
    Hence all three sets appearing in the equality of
    Claim~\ref{claim:singularlocus} are preimages of their images in
    $\exceptional$. Therefore it is enough to prove
    the claim for elements of $\exceptional$.

    Take $[\Rzero]\in
    \exceptional$  with corresponding homogeneous ideal $\Izero$.
    Take $F\in \Sym^3 \VV$ so that $\Izero = \Ann(F)$.
%    Suppose that $\Rzero
%    \not \in \Hilbshortzero$.
    The point $[\Rzero]$ is smooth if and only if $\dim \DS/\Izero^2 = 76 + 14 = 90$.
    Consider the Hilbert series $H$ of $\DS/\Izero^2$.
    By degree reasons, $\Izero^2$
    annihilates $\Sym^{\leq 3} \VV$. We now show that it annihilates also a
    $6$-dimension space of quartics. Notabene, by
    Example~\ref{ex:tangentvectors}, this space is the tangent space to
    deformations of $\DS/\Izero$ obtained by moving its support in $\VV$.
    \begin{lemma}\label{ref:quartics:lem}
        The ideal $\Izero^2$ annihilates the space $\VV\cdot F\subset \Sym^4
        \VV$.
    \end{lemma}
    \begin{proof}
        Let $\Dx\in \Ddual{V}$ and $x\in V$ be linear forms. Then $\Dx$ acts
        on the divider power polynomial ring $\Sym V$
        as a derivation, so that $\Dx\hook (xF) = (\Dx\hook x) F + x(\Dx\hook
        F) \equiv x(\Dx \hook F) \mod \DS \hook F$.
        Take any element $i\in \Izero_2$ and write it as $i = \sum
        \beta_i\beta_j$ with $\beta_i$ linear. Then
        \[
            i\hook (xF) = \sum \beta_i \hook (\beta_j\hook xF) \equiv \sum
            x(\beta_i\beta_j \hook F) = x(i\hook F) = 0\mod (\DS \hook F).
        \]
        Therefore $i\hook (xF) \in \DS\hook F$, hence is annihilated by
        $\Izero$. This proves that $\Izero_2\Izero$ annihilates $xF$. Other
        graded parts of $I^2$ annihilate $xF$ by degree reasons.
    \end{proof}

    By Lemma~\ref{ref:quartics:lem} and the discussion above we have
    $H_{\DS/I^2} = (1, 6, 21, 56, r, *)$ with $r\geq 6$.
    Therefore $\sum H_{\DS/I^2} = 84 + r + *$ and this equals $90$ if and only if $r = 6$ and $*$ consists of zeros.
    Now we show that if $r = 6$ then $*$ consists of zeros. For $J \subset
    \DS$ a homogeneous ideal, by $J^{\perp}_d$ we denote the forms in $\DP_d$
    annihilated by $J$, so that $\dim J^{\perp}_d + \dim J_d = \dim \DS_d$.
    \begin{lemma}\label{ref:onlyquatricsmatter:lem}
        Let $F\in \Sym^3 V$ and $\Izero = \Ann(F)\subset \DS$ be as above. Suppose that
        \[
            \dim\Ann\left(\Izero^2\right)^{\perp}_4 = 6.
        \]
        Then ${\Sym^5 \Ddual{V} \subset
        \Izero^2}$. In particular $H_{S/\Izero^2} = (1, 6, 21, 56, 6, 0)$, so
        that the tangent space to $\Hilbshort$ at $[\DS/\Izero]$ has dimension $76$.
        As a corollary, %, if $[\DS/\Izero]\not\in \Hilbshortzero$, then
        $[\DS/\Izero]$ is singular if and only if
        $\dim\left(\Izero^2\right)^{\perp}_4 > 6$.
    \end{lemma}
    \begin{proof}
        Suppose $\Sym^5 \Ddual{V} \not\subset \Izero^2$ and take non-zero $G\in \Sym^5
        \Ddual{V}$ annihilated by this ideal.
        By assumption $\dim\left(\Izero^2\right)^{\perp}_4 = 6$ and by
        Lemma~\ref{ref:quartics:lem} the $6$-dimensional space $V F$ is perpendicular to
        $\Izero^2$. Therefore $\left(\Izero^2\right)^{\perp} = V F$ and
        hence $\Ddual{V}\hook G \subset V F$.

        We first show that all linear forms are partials of $G$, in other
        words that $V \subset \DS\hook G$.
        Clearly ${0\neq \Ddual{V} \hook G \subset V F}$.
        Take a non-zero $x\in V$ such that
        $xF$ is a partial of $G$.
        Let $W^* = \left(x^{\perp}\right)_1 \subset \Ddual{V}$ be the
        space perpendicular to $x$.
        Let $xF = \DPel{x}{e+1}\tilde{F}$, where $\tilde{F}$ is not divisible by
        $x$. Then there exists an element $\sigma$ of $\Sym W^*$ such that
        $\sigma\hook (xF) = \DPel{x}{e+1}$. In particular $x\in \DS\hook (xF)$.
        Moreover $\DS_3\hook (xF) \equiv \DS_2\hook F = V \mod \kk x$.
        Therefore, $V \subset \DS\hook G$, so $V = \DS_4\hook G$.
        By symmetry of the Hilbert function, $\dim \Ddual{V}\hook G = \dim
        \DS_4\hook G =
        6$. Since $\Ddual{V}\hook G$ is annihilated by $\Izero^2$, by comparing
        the dimensions we conclude that
        \[\Ddual{V}\hook G = \VV F.\]

        Since $\Izero_3\hook(\Izero_2\hook G) = 0$ and $I_3^{\perp} =
        \kk F$, we have $\Izero_2\hook G \subset
        \kk F$, so $\dim \left(\DS_2\hook G + \kk F\right) \leq 6 + 1= 7$.
        For every linear $\Dx\in \Ddual{V}$ and $y\in V$ we have $\Dx\hook (yF)
        = (\Dx\hook y)F + y(\Dx\hook F) \equiv y(\Dx\hook F) \mod \kk F$.
        Therefore we have
        \[
            \DS_2 \hook G = \Ddual{V}\hook (\Ddual{V}\hook G) = \Ddual{V}\hook (\VV F)
            \equiv V(\Ddual{V}\hook F)\mod
            \kk F,
        \]
        thus $\dim V  (\Ddual{V}\hook F) \leq 7$. Take any two quadrics $q_1,
        q_2\in \Ddual{V}\hook F$. Then $V q_1 \cap V q_2$ is non-zero, so that
        $q_1$ and $q_2$ have a common factor. We conclude that $\Ddual{V}\hook F =
        y V$ for some $y$, but then $\dim V(\Ddual{V}\hook F) = \dim y\Sym^2 V >
        7$, a contradiction.
    \end{proof}

    \def\Ffamily{\mathcal{F}}%
    \def\Itwobundle{\mathcal{I}_2}%
    \def\Singdiv{E}%
    By Claim~\ref{claim:exceptional}, the map $[F]\to \Spec \Apolar{F}$ is an isomorphism
    $\cubicspacereg \to \exceptional$.
    For $R_0 = \Spec \Apolar{F}$ consider the statements
    \begin{itemize}
        \item $[\Rzero]\in \Hilbshortzero$,
        \item $[\Rzero]$ is singular,
    \end{itemize}
     as conditions on the form $F\in \cubicspacereg$.
    For a family $\Ffamily$ of forms parameterized by $T = \cubicspacereg$,
    constructed as in Proposition~\ref{ref:flatfamiliesforconstructible:prop}, we get a rank $120$ bundle $\Sym^2\Itwobundle$ with
    an evaluation morphism
    \begin{equation}\label{eq:eval}
        ev:\Sym^2\Itwobundle \to (\VV \Ffamily)^{\perp} \subset \Sym^4
        \VV^*\tensor_{\kk} \OT.
    \end{equation}
    The condition $(\Izero^2)^{\perp}_4 > 6$ is equivalent to degeneration of
    $ev$ on the fiber and thus it is \emph{divisorial} on $\cubicspacereg$.
    Let
    \[
        \Singdiv = (\det ev = 0) = \Sing \Hilbshortother \cap \exceptional.
    \]
    Recall that divisors on $\cubicspacereg$ and $\cubicspace$ are identified
    by restriction and closure. Since $\cubicspace$ is proper, we speak about
    the degree of $\Singdiv$ as the degree of its
    closure in $\cubicspace$.  We now check that $\Singdiv$  is prime of
    degree $10$.

    \begin{lemma}\label{ref:intersectiondegree:lem}
        Fix a basis $x_0, \ldots ,x_5$ of $\VV$ and let $F = x_0x_1x_3 -
        x_0\DPel{x_4}{2} + x_1\DPel{x_2}{2} + x_2x_4x_5 + x_3\DPel{x_5}{2}$. The line between $F$ and
        $\DPel{x_5}{3}$ intersects the divisor $\overline{\Singdiv} \subset
        \cubicspace$ in a finite scheme of degree
        $10$ supported at $\DPel{x_5}{3}$.
    \end{lemma}
    \def\fixed{fixed}%
    \begin{proof}
        \def\FF{\mathcal{F}}
        First, we check that every linear form is a partial of $F$, so that
        $F\in \cubicspacereg$. Second, consider a $14$-dimensional space
        \begin{align*}
            \fixed := \operatorname{span}(\alpha_0^2,\,
            \alpha_0\alpha_2,\, -\alpha_0\alpha_3+\alpha_2^2,\,
            &\alpha_0\alpha_4+\alpha_2\alpha_5,\,
            \alpha_0\alpha_5,\, \alpha_1^2,\, \alpha_1\alpha_2-\alpha_4\alpha_5,\,
            \alpha_1\alpha_3+\alpha_4^2,\\ &\alpha_1\alpha_4,\, \alpha_1\alpha_5,\,
            \alpha_2\alpha_3,\, \alpha_2\alpha_4-\alpha_3\alpha_5,\, \alpha_3^2,\,
        \alpha_3\alpha_4).
        \end{align*}
        Let $\FF = uF + v \DPel{x_5}{3}$. Then
        \[\fixed \oplus \kk\left(v \alpha_3\alpha_5 + u \alpha_0\alpha_1 - u \alpha_5^2\right)
        \subset \Ann(\FF{})_2\]
        and the equality holds for a general choice of $(u:v)\in \mathbb{P}^1$.
        One verifies that the determinant of $ev$ restricted to this line is equal, up to unit, to $u^{10}$.
        This can be conveniently checked near $\DPel{x_5}{3}$ by considering $\Sym^2 \Izero \to
        \Sym^4\VV^*/\VV \DPel{x_5}{3}$ and near $F$ by $\Sym^2 \Izero \to
        \Sym^4\VV^*/\VV x_3\DPel{x_5}{2}$. We note that the same equality holds in any
        characteristic other than $2,3$.
    \end{proof}

    \begin{proposition}\label{ref:divisorequality:prop}
        The divisor $\Singdiv = \Sing \Hilbshortother \cap \exceptional$ is
        prime of degree $10$.
        We have $\Singdiv = \Hilbshortzero\cap \exceptional = \DIR\cap
        \cubicspacereg$ as sets.
    \end{proposition}

    \begin{proof}
        \def\SL#1{\operatorname{SL}(#1)}
        Take two forms $[F_1], [F_2]\in\cubicspace$ and consider the
        intersection of $E$ with the line $\ell$ spanned by them.
        By Lemma~\ref{ref:intersectiondegree:lem} the restriction to $\ell$ of the evaluation morphism
        from~Equation \eqref{eq:eval} is finite of degree $10$.
        Hence also $\Singdiv$ is of degree $10$.

        Note that $E$ is
        $\SL{V}$-invariant. By a direct check, e.g.~conducted
        with the help of computer (e.g.~\cite{LiE}), we see that there are no $\SL{V}$-invariant
        polynomials in $\Sym^{\bullet}\Sym^3 \VV^*$ of degree less than ten.
        Therefore $E$ is prime.
        Smoothable schemes are singular, hence we have
        \[
            {\DIR} \subset \exceptional\cap \Hilbshortzero  \subset E.
        \]
        Since
        ${\DIR}$ is also a
        non-zero divisor, we have the equality of sets.
    \end{proof}

    \begin{remark}
        \def\II{\mathcal{I}}%
        In the proof of
        Lemma~\ref{ref:intersectiondegree:lem} we can
        avoid calculating the precise degree of the restriction of $ev$ to
        $\ell$, provided that we prove that $(\det ev = 0)_{\ell}$ is finite. Namely,
        let $\II \subset \Sym^{\bullet} V^{*}$
        be the relative apolar ideal sheaf on $\ell$. We look at $\II_2$. By
        the proof of Lemma~\ref{ref:intersectiondegree:lem} we have $\II_2  \simeq  \OO^{14} \oplus \OO(-1)$.
        Hence $\Sym^2 \II_2$ has determinant $\OO(-16)$.
        Similarly $V\Ffamily  \simeq \OO(-1)^{6}$, hence
        $(V\Ffamily)^{\perp} = \OO^{114} \oplus \OO(-1)^{6}$ and thus
        $\det ev_{|\ell}:\OO(-16)\to \OO(-6)$ is zero
        on a degree $10$ divisor. We conclude that $\Singdiv$ has degree $10$.
    \end{remark}

    \begin{proof}[Proof of Claim~\ref{claim:singularlocus}]
        Schemes corresponding to different elements in the fiber of $\pi$ are
        abstractly isomorphic. Therefore,
        we have $\Sing \Hilbshortother = \pi^{-1}\pi(\Sing \Hilbshortother)$
        and
        \[
            \Hilbshortzero\cap \Hilbshortother = \pi^{-1}\pi(\Hilbshortzero\cap \Hilbshortother).
        \]
        Hence the equality
        in Proposition~\ref{ref:divisorequality:prop} implies the equality in
        the Claim. The reducedness follows, because the scheme was defined via
        the closure: $\Hilbshortother = \overline{\Hilbshortother \setminus
        \Hilbshortzero}$. The last claim follows because $\DIR$ is prime and
        of codimension one thus its preimage under $\pi$ is also such.
    \end{proof}

    \begin{remark}
        If $[R]\in \Hilbshortother$ lies in $\Hilbshortzero$, then the tangent
        space to $\Hilbshort$ at $[R]$ has dimension at least $85 = 76 + 9$.
        This is explained geometrically by an elegant argument
        of \cite{Ranestad_Iliev__VSPcubic}, which we sketch below.
        Recall that we have an embedding $R \subset  \mathbb{A}^6 \cap
        \wedgetwo$.
        Define a rational map $\varphi\colon \mathbb{P}(\wedgetwo) \dashrightarrow
        \mathbb{P}(\wedgetwodual)$ as the composition $\wedgetwo \to \Lambda^4 \VV  \simeq
        \Lambda^2 \VV^*$ where the first map is $w\to w\wedge w$ and the
        second is an isomorphism coming from a choice of element of $\Lambda^6
        \VV$. Then $\varphi$ is defined by the $15$ quadrics vanishing on
        $\Gtwosix$ and it is birational with a natural inverse given by
        quadrics vanishing on $Gr(2, V^*)$.
        The $15$ quadrics in the ideal of $R$ are the restrictions of the $15$
        quadrics defining $\varphi$. Thus the map $\varphi\colon R\to \wedgetwodual$
        is defined by $15$ quadrics in the ideal of $R$.
        Since $\varphi^{-1}(\varphi(R))$ spans at most an $\mathbb{A}^6$ the
        coordinates of $\varphi^{-1}$ give $15 - 6 = 9$ quadratic relations between
        those quadrics. Therefore $\dim I_4 \leq \dim \Sym^2 I_2 - 9 \leq 126
        - 15$ and $r$ from the discussion above
        Lemma~\ref{ref:onlyquatricsmatter:lem} is at least $15$ so that the tangent space
        dimension is at least $85$.
    \end{remark}

    \begin{proof}[Proof of Theorem~\ref{ref:14pointsmain:thm}]

        Item~2.~By Claim~\ref{claim:singularlocus} the component $\Hilbshortother$ is
        reduced. Hence this part follows by Claim~\ref{claim:vectorbundle}.
        Item~3~is proved in Claim~\ref{claim:exceptional}.
        Then, $\Hilbshortother$ is smooth and connected, being a vector bundle
        over $\cubicspacereg$, hence Item~1 follows.
        Finally, Item~4~is proved in Claim~\ref{claim:singularlocus}.
    \end{proof}

    \begin{remark}\label{ref:uniquesmall:rmk}
        \def\ZZ{\mathcal{Z}}%
        The dimension of $\Hilbshortother$ equal to $76$ is smaller than the
        dimension of $\Hilbshortzero$ equal to $14\cdot 6 = 84$. This is the
        only known example of a component $\ZZ$ of the Gorenstein locus of
        Hilbert scheme of $d$ points on $\mathbb{A}^n$ such
        that $\dim \ZZ \leq dn$ and points of $\ZZ$ correspond to \emph{irreducible}
        subschemes. It is an interesting question (a special case of
        Problem~\ref{prob:smallcomponents}) whether other examples
        exist.
    \end{remark}

    \chapter{Small punctual Hilbert schemes}\label{sec:smallpunctual}

    In this section we fix a $\kk$-rational point $p\in \AA^n$, for example
    the origin, and consider the locus
    \begin{equation}\label{eq:punctual}
        \HilbGorpp{\AA^n} := \left\{ R \subset \AA^n\ |\ \Supp R = \{p\},
        R\mbox{ Gorenstein} \right\} \subset \HilbGorr{\AA^n}
    \end{equation}
    As explained in Section~\ref{ssec:introkregularity}, it is important for
    applications to
    give an upper bound for the dimension of $\HilbGorpp{\AA^n}$.
    In Proposition~\ref{ref:alignabledim:prop} we prove a lower bound
    $(r-1)(n-1)$ for this dimension. Theorem~\ref{ref:expecteddim:thm} below
    implies that this lower bound is attained for small degrees and $\chark = 0$. We
    follow~\cite{Michalek}.

    \begin{theorem}[The Hilbert scheme has expected
        dimension]\label{ref:expecteddim:thm}
        Let $\kk$ be a field of characteristic zero.  For $r\leq 9$ the
        dimension of $\HilbGorpp{\AA^n}$ is equal to $(r-1)(n-1)$.
    \end{theorem}

    Before we prove Theorem~\ref{ref:expecteddim:thm}, we explain the lower
    bound. For Hilbert schemes of points, we strived to present each given
    algebra as a limit of smooth algebras. Inside $\HilbGorpp{\AA^n}$ we do
    not have any smooth algebras (unless $r = 1$), so we need an equivalent.
    It is given by aligned (curvilinear) schemes, studied for example in~\cite{ia_deformations_of_CI}.
    \begin{definition}
        A finite local algebra $\DA$ is \emph{aligned} (or \emph{curvilinear}) if it is
        isomorphic to $\kk[\Dx]/\Dx^r$ for some $r$.
        A finite local algebra $\DA$ is \emph{alignable} if it is a finite flat
        limit of aligned algebras, i.e., there exists a flat family of algebras with
        fiber $A$ and general fiber aligned.
    \end{definition}

    \newcommand{\familyR}{\mathcal{Z}}
    The following result gives a lower bound for $\dim \HilbGorpp{\AA^n}$.
    \begin{proposition}[dimension of aligned
        schemes]\label{ref:alignabledim:prop}
        The locus of points of $\HilbGorpp{\AA^n}$ corresponding to aligned schemes
        has dimension $(r-1)(n-1)$.  Therefore, also the locus of alignable
        schemes has dimension $(r-1)(n-1)$.
    \end{proposition}
    We prove this proposition as a consequence of more general analysis of
    reembeddings. Let $\familyR \subseteq \HilbGorpp{\AA^n}$.
    We say that $\familyR$ is \emph{closed under isomorphisms} if every
    subscheme from $\HilbGorpp{\AA^n}$ isomorphic to a member of $\familyR$ belongs to $\familyR$.
    In other words, if $i, i' \colon R \hookrightarrow \AA^n$ are two
    embeddings of a finite scheme $R$ with support at $p$ and $i(R)$ is in the
    family $\familyR$, then also $i'(R)$ is in $\familyR$.

    We now perform a dimension count to see, how does $\dim \familyR$ behave
    under change of ambient from $\AA^n$ to $\AA^m$.
    While the underlying idea is easy, it is technically suitable to use
    advanced devises: flag Hilbert schemes (see~\cite[IX.7,
    p.~48]{Arbarello__Geometry_of_algebraic_curves} or
    \cite[Section 4.5]{Sernesi__Deformations}) and multigraded Hilbert schemes
    (see~\cite{Haiman_Sturmfels__multigraded}).
    Note, that we only consider these constructions for a scheme \emph{finite}
    over $\kk$, a very special case where the existence is almost obvious,
    compare Proposition~\ref{ref:connectedforfinite:prop}.
\def\familyRe{\mathcal{R}^n}
\def\familyRm{\mathcal{R}^m}
\def\familyRem{\mathcal{R}^{n,m}}
\begin{proposition}[invariance of codimension]\label{ref:codiminv}
    Let $m\geq n$ and $\familyRe \subseteq \HilbGorpp{\AA^n}$ be a
    Zariski-constructible subset closed under isomorphisms.
    Consider the subset $\familyRm \subseteq \HilbGorpp{\AA^m}$ consisting of all
    schemes isomorphic to an element of $\familyRe$.
    Then the family $\familyRm$ is a Zariski-constructible subset of
    $\HilbGorpp{\AA^m}$ and its dimension satisfies
    \[
        (r - 1)m - \dim \familyRm = (r-1)n - \dim \familyRe.
    \]
\end{proposition}

\newcommand{\dissk}[1]{\mathbb{D}^{#1}}%
\newcommand{\dissktrimmed}[1]{\underline{\mathbb{D}}^{#1}}%
\begin{proof}
    For technical reasons (to assure the existence of Hilbert flag schemes) we
    consider
    \[
        \dissktrimmed{n} := \Spec \kk[[\alpha_1, \ldots ,\alpha_n]]/(\alpha_1,
        \ldots ,\alpha_n)^{r}
    \]
    rather than $\Spec \kk[[\alpha_1, \ldots ,\alpha_n]]$. Let $R$ be a finite scheme of
    degree $r$. To give an embedding $R \subseteq \AA^n$ with support $p$
    is the same as to give as
    to give an embedding $R \subseteq \dissktrimmed{n}$. Therefore
    $\Hilbr{\AA^n, p}  \simeq \Hilbr{\dissktrimmed{n}}$ for every $n$ and
    $p$.

    \def\HilbFlag{\operatorname{HilbFlag}}%
    \def\Hilbdiske{\Hilbr{\dissktrimmed{n}}}%
    \def\Hilbdiskindisk{\Hilbr{\dissktrimmed{n} \subseteq \dissktrimmed{m}}}%
    \def\Hilbdiskn{\Hilbr{\dissktrimmed{m}}}%
    Consider the multigraded flag Hilbert scheme $\HilbFlag$
    parameterizing pairs of closed immersions $R \subseteq \dissktrimmed{n} \subseteq
    \dissktrimmed{m}$. It has natural
    projections $\pi_1$, $\pi_2$, mapping $R \subseteq \dissktrimmed{n} \subseteq
    \dissktrimmed{m}$ to $\dissktrimmed{n} \subseteq \dissktrimmed{m}$ and $R \subseteq
    \dissktrimmed{m}$ respectively, see diagram below.
    \[
        \begin{tikzcd}
            {} & \arrow{ld}{\pi_1}\HilbFlag{} \arrow{rd}{\pi_2} & \\
            \Hilbdiskindisk & & \Hilbdiskn
        \end{tikzcd}
    \]
    Note that $\pi_i$ are proper.
    We will now prove that $\familyRm$ is Zariski-constructible.
    Consider the automorphism group $G$ of $\dissktrimmed{m}$. It acts
    naturally on $\Hilbdiskindisk$  and $\HilbFlag$, making the morphism $\pi_1$
    equivariant. 
    The ideal of a $\dissktrimmed{n} \subseteq \dissktrimmed{m}$ is given by
    $m-n$ order one elements of the power series ring, linearly independent modulo
    higher order operators.  Therefore the action
    of $G$ on $\Hilbdiskindisk$ is transitive: for any two $(m-n)$-tuples as
    above there exists an automorphism of $\dissktrimmed{m}$
    mapping elements of the first tuple to the elements of the other tuple.

    Fix an embedding $\dissktrimmed{n} \subseteq
    \dissktrimmed{m}$, and hence an inclusion $i\colon\familyRe \hookrightarrow
    \HilbFlag$. Let $\familyRem = G\cdot i(\familyRe)$, then $\familyRm =
    \pi_2(\familyRem)$ and so it is Zariski-constructible. Note that
    $\familyRem = \pi_2^{-1}(\familyRm)$.

    It remains to compute the dimension of $\familyRm$. Let us redraw the
    previous diagram:
    \[
        \begin{tikzcd}
            {} & \arrow{ld}{\pi_1|_{\familyRem}}\familyRem \arrow{rd}{\pi_2|_{\familyRem}} & \\
            \Hilbdiskindisk & & \familyRm
        \end{tikzcd}
    \]
    Note that $\pi_1|_{\familyRem}$ is surjective because $\pi_1$ is $G$-equivariant 
       and $G$ acts transitively on the scheme $\Hilbdiskindisk$. 
    Furthermore, $\pi_1|_{\familyRem}$ has
        fibers isomorphic to $\familyRe$ because $\familyRe$ is closed under isomorphisms.
    Thus we obtain 
    \[
        \dim \familyRem = \dim \familyRe + \dim \Hilbdiskindisk.
    \]
    It remains to calculate the dimensions of $\Hilbdiskindisk$ and the fiber
    of $\pi_2$.
    An immersion $\varphi\colon \dissktrimmed{n} \subseteq
    \dissktrimmed{m}$ corresponds to a surjection
    \[
        \varphi^{*}:\kk[[\alpha_1,
        \ldots ,\alpha_m]]/(\alpha_1, \ldots ,\alpha_m)^{r}\to \kk[[\beta_1, \ldots ,\beta_n]]/(\beta_1,
        \ldots ,\beta_n)^{r}.
    \]
    Such surjective morphisms are parameterized
    by the images of generators $\varphi^{*}(\alpha_1), \ldots, \varphi^{*}(\alpha_m)$ in the maximal ideal $(\beta_1, \ldots, \beta_n)$.
    In fact, a general choice of those images gives a surjection. Let $M =
    \dim_{\kk} (\beta_1, \ldots, \beta_n)$ be the dimension of the \emph{ideal}
    $(\beta_1, \ldots , \beta_n)$. Then we have $mM$ parameters for the choice of
    $\varphi^{*}(\alpha_1), \ldots ,\varphi^{*}(\alpha_m)$. Two choices are equivalent if they
    have the same kernel, so that they differ by an automorphism of
    $\kk[[\beta_1, \ldots ,\beta_n]]/(\beta_1, \ldots ,\beta_n)^{r}$. This automorphisms
    group is $nM$ dimensional, thus $\dim \Hilbdiskindisk = mM - nM = (m-n)M$.

    Similarly, we may consider the fiber $\pi_2^{-1}(R)$ over a point $R\in
    \familyRm$ corresponding to a subscheme $R \subseteq \dissktrimmed{m}$. As above,
    the possible $\dissktrimmed{n}\subseteq \dissktrimmed{m}$ are
    parameterized by fixing the images of $\varphi^{*}(\alpha_1), \ldots ,\varphi^{*}(\alpha_{m})$
    in $\kk[[\beta_1, \ldots ,\beta_n]]/(\beta_1, \ldots ,\beta_n)^{r}$. The difference is
    that we have to ensure $R \subseteq \dissktrimmed{n}$. Algebraically, the
    images $\varphi^{*}(\alpha_1), \ldots , \varphi^{*}(\alpha_m)$ need to lie in the ideal
    $I(R) \subseteq (\beta_1, \ldots , \beta_n)$. Since $\dim_{\kk} I(R) = M-(r-1)$, the
    fiber has dimension $(m-n)(M - r +
    1)$.

    In particular, $\pi_2$ is equidimensional, so that the
    dimension of $\familyRm$  is given by the formula:
    \[
        \dim \familyRm = \dim \familyRe + (m-n)M - (m-n)(M-r+1) = \dim \familyRe + (m-n)(r-1).\qedhere
    \]
\end{proof}

\begin{proof}[Proof of Proposition~\ref{ref:alignabledim:prop}]
        If $n = 1$, then there is a unique closed subscheme of
        $\AA^n$ isomorphic to $\Spec \kk[\Dx]/\Dx^r$ and supported at $p$, thus the
        dimension is $0$ and the claim is satisfied.

        Now let $n$ be arbitrary.  By Proposition~\ref{ref:codiminv} the
        dimension $d$ from the statement satisfies $(r-1)n - d = (r-1)$,
        thus $d = (r-1)(n-1)$.
\end{proof}

\newcommand{\HilbGP}{\HilbGorpp{\AA^n}}%
Having proved the lower bound from Theorem~\ref{ref:expecteddim:thm}, we
proceed to prove that it is equal to the upper bound for degree up to nine.
\begin{definition}\label{ref:expecteddim:def}
    The \emph{expected dimension} of $\HilbGorpp{\AA^n}$ is the dimension of the family of alignable subschemes,
    i.e.~$(r-1)(n-1)$, see Proposition~\ref{ref:alignabledim:prop}.

    If a Zariski-constructible subset of the punctual Hilbert scheme has
    dimension less or equal than $(r-1)(n-1)$, then we call it
    \emph{negligible}. In particular, the set of alignable subschemes is
    negligible.
\end{definition}

The name negligible was tailored for the purposes of~\cite{Michalek}
and it is not standard.
    \begin{remark}\label{ref:invariance_of_negligible}
        Let $r$, $n$ be such that the Gorenstein punctual Hilbert scheme
        $\HilbGorpp{\AA^n}$ has expected dimension.
        Fix any $m\geq n$.
        Then the set of Gorenstein schemes in $\AA^m$ of degree
        $r$ and embedding dimension at most $n$ is negligible.
        Indeed, Proposition~\ref{ref:codiminv}
        implies that for every $m \geq n$ the set of schemes
        in $\HilbGorpp{\AA^m}$ with embedding dimension $n$ has dimension
        $(r-1)m - (r-1)n + (r-1)(n-1) = (r-1)(m-1)$.
    \end{remark}

    \begin{remark}
        Gorenstein schemes of degree at
        most $9$ are all (for $\chark\neq 2, 3$) smoothable by Theorem~\ref{ref:cjnmain:thm}. Since
        ultimately we want to analyse those, we will not emphasise
        smoothability. However, the reader should note that
        Definition~\ref{ref:expecteddim:def} is reasonable only with the
        smoothability assumption.
    \end{remark}

    In general, there exist Gorenstein non-alignable subschemes and non-negligible
    families, see Example~\ref{ref:1551:example} below. In the remaining part
    of this section we show that if the degree is small enough, then all
    subschemes are negligible, thus $\HilbGP$ has the expected dimension.

    We begin with the following result of Brian\c{c}on:
    \begin{theorem}[{\cite[Theorem V.3.2, p.~87]{briancon}}]\label{ref:briancon:obs}
        Let $\kk = \CC$ and $R \subset \AA^2$ be a finite local scheme (not necessarily
        Gorenstein). Then $R$ is alignable.
    \end{theorem}

    \begin{corollary}\label{ref:embeddingdimtwo:cor}
        Let $\chark = 0$. Let $\mathcal{R} \subset \HilbGorpp{\AA^n}$ be the family of local
        schemes with embedding dimension (see
        Section~\ref{ssec:hilbertfunctionGeneral} for definition) at most two.
        Then $\mathcal{R}$ is negligible.
    \end{corollary}
    \begin{proof}
        The schemes of embedding dimension at most two are precisely those
        embeddable in $\AA^2$, by Lemma~\ref{ref:embeddingdimension:lem}.
        By Remark~\ref{ref:invariance_of_negligible} it is enough to prove the claim for
        $n = 2$. For $\kk = \CC$ the family $\mathcal{R}$ is contained in
        the alignable locus and the claim follows from
        Theorem~\ref{ref:briancon:obs}. For $\kk = \mathbb{Q}$ the claim follows by
        base change, since the dimension is invariant under field
        extension (see, e.g., \cite[Chapter~8]{Eisenbud})
        and
        \[
            \HilbGorpp{\AA^n_{\mathbb{C}}} = \HilbGorpp{\AA^n_{\mathbb{Q}}}
            \times_{\Spec\mathbb{Q}} \Spec\mathbb{C}.
        \]
        For $\kk$ arbitrary of characteristic zero, the claim follows again by
        base change:
        \[
            \HilbGorpp{\AA^n_{\kk}} = \HilbGorpp{\AA^n_{\mathbb{Q}}}
            \times_{\Spec\mathbb{Q}}
            \Spec\kk.\qedhere
        \]
    \end{proof}
    Analogues of Brian\c{c}on results are false for higher embedding
    dimensions, see~\cite{ia_deformations_of_CI} and
    Example~\ref{ref:1551:examplebis}.  To analyse schemes of
    embedding dimension greater than two, we need a few results from the
    theory of finite Gorenstein algebras proved in Part~\ref{part:algebras},
    in particular Macaulay's Theorem for
    Gorenstein algebras, see Theorem~\ref{ref:MacaulaytheoremGorenstein:thm}.

Every finite local Gorenstein algebra $\DA$ of socle degree one is an apolar algebra of
a linear form, thus it is aligned. Therefore the set of socle degree one
algebras is negligible. In the following Lemma~\ref{ref:negligequadrics:lem}
we extend this result to socle degree two.
\begin{lemma}\label{ref:negligequadrics:lem}
    The set $\mathcal{H} \subseteq \HilbGorpp{\AA^n}$
    consisting of finite local Gorenstein algebras of socle degree two is
    negligible.
\end{lemma}
\begin{proof}
    All members of $\mathcal{H}$ have degree $r$ and socle degree $2$, hence
    their Hilbert function is equal to $(1, r-2, 1)$.
    In particular, $r-2\le n$, so using Proposition~\ref{ref:codiminv} similarly as in
    Remark~\ref{ref:invariance_of_negligible} we may assume $r-2 =n$.
    Algebras from $\mathcal{H}$ have degree $n+2$ and are parameterized by a set of dimension
    $\binom{n+2}{2} - (n+2)$, compare~\ref{ex:1nn1largenonsmoothable}. Therefore, $\mathcal{H}$ is negligible if
    \[
         \binom{n+2}{2} - (n+2) \leq (n+1)(n-1),
    \]
    which is true for every $n\geq 1$.
\end{proof}

Let $A$ be an algebra of socle degree $d\geq 3$.  Recall from
Section~\ref{ssec:hilbertfunctionGorenstein} the symmetric
decomposition $\Dhdvect{A}$ of the Hilbert function of $\DA$.
In particular $\Dhd{A}{d-2} = (0, q, 0)$, for some $q$ where $\Dhdvect{A}$.
In the following we investigate the case when $q> 0$; we perform a
remove-the-quadric-part trick, already used e.g.~in
Corollary~\ref{ref:squareshavedegenerations:cor}.
\begin{lemma}\label{ref:negligesquares:lem}
    Let $\chark \neq 2$.
    Consider the set $\mathcal{H}(q) \subseteq \HilbGorr{\AA^n}$
    consisting of finite local Gorenstein algebras of any socle degree $d\geq
    3$ satisfying $\Delta_{d-2} = (0, q, 0)$.  If all Gorenstein schemes of
    degree $d - q$ and embedding dimension at most $n$ are negligible, then
    $\mathcal{H}(q)$ is negligible.
\end{lemma}

\begin{proof}
    We argue by induction on $q$. In the base case $q = 0$ there is nothing to
    prove. By a base change (Proposition~\ref{ref:Gorensteinbasechange:prop}), we assume $\kk = \kkbar$.

    Take any scheme $\Spec A \in \mathcal{H}(q)$.
    By Proposition~\ref{ref:squares:prop} the algebra $\DA$ is isomorphic to the apolar algebra of 
    $f + \DPel{x_1}{2} + \ldots  + \DPel{x_q}{2}$, where $f$ is a polynomial in
    variables different from $x_1, \ldots ,x_q$. By
    Corollary~\ref{ref:squareshavedegenerations:cor}
    this algebra is an embedded limit of algebras of the form $B \times \kk$, where $B$
    has the same socle degree as $A$ and satisfies $\Delta_{B, d-2} = (0,
    q-1, 0)$.
    By induction, the set of schemes corresponding to such $B$ is negligible,
    i.e.~has dimension at most $((r-1)-1)(n-1)$.
    Therefore the set of schemes corresponding to $B\times \kk$ has dimension
    at most $(r-2)(n-1) + n = (r-1)(n-1) + 1$. Since $\mathcal{H}(q)$ lies on
    the border of this set, $\dim \mathcal{H}(q) \leq (r-1)(n-1) + 1 - 1 =
    (r-1)(n-1)$.
\end{proof}

\begin{lemma}\label{ref:neglige4:lem}
    Let $\chark = 0$ and $r\leq 10$. Let $\mathcal{R} \subseteq \HilbGorpp{\AA^n}$ be the
    subset of schemes corresponding to finite local Gorenstein algebras of socle degree
    at most four. Then $\mathcal{R}$ is negligible.
\end{lemma}

\begin{proof}
    By base change (Proposition~\ref{ref:Gorensteinbasechange:prop}), we reduce to the case $\kk = \kkbar$.
    The family $\familyR$ divides into finitely many families according to
    the Hilbert function and its symmetric decomposition. Therefore we may assume
    these are fixed in $\familyR$. Thus we may speak about the Hilbert
    function, the socle degree etc.

    We begin with a series of reductions.
    By induction and Remark~\ref{ref:invariance_of_negligible}, we may assume that the claim is true for schemes with
    embedding dimension less that $n$.
    Let $d$ be the socle degree of any member of $\familyR$. By
    Lemma~\ref{ref:negligequadrics:lem} we may assume that $d\geq 3$.
    By Lemma~\ref{ref:negligesquares:lem}, we may assume that $\Delta_{d-2}
    = (0, 0, 0)$.
    If $d = 3$, elements of $\familyR$ are parameterized by a set
    of dimension $\binom{n+3}{3} - (2n+2)$, compare
    Example~\ref{ex:1nn1largenonsmoothable}, and $10 \geq r = 2n+2$
    by Example~\ref{ex:hilbertfunctioncubics}, so
    $n\leq 4$. Then we
    need to check that $(r-1)(n-1) = (2n+1)(n-1) \geq \binom{n+3}{3} - (2n+2)$ for all
    $n\leq 4$.

    Similarly, if $d = 4$, then the Hilbert function
    has decomposition of the form $(1, a, b, a, 1) + (0, c, c, 0)$, where $a,
    b > 0$.
    We see that $r = 2 + 2a + 2c + b \leq 10$, $n = H(1) = a + c$.
    Moreover $b\leq \binom{a+1}{2}$ and from the Macaulay's Growth
    Theorem~\ref{ref:MacaulayGrowth:thm} and
    Lemma~\ref{ref:sumsofdhdaraOseqence:lem} it
    follows that either $b > 2$ or $a = b = 1$ or $a = b = 2$.

    Such algebras are parameterized by
    \begin{itemize}
      \item the choice of a quartic in $a$ variables, which gives at most dimension $\binom{a+3}{4}$,
      \item the choice of these $a$ variables out of the linear space of $a+c$ variables, which gives at most $ac$,
      \item and a choice of polynomial of degree $3$ in $a+c$ variables: $\binom{a+c+3}{3}$,
      \item minus the degree: $2 + 2a + 2c + b$, see
          Proposition~\ref{ref:apolarrepresentability:prop}.
    \end{itemize}
    Finally we get a parameter set of dimension at most
    \begin{equation}\label{eq:dimforfour}
        \binom{a+3}{4} + ac + \binom{a+c+3}{3} - (2 + 2a + 2c + b)
    \end{equation}
    Now one needs to the check that for all $a, b, c$ such that $2 + 2a +
    2c + b \leq 10$ satisfying the constraints above, the number
    \eqref{eq:dimforfour} is not higher than $(r - 1)(n-1)$.
\end{proof}

\begin{lemma}\label{ref:neglige:lem}
    Let $\chark = 0$ and $r \leq 9$. Then the whole Gorenstein punctual Hilbert scheme
      \[
          \HilbGorpp{\AA^n}
      \]
      is  negligible.
\end{lemma}

\begin{proof}
    Let $\familyR := \HilbGorpp{\AA^n}$.
    As before, we may fix a Hilbert function $H$ with symmetric
    decomposition $\Delta$, and a socle degree $d$.
    By Corollary~\ref{ref:embeddingdimtwo:cor} we may assume that the embedding
    dimension is at least three.
    By Lemma~\ref{ref:negligesquares:lem} we may assume that
    $\Delta_{d-2} = (0, 0, 0)$.
    By Lemma~\ref{ref:neglige4:lem} we may assume $d > 4$.

    We will prove that no decomposition $\Delta$ satisfying all above constrains exists.

    Let $e_i := \Delta_{A, i}(1)$. Then $H(1) = \sum e_i$.
    Note that $\Delta_{0} = (1, e_1,  \ldots , e_1, 1)$ is a vector of
    degree $d+1\geq 6$, thus its sum is at least $4 + 2e_1$.
    Note that by symmetry of $\Delta_{i}$ we have $e_i =
    \Delta_{i}(d-i-1)$ and since $s-i-1 > 1$, we have $\sum_j
    \Delta_{i}(j) \geq 2e_i$. Summing up
    \[r = \sum H = \sum_i
    \sum_j \Delta_{i}(j)\geq 4 + 2\sum_i e_i \geq 4 + 2\cdot 3 = 10.\]
    This contradicts the assumption $r\leq 9$.
\end{proof}

We now conclude the proof of our main theorem.
\begin{proof}[Proof of Theorem~\ref{ref:expecteddim:thm}]
    The lower bound follows from Proposition~\ref{ref:alignabledim:prop} and
    the upper bound from Lemma~\ref{ref:neglige:lem}.
\end{proof}

\begin{example}\label{ref:1551:examplebis}
    The dimension of the locus of alignable subschemes in
    $\HilbGorpuncarged{12}{\AA^5, p}$ is
    $(12-1)(5-1) = 44$. This locus is Zariski-irreducible and its general member is, by definition,
    isomorphic to $\Spec \kk[\Dx]/\Dx^{12}$. The subset $\familyR$ of
    $\HilbGorpuncarged{12}{\AA^5, p}$ parameterizing subschemes with
    Hilbert function $(1, 5, 5, 1)$ has dimension $\binom{5+3}{3} - 12 = 44$,
    see Example~\ref{ex:1nn1largenonsmoothable}, thus $\familyR$ is not contained
    in the locus of alignable algebras, i.e.~a general subscheme with Hilbert
    function $(1, 5, 5, 1)$ is not alignable. The subschemes in $\familyR$ are
    smoothable by Theorem~\ref{ref:cjnmain:thm}.
\end{example}

\begin{example}
    As in Example~\ref{ref:1551:examplebis}, by dimension count we see that a general
    irreducible subscheme of $\mathbb{A}^7$ with Hilbert function $(1, 7, 7, 1)$ is
    non-alignable. Such subschemes are smoothable, see Remark~\ref{ref:1nn1:rmk}.
\end{example}

\small
\newcommand{\etalchar}[1]{$^{#1}$}

\end{document}